\documentclass[11pt,twoside]{report}
\usepackage{fancyheadings}
\usepackage{mathscinet}
\usepackage{graphicx}
\usepackage{bm}	
\usepackage{amsmath}
\usepackage{amsthm}	
\usepackage{amsfonts}	
\usepackage{amssymb}
\usepackage{mathrsfs}
\usepackage{latexsym}
\usepackage{enumerate}
\usepackage{dsfont}
\usepackage{amssymb}
\usepackage{relsize}
\usepackage{yhmath}
\usepackage{mathtools}
\usepackage{array}
\usepackage{parskip}
\usepackage[all, knot]{xy}
\usepackage[super]{nth}
\input xy 
\xyoption{all}
\xyoption{2cell} 
\xyoption{v2}
\xyoption{arc}

\DeclareMathOperator{\Lie}{Lie}
\DeclareMathOperator{\End}{End}
\DeclareMathOperator{\tr}{tr}

\DeclareMathOperator{\Bun}{{Bun}}
\DeclareMathOperator{\frBun}{{frBun}}
\DeclareMathOperator{\vect}{{Vect}}

\DeclareMathOperator{\frvect}{{frVect}}

\DeclareMathOperator{\id}{id}
\DeclareMathOperator{\hol}{hol}
\DeclareMathOperator{\Map}{Map}
\DeclareMathOperator{\Mat}{Mat}

\DeclareMathOperator{\ad}{ad}

\DeclareMathOperator{\ev}{ev}
\DeclareMathOperator{\pr}{pr}

\DeclareMathOperator{\Gr}{Gr}
\DeclareMathOperator{\Hom}{Hom}
\DeclareMathOperator{\struct}{Struct_{\Omega}}
\DeclareMathOperator{\struuct}{Struct}

\DeclareMathOperator{\Man}{Man}
\DeclareMathOperator{\gagrp}{AbGrp^\bullet}

\DeclareMathOperator{\rank}{rank}
\DeclareMathOperator{\bimon}{BimonCat}
\DeclareMathOperator{\rig}{SemiRing}

\theoremstyle{plain}
\newtheorem{theorem}{Theorem}[section]
\newtheorem{corollary}[theorem]{Corollary}
\newtheorem{lemma}[theorem]{Lemma}
\newtheorem{proposition}[theorem]{Proposition}

\theoremstyle{definition}
\newtheorem{definition}[theorem]{Definition}

\theoremstyle{remark}
\newtheorem{example}[theorem]{Example}

\newtheorem{remark}[theorem]{Remark}

\newcommand{\cC}{{\mathcal C}}
\renewcommand{\cH}{{\mathcal H}}

\newcommand{\cS}{{\mathcal S}}
\newcommand{\cA}{{\mathcal A}}

\newcommand{\cF}{{\mathcal F}}
\newcommand{\cK}{\mathcal{K}^{-1}}
\newcommand{\cKr}{{\mathcal{K}}_{\RR}^{-1}}
\renewcommand{\cL}{{\mathcal L}}

\renewcommand{\cL}{{\mathcal L}}

\newcommand{\cV}{{\mathcal V}}
\newcommand{\cM}{{\mathcal M}}
\renewcommand{\cR}{{\mathcal R}}
\newcommand{\CC}{{\mathbb C}}

\newcommand{\RR}{{\mathbb R}}
\newcommand{\ZZ}{{\mathbb Z}}
\newcommand{\QQ}{{\mathbb Q}}

\renewcommand{\d}{\delta}
\newcommand{\g}{\mathfrak{g}}

\renewcommand{\u}{\mathfrak{u}}
\newcommand{\sQ}{\mathsf{Q}}
\newcommand{\sA}{\mathsf{A}}
\newcommand{\sP}{\mathsf{P}}
\newcommand{\sF}{\mathsf{F}}
\newcommand{\sE}{\mathsf{E}}
\newcommand{\sR}{\mathsf{R}}
\newcommand{\sB}{\mathsf{B}}
\newcommand{\sH}{\mathsf{H}}

\newcommand{\sh}{\mathsf{h}}
\newcommand{\dK}{\check{\mathcal{K}}^{-1}}
\newcommand{\dl}{\check{\mathcal{L}}^{-1}}

\newcommand{\dL}{\check{\mathcal{L}}'^{-1}}
\newcommand{\edK}{\check{\mathcal{K}}^{0}}

\newcommand{\x}{\times}
\newcommand{\lan}{\langle}
\newcommand{\ran}{\rangle}
\newcommand{\llan}{\lan\mspace{-5mu}\lan}
\newcommand{\rran}{\ran\mspace{-5mu}\ran}

\newcommand{\sone}{{S^1}}
\newcommand{\lo}{\longrightarrow}
\newcommand{\ovect}{\Omega\!\vect}
\newcommand{\ohect}{\Omega\!\vect_{\mathrm{H}}}
\newcommand{\frhect}{\frvect_{\mathrm{H}}}
\newcommand{\oast}{{\mspace{4mu}\odot\mspace{-11.5mu}\star\mspace{6mu}}} %used to be \circledast
\newcommand{\bp}{{x_0}}

\renewcommand{\d}{\partial}
\newcommand{\vp}{{\vphantom{-1}}}

\newcommand{\simto}{\xrightarrow{\;\;\sim\;\;}}
\newcommand{\im}{\mathrm{im}}

\newcommand{\CS}{{C \mspace{-1.9mu}S}}
\newcommand{\TWZ}{{\text{TWZ}}}
\newcommand{\Ch}{{C \mspace{-1.9mu}h}}
\newcommand{\cw}{c \mspace{-1mu}w}

\newcommand{\I}{{[0,1]}}
\newcommand{\J}{{[0,\tfrac{\pi}{2}]}}
\newcommand{\deR}{{\mathrm{deR}}}

\newcommand{\dlim}{\varinjlim}
\newcommand{\sBun}{{S^1\mspace{-4mu}\Bun}}

\newcommand{\uc}{{A}} %UNIVERSAL CONNECTION
\newcommand{\uf}{{F}} %UNIVERSAL CURVATURE
\newcommand{\hc}{{\widehat{A}}} %HOLONOMY PULLBACK OF UNIVERSAL CONNECTION
\newcommand{\hf}{{\widehat F}} %HOLONOMY PULLBACK OF UNIVERSAL CURVATURE

\newcommand{\dtz}[1]{\frac{d}{dt} \Big\{ {#1} \Big\}_{t=0}}

\renewcommand{\theequation}{\arabic{chapter}.\arabic{section}.\arabic{equation}}

\begin{document}

% Change the page numbering system for the preamble
\pagenumbering{roman}
 \addtolength{\headwidth}{\marginparsep}
 \addtolength{\headwidth}{\marginparwidth}

% Import the title page
\begin{titlepage}
%
% Add some extra space at the top, because the pages are so tall:
%
\vspace*{\fill}

%
% Do the main title in a LARGE bold centred font:
%
\begin{huge}
\begin{bf}
\begin{center}
The Caloron Correspondence and \\
\vspace{1mm}
Odd
 Differential K-theory\end{center}
\end{bf}
\end{huge}
%
% Leave a gap between the title and the author:
%
\vspace*{\fill}
%%\vspace{24.0mm}
%
% Do the author in a large centred font:
%
\begin{Large}
\begin{center}
Vincent S. Schlegel
\end{center}
\end{Large}

\begin{large}
%
% Leave a gap between the author and the description:
%
\vspace*{\fill}

%
% Do the description in a large italic centred font:
%
{{
\begin{center}
\begin{normalsize}
Thesis
submitted in partial fulfilment\\ of the requirements for the degree of\\
\vspace{1mm}
\end{normalsize}
\emph{
Master of Philosophy\\}
\begin{normalsize}
\vspace{1mm}
in\\
\vspace{1mm}
\end{normalsize}
\emph{
Pure Mathematics\\}
\begin{normalsize}
\vspace{1mm}
at
The University of Adelaide
\end{normalsize}
\end{center}
}}
%
% Leave a gap between the description and the department:
%
%
% Do the department in a large centred font:
%
%
% Leave a gap between the department and the crest:
%
\vspace*{\fill}
\vspace*{\fill}
\begin{center} % place this in the centre
\includegraphics[height=4cm]{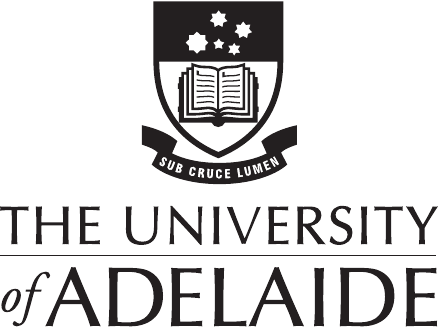}
\end{center}
\vspace{3mm}
\begin{center}

\begin{bf}
School of Mathematical Sciences
\end{bf}

\end{center}
%
% Leave a gap between the crest and the date:
%
\vspace{6.0mm}
%
% Do today's date in a large centred font:
%
\begin{center}
July 2013
\end{center}
\end{large}
%
% Add some space for luck?:
%
\vspace*{\fill}
\end{titlepage}

\newpage
\thispagestyle{empty}
\mbox{}

% Import the various things

\pagestyle{empty}
\addtocontents{toc}{\protect\thispagestyle{empty}}
\tableofcontents
\cleardoublepage
\setcounter{page}{1}
\chapter*{Abstract}
\typeout{Abstract}
\label{ch:abstract}
\addcontentsline{toc}{chapter}{Abstract}

The caloron correspondence (introduced in \cite{MS} and generalised in \cite{HMV,MV,V}) is a tool that gives an equivalence between principal $G$-bundles based over the manifold $M\x\sone$ and principal $LG$-bundles on $M$, where $LG$ is the Fr\'{e}chet Lie group of smooth loops in the Lie group $G$.
This thesis uses the caloron correspondence to construct certain differential forms called \emph{string potentials} that play the same role as Chern-Simons forms for loop group bundles.
Following their construction, the string potentials are used to define degree $1$ differential characteristic classes for $\Omega U(n)$-bundles.

The notion of an \emph{$\Omega$ vector bundle} is introduced and a caloron correspondence is developed for these objects.
Finally, string potentials and $\Omega$ vector bundles are used to define an $\Omega$ bundle version of the structured vector bundles of \cite{SSvec}.
The \emph{$\Omega$ model} of odd differential $K$-theory is constructed using these objects and an elementary differential extension of odd $K$-theory appearing in \cite{TWZ}.

\newpage
\thispagestyle{plain}
\mbox{}

\chapter*{Signed statement}
%\label{ch: signed statement}
\addcontentsline{toc}{chapter}{Signed statement}
I certify that this work contains no material which has been accepted for the award of any other 
degree or diploma in any university or other tertiary institution and, to the best of my knowledge and
belief, contains no material previously published or written by another person, except where due
reference has been made in the text. In addition, I certify that no part of this work will, in the future,
be used in a submission for any other degree or diploma in any university or other tertiary institution 
without the prior approval of the University of Adelaide and where applicable, any partner institution
responsible for the joint-award of this degree.
\vspace{5mm}

I give consent to this copy of my thesis, when deposited in the University Library, being made
available for loan and photocopying, subject to the provisions of the Copyright Act 1968.
\vspace{5mm}

I also give permission for the digital version of my thesis to be made available on the web, via the
University's digital research repository, the Library catalogue and also through web search engines, unless permission has been granted by the University to restrict access for a period of time.

\vspace{20mm}

\noindent
{\sc SIGNED}: {\tt ............................} \\
\vspace{7mm}

{\sc DATE}: {\tt ..............................}

\newpage
\thispagestyle{plain}
\mbox{}
\chapter*{}
\label{ch:dedication}
%\addcontentsline{toc}{chapter}{Dedication}
\vspace*{\fill}

\begin{center}

\emph{In Erinnerung an meinen Opa,}
\vskip1mm

\begin{Large}
 Ermut Geldmacher\\
\vspace{2mm}
\end{Large}
\begin{large}
(20. Juni 1923 bis 24. Februar 2009)
\end{large}

\vspace{1mm}

\emph{einen geselligen und freundlichen Mann,\\
dessen Kreativit\"{a}t mich nach wie vor inspiriert.}

%To the memory of my grandfather,\\
%Ermut Geldmacher.\\
%(\nth{20} June 1923---\nth{24} February 2009)\\
%A kind, gregarious man \\whose creativity continues to inspire me.
\end{center}

\vspace*{\fill}
\vspace*{\fill}

\newpage
\thispagestyle{plain}
\mbox{}
\chapter*{Acknowledgements}
\typeout{Acknowledgements}
\label{ch:acknowledgements}
\addcontentsline{toc}{chapter}{Acknowledgements}

I would like to sincerely thank my supervisors, Michael and Pedram, for allowing me to undertake my research with a great deal of autonomy and also for their patience during the numerous times each week I would drop by unannounced.
Their guidance, insight and mathematical savvy, far exceeding my own, has proved invaluable throughout the course of my project. I am happy to say that I have learned a lot from them.
By the same token, my thanks go to Raymond Vozzo and David Roberts for their sheer abundance of patience as well as the many useful and interesting conversations that I have had with them.

I am very grateful to have been part of the amazing community of students and staff that is the School of Mathematical Sciences at the University of Adelaide; there are too many to name.
In particular, I would like to thank my fellow students David Price, Nic Rebuli, Ben Rohrlach and Mingmei Teo for the readiness and gusto with which they would ingest vast quantities of caffeine each day, just so we could have a break from the mathematics by way of frequently hilarious conversation (at least now I will not claim that I can perform a one-man alley-oop).
Thanks are also due to Nicholas Buchdahl and Finnur L\'{a}russon, who between them are largely responsible for teaching me most of the mathematics I now know and who were always willing to discuss with me the finer points of whatever it was that I didn't understand on any given week.

I would also like to thank my family.
Especial thanks go to my mother Karin, who is precisely the type of person I aspire to be, and my step-father Jason, who has been a constant source of motivation throughout the course of my education.
I am grateful to my little brother Ned for reminding me what it is like to be a child (even though he never lets me win our sword-fights) and to my not-so-little brother Lukas for his humorous anecdotes.
Special thanks are due also to the Braunack-Mayers; Annette, Erik and Anna, as well as Jakob and Jacqui Mayer, for their constant support and encouragement.
Their company was an unexpected blessing that has had an inestimable effect on my well-being throughout the past year or so (not to mention their willingness in trying my food as I attempted to refine my culinary skills).

Finally, I would like to thank my beautiful girlfriend Lydia Braunack-Mayer; without her infinite patience and enchanting company I would surely be a gibbering wreck by now.
\newpage
\thispagestyle{plain}
\mbox{}

\cleardoublepage
\pagenumbering{arabic}
%\listoftables
%\listoffigures
\pagestyle{plain}
% What do you want at the top and bottom of a page (see fancyheadings.sty)
%\rhead{\thepage}
%\cfoot{}
% Import the chapter files
\renewcommand{\theequation}{I.\arabic{equation}}
\chapter*{Introduction \label{ch:intro}}
\addcontentsline{toc}{chapter}{Introduction}

The caloron correspondence appeared initially in the guise of a bijection between certain isomorphism classes of periodic instantons, or \emph{calorons}, on $\RR^4$ and isomorphism classes of monopoles on $\RR^3$.

Considering monopoles for loop groups and their twistor theory, Garland and Murray established in \cite{GM1} a correspondence between $SU(n)$-calorons on $\RR^4$ and monopoles on $\RR^3$ with structure group $L SU(n)$, the Fr\'{e}chet Lie group of smooth loops in $SU(n)$.
By virtue of being periodic, a caloron on $\RR^4$ may be naturally viewed as an instanton on $\RR^3\x\sone$, thus the work of Garland and Murray may be viewed as establishing a relationship between certain geometric data on $SU(n)$-bundles over $\RR^3 \x \sone$ and similar data on $L SU(n)$-bundles on $\RR^3$.

The underlying principle of this original \emph{caloron correspondence}---as it was first described by Murray and Stevenson \cite{MS}---is that, for any compact Lie group $G$ and manifold $M$, there is a bijective correspondence between $G$-bundles over $M\x\sone$ and $L G$-bundles over $M$, with $LG$ the Fr\'{e}chet Lie group of smooth loops in $G$.
This procedure gives a sort of fake dimensional reduction, whereby the circle direction of the $G$-bundle $P \to M\x\sone$ is hidden in the fibres to obtain an $LG$-bundle $\sP \to M$ and vice-versa.

The caloron correspondence may be thought of as the bundle-theoretic generalisation of the following simple observation.
If  $X, Y$ and $Z$ are sets, then denoting by $Y^X$ the set of all functions $X\to Y$, there is a bijection
\begin{equation}
\label{eqn:setcaloron}
c \colon Z^{X\x Y} \simto \big(Z^Y\big){}^X
\end{equation}
given by sending $f \mapsto \check f$ with
\[
\check f(x)(y) := f(x,y)
\]
for $x \in X$ and $y\in Y$.
In the case that $X, Y$ and $Z$ are finite-dimensional manifolds, let $\Map(X,Y)$ be the set of all smooth maps $X\to Y$.
If $Y$ is compact then $\Map(Y,Z)$ becomes a (smooth) Fr\'{e}chet manifold.
The map \eqref{eqn:setcaloron} now gives a method by which one may study smooth maps from $X$ into the infinite-dimensional manifold $\Map(Y,Z)$ by instead studying smooth maps from $X\x Y$ into $Z$.
In fact, in the case that $X = M$, $Y=\sone$ and $Z = G$, the map \eqref{eqn:setcaloron} gives a bijective correspondence between the space of sections of the trivial $G$-bundle over $M\x\sone$ and the space of sections of the trivial $LG$-bundle over $M$.
For general (non-trivial) $G$-bundles, the caloron correspondence is a twisted version of this equivalence.

The caloron correspondence outlined thus far gives a means by which one may more easily study $LG$-bundles, which are necessarily infinite-dimensional manifolds, by instead studying their finite-dimensional $G$-bundle counterparts.
Perhaps more importantly, the caloron correspondence may be extended to incorporate geometric data.
In \cite{MS} it was shown that a $G$-bundle over $ M\x\sone$ equipped with a $G$-connection is equivalent to an $LG$-bundle over $M$ equipped with an $LG$-connection and an extra geometric datum called a \emph{Higgs field}, which is essentially the component of a connection on a $G$-bundle over $M\x\sone$ in the $\sone$ direction.

Using this \emph{geometric} caloron correspondence together with the machinery of bundle gerbes, Murray and Stevenson developed a useful generalisation of string classes.
String classes first appeared in the work of Killingback \cite{K} on string structures; the string theory versions of the spin structures that are important in quantum field theory.
Taking a compact Lie group $G$ one may ask whether a given $LG$-bundle $\sP \to M$ admits a lifting of the structure group to the Kac-Moody group $\widehat{LG}$, which is a central extension of $LG$ by $\sone$ (see \cite{PS}, for example).
The obstruction to such a lift is a class in the degree three integral cohomology of $M$.
In the case that $\sP = LQ \to LM$ is given by taking smooth loops in a $G$-bundle $Q \to M$, Killingback showed that this obstruction is given by transgressing the first Pontryagin class $p_1(Q)$ of $Q$.
Thus the \emph{string class} is
\[
s(\sP) = \widehat{\int_\sone} \ev^\ast p_1(Q) \in H^3(LM)
\]
where $\ev \colon \sone \x LM \to M$ is the evaluation map and $\widehat{\int_\sone}$ denotes integration over the fibre in integral cohomology.
The string class of $\sP$ measures the obstruction to $\sP$ having \emph{string structure}; i.e.~a lifting to an $\widehat{LG}$-bundle.
String structures are important in string theory because, as the work of Killingback shows, the loop bundle $LQ \to LM$ has a Dirac-Ramond operator if and only if $LQ$ has a string structure.

Murray and Stevenson used the caloron correspondence to extend the work of Killingback by first defining the string class for all $LG$-bundles $\sP \to M$ and showing that it satisfies
\[
s(\sP) = \widehat{\int_\sone} p_1(P)
\]
where $p_1(P)$ is the first Pontryagin class of the caloron transform $P$ of $\sP$.
They also showed, using bundle gerbes, that a de Rham representative for the string class is given by
\[
-\frac{1}{4\pi^2} \int_{\sone} \lan \sF, \nabla\Phi \ran
\]
where $\sF$ is the curvature of a chosen $LG$-connection on $\sP$, $\nabla\Phi$ is the covariant derivative of a chosen Higgs field $\Phi$  on $\sP$ and $\lan \cdot,\cdot\ran$ is the (normalised) Killing form. 

In his PhD thesis \cite{V} and together with Murray in \cite{MV}, Vozzo generalised the caloron correspondence to principal bundles with structure group $\Omega G$; the Fr\'{e}chet Lie group of smooth loops in $G$ based at the identity.
The key innovation here is the use of framings to establish a correspondence between $\Omega G$-bundles over $M$ and $G$-bundles over $M\x\sone$ equipped with a distinguished section over $M\x\{0\}$.
As before, this correspondence generalises to incorporate connective data, which must necessarily be compatible with the framing data on the $G$-bundle side.

Murray and Vozzo also defined (higher) string classes, which are characteristic classes for $\Omega G$-bundles that live in odd integral cohomology.
Fixing the $\Omega G$-bundle $\sP \to M$ and choosing an $\Omega G$-connection $\sA$ and Higgs field $\Phi$, explicit de Rham representatives for these characteristic classes called \emph{string forms} are given by
\[
s_f(\sA,\Phi) = k \int_\sone f(\nabla\Phi,\!\underbrace{\sF,\dotsc,\sF}_{\text{$k-1$ times}}\!),
\]
where $f$ is an $\ad$-invariant symmetric polynomial on the Lie algebra $\g$ of $G$ of degree $k$ and $\sF$, $\nabla \Phi$ are as above.
If $A$ denotes the corresponding connection on the caloron transform $P$ of $\sP$, it turns out that the string forms satisfy
\[
s_f(\sA,\Phi) = \widehat{\int_\sone} \cw_f(A),
\]
where $\cw_f(A)$ is the Chern-Weil form associated to $f$ and $A$.
The higher string classes give a version of Chern-Weil theory for loop group bundles different from the more analytic approach of \cite{PRose}.
In addition, by considering the path fibration $PG \to G$ which is a smooth model for the universal $\Omega G$-bundle, Murray and Vozzo show that the construction of the higher string classes provides a geometric interpretation of Borel's transgression map $\tau \colon H^{2k}(BG) \to H^{2k-1}(G)$ (see \cite{B2} for details).

This thesis grew out of the attempt to answer a natural question that arises when one contrasts the theory of string classes for loop group bundles to the familiar Chern-Weil theory.
In Chern-Weil theory, differential form representatives for characteristic classes of the $G$-bundle $P \to X$ are given in terms of the curvature of a chosen connection $A$ on $P$.
Whilst the characteristic cohomology classes of the bundle $P$ are necessarily independent of the choice of $A$, the differential form representatives are not.
There are well-known differential forms, the Chern-Simons forms introduced in \cite{CS}, that measure the dependence of the Chern-Weil forms on the choice of connection. 
It is natural to ask, therefore, whether such forms exists in the context of loop group bundles and string classes.

The first part of this thesis deals with the construction of the \emph{string potentials}, which are the analogues of Chern-Simons forms for loop group bundles.
Like Chern-Simons forms, the string potentials come in two different flavours: one has \emph{relative} string potentials, which live on the base space of a loop group bundle and encode the dependence of the string forms on the choice of connection and Higgs field; and \emph{total} string potentials, which live on the total space and carry secondary geometric data associated to a particular choice of connection and Higgs field.

Within the framework of the differential characters of Cheeger and Simons \cite{ChS1}, the total Chern-Simons forms become \emph{differential} characteristic classes (characteristic classes valued in differential cohomology).
This thesis hints at a similar interpretation for the total string potentials by constructing such classes in a limited setting.

The interpretation of the relative string potential forms is more involved and proceeds by analogy with the codification of relative Chern-Simons forms given by Simons and Sullivan in \cite{SSvec}.
In that paper, the authors use relative Chern-Simons forms to define an equivalence relation on the space of connections on a given vector bundle.
The space of isomorphism classes of \emph{structured vector bundles}, i.e.~vector bundles equipped with such an equivalence class of connections, determines a functor from the category of compact manifolds with corners to the category of abelian semi-rings.
Passing to the Grothendieck group completion, one obtains a multiplicative differential extension of the even-degree part of topological $K$-theory.
By a result of Bunke and Schick \cite{BS2} this differential extension, denoted here by $\edK$, is isomorphic to any other differential extension of even $K$-theory via a unique isomorphism.

The Simons-Sullivan model of even differential $K$-theory is built upon vector bundles rather than principal bundles\footnote{though, of course, the two are naturally related by the frame bundle and associated vector bundle functors.}, since the even topological $K$-theory of a compact manifold $M$ has a natural construction in terms of vector bundles over $M$.
Topological $K$-theory is a generalised cohomology theory and as such  has a `homotopy-invariant' representation as homotopy classes of maps into a spectrum $KU$.
By the well-known Bott Periodicity Theorem this spectrum is $2$-periodic and in fact
\[
K^0(M) \cong [M,BU\x\ZZ]\;\;\mbox{ and }\;\; K^{-1}(M) \cong [M,U],
\]
where $U = \dlim U(n)$ is the stabilised unitary group and $BU$ is its classifying space.
Using this representation, it is clear why even $K$-theory $K^0(M)$ may be represented by vector bundles over $M$.
In fact, as $M$ is taken to be a smooth manifold, one may define $K^0(M)$ using only smooth vector bundles.

The odd $K$-theory of $M$ is a little more subtle and is usually defined in terms of vector bundles over $\Sigma M^+$, the reduced suspension of $M^+ := M\sqcup \{\ast\}$.
This is problematic when attempting to construct a differential extension after the fashion of Simons-Sullivan as $\Sigma M^+$ is rarely a smooth manifold so it is not clear how to incorporate differential form data.
The homotopy-theoretic model for $K^{-1}(M)$ gives a clue as to how to resolve this issue: by pulling back the path fibration $PU \to U$ one may construct odd $K$-theory using $\Omega U$-bundles, or rather their associated vector bundles, over $M$.
The benefits of this are two-fold since the building blocks of the theory are bundles over $M$ that may additionally  be taken to be smooth without loss of generality.

The latter part of this thesis introduces $\Omega$ vector bundles, which are the associated vector bundles of $\Omega GL_n(\CC)$- and $\Omega U(n)$-bundles.
As with their frame bundles, there is a caloron correspondence for $\Omega$ vector bundles that may be extended to incorporate the appropriate connective data.
A model for odd topological $K$-theory is given in terms of $\Omega$ vector bundles and the odd Chern character is computed in this model in terms of characteristic classes of the underlying $\Omega$ vector bundles.
Using the relative string potentials to define an equivalence relation on connective data, this model is refined to give a differential extension of odd $K$-theory: the $\Omega$ model.
Using the work of Bunke and Schick \cite{BS3,BS2,BS1} and Tradler, Wilson and Zeinalian \cite{TWZ} it is shown that the $\Omega$ model is isomorphic to the odd part of differential $K$-theory, thereby giving the desired codification of relative string potentials.

An outline of this thesis is:

\textbf{Chapter \ref{ch:two}.}
This chapter gives a detailed review of the construction of the caloron correspondence as formulated by Murray and Stevenson \cite{MS} for free loop groups and Murray and Vozzo \cite{MV} for based loop groups.
Following this, an in-depth exposition of the construction of string forms and string classes is presented.

\textbf{Chapter \ref{ch:three}.}
This chapter describes the construction of the relative and total string potential forms for loop group bundles and collects some facts about these objects used in subsequent chapters.
Following a brief review of differential cohomology, in particular Cheeger-Simons differential characters, the total string potentials are used to construct degree $1$ differential characteristic classes for $\Omega U(n)$-bundles.

\textbf{Chapter \ref{ch:four}.}
This chapter focusses on the introduction of $\Omega$ vector bundles.
These objects are Fr\'{e}chet vector bundles with typical fibre $LV$ and structure group $\Omega G$ for some complex vector space $V$ and matrix group $G\subseteq GL(V)$ with its standard action on $V$.
A caloron correspondence is developed relating $\Omega$ vector bundles over $M$ to framed vector bundles over $M \x\sone$ that respects the frame bundle functor and principal bundle caloron correspondence.
A version of the Serre-Swan theorem is proved for $\Omega$ vector bundles, which shows that every $\Omega$ vector bundle over compact $M$ may be regarded as a smoothly-varying family of modules for the ring $L\CC$ over $M$.
This module structure is used to define connective data (\emph{module connections} and \emph{vector bundle Higgs fields}) on $\Omega$ vector bundles, which fit into a geometric caloron correspondence for vector bundles.

After introducing the analogue of Hermitian structures for $\Omega$ vector bundles, together with an associated caloron correspondence,  a model for odd $K$-theory is defined by applying the Grothendieck group completion to the abelian semi-group of isomorphism classes of $\Omega$ vector bundles.
The odd Chern character is computed in this model of odd $K$-theory in terms of string forms of the underlying $\Omega$ vector bundles.

\textbf{Chapter \ref{ch:five}.}
Based on the results of Chapters \ref{ch:three} and \ref{ch:four} and following a review of the Simons-Sullivan construction of \cite{SSvec}, a differential extension of odd $K$-theory is constructed in terms of $\Omega$ vector bundles.
This construction uses the relative string potential forms to generate an equivalence relation on the space of module connections and Higgs fields of a given $\Omega$ vector bundle, an equivalence class of which is called a \emph{string datum}.
The $\Omega$ model is given by applying the Grothendieck group completion device to the abelian semi-group of (a certain collection of) isomorphism classes of \emph{structured $\Omega$ vector bundles}; $\Omega$ vector bundles equipped with string data. 

Bunke and Schick showed in \cite{BS2} that differential extensions of odd $K$-theory are non-unique and that additional structure is required in order to obtain differential $K$-theory, which is unique up to unique isomorphism.
Nevertheless, by relating the $\Omega$ model to a differential extension appearing in a recent paper of Tradler, Wilson and Zeinalian \cite{TWZ}, the caloron transform is used to show that the $\Omega$ model defines the odd part of differential $K$-theory.
The effect of this it two-fold, as it provides a sort of homotopy-theoretic interpretation of the $\Omega$ model as well as a proof that the $\TWZ$ differential extension defines odd differential $K$-theory, a result not previously obtained.

\textbf{Appendices.}
Appendix \ref{app:frechet} provides background material on Fr\'{e}chet spaces and Fr\'{e}chet manifolds, a proof that the path fibration $PG\to G$ gives a model for the universal $\Omega G$-bundle and some results on direct limits of directed systems of manifolds.
Appendix \ref{app:int} discusses the integration over the fibre operations on differential forms and in singular cohomology.
Appendix \ref{app:diff} records the Bunke-Schick definition of differential extensions together with some results that are required in this thesis.

\textbf{Remark on conventions.}
Unless stated otherwise all smooth finite-dimensional manifolds are taken to be paracompact and Hausdorff (so that they admit smooth partitions of unity) and all maps between smooth manifolds are smooth.
All unadorned cohomology groups $H^\bullet$ represent integer-valued singular cohomology and $\Omega_{d=0}(M)$ denotes the space of closed differential forms on the smooth manifold $M$.
The symbol $G$ shall usually denote a smooth connected finite-dimensional Lie group, with $\Theta$ its (left-invariant) Maurer-Cartan form and $\g = \Lie(G)$ its Lie algebra.
The terms `$G$-bundle' and `principal $G$-bundle' are used interchangeably.
The circle group $\sone$ is regarded as the quotient of $\RR$ modulo the equivalence relation $x\sim y \Leftrightarrow x = y + 2k \pi$ for some $k\in \ZZ$ and the equivalence class of $0$ defines a distinguished basepoint for $\sone$, which is also denoted $0$.
The integration over the fibre operation $\widehat{\int_\sone}$ is always taken with respect to the canonical orientation on $\sone$ inherited from $\RR$.
The Fr\'{e}chet Lie group of smooth maps $\sone \to G$ with pointwise group operations is denoted by $LG$ and the subgroup of those maps sending $0 \in \sone$ to the identity in $G$ is denoted $\Omega G$.

%\newpage
%\mbox{}
\renewcommand{\theequation}{\arabic{chapter}.\arabic{section}.\arabic{equation}}
\chapter{The caloron correspondence\label{ch:two}}

This chapter gives a detailed treatment of the underlying constructions needed throughout this thesis.
It begins with an in-depth review of the caloron correspondence for $LG$-bundles, then continues by reviewing a variant of the correspondence---the `based case'---that is of singular importance in the sequel.
Finally, some results are collected from the classifying theory of $\Omega G$-bundles.

To aid the exposition some technical material that, strictly speaking, should form a part of the discourse has been relegated to the appendices.

%%%%%%%%%%%%%%%%%%%%%%%%%%%%%%%%%%%%%%%%%%%%%%%%%%%%%%%%%%%%%%%%%%%%%%%%%%%%%%%%%%%%
%%%%%%%%%%%%%%%%%%%%%%%%%%%%%%%%%%%%%%%%%%%%%%%%%%%%%%%%%%%%%%%%%%%%%%%%%%%%%%%%%%%%
%%%%%%%%%%%%%%%%%%%%%%%%%%%%%%%%%%%%%%%%%%%%%%%%%%%%%%%%%%%%%%%%%%%%%%%%%%%%%%%%%%%%

\section{The caloron correspondence}
\label{S:caloron}
The caloron correspondence is a bijective correspondence between isomorphism classes of $G$-bundles over $M\x\sone$ and isomorphism classes of $LG$-bundles over $M$.
The correspondence may be thought of as a sort of fake dimensional reduction; given a $G$-bundle over $M\x\sone$ it allows one to simplify the base manifold by `hiding' the circle direction in the fibres resulting in an $LG$-bundle over $M$.

The underlying idea of the caloron correspondence appeared initially in \cite{GM1} in the study of the relationship between $LG$-valued monopoles on $\RR^3$ and \emph{calorons}---periodic $G$-instantons on $\RR^4$.
In \cite{MS}, the authors realised the caloron correspondence as a relationship between $G$-bundles over $M\x\sone$ and $LG$-bundles over $M$, for any manifold $M$.

The caloron correspondence enables one to represent the total space of an $LG$-bundle, which is necessarily an infinite-dimensional Fr\'{e}chet manifold (cf. Appendix \ref{app:frechet}), in terms of a finite-dimensional manifold.
However, at this level the caloron correspondence is not especially exotic.
What is perhaps surprising is that the caloron correspondence allows one to transfer certain \emph{geometric} data, such as connections and Higgs fields (Definition \ref{defn:higgs}), from the infinite-dimensional setting to the finite-dimensional setting and vice-versa.
The caloron correspondence thus becomes a powerful tool for elucidating the properties of loop group bundles.
In particular, via a modification of the classical Chern-Weil theory, it allows one to construct explicitly a suite of characteristic classes for such bundles.

%%%%%%%%%%%%%%%%%%%%%%%%%%%%%%%%%%%%%%%%%%%%%%%%%%%%%%%%%%%%%%%%%%%%%%%%%%%%%%%%%%%%
%%%%%%%%%%%%%%%%%%%%%%%%%%%%%%%%%%%%%%%%%%%%%%%%%%%%%%%%%%%%%%%%%%%%%%%%%%%%%%%%%%%%
%%%%%%%%%%%%%%%%%%%%%%%%%%%%%%%%%%%%%%%%%%%%%%%%%%%%%%%%%%%%%%%%%%%%%%%%%%%%%%%%%%%%

\subsection{The caloron correspondence}
\label{SS:caloron}
Before detailing the caloron correspondence, some basic definitions are needed.
\vspace{11pt}
\begin{definition}
\label{defn:bundlecat}
Let $\Bun_{G}$  be the category whose objects are principal $G$-bundles and whose morphisms are $G$-bundle maps, that is, smooth $G$-equivariant maps on the total spaces.

For a fixed manifold $M$, let $\Bun_{G}(M)$ be the groupoid of all $G$-bundles $P \to M$ with morphisms those bundle maps covering the identity on $M$.
There is a canonical faithful functor $\Bun_G(M) \to \Bun_G$.
\end{definition}
\vspace{11pt}
\begin{definition}
\label{defn:bundlescat}
Define $\sBun$ to be the category whose objects are those $G$-bundles of the form $P \to M\x\sone$ for some manifold $M$.
The morphisms of $\sBun_G$ are given by $G$-bundle maps covering maps of the form $\widetilde f \x\id \colon N\x\sone \to M\x\sone$.

As above, for fixed $M$ let $\sBun_{G}(M)$ be the groupoid of all $G$-bundles $P \to M \x\sone$ with morphisms those bundle maps covering the identity on $M\x\sone$.
There is a canonical faithful functor $\sBun_G(M) \to \sBun_G$.
\end{definition}

The caloron correspondence may now  be phrased succinctly as an equivalence of categories between $\Bun_{LG}$ and $\sBun_G$.
This equivalence is given by a pair of functors
\[
\cC \colon \Bun_{LG} \lo \sBun_G
\]
and
\[
\cC^{-1} \colon \sBun_{G} \lo \Bun_{LG},
\]
called the \emph{caloron transform} and \emph{inverse caloron transform} respectively.
It is important to notice that the notation $\cC^{-1}$ is somewhat misleading since the inverse caloron transform is only a pseudo-inverse for $\cC$; i.e. $\cC^{-1}$ is the inverse of $\cC$ only up to a natural isomorphism. 

The action of the \emph{caloron transform} on objects is given by sending the $LG$-bundle $\sQ \to M$ to the space
\[
\cC(\sQ) := (\sQ \x\sone \x G) / LG,
\]
the quotient taken with respect to the $LG$-action $(q,\theta,g) \cdot \gamma = (q\gamma,\theta,\gamma(\theta)^{-1} g)$, where one denotes the equivalence class of $(q,\theta,g)$ by $[q,\theta,g]$.
A straightforward argument using the local triviality of $\sQ$ shows that the space $\cC(\sQ)$ has a smooth manifold structure.
Moreover, writing $\pi \colon \sQ \to M$ for the projection, $\cC(\sQ)$ is naturally a $G$-bundle over $M\x\sone$ with (free and transitive) right action
\[
[q,\theta,g] \cdot h := [q,\theta,gh]
\]
and projection map
\[
[q,\theta,g] \longmapsto (\pi(q),\theta).
\]
The action of $\cC$ on the morphism $f \colon \sQ \to \sP$ is given by
\begin{align*}
\cC(f) \colon \cC(\sQ) &\lo \cC(\sP)\\
[q,\theta,g] &\longmapsto [f(q),\theta,g]
\end{align*}
which is clearly well-defined, $G$-equivariant and covers a map $\widetilde{f}\x\id \colon N\x\sone \to M\x\sone$, hence is a morphism in $\sBun_G$.

Conversely, the \emph{inverse caloron transform} $\cC^{-1}$ works by taking sections.
That is, $\cC^{-1}$ sends the $G$-bundle $P \to M\x\sone$ to the $LG$-bundle $\cC^{-1}(P) \to M$ whose fibre over $m \in M$ is precisely the set of sections $\Gamma(\{m\}\x\sone, P)$ with the obvious pointwise $LG$-action.
This construction may be realised globally as follows.
Applying the functor $L = \Map(\sone,\cdot)$ of smooth maps to $P \to M\x\sone$ gives the $LG$-bundle $LP \to L(M\x\sone)$.
There is a natural smooth map
\begin{equation}
\label{eqn:etamap}
\eta \colon M \lo L(M\x\sone)
\end{equation}
sending $m \in M$ to the loop $\theta \mapsto (m,\theta)$. 
Using $\eta$ to pull back $LP$ yields the $LG$-bundle $\cC^{-1}(P) \to M$
\[
\xy
(0,25)*+{\cC^{-1}(P):= \eta^\ast LP}="1";
(50,25)*+{LP}="2";
(0,0)*+{M}="3";
(50,0)*+{L(M\x\sone)}="4";
{\ar^{} "1";"2"};
{\ar^{} "2";"4"};
{\ar^{} "1";"3"};
{\ar^{\eta} "3";"4"};
\endxy
\]
and from the construction the fibre $\cC^{-1}(P)_m = \Gamma(\{m\}\x\sone,P)$ for all $m \in M$.
The action of $\cC^{-1}$ on the morphism $f \colon P \to Q$ is simply the induced map
\begin{align*}
\cC^{-1}(f) \colon \cC^{-1}(P) &\lo \cC^{-1}(Q)\\
p &\longmapsto f \circ p,
\end{align*}
which is a morphism in $\Bun_{LG}$.
\vspace{11pt}
\begin{remark}
One may view the  inverse caloron transform in sheaf theory terms, namely  if $\Gamma_P$ is the sheaf of smooth sections of the $G$-bundle $P\to M\x\sone$ then  $\cC^{-1}(P)$ is the $LG$-bundle whose sheaf of smooth sections is exactly $\pi_\ast \Gamma_P$, where $\pi \colon M\x\sone \to M$ is the projection. 
\end{remark}

The caloron transform functors defined above give a way of constructing a $G$-bundle given an $LG$-bundle and conversely.
It remains to be seen that $\cC$ and $\cC^{-1}$ do indeed define an equivalence of categories between $\sBun_G$ and $\Bun_{LG}$.
That is, one requires the existence of natural isomorphisms
\[
\alpha \colon \cC^{-1} \circ \cC \lo \id_{\Bun_{LG}}
\]
and
\[
\beta \colon \cC \circ \cC^{-1} \lo \id_{\sBun_{G}}\!\!.
\]
To construct the natural isomorphism $\alpha$, one first takes the $LG$-bundle $\sQ \to M$.
Denoting $Q := \cC(\sQ)$, define the smooth map
\[
\hat \eta \colon \sQ \lo L Q
\]
by setting $\hat\eta(q)(\theta) := [q,\theta,1]$ for $\theta\in\sone$.
This map is $LG$-equivariant since for any $\gamma \in LG$
\begin{align*}
\hat\eta(q\gamma)(\theta) &= [q\gamma,\theta,1]\\
&=[q,\theta,\gamma(\theta)] \\
&=\left(\hat\eta(q) \cdot \gamma\right)(\theta).
\end{align*}
This gives the bundle map
\[
\xy
(0,25)*+{\sQ}="1";
(50,25)*+{L Q}="2";
(0,0)*+{M}="3";
(50,0)*+{L(M\x\sone)}="4";
{\ar^{\hat\eta} "1";"2"};
{\ar^{} "2";"4"};
{\ar^{} "1";"3"};
{\ar^{\eta} "3";"4"};
\endxy
\]
which, recalling the construction of $\cC^{-1}$, defines an isomorphism $\alpha_\sQ \colon \cC^{-1}(\cC(\sQ)) \to \sQ$ of $LG$-bundles.
It is easy to see that for any morphism $f \colon \sQ \to \sP$ in $\Bun_{LG}$ one has $\alpha_\sP \circ \cC^{-1}(\cC(f)) = f \circ \alpha_\sQ$ so that $\alpha$ is indeed a natural transformation from $\cC^{-1} \circ \cC$ to $\id_{\Bun_{LG}}$.

To define the natural isomorphism $\beta$, first take any $G$-bundle $P \to M\x\sone$.
Since the construction of $\sP := \cC^{-1}(P)$ is such that
\begin{equation}
\label{eqn:secidentg}
\sP_m = \Gamma(\{m\}\x\sone,P)
\end{equation}
for all $m\in M$, it follows that the fibre of $\cC(\sP)$ over $(m,\theta) \in M\x\sone$ is
\[
(\sP_m \x \{\theta\} \x G) / LG = \left\{ p \colon \sone \to P \mid \pi \circ p (\theta) = (m,\theta) \text{ for all $\theta \in \sone$}\right\} \x \{\theta\} \x G
\]
with $\pi \colon P\to M$ the projection.
Then $\beta_P$ is the map
\[
\beta_P \colon [p,\theta,g] \longmapsto p(\theta)g,
\]
which is an isomorphism of $G$-bundles covering the identity.
For any morphism $f\colon P \to Q$ in $\sBun_G$, it is clear that $\beta_Q \circ \cC(\cC^{-1}(f)) = f \circ \beta_P$ so that $\beta$ is a natural transformation from $\cC\circ\cC^{-1}$ to $\id_{\sBun_G}$.
This completes the proof of
\vspace{11pt}
\begin{theorem}[\cite{MS,MV}]
\label{theorem:lgcaloron}
The caloron correspondence
\[
\cC \colon \Bun_{LG} \lo \sBun_G \;\;\mbox{and}\;\;\; \cC^{-1} \colon  \sBun_G \lo \Bun_{LG}
\]
is an equivalence of categories that, for any manifold $M$, restricts to an equivalence of groupoids
\[
\cC \colon \Bun_{LG}(M) \lo \sBun_G(M) \;\;\mbox{and}\;\;\; \cC^{-1} \colon  \sBun_G(M) \lo \Bun_{LG}(M).
\]
\end{theorem}

The following result establishes that the caloron correspondence truly is the bundle-theoretic version of the bijection $c \colon f \to \check f$ of \eqref{eqn:setcaloron}.
\vspace{11pt}
\begin{lemma}
Take any $G$-bundle $P \to M\x\sone$.
If $\{U_\alpha\}$ is an open cover of $M$ for which there are local sections $s_\alpha\in \Gamma(U_\alpha\x\sone, P)$ then the $LG$-bundle $\sP := \cC^{-1}(P)$ has local sections $\check s_\alpha \in \Gamma(U_\alpha,\sP)$.

Moreover, if $\tau_{\alpha\beta}$ are the transition functions of $P$ with respect to the sections $s_\alpha$ then the transition functions of $\sP$ with respect to the sections $\check s_\alpha$ are precisely $\check \tau_{\alpha\beta}$.

The converse is also true.
\end{lemma}
\begin{proof}
Since $\sP$ is constructed by looping $P$, the map $\check s_\alpha(m)(\theta) := s_\alpha(m,\theta)$ is a section of $\sP$ over $U_\alpha$.
Moreover, on the intersections $U_{\alpha\beta} :=U_\alpha \cap U_\beta$ one has
\[
s_\beta(m,\theta) = s_\alpha(m,\theta) \cdot   \tau_{\alpha\beta}(m,\theta)\;\;\mbox{ and }\;\;\check s_\beta(m) = \check s_\alpha(m) \cdot   \upsilon_{\alpha\beta}(m)
\]
so evaluating the latter expression at $\theta \in \sone$ gives
\[
\upsilon_{\alpha\beta}(m)(\theta) = \tau_{\alpha\beta}(m,\theta)
\]
and hence $\upsilon_{\alpha\beta} = \check \tau_{\alpha\beta}$.
The converse is essentially the above argument.
\end{proof}

%%%%%%%%%%%%%%%%%%%%%%%%%%%%%%%%%%%%%%%%%%%%%%%%%%%%%%%%%%%%%%%%%%%%%%%%%%%%%%%%%%%%
%%%%%%%%%%%%%%%%%%%%%%%%%%%%%%%%%%%%%%%%%%%%%%%%%%%%%%%%%%%%%%%%%%%%%%%%%%%%%%%%%%%%
%%%%%%%%%%%%%%%%%%%%%%%%%%%%%%%%%%%%%%%%%%%%%%%%%%%%%%%%%%%%%%%%%%%%%%%%%%%%%%%%%%%%

\subsection{Higgs fields and connections}
As mentioned previously, the true power of the caloron correspondence lies in its ability to translate certain geometric data from loop group bundles to finite-dimensional bundles and vice-versa.
More specifically, there is a functorial equivalence between $G$-bundles $P \to M\x\sone$ equipped with a $G$-connection and $LG$-bundles $\sQ \to M$ equipped with an $LG$-connection and Higgs field (Definition \ref{defn:higgs}).
\vspace{11pt}
\begin{definition}[\cite{MS,MV}]
\label{defn:higgs}
A \emph{Higgs field} on the $LG$-bundle $\sQ \to M$ is a smooth map $\Phi \colon \sQ \to L\g$ that satisfies the \emph{twisted equivariance condition}
\begin{equation}
\label{eqn:twisteq}
\Phi(q\gamma) = \ad(\gamma^{-1})\Phi(q) + \gamma^{-1} \d\gamma
\end{equation}
for all $q\in \sQ$ and $\gamma\in L G$, where $\d$ denotes differentiation of the loop $\gamma$ in the $\sone$ direction.
The space of Higgs fields $\cH_\sQ$ on a fixed $LG$-bundle $\sQ \to M$ is an affine space.
\end{definition}

As a result of the caloron correspondence for bundles with connection, it will become apparent that Higgs fields really encode the $\sone$ component of a connection on a $G$-bundle over $M\x\sone$.
The following result guarantees the existence of Higgs fields on any $LG$-bundle $\sQ \to M$.
\vspace{11pt}
\begin{lemma}[\cite{MV,V}]
\label{lemma:hfexist}
Higgs fields exist.
\end{lemma}
\begin{proof}
It is evident that a convex combination of Higgs fields is again a Higgs field and that any trivial $LG$-bundle admits the trivial Higgs field
\begin{align*}
LG &\lo L\g\\
\gamma &\longmapsto \gamma^{-1}\d\gamma.
\end{align*}
The result follows by choosing a smooth partition of unity for $M$ subordinate to a given trivialisation.
\end{proof}

Having established this result, one is almost in a position to describe the geometric caloron correspondence.
\vspace{11pt}
\begin{definition}
\label{defn:gbuncatconn}
Let $\sBun^{c}_{G}$ to be the category whose objects are objects of $\sBun_G$ equipped with $G$-connections and whose morphisms are the connection-preserving morphisms of $\sBun_G$.

For a fixed manifold $M$, define $\sBun^{c}_{G}(M)$ to be the groupoid with objects the $G$-bundles $P \to M\x\sone$ equipped with connection, with morphisms those connection-preserving bundle maps covering the identity.

Moreover, in the case that the group is a loop group, define $\Bun^c_{LG}$ to be the category whose objects are $LG$-bundles equipped with $LG$-connections and Higgs fields.
Morphisms of this category are $LG$-bundle maps that preserve the additional structure.
For a fixed manifold $M$, the groupoid $\Bun^c_{LG}(M)$ is defined similarly to the above.
\end{definition}

To define the \emph{geometric caloron transform}, a functor
\[
\cC \colon \Bun^c_{LG} \lo \sBun^{c}_{G},
\]
first take the $LG$-bundle $\sQ \to M$ equipped with $LG$-connection $\sA$ and Higgs field $\Phi$.
Define the 1-form $A$ on $\sQ\x\sone\x G$ by
\begin{equation}
\label{eqn:con:lgcalfinconn}
A_{(q,\theta,g)} : = \ad(g^{-1}) \big( \sA_q(\theta) + \Phi(q) (\theta) d\theta \big) + \Theta_g,
\end{equation}
with $\Theta$ the Maurer-Cartan form on $G$.
\vspace{11pt}
\begin{lemma}[\cite{MS,MV,V}]
\label{lemma:conndescend}
The 1-form $A$ defined by equation \eqref{eqn:con:lgcalfinconn} descends to a $G$-connection, also called $A$, on the caloron transform $Q:= \cC(\sQ)  = (\sQ \x\sone \x G) / LG$.
\end{lemma}
\begin{proof}
The proof is essentially that of \cite[Proposition 3.9]{MV}.
To show that equation \eqref{eqn:con:lgcalfinconn} determines a well-defined 1-form on $\cC(\sQ)$, one must show that it is basic for the projection $\sQ \x\sone\x G \to \cC(\sQ)$.

Take any $X \in T_{[q,\theta,g]}Q$ and suppose that $\hat X$, $\hat X'$ are lifts of the vector $X$ to $\sQ\x\sone\x G$.
Without loss of generality, one may suppose that 
\[
\hat X \in T_{(q,\theta,g)}(\sQ\x\sone\x G)\;\;\mbox{and}\;\;\hat X' \in T_{(q,\theta,g)\gamma}(\sQ\x\sone\x G).
\]
Since $\hat X$ and $\hat X'$ are both lifts of $X$, one has that the pushforward $dR_\gamma\hat X$ of $\hat X$ by the (right) action $R_\gamma$ of $\gamma$  satisfies $dR_\gamma \hat X = \hat X' + V$ for some vertical vector $V$.
It is therefore sufficient to show that $A$ is invariant under the $LG$-action and that it annihilates vertical vectors.
For simplicity of calculation, one supposes $G$ to have a faithful matrix representation\footnote{since this thesis deals exclusively with $G$ a compact Lie group or $G= GL_n(\CC)$ for some $n$, this assumption is not restrictive. This assumption is not required in the general case.} so that the exponential map may be written as $\exp(t\xi) = 1 + t \xi + \mathcal{O}(t^2)$.
Thus any vertical vector at $(q,\theta,g)$ is of the form
\[
V = \dtz{(q,\theta,g) \cdot \exp(t\xi)} =\big(\xi^\#_q, 0, -\xi(\theta) g\big)
\]
for some $\xi \in L\g$, where $\xi^\#$ is the fundamental vector field on $\sQ$ generated by $\xi$.
The action of $A$ on such a vector is
\[
A_{(q,\theta,g)}(V) = \ad(g^{-1})(\xi(\theta)) - g^{-1} \xi(\theta) g = 0.
\]
To see that $A$ is basic suppose that
\[
\hat X_{(q,\theta,g)} = \dtz{\big(\gamma_\chi(t), \theta+tx, g\exp(t\zeta)\big)} = (\chi,x,g\zeta).
\]
Then for $\gamma \in LG$, the pushforward $dR_\gamma(\hat X)$ is given by
\begin{align*}
dR_\gamma(\hat X)_{(q,\theta,g)\gamma} &= \dtz{\big(\gamma_\chi(t)\gamma, \theta+tx, \gamma(\theta+tx)^{-1} g\exp(t\zeta)\big) } \\
&= \left(dR_\gamma(\chi), x, \gamma(\theta)^{-1}g\left[\zeta - x\ad(g^{-1})(\d\gamma(\theta))\gamma(\theta)^{-1}  \right] \right).
\end{align*}
Consequently,
\begin{align*}
R_\gamma^\ast A_{(q,\theta,g)} (\hat X) & = A_{(q\gamma,\theta,\gamma(\theta)^{-1}g)} (dR_\gamma(\hat X)) \\
%%%
&=\ad(g^{-1}\gamma(\theta))\left( \ad(\gamma(\theta)^{-1}) \sA_q (\chi)(\theta) + x \ad(\gamma(\theta)^{-1})\Phi(q)(\theta)  + x \gamma(\theta)^{-1}\d\gamma(\theta) \right)\\
&\qquad +\zeta  -x\ad(g^{-1})(\d\gamma(\theta))\gamma(\theta)^{-1} \\
%%%
&= \ad(g^{-1})\left( \sA_q(\chi)(\theta) + x \Phi(q)(\theta)\right) + \zeta\\
%%%
& = A_{(q,\theta,g)}(\hat X).
\end{align*}
This shows that the $1$-form defined in \eqref{eqn:con:lgcalfinconn} does indeed descend to a form on $Q$, which shall also be called $A$.

To see that $A$ is a $G$-connection, one must show that it reproduces the Lie algebra generators of fundamental vector fields on $Q$ and that it is equivariant for the $G$-action.
Notice that the vertical vector generated by $\xi \in \g$ at $[q,\theta,g] \in Q$ is
\[
V = \dtz{\big[q,\theta,g\exp(t\xi)\big]} = [0,0,g\xi]
\]
so that $A_{[q,\theta,g]} (V) = \xi$ as required.
It remains only to show that $R_h^\ast A(X) = \ad(h^{-1}) A(X)$ for all $h \in G$ and vector fields $X$.
If $X = [\chi,x,g\xi] \in T_{[q,\theta,g]}Q$ then
\begin{align*}
R_h^\ast A_{[q,\theta,g]}(X) &= A_{[q,\theta,gh]}([\chi,x,gh\ad(h)^{-1} \xi])\\
%%%
&= \ad(h^{-1}g^{-1})\left(\sA_q(\chi)(\theta) + x\,\Phi(q)(\theta)\right) + \ad(h^{-1}) \xi \\
&= \ad(h^{-1}) A_{[q,\theta,g]}(X) 
\end{align*}
which completes the proof.
\end{proof}

Having this result, one defines the geometric caloron transform of the $LG$-bundle $\sQ \to M$ with $LG$-connection $\sA$ and Higgs field $\Phi$ as the $G$-bundle $\cC(\sQ) \to M\x\sone$ equipped with the connection $A$ determined by \eqref{eqn:con:lgcalfinconn}.
Since a morphism $f \colon \sQ \to \sP$ in $\Bun^c_{LG}$ is required to respect the geometric data, the expression \eqref{eqn:con:lgcalfinconn} and the definition of $\cC$ together imply that the $G$-bundle morphism $\cC(f) \colon \cC(\sQ)\to \cC(\sP)$ as defined previously respects the connective data and so is a morphism of $\sBun^{c}_G$.

To define the \emph{geometric inverse caloron transform}, a functor
\[
\cC^{-1} \colon \sBun^{c}_G \lo \Bun_{LG}^c,
\]
first take a $G$-bundle $P \to M\x\sone$ with connection $A$.
Recalling that $\sP := \cC^{-1}(P)$ is defined essentially by looping $P$, one defines the $LG$-connection $\sA$ on $\sP$ via the expression
\begin{equation}
\label{eqn:con:lgcalfrecconn}
\sA_q(\chi)(\theta) := A_{q(\theta)} (\chi(\theta)),
\end{equation}
as the tangent vector $\chi \in T_q\sP$ is naturally equivalent to a section of the pullback vector bundle $q^\ast TP \to \sone$ (see \cite[Example 4.3.3]{H} or Appendix \ref{app:frechet}).
It is immediate that $\sA$ satisfies the properties required of an $LG$-connection simply by virtue of the fact that $A$ is a connection.

Writing $\d$\label{page:d} for the canonical vector field on $\sone$, one defines a Higgs field $\Phi$ on $\sP$ via the expression
\begin{equation}
\label{eqn:con:lgcalhiggs}
\Phi(q)(\theta) := q^\ast A_\theta (\d).
\end{equation}
To see that $\Phi$ satisfies the twisted equivariance condition \eqref{eqn:twisteq}, take $\gamma \in LG$ so that
\begin{align*}
\Phi(q\gamma)(\theta) &= (q\gamma)^\ast A_\theta(\d) \\
&= A_{q(\theta)\gamma(\theta)} \left(dR_{\gamma(\theta)}(\d q(\theta)) + (\gamma(\theta)^{-1}\d\gamma(\theta))^\#\right) \\
&= \ad(\gamma(\theta)^{-1})\Phi(q)(\theta) + \gamma(\theta)^{-1} \d \gamma(\theta)
\end{align*}
as required.
Notice that this formulation justifies the remark following Definition \ref{defn:higgs}, since the Higgs field constructed above really is the $\sone$ component of the connection $A$.

The geometric inverse caloron transform is given by sending the $G$-bundle $P \to M\x\sone$ with $G$-connection $A$ to the $LG$-bundle $\cC^{-1}(P) \to M$ with the $LG$-connection $\sA$ and Higgs field $\Phi$ given by \eqref{eqn:con:lgcalfrecconn} and \eqref{eqn:con:lgcalhiggs} respectively.
As for $\cC$, the action of $\cC^{-1}$ on morphisms is straightforward by virtue of the fact that morphisms in $\sBun^{c}_{G}$ necessarily preserve the connective data.
\vspace{11pt}
\begin{theorem}[\cite{MS,MV}]
\label{theorem:lggcaloron}
The geometric caloron correspondence
\[
\cC \colon \Bun^c_{LG} \lo \sBun^{c}_G \;\;\mbox{and}\;\;\; \cC^{-1} \colon  \sBun^{c}_G \lo \Bun^c_{LG}
\]
is an equivalence of categories that, for any manifold $M$, restricts to an equivalence of groupoids
\[
\cC \colon \Bun^c_{LG}(M) \lo \sBun^{c}_G(M) \;\;\mbox{and}\;\; \cC^{-1} \colon  \sBun^{c}_G(M) \lo \Bun^c_{LG}(M).
\]
\end{theorem}
\begin{proof}
It suffices to show that the natural isomorphisms $\beta_P$ and $\alpha_\sQ$ of Theorem \ref{theorem:lgcaloron} preserve the connective data.
If one starts with the $LG$-bundle $\sQ \to M$ with connection $\sA$ and Higgs field $\Phi$ then the connection $\sA'$ on $\sQ' := \cC^{-1}(\cC(\sQ))$ is given by
\[
\sA'_q(\chi)(\theta) := A_{q(\theta)}(\chi(\theta))
\]
where $Q := \cC(\sQ)$ is equipped with caloron-transformed connection $A$.
Recall that the natural isomorphism $\alpha_\sQ \colon \cC^{-1}(\cC(\sQ)) \to \sQ$ satisfies
\[
\alpha_\sQ^{-1}(q) := \big( \theta\mapsto [q,\theta,1] \big)
\] 
so that
\[
(\alpha_\sQ^{-1})^\ast \sA'_{q}(\chi)(\theta) = A_{[q,\theta,1]}( [\chi,0,0]) := \sA_q(\chi)(\theta).
\]
Thus $\alpha_\sQ$ preserves the $LG$-connections and an analogous argument holds for the Higgs fields.

On the other hand, if one begins with a $G$-bundle $P \to M\x\sone$ with connection $A$, then the connection $A'$ on $P' := \cC(\cC^{-1}(P))$ is given by
\[
A'_{[p,\theta,g]} = \ad(g^{-1})\big(\sA_p(\theta) + \Phi(p)(\theta)d\theta \big) + \Theta_g,
\]
with $\sA$ and $\Phi$ repsectively the connection and Higgs field on the (inverse) caloron transform $\sP:= \cC^{-1}(P)$ of $P$.
Recall that $\beta_P$ is given by sending
\[
[p,\theta,g] \longmapsto p(\theta)g,
\]
so considering the tangent vector
\[
X_{[p,\theta,g]} = \dtz{\big[\gamma_\chi(t), \theta+tx, g\exp(t\zeta)\big]} = [\chi,x,\zeta]
\]
one obtains
\[
d\beta_P(X)_{p(\theta)g} = \dtz{\gamma_\chi(t)(\theta+tx) g \exp(t\zeta)} = \chi(\theta)g + x\, \d p(\theta) g + \zeta^\#_{p(\theta)g}.
\]
Hence
\begin{align*}
\beta_P^\ast A_{[p,\theta,g]} (X) &= \ad(g^{-1})\left( A_{p(\theta)} (\chi(\theta)) + x A_{p(\theta)}(\d p(\theta)) \right) + \zeta = A'_{[p,\theta,g]}(X).
\end{align*}
This completes the proof.
\end{proof}

%%%%%%%%%%%%%%%%%%%%%%%%%%%%%%%%%%%%%%%%%%%%%%%%%%%%%%%%%%%%%%%%%%%%%%%%%%%%%%%%%%%%
%%%%%%%%%%%%%%%%%%%%%%%%%%%%%%%%%%%%%%%%%%%%%%%%%%%%%%%%%%%%%%%%%%%%%%%%%%%%%%%%%%%%
%%%%%%%%%%%%%%%%%%%%%%%%%%%%%%%%%%%%%%%%%%%%%%%%%%%%%%%%%%%%%%%%%%%%%%%%%%%%%%%%%%%%

\subsection{String classes}
\label{SS:stringclasses}
The enhanced caloron correspondence of Theorem \ref{theorem:lggcaloron} turns out to be a very useful tool, particularly for constructing characteristic classes for $LG$-bundles.
The procedure is a relatively simple variation of the standard Chern-Weil theory and relies on the caloron transform and integration over the fibre (cf. Appendix \ref{app:int}).

First, one recalls briefly the theory of characteristic classes; classical references for this material are \cite{KN2,Steen}.
For any Lie group $G$ there is a \emph{universal} $G$-bundle $EG \to BG$ such that the total space $EG$ is contractible.
A key property of the universal bundle is that for any (topological) $G$-bundle $P \to M$ there is a \emph{classifying map} $f \colon M \to BG$ such that the pullback $f^\ast EG$ is isomorphic to $P$.
The homotopy class of the classifying map $f$ is uniquely determined by $P$ and for any two homotopic maps $f_0,f_1\colon M\to BG$ the pullbacks $f_0^\ast EG$ and $f_1^\ast EG$ are isomorphic as $G$-bundles over $M$.
This establishes a bijective correspondence between principal $G$-bundles over $M$ and homotopy classes of maps $M \to BG$.
It is important to notice that in general neither $EG$ or $BG$ are assumed to be smooth manifolds and that they are unique only up to homotopy equivalence.

A \emph{characteristic class} associates to a $G$-bundle $P\to M$ a class $c(P)$ in the cohomology of $M$ that is natural in the sense that if $f \colon N \to M$ is a continuous (or, in the setting of the caloron correspondence, smooth) map then $f^\ast c(P) = c(f^\ast(P))$.
Since all $G$-bundles are pullbacks of the universal $G$-bundle $EG \to BG$, characteristic classes correspond precisely with cohomology classes of $BG$.

One important method for manufacturing characteristic classes is Chern-Weil theory.
Let
\[
\g^{\otimes k} := \underbrace{\g\otimes\dotsb\otimes\g}_{\text{$k$ times}},
\]
then  an \emph{invariant polynomial} of degree $k$ on $\g$ is a symmetric multilinear map $\g^{\otimes k} \to \RR$ that is invariant under the adjoint action of $G$ on $\g^{\otimes k}$.
The set of invariant polynomials of degree $k$ is denoted $I^k(\g)$.
Invariant polynomials multiply in a natural way so that $I^\bullet(\g) = \bigoplus_{i=1}^\infty I^k(\g)$ is a graded algebra.

Denote by $\Omega^p(M;\g^{\otimes q})$ the space of differential $p$-forms on $M$ taking values in $\g^{\otimes q}$.
There is a wedge product $\wedge \colon \Omega^p(M;\g^{\otimes q}) \x\Omega^{p'}(M;\g^{\otimes {q'}}) \to \Omega^{p+p'}(M;\g^{\otimes {q+q'}})$ given by
\[
\alpha \wedge \beta\, (X_1,\dotsc,X_{p+p'}) := \sum_{\sigma\in S_{p+p'}} (-1)^{|\sigma|}\alpha\big(X_{\sigma(1)},\dotsc,X_{\sigma(p)}\big) \otimes \beta\big(X_{\sigma(p+1)},\dotsc,X_{\sigma(p+p')}\big)
\]
for vector fields $X_1,\dotsc,X_{p+p'}$ on $M$, where $S_k$ is the group of permutations on $\{1,\dotsc,k\}$ and $|\sigma|$ is the sign of the permutation $\sigma$.
If $\alpha \in \Omega^p(M;\g)$ and $\beta \in \Omega^{p'}(M;\g)$ set
\[
[\alpha,\beta](X_1,\dotsc,X_{p+p'}) :=\sum_{\sigma\in S_{p+p'}} (-1)^{|\sigma|}\big[\alpha\big(X_{\sigma(1)},\dotsc,X_{\sigma(p)}\big), \beta\big(X_{\sigma(p+1)},\dotsc,X_{\sigma(p+p')}\big)\big]
\]
and there is also an exterior derivative $d \colon \Omega^p(M;\g^{\otimes q}) \to \Omega^{p+1}(M;\g^{\otimes q})$.
If $f \in I^k(\g)$ and $\omega_i \in \Omega^{p_i}(M;\g)$ for $i = 1,\dotsc,k$ write
\[
f( \omega_1,\dotsc,\omega_k):= f( \omega_1\wedge\dotsb\wedge\omega_k) \in \Omega^{p_1+\dotsb+p_k}(M).
\]
The $\ad$-invariance of $f $ implies
\[
\sum_{i=1}^k (-1)^{(p_1+\dotsb+p_i)}f(\omega_1,\dotsc,[\omega_i,\varpi],\dotsc,\omega_k) = 0
\]
for any $\varpi \in \Omega^{1}(M;\g)$ (cf. \cite{CS} or \cite[Lemma 3.2.6]{V} for the case that $\varpi$ has degree $p \geq 1$).
The following is well-known
\vspace{11pt}
\begin{theorem}[Chern-Weil Homomorphism]
\label{theorem:chernweil}
Given a $G$-bundle $P\to M$ with connection $A$ and curvature $F$, for any $f \in I^k(\g)$ the real-valued $2k$-form
\[
cw_f(A) := f( \underbrace{F,\dotsc,F}_{\text{$k$ times}})
\]
on $P$ is closed and descends to a form on $M$ whose class in de Rham cohomology is independent of the choice of connection $A$.
Using the de Rham isomorphism, this defines a map
\[
cw \colon I^\bullet(\g) \lo H^{2\bullet}(M;\RR),
\]
which is an algebra homomorphism.
\end{theorem}

For a detailed treatment of Chern-Weil theory, including a proof of this result, see \cite{Dupont}.
An immediate consequence is that the cohomology class of $cw_f(A)$ gives a characteristic class $cw_f(P) \in H^{2k} (M;\RR)$ for any $f \in I^k(g)$.
By a result of H. Cartan \cite{Cartan}, if $G$ is compact then all characteristic classes for $G$-bundles in $\RR$-valued cohomology are Chern-Weil classes (see also \cite[Theorem 8.1]{Dupont}).

In a similar vein, one may use the caloron correspondence and the Chern-Weil homomorphism to construct characteristic classes for $LG$-bundles; classes constructed in this manner are called \emph{string classes}.
To construct the string classes of the $LG$-bundle $\sQ \to M$, choose any $LG$-connection $\sA$ and Higgs field $\Phi$ on $\sQ$.
Applying the caloron transform gives the $G$-bundle $Q\to M\x\sone$ equipped with the $G$-connection $A$.
\vspace{11pt}
\begin{definition}[\cite{MV,V}]
The \emph{string form} associated to $f \in I^k(\g)$ is
\begin{equation}
\label{eqn:stringform}
s_f(\sA,\Phi) := \widehat{\int_\sone} cw_f(A),
\end{equation}
which is a closed $(2k-1)$-form on $M$.
The operation $\widehat{\int_\sone}$ here denotes integration over the fibre of differential forms (see Appendix \ref{app:int}).
\end{definition}

Since the exterior derivative commutes with integration over the fibre (Lemma \ref{lemma:intwithd}) Theorem \ref{theorem:chernweil} implies that the cohomology class of $s_f(\sA,\Phi)$ depends neither on the $LG$-connection $\sA$ nor on the Higgs field $\Phi$---this can also be seen immediately as a result of the construction of the string potential forms in Chapter \ref{ch:three}, in particular Theorem \ref{theorem:antistring}.
Therefore, taking the cohomology class and applying the de Rham isomorphism gives the \emph{string class}
\begin{equation}
\label{eqn:stringclass}
s_f(\sQ) \in H^{2k-1}(M;\RR),
\end{equation}
which is a characteristic class for any $f \in I^k(\g)$.

This construction of the string classes naturally provides differential form representatives that seem to depend on data on the caloron transform.
However, the caloron correspondence allows one to write these representatives entirely in terms of data on the underlying $LG$-bundle.
Namely, given the $LG$-bundle $\sQ \to M$ with connection $\sA$ and Higgs field $\Phi$, let $Q:=\cC(\sQ)$ be its caloron transform with caloron-transformed connection $A$.
The string form associated to the invariant polynomial $f \in I^k(\g)$ is then
\[
s_f(\sA,\Phi) = \widehat{\int_\sone} f( \underbrace{F,\dotsc,F}_{\text{$k$ times}})
\]
with $F$ the curvature of the connection $A$ on $Q$.
Recall $F = dA + \tfrac{1}{2}[A,A]$ and that, in this case, the connection $A$ is given by \eqref{eqn:con:lgcalfinconn} so one may write $F$ in terms of $\sA$ and $\Phi$; first notice that
\begin{align*}
[A,A] &=  \left[ \ad(g^{-1}) \left( \sA + \Phi d\theta \right) + \Theta, \ad(g^{-1}) \left( \sA + \Phi d\theta \right) + \Theta\right]\\
%%%
&= \ad(g^{-1}) \left([\sA,\sA] + 2[\sA,\Phi] \wedge d\theta \right) + 2[\Theta,\ad(g^{-1})\sA] + 2[\Theta,\ad(g^{-1})\Phi]\wedge d\theta.
\end{align*}
Recall also that $d\omega (X,Y) := \left\{X(\omega(Y)) - Y(\omega(X))- \omega([X,Y]) \right\}$ for a 1-form $\omega$ and vector fields $X, Y$.
If $X$ and $Y$ are the vector fields on $Q$ whose values at the point $[q,\theta,g] \in Q$ are
\begin{align*}
X_{[q,\theta,g]} &= \dtz{\big(\gamma_\chi(t), \theta+tx, g\exp(t\zeta)\big)} = (\chi,x,g\zeta)\\
%%%
Y_{[q,\theta,g]} &= \dtz{\big(\gamma_\kappa(t), \theta+ty, g\exp(t\xi)\big)} = (\kappa,y,g\xi)
\end{align*}
then
\[
[X,Y]_{[q,\theta,g]} = ([\chi,\kappa],0,g[\zeta,\xi]).
\]
First, one calculates $d(\ad(g^{-1})\sA)$ by noticing that at the point $[q,\theta,g]$
\begin{align*}
X \big(\ad(g^{-1})\sA(Y)\big) &= \dtz{(1-t\zeta)g^{-1} \sA_{\gamma_{\chi}(t)}(Y)(\theta+tx) g(1+t\zeta)}\\
%%%
&= \dtz{\ad(g^{-1})\sA_{\gamma_{\chi}(t)}(\kappa)(\theta)} + x\, \ad(g^{-1}) \d\sA_q(\kappa)(\theta) \\
&\qquad\qquad - \left[\zeta, \ad(g^{-1})\sA_q(\kappa)(\theta)\right]
\end{align*}
and hence
\[
d(\ad(g^{-1})\sA) = \ad(g^{-1})d\sA + \d\sA\wedge d\theta -\left[\Theta,\ad(g^{-1})\sA\right],
\]
where $\d \sA$ is the $L\g$-valued $1$-form on $\sQ$ given by differentiating $\sA$ in the $\sone$ direction. 

Applying the same argument to $\ad(g^{-1})\Phi \,d\theta$ this all gives
\[
d(\ad(g^{-1})\Phi d\theta) = \ad(g^{-1}) d\Phi\wedge d\theta -\left[\Theta,\ad(g^{-1})\Phi\right] \wedge d\theta,
\]
so recalling that the Maurer-Cartan form $\Theta$ satisfies $d\Theta + \frac{1}{2}[\Theta,\Theta] = 0$ gives
\[
F = \ad(g^{-1}) \left( d\sA + \frac{1}{2} [\sA,\sA] + (d\Phi + [\sA,\Phi] -\d\sA) \wedge d\theta  \right)\!.
\]
Writing $\sF = d\sA +\frac{1}{2}[\sA,\sA]$ for the curvature of $\sA$ and defining the \emph{Higgs field covariant derivative}\footnote{one might be tempted to call $\nabla \Phi$ the `Higgs field curvature', however there is a generalised version of the caloron correspondence in which an additional term $\sF_\Phi$ appears on the right-hand side of the expression \eqref{eqn:curv} \cite[pp.~238]{HMV}. In this context it is more appropriate to call $\sF_\Phi$ the Higgs field curvature, so this terminology is used here for consistency.} $\nabla\Phi := d\Phi + [\sA,\Phi] -\d\sA$, one obtains
\begin{equation}
\label{eqn:curv}
F = \ad(g^{-1}) (\sF + \nabla\Phi\wedge d\theta).
\end{equation}
Using the properties of $f \in I^k(\g)$ yields the expression
\begin{equation}
\label{eqn:stringformex}
s_f(\sA,\Phi) = k \int_\sone f(\nabla\Phi,\underbrace{\sF,\dotsc,\sF}_{\text{$k-1$ times}})
\end{equation}
for the string form associated to $f$.
Note that the integration symbol in \eqref{eqn:stringformex} denotes the standard integration operation on functions $\sone \to \RR$ and not integration over the fibre.%
\vspace{11pt}
\begin{example}
\label{example:kstring}
In \cite{K} Killingback studied the notion of a \emph{string structure}, which is the string-theoretic analogue of a spin structure.
If $G$ is a compact, simple, simply-connected Lie group then it is well-known (see \cite{PS}, for example) that there is a universal central extension
\[
0 \lo \sone \lo \widehat{LG} \lo LG\lo 0
\]
of the loop group $LG$.
A string structure on $M$ is then given by a lifting of the $LG$-bundle $\sQ \to M$ to an $\widehat{LG}$-bundle $\widehat{\sQ} \to M$ and there is an integral three-class---the original string class---that measures the obstruction to such a lift.
Using the machinery of bundle gerbes, which  give smooth geometric representatives for degree three integral cohomology through their Dixmier-Douady classes, Murray and Stevenson \cite{MS} gave an explicit formula for a de Rham representative of this class, namely
\[
-\frac{1}{4\pi^2}\int_\sone \lan \nabla\Phi,\sF \ran.
\]
In this expression, $\sF$ and $\nabla\Phi$ are respectively the curvature of a connection and covariant derivative of a Higgs field on $\sQ$ and $\lan\cdot,\cdot\ran$ is the Killing form on $\g$ normalised so that the longest root has length $\sqrt{2}$.
It is clear that this string class is simply the string class (in the sense of \eqref{eqn:stringclass}) corresponding to the $\ad$-invariant polynomial
\[
f(\cdot,\cdot) := - \frac{1}{8\pi^2} \lan\cdot,\cdot\ran.
\]
Notice in particular that if $Q$ is the caloron transform of $\sQ$ equipped with the caloron-transformed connection $A$ then
\[
s_f(\sQ) = \widehat{\int_\sone} p_1(Q)
\]
where $p_1(Q)$ is the first Pontryagin class of $Q$.
\end{example}

%%%%%%%%%%%%%%%%%%%%%%%%%%%%%%%%%%%%%%%%%%%%%%%%%%%%%%%%%%%%%%%%%%%%%%%%%%%%%%%%%%%%
%%%%%%%%%%%%%%%%%%%%%%%%%%%%%%%%%%%%%%%%%%%%%%%%%%%%%%%%%%%%%%%%%%%%%%%%%%%%%%%%%%%%
%%%%%%%%%%%%%%%%%%%%%%%%%%%%%%%%%%%%%%%%%%%%%%%%%%%%%%%%%%%%%%%%%%%%%%%%%%%%%%%%%%%%

\section{The based case}
\label{S:based}
In Section \ref{S:caloron}, the main objects under consideration were $LG$-bundles over $M$ and the corresponding $G$-bundles over $M\x\sone$.
It turns out that with a little more work, one may extend all of the results of Section \ref{S:caloron} to give a caloron correspondence for principal bundles whose structure group is the \emph{based} loop group
\[
\Omega G := \{ \gamma \in LG \mid \gamma(0) = 1 \},
\]
a Fr\'{e}chet Lie subgroup of $LG$.
The key innovation here is the use of framings (Definition \ref{defn:framedbundle}) on the finite-dimensional side to reduce the structure group on the Fr\'{e}chet side from $LG$ to $\Omega G$.
The discussion presented here is largely based off of \cite{MV,V}.
\vspace{11pt}
\begin{definition}
\label{defn:framedbundle}
Let $P \to X$ be a $G$-bundle and $X_0 \subset X$ a submanifold.
Then $P$ is \emph{framed} over $X_0$ if there is a distinguished section $s_0 \in \Gamma(X_0,P)$.
Write $P_0 = s_0(X_0) \subset P$ for the image of $s_0$.
\end{definition}

In what follows, if $P \to M\x\sone$ is a $G$-bundle then, unless stated otherwise, the framing shall be taken over the submanifold $M_0 := M\x\{0\}$.
\vspace{11pt}
\begin{definition}
\label{defn:frbundlecat}
Define $\frBun_{G}$ to be the category whose objects are framed $G$-bundles $P \to M\x\sone$ and whose morphisms are $G$-bundle maps that preserve the framings and cover a map of the form $\widetilde f\x\id \colon N \x\sone \to M\x\sone$.

For a fixed manifold $M$, let $\frBun_{G}(M)$ be the groupoid with objects the framed $G$-bundles $P \to M\x\sone$ with morphisms those bundle maps preserving the framing and covering the identity on $M\x\sone$.
\end{definition}

The \emph{based caloron transform} is then the functor
\[
\cC \colon \Bun_{\Omega G}\lo \frBun_G
\]
constructed as in Section \ref{SS:caloron}, i.e.~by sending the $\Omega G$-bundle $\sQ \to M$ to the associated bundle
\[
\cC(\sQ) := (\sQ\x\sone\x G) /\Omega G.
\]
This bundle has a canonical framing given by
\[
s_0 (m,0) := [q,0,1],
\]
where $q$ is any point in the fibre of $\sQ$ over $m$, and so is an object of $\frBun_G$.
The action of $\cC$ on morphisms is given analogously to the free loop case, noting of course that if $f \colon \sQ \to \sP$ is a morphism in $\Bun_{\Omega G}$,  then $\cC(f) \colon \cC(\sQ)\to \cC(\sP)$ preserves the framings and therefore gives a morphism in $\frBun_G$.

Conversely, the \emph{based inverse caloron transform} is a functor
\[
\cC^{-1} \colon \frBun_G \lo \Bun_{\Omega G}.
\]
defined similary to the free loop case, i.e.~by first looping and then pulling back by the map $\eta$ of \eqref{eqn:etamap}.
The distinction here is that instead of applying the smooth loop functor $L = \Map(\sone,\cdot)$ to the framed $G$-bundle $P \to M\x\sone$ one takes \emph{based} loops.
If $X$ is a smooth (finite-dimensional) manifold with submanifold $X_0 \subset X$, define the based loop manifold
\[
\Omega_{X_0}X := \{p\colon \sone \to X \text{ smooth } \mid p(0) \in X_0 \}.
\]
The based inverse caloron transform is given by first taking the based loop bundle
\[
\Omega_{P_0} P \lo \Omega_{M_0}(M\x\sone),
\]
which is an $\Omega G$-bundle, and then pulling back by $\eta$ to obtain the $\Omega G$-bundle $\cC^{-1}(P) \to M$.
The action of $\cC^{-1}$ on morphisms is defined similarly to the free loop case and, as before,
\vspace{11pt}
\begin{theorem}[\cite{MV,V}]
\label{theorem:equiv}
The based caloron correspondence
\[
\cC \colon \Bun_{\Omega G} \lo \frBun_G \;\;\mbox{and}\;\;\; \cC^{-1} \colon  \frBun_G \lo \Bun_{\Omega G}
\]
is an equivalence of categories that, for any manifold $M$, restricts to an equivalence of groupoids
\[
\cC \colon \Bun_{\Omega G}(M) \lo \frBun_G(M) \;\;\mbox{and}\;\; \cC^{-1} \colon  \frBun_G(M) \lo \Bun_{\Omega G}(M).
\]
\end{theorem}
\begin{proof}
The natural isomorphisms $\alpha$ and $\beta$ are constructed in an analogous fashion to the natural isomorphisms of Theorem \ref{theorem:lgcaloron} and, following the arguments presented there, are easily seen to satisfy the required properties.
\end{proof}

As in the free loop case, there is an extension of the based caloron correspondence to bundles with connection.
In this setting, connections on the $G$-bundle side are required to satisfy a compatibility condition with respect to the framings.
\vspace{11pt}
\begin{definition}
\label{caloron:framedconnection}
Let $P \to X$ be a framed $G$-bundle with framing $s_0 \in \Gamma(X_0,P)$.
A connection $A$ on $P$ is \emph{framed (with respect to $s_0$)} if $s_0^\ast A = 0$.
\end{definition}

This framing condition is required to guarantee that the connections constructed on the Fr\'{e}chet side are indeed valued in $\Omega\g = \Lie(\Omega G)$.
Framed connections exist on framed bundles provided, as is assumed in this thesis, that the base manifold admits smooth partitions of unity \cite[Lemma 3.5]{MV}.
\vspace{11pt}
\begin{definition}
Let $\frBun^{c}_{G}$ to be the category whose objects are objects of $\frBun_G$ equipped with framed  $G$-connections and whose morphisms are the connection-preserving morphisms of $\frBun_G$.

For a fixed manifold $M$, $\frBun^c_G(M)$ denotes the groupoid with objects the framed $G$-bundles $P \to M\x\sone$ equipped with framed connection and with morphisms the connection-preserving bundle maps covering the identity that also preserve the framing.
\end{definition}

In order to formulate the geometric caloron correspondence for based loop group bundles, one requires the correct notion of Higgs field for $\Omega G$-bundles.
It turns out that this is given exactly by replacing $LG$ in Definition \ref{defn:higgs} by $\Omega G$.
\vspace{11pt}
\begin{remark}
Higgs fields on $\Omega G$-bundles still map into $L\g$, not $\Omega\g$ as one might suspect.
This is because in general $\gamma^{-1}\d\gamma \notin \Omega \g$ for $\gamma \in \Omega G$, so a Higgs field cannot both map into $\Omega\g$ and satisfy the twisted equivariance condition.
\end{remark}
Lemma \ref{lemma:hfexist} may be easily adapted to show that every $\Omega G$-bundle admits a Higgs field.

Define the categories $\Bun_{\Omega G}^c$ and $\Bun_{\Omega G}^c(M)$ exactly as in Definition \ref{defn:gbuncatconn} (with $LG$ replaced by $\Omega G$).
The \emph{based geometric caloron transform} is the functor
\[
\cC\colon \Bun_{\Omega G}^c \lo \frBun_G^c
\]
defined by sending the $\Omega G$-bundle $\sQ \to M$ with $\Omega G$-connection $\sA$ and Higgs field $\Phi$ to the framed $G$-bundle $\cC(\sQ) \to M\x\sone$ (as defined above) equipped with the connection $A$  given by the expression \eqref{eqn:con:lgcalfinconn}.
To see that this is well-defined, one must verify that the connection $A$ is framed.
Recall that the framing of $\cC(\sQ)$ is the section
\[
s_0 \colon (m,0) \longmapsto [q,0,1].
\]
Therefore, taking any $X \in  T_{(m,0)}M_0$ (so that $ds_0 (X)= [\chi,0,0]$ for some tangent vector $\chi$ to $\sQ$), one obtains
\[
s_0^\ast A_{(m,0)}(X) = A_{[q,0,1]}([\chi,0,0]) = \sA_{q}(\chi)(0) = 0
\]
since $\sA$ is valued in $\Omega\g$.
The action of $\cC$ on morphisms is essentially the same as in the free loop case.

Conversely, the \emph{based inverse geometric caloron transform} is a functor
\[
\cC^{-1} \colon   \frBun_G^c\lo \Bun_{\Omega G}^c
\]
sends the framed $G$-bundle $P \to M\x\sone$ with framed connection $A$ to an $\Omega G$-bundle $\sP \to M$ equipped with $\Omega G$-connection $\sA$ and Higgs field $\Phi$.
Here $ \sP := \cC^{-1}(P)$ is the based inverse caloron transform as above and the connection $\sA$ and Higgs field $\Phi$ are given respectively by \eqref{eqn:con:lgcalfrecconn} and \eqref{eqn:con:lgcalhiggs}.
Since it is clear that $\Phi$ is well-defined, the only thing that needs to be shown is that the connection $\sA$ is indeed an $\Omega G$-connection.
To see this, take any $\chi \in T_p \sP$ noting that $\chi(0) \in T_{p(0)}P_0$ is in the image of the map $ds_0 \colon TM_0 \to TP_0$.
Thus
\[
\sA_p(\chi)(0) = A_{p(0)} (\chi(0)) = s_0^\ast A_{(m,0)}(X) = 0 
\]
where $p(0)$ is in the fibre of $P$ over $(m,0)$ and $ds_0 (X) = \chi(0)$.
Once again, the action of $\cC^{-1}$ on morphisms is essentially the same as in the free loop case.

One may verify that the natural isomorphisms of Theorem \ref{theorem:equiv} respect the connective data, the arguments proceeding exactly as in the free loop case, so that
\vspace{11pt}
\begin{theorem}[\cite{MV,V}]
\label{theorem:gequiv}
The based geometric caloron correspondence
\[
\cC \colon \Bun^c_{\Omega G} \lo \frBun^c_G \;\;\mbox{and}\;\;\; \cC^{-1} \colon  \frBun^c_G \lo \Bun^c_{\Omega G}
\]
is an equivalence of categories that, for any manifold $M$, restricts to an equivalence of groupoids
\[
\cC \colon \Bun^c_{\Omega G}(M) \lo \frBun^c_G(M) \;\;\mbox{and}\;\; \cC^{-1} \colon  \frBun^c_G(M) \lo \Bun^c_{\Omega G}(M).
\]
\end{theorem}

%%%%%%%%%%%%%%%%%%%%%%%%%%%%%%%%%%%%%%%%%%%%%%%%%%%%%%%%%%%%%%%%%%%%%%%%%%%%%%%%%%%%
%%%%%%%%%%%%%%%%%%%%%%%%%%%%%%%%%%%%%%%%%%%%%%%%%%%%%%%%%%%%%%%%%%%%%%%%%%%%%%%%%%%%
%%%%%%%%%%%%%%%%%%%%%%%%%%%%%%%%%%%%%%%%%%%%%%%%%%%%%%%%%%%%%%%%%%%%%%%%%%%%%%%%%%%%

\subsection{The path fibration and string classes}
\label{S:path}
The argument that constructed the string classes of Section \ref{SS:stringclasses} may easily be adapted in order to construct characteristic classes for $\Omega G$-bundles, which are also called \emph{string classes}.

Explicitly, for an $\Omega G$-bundle $\sQ \to M$ equipped with connection $\sA$ and Higgs field $\Phi$ and for any  invariant polynomial $f \in I^k(\g)$, one obtains the closed \emph{string form}
\[
s_f(\sA,\Phi) = k \int_\sone f(\nabla\phi,\underbrace{\sF,\dotsc,\sF}_{\text{$k-1$ times}})
\]
noticing that this is exactly the expression \eqref{eqn:stringformex}---the only thing that has changed is that the connection and Higgs field now live on an $\Omega G$-bundle.
As in the free loop case, the corresponding class in the real-valued cohomology of $M$ is the \emph{string class}
\[
s_f(\sQ) \in H^{2k-1}(M;\RR),
\]
which is independent of the choice of $\sA$ and $\Phi$.

So far, there have not been any particularly novel features of the based case.
One advantage of working with based loops instead of free loops is that there is a \emph{smooth} model for the universal $\Omega G$-bundle---the \emph{path fibration}---that allows one to explicitly compute the universal string classes.
\[
\xy
(20,25)*+{\Omega G}="1";
(50,25)*+{PG}="2";
(50,0)*+{G}="4";
{\ar^{} "1";"2"};
{\ar^{\ev_{2\pi}} "2";"4"};
\endxy
\]
Let $PG$ be the space of all smooth maps $p\colon \RR\to G$ such that $p(0)$ is the identity and $p^{-1}\d p$ is periodic with period $2\pi$.
There is a natural action of $\Omega G$ on $PG$ that gives $PG$ the structure of an $\Omega G$-bundle over $G$.
The projection $PG \to G$ is simply evaluation of paths at $2\pi$ and it turns out that $PG$ is (smoothly) contractible.
Therefore $PG\to G$ is a model for the universal $\Omega G$-bundle, which also shows that $B\Omega G = G$ in this case.
For a rigorous treatment see Appendix \ref{app:frechet}.

Since $PG$ is a smooth manifold one may talk about connective data directly on the universal $\Omega G$-bundle.
Recall that a tangent vector $\chi \in T_pPG$ is canonically identified with a section of $p^\ast TG \to \RR$ so that a vertical vector for $PG\to G$ is a tangent vector such that
\[
\chi(2\pi) = 0.
\]
Choosing a smooth function $\alpha \colon \RR\to \RR$ satisfying $\alpha(t) = 0$ for $t\leq0$ and $\alpha(t) = 1$ for $t\geq 2\pi$, a 
complementary horizontal subspace at $p \in PG$ is given by
\[
H_p = \{ \chi \in T_pPG \mid \chi(\theta) = \alpha(\theta) dR_{p(\theta)}(\xi) \mbox{ for some } \xi \in \g  \}
\]
(see also \cite[pp.~552--553]{MS}).
The horizontal projection of any tangent vector $\chi \in T_pPG$ is
\[
\chi^h(\theta) = \alpha(\theta)dR_{p(2\pi)^{-1}p(\theta)} \left(\chi(2\pi)\right).\
\]
and the value of the $\Omega G$-connection corresponding to this splitting of the tangent bundle at $p \in PG$ is
\[
\sA_\infty(\theta) = \Theta(\theta) - \alpha(\theta) \ad\left(p(\theta)^{-1}\right) \ev_{2\pi}^\ast\widehat\Theta
\]
with $\Theta$ the Maurer-Cartan form on $\Omega G$\footnote{given by $\Theta_\gamma(\xi)(\theta) := (\Theta_G){}_{\gamma(\theta)}(\xi(\theta))$ where $\gamma\in \Omega G$, $\xi \in T_\gamma \Omega G$ and $\Theta_G$ momentarily denotes the Maurer-Cartan form on $G$.} and $\widehat\Theta$ the \emph{right-invariant} Maurer-Cartan form on $G$.
There is a canonical Higgs field on $PG$ given by
\[
\Phi_\infty(p) = p^{-1}\d p
\]
and some straightforward calculations (as in \cite[Section 3.1.2]{V}) show that
\begin{equation}
\label{eqn:universalcurv}
\sF_\infty = -\frac{1}{2}\alpha(1-\alpha) \ad(p^{-1}) \ev_{2\pi}^\ast\left([\widehat\Theta,\widehat\Theta] \right)\;\;\mbox{ and }\;\; \nabla\Phi_\infty = \frac{d\alpha}{d\theta} \ad(p^{-1}) \ev_{2\pi}^\ast\big(\widehat\Theta\big).
\end{equation}
Assuming a fixed choice of smooth function $\alpha$, $\sA_\infty$ and $\Phi_\infty$ are respectively the \emph{standard} connection and Higgs field for $PG$.

One is now in a position to explicitly calculate the universal string forms, which are odd-degree forms on $G$, as follows.
Taking any $f \in I^k(\g)$ gives
\vspace{11pt}
\begin{lemma}[\cite{MV,V}]
\label{lemma:universalstringform}
The string form of the standard connection and Higgs field of the path fibration over $G$ is
\[
\tau_f :=  \left(-\frac{1}{2}\right)^{k-1}\frac{k!(k-1)!}{(2k-1)!} f\big(\Theta, \underbrace{[\Theta,\Theta],\dotsc,[\Theta,\Theta]}_{\text{$k-1$ times}} \big)
\]
with $\Theta$ the Maurer-Cartan form on $G$.
\end{lemma}
\begin{proof}
Plugging the expressions from \eqref{eqn:universalcurv} into the formula \eqref{eqn:stringformex} and using the fact that $\Theta_g = \ad(g^{-1}) \widehat\Theta_g$ gives
\[
s_f(\sA_\infty,\Phi_\infty) = k \left(-\frac{1}{2}\right)^{k-1} \left( \int_\sone \alpha^{k-1}(1-\alpha)^{k-1}\frac{d\alpha}{d\theta}\,d\theta \right) f\big(\Theta, \underbrace{[\Theta,\Theta],\dotsc,[\Theta,\Theta]}_{\text{$k-1$ times}} \big).
\]
The integral is evaluated as
\begin{align*}
 \int_\sone \alpha^{k-1}(1-\alpha)^{k-1} \frac{d\alpha}{d\theta}  \,d\theta = \int_0^1 t^{k-1}(1-t)^{k-1}\,dt = \frac{(k-1)!(k-1)!}{(2k-1)!}
\end{align*}
using the beta function, which completes the proof.
\end{proof}

\begin{remark}
\label{remark:trans}
The de Rham cohomology classes of the forms $\tau_f$ appearing in Lemma \ref{lemma:universalstringform} are well-known (cf. \cite{CS,HL}) to be precisely the cohomology classes on $G$ obtained by transgressing classes on $BG$.
$H^\bullet(G;\RR)$ is generated as an exterior algebra by (finitely many) such transgressed classes \cite[Theorem 18.1]{B2}.
\end{remark}

Another major advantage that comes from working with based loop group bundles is that classifying maps are easy to describe.
Namely, given an $\Omega G$-bundle $\sQ \to M$ a choice of Higgs field $\Phi$ on $\sQ$ is equivalent to choosing a smooth classifying map $M \to G$ for $\sQ$ in the following manner.
At the point $q\in \sQ$ the equation
\begin{equation}
\label{eqn:hfholonomy}
\Phi(q) = g(q)^{-1}\d g(q)
\end{equation}
for $g = g(q) \in PG$ has a unique solution by the Picard-Lindel\"{o}f Theorem.
The \emph{holonomy} of the Higgs field $\Phi$ is then the map $\hol_\Phi \colon \sQ \to PG$ that sends
\[
q \longmapsto \hol_\Phi(q) := g
\]
with $g$ satisfying \eqref{eqn:hfholonomy}.
As $\Phi$ is smooth, $\hol_\Phi$ is also smooth and notice also that $\hol_\Phi$ is $\Omega G$-equivariant, for if $\gamma \in \Omega G$ and $g = \hol_\Phi(q)$ then
\begin{align*}
(g\gamma)^{-1} \d(g\gamma) = \ad(\gamma^{-1}) g^{-1}\d g + \gamma^{-1}\d\gamma = \Phi(q\gamma)
\end{align*}
implies that $\hol_\Phi(q\gamma) = \hol_\Phi(q)\gamma$.
Hence $\hol_\Phi$ descends to a map $M\to G$, which shall also be called $\hol_\Phi$, proving
\vspace{11pt}
\begin{proposition}[\cite{MV}]
\label{prop:higgsfieldholonomy}
If $\sQ \to M$ is an $\Omega G$-bundle and $\Phi$ is any Higgs field on $\sQ$ then $\hol_\Phi \colon M\to G$ is a smooth classifying map for $\sQ$.
\end{proposition}

Writing
\[
\tau \colon I^\bullet(\g) \lo H^{2\bullet-1} (G;\RR)
\]
for the map $f \mapsto [\tau_f]$, with $\tau_f$ the forms of Lemma \ref{lemma:universalstringform}, by naturality of the string classes and using \eqref{eqn:stringform} one obtains
\vspace{11pt}
\begin{theorem}[\cite{MV}]
\label{theorem:stringdiag1}
If $\sQ \to M$ is an $\Omega G$-bundle and
\[
s(\sQ) \colon I^\bullet(\g) \lo H^{2\bullet-1}(M;\RR)
\]
is the map $f \mapsto s_f(\sQ)$ then the diagram
\[
\xy
(0,25)*+{I^\bullet(\g)}="1";
(50,25)*+{H^{2\bullet}(M\x\sone;\RR)}="2";
(0,0)*+{H^{2\bullet-1}(G;\RR)}="3";
(50,0)*+{H^{2\bullet-1}(M;\RR)}="4";
{\ar^{cw(Q)} "1";"2"};
{\ar^{\widehat{\int_\sone}} "2";"4"};
{\ar^{\tau} "1";"3"};
{\ar^{s(\sQ)} "1";"4"};
{\ar^{\hol_\Phi^\ast} "3";"4"};
\endxy
\]
commutes for any choice of Higgs field $\Phi$ on $\sQ$, where $Q := \cC(\sQ)$ is the caloron transform.
Here $\widehat{\int_\sone} \colon H^\bullet(M\x\sone;\RR) \to H^{\bullet-1}(M;\RR)$ is the integration over the fibre map in singular cohomology.
\end{theorem}

An immediate consequence of this is that for any choice of connection $\sA$ and Higgs field $\Phi$ on $\sQ$ there is some $(2k-2)$-form $\omega$ on $M$ such that
\[
s_f(\sA,\Phi) = \hol_\Phi^\ast\tau_f + d\omega.
\]
As shown in Proposition \ref{prop:canform} below, for each choice of $\sA$ and $\Phi$ there is a corresponding \emph{canonical} choice of $\omega$.
The fact that there is a canonical choice of $\omega$ plays an important role in relating the $\Omega$ model of Chapter \ref{ch:five} to the TWZ differential extension of odd $K$-theory in Section \ref{S:other}.

It is important for the construction of these canonical forms to understand how the Higgs field holonomy is related to the conventional notion of holonomy of a connection.
First take a $G$-bundle $\pi\colon P \to X$ with framing $s_0$ over $X_0$ and framed connection $A$.
Writing $P_0 = s_0(X_0)$ as before and taking based loops\footnote{recalling that, for example, $\Omega_{X_0}X = \{\gamma \in LX \mid \gamma(0) \in X_0 \}$.} gives the $\Omega G$-bundle $\Omega_{P_0} P \to \Omega_{X_0} X$, as in the construction of the based inverse caloron transform functor.
Take a loop $p \in \Omega_{P_0}P$ and project down to $\Omega_{X_0}X$ to obtain $\pi\circ p$.
Taking the horizontal lift $\widehat{p}$ of $\pi\circ p$ through $s_0(\pi\circ p(0))$ with respect to the connection $A$, one has
\[
p = \widehat{p} \hol(p)
\]
for some $\hol(p) \in PG$.
Notice that $\hol$ is $\Omega G$-equivariant since $\widehat{p\gamma} = \widehat{p}$ for each $\gamma \in \Omega G$ and so
\[
p\gamma = \widehat{p\gamma}\hol(p\gamma) \Longrightarrow \hol(p\gamma) = \hol(p)\gamma.
\]
Thus, $\hol$ descends to a map $\hol \colon \Omega_{X_0} X \to G$ such that the diagram
\[
\xy
(0,25)*+{\Omega_{P_0}P}="1";
(50,25)*+{PG}="2";
(0,0)*+{\Omega_{X_0}X}="3";
(50,0)*+{G}="4";
{\ar^{\hol} "1";"2"};
{\ar^{\hol} "3";"4"};
{\ar^{} "1";"3"};
{\ar^{} "2";"4"};
\endxy
\]
commutes.
In particular, $\hol$ is a classifying map for $\Omega_{P_0}P$ and agrees with the traditional notion of holonomy (see also \cite[Section 3]{CM1}).
Recalling the map $\eta \colon M \to \Omega_{M_0}(M\x\sone)$ of \eqref{eqn:etamap} used in the definition of the caloron correspondence one has
\vspace{11pt}
\begin{lemma}[\cite{MV,V}]
\label{lemma:hol}
Let $\sQ \to M$ be an $\Omega G$-bundle with connection $\sA$ and Higgs field $\Phi$.
Write $Q \to M\x\sone$ for its caloron transform, with the framed connection $A$.
Then the Higgs field holonomy is related to the usual holonomy (with respect to $A$) by
\[
\hol_\Phi = \hol\circ\,\eta.
\]
\end{lemma}
For a proof of this fact see \cite[Lemma 3.1.1]{V}.
In \cite[Lemma 3.1.2]{V} it is also shown that\footnote{in fact, the argument presented there proves this only for the case $k= 4$ but the argument is sufficiently general to hold for all $k$.}
\vspace{11pt}
\begin{lemma}[\cite{V}]
\label{lemma:etaevform}
For differential $k$-forms on $M\x\sone$
\[
\eta^\ast \widehat{\int_\sone} \ev^\ast = \widehat{\int_\sone}
\]
where $\ev\colon \Omega_{M_0}(M\x\sone)\x\sone \to M\x\sone$ is the evaluation map $(\gamma,\theta)\mapsto \gamma(\theta)$.
\end{lemma}
With these results, one is able to prove
\vspace{11pt}
\begin{proposition}
\label{prop:canform}
Let $\sQ \to M$ be an $\Omega G$-bundle.
For any $f \in I^k(\g)$, and choice of $\Omega G$-connection $\sA$ and Higgs field $\Phi$ on $\sQ$, there is a canonical choice of $(2k-2)$-form $\chi$ on $M$ such that
\[
s_f(\sA,\Phi) = \hol_\Phi^\ast \tau_f + d\chi.
\]
\end{proposition}
\begin{proof}
The proof is essentially that of \cite[Proposition 3.2.7]{V} and \cite[Proposition 4.15]{MV}.
Let $Q \to M\x\sone$ be the caloron transform, with framing $s_0 \in \Gamma(M_0,Q)$ and framed connection $A$.
Pulling back by the evaluation map $\ev \colon \Omega_{M_0}(M\x\sone)\x\I \to M\x\sone$ one obtains the $G$-bundle $\ev^\ast \!Q \to \Omega_{M_0}(M\x\sone)\x\I$.
This bundle is in fact trivial, since it has a section
\[
h \colon \Omega_{M_0}(M\x\sone)\x\I \lo \ev^\ast \!Q
\] 
given by $h(p,t) := \widehat p(t)$, where $\widehat p$ is the horizontal lift of $p$ through $s_0(p(0))$ with respect to the connection $\uc$ on $Q$.
Setting
\[
\hc := h^\ast \ev^\ast \uc,
\]
since $\Omega_{M_0}(M\x\sone)\x\I$ is a product manifold one may write
\[
h^\ast \ev^\ast \uf = -\frac{\d}{\d t} \hc \wedge dt + \hf = -\d_t \hc\wedge dt +\hf
\]
where $\uf$ is the curvature of $\uc$ and, if $\varsigma_t \colon p \mapsto (p,t)$ is the slice map, $\varsigma_t^\ast \hf$ is the curvature of $\varsigma_t^\ast \hc$.

Putting this aside for one moment, considering the Chern-Weil forms on $M\x\sone$ one has
\[
s_f(\sA,\Phi)  = \widehat{\int_\sone} cw_f(\uc) = \eta^\ast \widehat{\int_\sone} \ev^\ast cw_f(\uc) = \eta^\ast \widehat{\int_\sone}  cw_f(\ev^\ast \!\uc)
\]
for any $f \in I^k(\g)$ by Lemma \ref{lemma:etaevform}.
Inserting the above expression for $h^\ast \ev^\ast F$ and treating $\sone$ as $\I$ with endpoints identified\footnote{so as to avoid an excess of factors of $2\pi$ that are integrated out in any case.} gives
\[
c_f:= \widehat{\int_\sone} cw_f(\ev^\ast \!\uc) = -k \int_0^1 f\big(\d_t\hc,\underbrace{\hf,\dotsc,\hf}_{\text{$k-1$ times}}\big)\, dt.
\] 
Using the formula $\hf = d\hc +\frac{1}{2}[\hc,\hc]$, this becomes
\[
c_f = -k \sum_{i=0}^{k-1} \frac{1}{2^i} \frac{(k-1)!}{i!(k-1-i)!}\int_0^1 f\big(\d_t\hc,\underbrace{d\hc,\dotsc,d\hc}_{\text{$k-1-i$ times}},\underbrace{[\hc,\hc],\dotsc,[\hc,\hc]}_{\text{$i$ times}}\big)\, dt\\
\]
For $i= 0,\dotsc, k-1$ set
\[
f_i :=  f\big(\d_t\hc,\underbrace{d\hc,\dotsc,d\hc}_{\text{$k-1-i$ times}},\underbrace{[\hc,\hc],\dotsc,[\hc,\hc]}_{\text{$i$ times}}\big),
\]
which, acting on tangent vectors at a point $p\in \Omega_{M_0}(M\x\sone)$, is a map $\sone \to \RR$.
Using the shorthand $\hc^j$ to mean $\hc$ repeated  $j$ times as an argument of $f$, using the Leibniz rule for $d$ gives
\begin{multline*}
\int_0^1 f_i(t)\, dt = \int_0^1 f\big(d(\d_t\hc),\hc,d\hc^{k-2-i},[\hc,\hc]^i \big)\,dt 
\\ + i \int_0^1 f\big( \d_t\hc,\hc,d\hc^{k-2-i},d[\hc,\hc],[\hc,\hc]^{i-1} \big)\, dt 
-d \int_0^1 f\big(\d_t\hc,\hc,d\hc^{k-2-i},[\hc,\hc]^i \big)\,dt
\end{multline*}
and integrating by parts in the $\I$ direction gives
\begin{multline*}
\int_0^1 f_i(t) \,dt = F_i(1) - F_i(0) - (k-1-i) \int_0^1 f\big(\hc,\d_t(d\hc),d\hc^{k-2-i},[\hc,\hc]^i \big)\,dt\\
-i \int_0^1 f\big(\hc,d\hc^{k-1-i},\d_t[\hc,\hc],[\hc,\hc]^{i-1} \big)\,dt
\end{multline*}
where $F_i := f\big(\hc,d\hc^{k-1-i},[\hc,\hc]^i \big)$.

By the $\ad$-invariance of $f$,
\begin{multline*}
f\big( \d_t\hc,\hc,d\hc^{k-2-i},d[\hc,\hc],[\hc,\hc]^{i-1} \big) \\
%%%
= 2 f\big([\d_t\hc,\hc],\hc,d\hc^{k-1-i},[\hc,\hc]^{i-1} \big) -2  f\big(\d_t\hc,d\hc^{k-i-1},[\hc,\hc]^{i} \big)\\
%%%
 \quad\qquad- 2(k-2-i)  f\big(\d_t\hc,\hc,[d\hc,\hc],d\hc^{k-2-i},[\hc,\hc]^{i-1} \big)\\
%%%
= f\big(\d_t[\hc,\hc],\hc,d\hc^{k-1-i},[\hc,\hc]^{i-1} \big)  -2  f\big(\d_t\hc,d\hc^{k-i-1},[\hc,\hc]^{i} \big)\\
%%%
 \quad\qquad- (k-2-i)  f\big(\d_t\hc,\hc,d[\hc,\hc],d\hc^{k-2-i},[\hc,\hc]^{i-1} \big).
\end{multline*}
Using this and combining the two expressions for $\int_0^1 f_i(t)\, dt$ above gives
\begin{multline*}
(k+i) \int_0^1 f_i(t)\,dt = F_i(1) - F_i(0) -(k-1-i) d \int_0^1  f\big(\d_t\hc,\hc,d\hc^{k-2-i},[\hc,\hc]^i \big)\,dt.
\end{multline*}
This gives the expression
\[
c_f =  \sum_{i=0}^{k-1} \left[ -\frac{1}{2^i} \frac{k(k-1)!}{i!(k-1-i)!}\frac{1}{k+i} \left\{F_i(1) - F_i(0) \right\}\right] + d\beta
\]
where
\[
\beta := \sum_{i=0}^{k-1} \frac{k}{2^i} \frac{(k-1)!}{i!(k-2-i)!} \int_0^1  f\big(\d_t\hc,\hc,d\hc^{k-2-i},[\hc,\hc]^i \big)\,dt.
\] 
Notice that $h(p,0) = h(p,1) \hol(p)$ so that
\[
\hc_0 = \ad(\hol^{-1})\hc_1 + \hol^{-1} d\hol
\]
where $\hc_t:= \varsigma_t^\ast \hc$.
Since $h(p,0) \in P_0 = s_0(M_0)$ for all $p \in \Omega_{M_0}(M\x\sone)$, the composition $\ev\circ \,h\circ \varsigma_0$ sends the loop $p$ to $s_0(p(0))$.
Recall that $A$ is framed with respect to $s_0$,  so $\hc_0 = 0$ which gives 
\[
\hc_1 = - (d\hol)\hol^{-1}
\]
and $F_i(0) = 0$ for each $i = 0,\dotsc,k-1$.
Thus, if $\widehat\Theta$ is the right-invariant Maurer-Cartan form on $G$
\[
\hc_1 = -\hol^\ast\widehat\Theta
\]
and hence
\[
d\hc_1 = - \frac{1}{2}\left[(d\hol)\hol^{-1}, (d\hol)\hol^{-1}\right]
\]
since $d\widehat\Theta +\tfrac{1}{2}[\widehat\Theta,\widehat\Theta] = 0$.
Using $\Theta_g = \ad(g^{-1})\widehat\Theta_g$ and the $\ad$-invariance of $f$ gives
\[
F_i(1) =  f\big(\hc,d\hc^{k-1-i},[\hc,\hc]^i\big) =  -\left(-\frac{1}{2}\right)^{k-1-i} \hol^\ast f\big(\Theta,[\Theta,\Theta]^{k-1}\big) 
\]
so that
\[
c_f = \hol^\ast \left(-\frac{1}{2}\right)^{k-1} k\sum_{i=0}^{k-1} \frac{(k-1)!}{i!(k-1-i)!}\frac{(-1)^i}{k+i} f\big(\Theta,[\Theta,\Theta]^{k-1}\big) + d\beta.
\]
Recalling the beta function
\[
\frac{(k-1)!(k-1)!}{(2k-1)!} = B(k,k) := \int_0^1 t^{k-1}(1-t)^{k-1} \,dt = \sum_{i=0}^{k-1} \frac{(k-1)!}{i!(k-1-i)!}\frac{(-1)^i}{k+i}
\]
finally gives
\[
c_f = \hol^\ast  \left(-\frac{1}{2}\right)^{k-1} \frac{k!(k-1)!}{(2k-1)!} f\big(\Theta,[\Theta,\Theta]^{k-1}\big) + d\beta.
\]
Pulling back by $\eta$ and using $\hol_\Phi = \hol \circ\,\eta$ gives
\[
s_f(\sA,\Phi) = \hol_\Phi^\ast  \left(-\frac{1}{2}\right)^{k-1} \frac{k!(k-1)!}{(2k-1)!} f\big(\Theta,[\Theta,\Theta]^{k-1}\big) + d\chi
\]
and hence the result, since
\[
\chi := \eta^\ast \beta = \eta^\ast \sum_{i=0}^{k-1} \frac{k}{2^i} \frac{(k-1)!}{i!(k-2-i)!} \int_0^1  f\big(\d_t\hc,\hc,d\hc^{k-2-i},[\hc,\hc]^i \big)\,dt
\]
is a $(2k-2)$-form on $M$ determined entirely by $\sA$ and $\Phi$.
Note that since $\widehat A$ depends on both $\sA$ and $\Phi$, the form $\chi$ also depends on both of these data.
\end{proof}

Unwinding the proof above gives
\vspace{11pt}
\begin{corollary}
\label{cor:canpullback}
If $g \colon N \to M$ is smooth, on the pullback $g^\ast \sQ \to N$ equipped with the pullback connection $g^\ast\sA$ and Higgs field $g^\ast \Phi$ one has
\[
s_f(g^\ast\sA,g^\ast\Phi) = \hol_{g^\ast\Phi}^\ast\tau_f + d\chi'
\]
for some canonical form $\chi'$.
Then $\chi' = g^\ast\chi$, with $\chi$ as in Proposition \ref{prop:canform}.
\end{corollary}

%\newpage
%\mbox{}
\chapter{The string potentials\label{ch:three}}

The discussion of the caloron correspondence in Chapter \ref{ch:two} culminated in the construction of characteristic classes for loop group bundles.
Since these classes live in the cohomology of the base manifold and depend only on the isomorphism class of the bundle, they define topological invariants.
More importantly, the construction provides explicit differential form representatives for these classes, given some data on the total space.

This chapter is devoted to the construction of new objects on loop group bundles: the \emph{string potentials}.
The string potentials are differential forms measuring the dependence of the string forms  (of some given loop group bundle) on a particular choice of connection and Higgs field.
In this manner, the string potentials may be thought of as `looped' versions of the well-known Chern-Simons forms.

%%%%%%%%%%%%%%%%%%%%%%%%%%%%%%%%%%%%%%%%%%%%%%%%%%%%%%%%%%%%%%%%%%%%%%%%%%%%%%%%%%%%
%%%%%%%%%%%%%%%%%%%%%%%%%%%%%%%%%%%%%%%%%%%%%%%%%%%%%%%%%%%%%%%%%%%%%%%%%%%%%%%%%%%%
%%%%%%%%%%%%%%%%%%%%%%%%%%%%%%%%%%%%%%%%%%%%%%%%%%%%%%%%%%%%%%%%%%%%%%%%%%%%%%%%%%%%

\section{Construction of the string potentials}
\label{S:antistring}

In the setting of finite-dimensional principal bundles, recall that a connection $A$ on the $G$-bundle $P\to M$ determines the Chern-Weil forms: closed real-valued differential forms
\[
cw_f(A) := f(\underbrace{F,\dotsc,F}_{\text{$k$ times}}) \in \Omega^{2k}(M),
\]
where $f \in I^k(\g)$ and $F$ is the curvature form of $A$.
As discussed in Section \ref{SS:stringclasses} (in particular Theorem \ref{theorem:chernweil}) the cohomology class associated to $cw_f(A)$ is independent of the chosen connection $A$ and defines a characteristic class of the bundle.
The characteristic classes that arise in this manner are primary \emph{topological} invariants of the bundle $P\to M$.

In \cite{CS}, the study of the dependence of the forms $cw_f(A)$ on the connection $A$ led to the discovery of secondary \emph{geometric} invariants called Chern-Simons forms.
The Chern-Simons form associated to the connection $A$ and polynomial $f \in I^k(\g)$ is the $(2k-1)$-form on $P$ given by
\[
\CS_f(A) := \sum_{j=1}^{k-1} \left(-\frac{1}{2}\right)^{j}\frac{k!(k-1)!}{(k+j)! (k-1-j)!} f(A, \underbrace{[A,A],\dotsc,[A,A]}_{\text{$j$ times}}, \!\!\!\!\underbrace{F,\dotsc,F}_{\text{$k-j-1$ times}}\!\!\!\! ).
\]
\cite[equation (3.5)]{CS}.
An important property of the Chern-Simons forms is that
\[
d\CS_f(A) = \pi^\ast \cw_f(A)
\]
with $\pi \colon P\to M$ the projection \cite[Proposition 3.2]{CS}.
The discovery of the Chern-Simons forms motivated the work of Cheeger and Simons \cite{ChS1}, in which the theory of differential characters was developed.
Differential characters provide a refinement of integral cohomology that naturally includes differential forms: refer to \cite{Green} for a detailed treatment.
In \cite{ChS1} it is shown that the Chern-Simons forms descend to the base manifold $M$ as differential characters and, in fact, that these differential characters define differential characteristic classes for the $G$-bundle $P\to M$.

The initial situation for loop group bundles is not too dissimilar to the finite-dimensional setting: associated to a connection $\sA$ and Higgs field $\Phi$ on the $LG$-bundle $\sQ \to M$ there are the string forms
\[
s_f(\sA,\Phi) = k \int_\sone f(\nabla\Phi,\underbrace{\sF,\dotsc,\sF}_{\text{$k-1$ times}}) \in \Omega^{2k-1}(M)
\]
that play the role of the Chern-Weil forms.
Before introducing the \emph{total} string potentials, which play the role of the Chern-Simons forms in the loop group setting, the discussion  focusses on the \emph{relative} string potentials.
These latter objects are constructed by analogy with the relative Chern-Simons forms---see \cite[Appendix A]{FreedR} for a discussion of these objects.

The following result is used in the construction of the total string potentials.
\vspace{11pt}
\begin{lemma}
\label{lemma:intcube}
Let $\pi\colon \sQ \to M$ be a loop group bundle with caloron transform the $G$-bundle $\Pi\colon Q \to M\x\sone$.
Then the diagram
\[
\xy
(0,25)*+{\Omega^\bullet(\sQ\x\sone\x G)}="1";
(50,25)*+{\Omega^{\bullet-1}(\sQ\x G)}="2";
(0,0)*+{\Omega^\bullet(\sQ\x\sone)}="3";
(50,0)*+{\Omega^{\bullet-1}(\sQ)}="4";
(-40,13)*+{\Omega^\bullet(Q)}="5";
(-40,-12)*+{\Omega^{\bullet}(M\x\sone)}="6";
(10,-12)*+{\Omega^{\bullet-1}(M)}="7";
{\ar^{\widehat{\int_\sone}} "1";"2"};
{\ar^{s^\ast} "2";"4"};
{\ar^{(s\x\id)^\ast} "1";"3"};
{\ar^{\widehat{\int_\sone}} "3";"4"};
{\ar^{} "5";"1"};
{\ar^{(\pi \x\id)^\ast} "6";"3"};
{\ar^<<<<<<<<<<<<<<<<<{\widehat{\int_\sone}} "6";"7"};
{\ar^{\Pi^\ast} "6";"5"};
{\ar^{\pi^\ast} "7";"4"};
\endxy
\]
commutes, with $s$ the section $q \mapsto (q,1)$ and the map $\Omega^\bullet(Q) \to \Omega^\bullet(\sQ\x\sone\x G)$ given by the pullback of the quotient map.
\end{lemma}
\begin{proof}
The result follows from Lemma \ref{lemma:intpullsquare} and the fact that
\[
\xy
(0,25)*+{\sQ\x\sone\x G}="1";
(50,25)*+{Q}="2";
(0,0)*+{\sQ\x\sone}="3";
(50,0)*+{M\x\sone}="4";
{\ar^{} "1";"2"};
{\ar^{\Pi} "2";"4"};
{\ar_{s\x\id} "3";"1"};
{\ar^{\pi\x\id} "3";"4"};
\endxy
\]
commutes, which is consequence of the construction of the caloron transform functor $\cC$ and the fact that pullback is contravariantly functorial.
\end{proof}

\begin{remark}
Take the loop group bundle $\sQ \to M$ equipped with connection $\sA$ and Higgs field $\Phi$.
Using the formula \eqref{eqn:con:lgcalfinconn} for the caloron transform connection $A$ on $Q:= \cC(\sQ)$ gives
\[
(s\x\id)^\ast A_{(q,\theta)} = \sA_q(\theta) + \Phi(q)(\theta)d\theta
\]
on $\sQ\x\sone$ and hence, using \eqref{eqn:curv},
\[
(s\x\id)^\ast \left(dA + \tfrac{1}{2}[A,A] \right) = \sF_q(\theta) + \nabla\Phi_q(\theta) \wedge d\theta.
\]
The diagram of Lemma \ref{lemma:intcube} then implies that whenever an expression on $\sQ$ or $M$ is given in terms of the connection $A$ on $Q$ or its curvature $F$ and then integrated over the fibre, the (schematic) expressions
\begin{equation}
\label{eqn:localconn}
A = \sA + \Phi d\theta \;\;\mbox{ and }\;\; F = \sF + \nabla\Phi\wedge d\theta
\end{equation}
may be used instead of \eqref{eqn:con:lgcalfinconn} and \eqref{eqn:curv}.
\end{remark}

In order to construct the string potentials, one requires a notion of \emph{smooth $n$-cubes} in the spaces of connections and Higgs fields on the loop group bundle $\sQ \to M$.
In general there is no obvious smooth structure on these spaces so a key concept in the construction of the string potentials is the following
\vspace{11pt}
\begin{definition}
\label{defn:smoothconn}
Let $P \to M$ be a fixed $G$-bundle.
A \emph{smooth $n$-cube of $G$-connections} on $P$ is a $G$-connection $\hat A$ on the $G$-bundle $P\x \I^n \to M\x \I^n$ satisfying $\imath_{X} \hat A = 0$ for any vector field $X$ on $P\x\I^n$ that is vertical for the projection $P\x\I^n \to P$.
Here $\imath_X$ denotes contraction with the vector field $X$.
\end{definition}

Consequently, any such $\hat A$ may be written at $(p,t_1,\dotsc,t_n) \in P\x\I^n$ as
\[
\hat{A}_{(p,t_1,\dotsc,t_n)}  = f(p,t_1,\dotsc,t_n) \pr_1^\ast \omega_{(p,t_1,\dotsc,t_n)}
\]
for some $1$-form $\omega$ on $P$ and smooth function $f$ on $P\x\I^n$.
Taking the Lie derivative $\cL$ gives
\begin{equation}
\label{eqn:connddt}
\left(\frac{d}{dt_i}\hat A\right)\mspace{-4mu}{}_{(p,t_1,\dotsc,t_n)}:= \left(\cL_{\d_{t_i}} \hat{A}\right)\mspace{-3mu}{}_{(p,t_1,\dotsc,t_n)} = \frac{\d \!f}{\d t_i}(p,t_1,\dotsc,t_n) \pr_1^\ast \omega_{(p,t_1,\dotsc,t_n)}
\end{equation}
where $\d_{t_i}$ is the vector field on $P\x\I^n$ generated by the $i$-th coordinate function $t_i$.

In the case that $\sQ \to M$ is a loop group bundle, there is the related notion
\vspace{11pt}
\begin{definition}
\label{defn:smoothhiggs}
A \emph{smooth $n$-cube of Higgs fields} on the loop group bundle $\sQ \to M$ is a Higgs field $\hat\Phi$ on $\sQ\x\I^n\to M\x\I^n$ with respect to the loop group action
\[
(q,t_1,\dotsc,t_n) \cdot \gamma:= (q\gamma,t_1,\dotsc,t_n).
\]
\end{definition}

As with \eqref{eqn:connddt},
\begin{equation}
\label{eqn:higgsddt}
\left(\frac{d}{dt_i}\,\hat\Phi \right)(q,t_1,\dotsc,t_n) := \left( \cL_{\d_{t_i}} \hat\Phi \right)(q,t_1,\dotsc,t_n) = \frac{\d \hat\Phi}{\d t_i}(q,t_1,\dotsc,t_n).
\end{equation}
For any product manifold of the form $X\x\I^n$ define the smooth \emph{slice map}
\begin{align*}
\varsigma_{(t_1,\dotsc,t_n)} \colon X &\lo X\x\I^n\\
x &\longmapsto (x,t_1,\dotsc,t_n).
\end{align*}
Fix a loop group bundle $\sQ \to M$, recall the notation $\cH_\sQ$ for the space of Higgs fields on $\sQ$ (Definition \ref{defn:higgs}) and write $\cA_\sQ$ for the affine space of connections on $\sQ$.
\vspace{11pt}
\begin{definition}
\label{defn:smoothpath}
A \emph{smooth map} $\gamma \colon \I^n \to \cA_\sQ\x\cH_\sQ$ is an assignment
\[
\gamma\colon {(t_1,\dotsc,t_n)} \longmapsto \left(\sA_{(t_1,\dotsc,t_n)},\Phi_{(t_1,\dotsc,t_n)}\right)
\]
with
\[
\sA_{(t_1,\dotsc,t_n)} = \varsigma_{(t_1,\dotsc,t_n)}^\ast \sA^\gamma\;\;\mbox{ and }\;\;\Phi_{(t_1,\dotsc,t_n)} = \varsigma_{(t_1,\dotsc,t_n)}^\ast\Phi^\gamma
\] 
for some smooth $n$-cube of connections $\sA^\gamma$ and smooth $n$-cube of Higgs fields $\Phi^\gamma$ on $\sQ$.
Thus, specifying such a smooth map $\gamma$ is equivalent to specifying smooth $n$-cubes $\sA^\gamma$ and $\Phi^\gamma$ in $\cA_\sQ$ and $\cH_\sQ$.
\end{definition}
\vspace{11pt}
\begin{remark}
When the base manifold $M$ is compact, $\cA_\sQ$ and $\cH_\sQ$ may be given Fr\'{e}chet manifold structures as in Appendix \ref{app:frechet} or \cite{H}.
In this case, a map $\gamma$ is smooth in the sense of Definition \ref{defn:smoothpath} if and only if it is smooth as a map between (Fr\'{e}chet) manifolds.
\end{remark}

%%%%%%%%%%%%%%%%%%%%%%%%%%%%%%%%%%%%%%%%%%%%%%%%%%%%%%%%%%%%%%%%%%%%%%%%%%%%%%%%%%%%
%%%%%%%%%%%%%%%%%%%%%%%%%%%%%%%%%%%%%%%%%%%%%%%%%%%%%%%%%%%%%%%%%%%%%%%%%%%%%%%%%%%%
%%%%%%%%%%%%%%%%%%%%%%%%%%%%%%%%%%%%%%%%%%%%%%%%%%%%%%%%%%%%%%%%%%%%%%%%%%%%%%%%%%%%

\subsection{The relative string potentials}
With all this machinery in place, one is able to define the relative string potentials
\vspace{11pt}
\begin{definition}
\label{defn:anti}
For any $f \in I^k(\g)$ and smooth path $\gamma \colon \I \to \cA_\sQ \x\cH_\sQ$, the associated \emph{(relative) string potential} is
\[
S_f(\gamma) := k \int_0^1 \int_\sone \Big( (k-1) f\big(\sB_t , \underbrace{\sF_t,\dotsc,\sF_t}_{k-2\;\text{times}}, \nabla_t\Phi \big) + f\big(\underbrace{\sF_t,\dotsc,\sF_t}_{k-1\;\text{times}},\varphi_t \big)\Big) \,dt \in \Omega^{2k-2}(\sQ)
\]
where $\sF_t$ is the curvature of $\sA_t$, $\nabla_t\Phi$ is the Higgs field covariant derivative of $\Phi_t$ and
\[
\sB_t := \varsigma_t^\ast(\cL_{\d_t} \sA^\gamma)\;\;\mbox{ and }\;\;\varphi_t := \varsigma_t^\ast(\cL_{\d_t} \Phi^\gamma)
\]
are the `time derivatives' of $\sA_t$ and $\Phi_t$ at $t\in \I$.
\end{definition}
\vspace{11pt}
\begin{theorem}
\label{theorem:antistring}
The string potential of Definition \ref{defn:anti} descends to a form on $M$, also called $S_f(\gamma)$, which satisfies the equation
\[
dS_f(\gamma) = s_f(\sA_1,\Phi_1) - s_f(\sA_0,\Phi_0).
\]
\end{theorem}
\begin{proof}
Start with the smooth path $\gamma \colon \I \to \cA_\sQ\x\cH_\sQ$, with $\sA^\gamma$ and $\Phi^\gamma$ the corresponding smooth $1$-cubes (i.e. connection and Higgs field) on $\sQ\x \I$.
Write $\sF^\gamma$ for the curvature of $\sA^\gamma$ and $\nabla\Phi^\gamma$ for the Higgs field covariant derivative of $\Phi^\gamma$.
Cartan's Magic Formula ($\{\imath,d\} = \cL$) gives
\begin{align*}
\imath_{\d_t} \sF^\gamma &= \imath_{\d_t} \!\left(d\sA^\gamma +\tfrac{1}{2} [\sA^\gamma,\sA^\gamma ] \right) = \cL_{\d_t} \sA^\gamma\text{ and}\\
%%%
\imath_{\d_t} \nabla\Phi^\gamma &= \imath_{\d_t} \!\left(d\Phi^\gamma +[\sA^\gamma,\Phi^\gamma] - \d\sA^\gamma \right) = \cL_{\d_t}\Phi
\end{align*}
so that $\varsigma_t^\ast \imath_{\d_t} \sF^\gamma = \sB_t$ and $\varsigma_t^\ast \imath_{\d_t} \nabla\Phi^\gamma = \varphi_t$.
Then 
\begin{align*}
\varsigma^\ast_t \imath_{\d_t} f\big((\sF^\gamma)^{k-1},\nabla\Phi^\gamma\big) &= (k-1) \varsigma_t^\ast f\big( \imath_{\d_t} \sF^\gamma, (\sF^\gamma)^{k-2},\nabla\Phi^\gamma \big) + \varsigma_t^\ast f\big( (\sF^\gamma)^{k-1},\imath_{\d_t}\!
 \nabla\Phi^\gamma \big)\\
%%%
&= (k-1) f\big( \sB_t,\sF_t^{k-2},\nabla_t\Phi \big) + f\big(\sF_t^{k-1},\varphi_t\big)
\end{align*}
so that
\begin{equation}
\label{eqn:alternateanti}
S_f(\gamma) = \int_0^1  \varsigma_t^\ast \imath_{\d_t} s_f(\sA^\gamma,\Phi^\gamma)\,dt.
\end{equation}
Since the string form $s_f(\sA^\gamma,\Phi^\gamma)$ lives on $M\x\I$, this shows that $S_f(\gamma)$ descends to $M$.

The proof is completed by the calculation
\begin{align*}
dS_f(\gamma) &= \int_0^1  d \varsigma_t^\ast \imath_{\d_t} s_f(\sA^\gamma,\Phi^\gamma)\,dt\\
%%%
&= \int_0^1 \varsigma_t^\ast \left[\cL_{\d_t} s_f(\sA^\gamma,\Phi^\gamma)\right]\,dt\\
&= \int_0^1  \,\Bigg[\! \frac{d}{dt} s_f(\sA^\gamma,\Phi^\gamma) \! \Bigg]\!(t)\,dt\\
&= s_f(\sA_1,\Phi_1) - s_f(\sA_0,\Phi_0),
\end{align*}
using Stokes' Theorem and Cartan's Magic Formula.
\end{proof}

There is a close relationship \eqref{eqn:stringform} between string forms and Chern-Weil forms, so it seems reasonable to expect a similar relationship between the relative string potentials and the relative Chern-Simons forms.
Via the caloron correspondence one can see that the smooth $n$-cubes $\hat\sA$ of connections and $\hat\Phi$ of Higgs fields on $\sQ$ determine a smooth $n$-cube $\hat A$ of connections on $Q:= \cC(\sQ)$ and conversely.
Thus, a smooth map
\[
\gamma \colon \I^n \lo \cA_\sQ\x\cH_\sQ
\]
is equivalent to a smooth map
\[
\cC\gamma \colon \I^n \lo \cA_Q,
\]
with $\sA_Q$ the space of all connections on $Q$.
The relative Chern-Simons form given by the map $\cC\gamma$ associated to $f \in I^k(\g)$ is
\[
\CS_f(\cC\gamma) = k\int_0^1 f\big(B_t, \underbrace{F_t,\dotsc,F_t}_{\text{$k-1$ times}} \big)\,dt
\]
where $F_t = \varsigma_t^\ast \hat F$ is the curvature of the connection $A_t$ at `time' $t\in\I$ and $B_t = \varsigma_t^\ast \cL_{\d_t} \hat A$ is the time derivative at $t$.
Using \eqref{eqn:localconn},
\begin{align*}
\widehat{\int_\sone} \CS_f(\cC\gamma) &= k \widehat{\int_\sone} \int_0^1 f\big(\sB_t+\varphi_t d\theta, (\sF_t + \nabla\Phi\wedge d\theta)^{k-1} \big)\,dt \\
%%%
&=k \int_0^1 \left[\widehat{\int_\sone} \left( f\big(\varphi_t, \sF_t^{k-1} \big)\wedge d\theta + (k-1) f\big(\sB_t, \sF_t^{k-2},\nabla_t\Phi \big)\wedge d\theta \right) \right]dt \\
%%%
&=k \int_0^1 \int_\sone \Big( (k-1) f\big(\sB_t ,\sF_t^{k-2}, \nabla_t\Phi \big) + f\big(\sF_t^{k-1},\varphi_t \big)\Big)\,dt
\end{align*}
so that
\begin{equation}
\label{eqn:csanti}
S_f(\gamma) = \widehat{\int_\sone} \CS_f(\cC\gamma),
\end{equation}
which is the direct analogue of the expression \eqref{eqn:stringform} for the string potentials.

Suppose that $\psi \colon N \to M$ is a smooth map and $\sQ \to M$ is a loop group bundle.
Given a smooth path $\gamma \colon \I \to \cA_\sQ\x\cH_\sQ$, one may pull back the corresponding smooth $1$-cubes to obtain a smooth path $\psi^\ast\gamma \colon \I \to \cA_{\psi^\ast\sQ} \x\cH_{\psi^\ast\sQ}$.
It is immediate that
\begin{equation}
\label{eqn:stringpotpullback}
S_f(\psi^\ast\gamma) = \psi^\ast S_f(\gamma)
\end{equation}
for any $f \in I^k(\g)$ so the relative string potentials are natural.

The following result shows that, modulo exact forms, the string potential $S_f(\gamma)$ depends only on the endpoints of the smooth path $\gamma$.
This result is key to the construction of the $\Omega$ model for odd differential $K$-theory in Chapter \ref{ch:five}.
\vspace{11pt}
\begin{proposition}
\label{prop:closedend}
If $\gamma_0,\gamma_1 \colon \I \to \cA_\sQ\x\cH_\sQ$ are smooth paths with the same endpoints then $S_f(\gamma_1) - S_f(\gamma_0)$ is exact.
\end{proposition}
\begin{proof}
An explicit antiderivative is constructed by taking $\Gamma \colon \I^2 \to \cA_\sQ \x\cH_\sQ$, which is smooth in the sense of Definition \ref{defn:smoothpath}, such that $\Gamma(t,0) = \gamma_0(t)$ and $\Gamma(t,1) = \gamma_1(t)$.
Write $\Gamma(t,s) = \gamma_s(t)$, noting that $\gamma_s \colon \I \to \cA_\sQ\x\cH_\sQ$ is then a smooth path for each $s\in\I$.
In the case that $\gamma_0$ and $\gamma_1$ have the same endpoints, the map $\Gamma$ may be chosen so that the restrictions $\Gamma_i(s):= \Gamma(i,s)$ are constant for $i=0,1$.

Taking such a $\Gamma$, define
\[
H_f(\Gamma) := \int_{\I^2} \varsigma_{(t,s)}^\ast \imath_{\d_s} \imath_{\d_t} s_f(\sA^\Gamma,\Phi^\Gamma)
\]
and notice that $\varsigma_{(t,s)} = \varsigma_s \circ \varsigma_t = \varsigma_t\circ \varsigma_s$.
Using Fubini's Theorem, 
\begin{align*}
\int_{\I^2}  \varsigma_{(t,s)}^\ast \imath_{\d_s} \frac{d}{dt} s_f(\sA^\Gamma,\Phi^\Gamma) &=  \int_0^1 \varsigma_s^\ast \imath_{\d_s} \int_0^1 \varsigma_t^\ast \frac{d}{dt} s_f(\sA^\Gamma,\Phi^\Gamma) \,dt\,ds\\
%%%
&=  \int_0^1 \varsigma_s^\ast \imath_{\d_s} \!\left(s_f(\sA^{\Gamma_1},\Phi^{\Gamma_1}) - s_f(\sA^{\Gamma_0},\Phi^{\Gamma_0}) \right)\,ds \\
&=0
\end{align*}
since the constancy condition implies $\imath_{\d_s} s_f(\sA^{\Gamma_i},\Phi^{\Gamma_i})= 0$ for $i=0,1$.
Using Cartan's Magic Formula, Stokes' Theorem, Fubini's Theorem and \eqref{eqn:alternateanti} then gives
\begin{align*}
dH_f(\Gamma) &= \int_{\I^2}  \varsigma_{(t,s)}^\ast d\big(  \imath_{\d_s} \imath_{\d_t} s_f(\sA^\Gamma,\Phi^\Gamma)\big)\\
%%%
&= \int_{\I^2}  \varsigma_{(t,s)}^\ast \left(\frac{d}{ds} \imath_{\d_t}  s_f(\sA^\Gamma,\Phi^\Gamma)-\imath_{\d_s} \frac{d}{dt} s_f(\sA^\Gamma,\Phi^\Gamma)\right) \\
%%%
&= \int_{\I^2}\varsigma_t^\ast \circ \varsigma_s^\ast \left( \frac{d}{ds} \imath_{\d_t}  s_f(\sA^\Gamma,\Phi^\Gamma) \right) \\
%%%
&= \int_0^1  \varsigma_t^\ast \imath_{\d_t} \left(s_f(\sA^{\gamma_1},\Phi^{\gamma_1})- s_f(\sA^{\gamma_0},\Phi^{\gamma_0})\right)\,dt \\
%%%
&= S_f(\gamma_1) - S_f(\gamma_0)
\end{align*}
so that the difference $S_f(\gamma_1) - S_f(\gamma_0)$ is exact as required.
\end{proof}

%%%%%%%%%%%%%%%%%%%%%%%%%%%%%%%%%%%%%%%%%%%%%%%%%%%%%%%%%%%%%%%%%%%%%%%%%%%%%%%%%%%%
%%%%%%%%%%%%%%%%%%%%%%%%%%%%%%%%%%%%%%%%%%%%%%%%%%%%%%%%%%%%%%%%%%%%%%%%%%%%%%%%%%%%
%%%%%%%%%%%%%%%%%%%%%%%%%%%%%%%%%%%%%%%%%%%%%%%%%%%%%%%%%%%%%%%%%%%%%%%%%%%%%%%%%%%%

\subsection{The total string potentials}
\label{S:totalstring}
Having constructed the relative string potential forms, it is now possible to define the \emph{total} string potentials using a similar method to the construction of the Chern-Simons forms (cf. \cite[(3.1)]{CS}).
For a fixed loop group bundle $\pi \colon \sQ \to M$, the total string potentials depend only on a choice of connection and Higgs field on $\sQ$, rather than a smooth path $\gamma$, but live on the total space of the bundle rather than the base.

Consider the pullback bundle
\[
\xy
(0,25)*+{\pi^\ast\sQ}="1";
(50,25)*+{\sQ}="2";
(0,0)*+{\sQ}="3";
(50,0)*+{M}="4";
{\ar^{} "1";"2"};
{\ar^{} "2";"4"};
{\ar^{} "1";"3"};
{\ar^{\pi} "3";"4"};
\endxy
\]
which is a trivial loop group bundle over $\sQ$.
The diagonal map $\Delta \colon \sQ \to \pi^\ast \sQ$ determines the global trivialisation
\[
(q,q\gamma) \longmapsto (q,\gamma),
\]
so pulling back the product connection and the trivial Higgs field on the trivial bundle over $\sQ$ gives a connection $\Theta$ and Higgs field $\Psi$ on $\pi^\ast\sQ$.
There is a  canonical smooth path joining $(\Theta,\Psi)$ to $(\sA,\Phi)$, namely
\[
\nu \colon t \longmapsto ((1-t) \Theta + t\sA, (1-t)\Psi + t\Phi).
\]
The total string potential of the pair $(\sA,\Phi)$ is defined by pulling the form $S_f(\nu)$ on $\pi^\ast\sQ$ back to $\sQ$ using the diagonal section, that is\\

\begin{definition}
The \emph{total string potential} of the connection $\sA$ and Higgs field $\Phi$ on $\sQ$ associated to $f \in I^k(\g)$ is
\begin{equation}
\label{eqn:totalanti}
S_f(\sA,\Phi) := \Delta^\ast S_f(\nu) \in \Omega^{2k-2}(\sQ)
\end{equation}
where $\Delta$ is the diagonal map and $S_f(\nu)$ is the string potential of the path $\nu$ as per Definition \ref{defn:anti}.
\end{definition}

\vspace{11pt}
\begin{proposition}
\label{prop:totalstringclosed}
The total string potential satisfies
\[
dS_f(\sA,\Phi) = \pi^\ast s_f(\sA,\Phi),
\]
in particular, $dS_f(\sA,\Phi)$ descends to a form on $M$.
\end{proposition}
\begin{proof}
By Theorem \ref{theorem:antistring}
\[
dS_f(\nu) = \pi^\ast  s_f(\sA,\Phi) - s_f(\Theta,\Psi),
\]
which is a form on $\pi^\ast\sQ$.
Pulling back by the diagonal section $\Delta$ gives the result since all $\Theta$ and $\Psi$ terms vanish.
\end{proof}

Since the total string potential depends only on $\sA$ and $\Phi$, one might expect to be able to write $S_f(\sA,\Phi)$ in terms of $\sA$ and $\Phi$ alone.
In fact
\vspace{11pt}
\begin{proposition}
\label{prop:totalantiex}
The total string potential of the connection $\sA$ and Higgs field $\Phi$ associated to $f\in I^k(\g)$ is given by the formula
\begin{multline}
\label{eqn:totalantiex}
S_f(\sA,\Phi) =  \int_\sone \bigg[\sum_{i=0}^{k-1}c_i f\big( \Phi,[\sA,\sA]^i, \sF^{k-1-i} \big)\\
+ \sum_{i=1}^{k-1} \Big[ 2 i c_i f\big(\sA,[\sA,\Phi],[\sA,\sA]^{i-1},\sF^{k-1-i} \big) \\
- 2(k+i)c_i f\big( \sA,\nabla\Phi,[\sA,\sA]^{i-1},\sF^{k-1-i} \big) \Big]\bigg]
\end{multline}
where the coefficients are
\[
c_i = \left(-\frac{1}{2}\right)^i\frac{k!(k-1)!}{(k+i)!(k-1-i)!}.
\]
\end{proposition}
\begin{proof}
Pulling back by $\Delta$ eliminates all $\Theta$ and $\Psi$ terms, so comparing with Definition \ref{defn:anti} one obtains
\[
S_f(\sA,\Phi) =k \int_0^1  \int_\sone \Big( (k-1) f\big(\sB_t , \sF_t^{k-2}, \nabla_t\Phi \big) + f\big(\sF_t^{k-1},\varphi_t \big)\Big)\, dt
\]
with $\sA_t = t\sA$ and $\Phi_t = t\Phi$ or, using \eqref{eqn:csanti} and \eqref{eqn:localconn},
\[
S_f(\sA,\Phi) = k \int_0^1 \left[ \widehat{\int_\sone} f\big(B_t, F_t^{k-1} \big) \right]dt
\]
where $F_t = tF -\tfrac{t}{2}(1-t) [A,A]$.
Hence
\begin{align*}
\widehat{\int_\sone} f\big(B_t, F_t^{k-1} \big) &= \widehat{\int_\sone} f\left(A,\left(tF -\tfrac{t}{2}(1-t) [A,A]\right)^{k-1}\right)\\
&= \sum_{i=0}^{k-1}(-1)^i t^{k-1}(1-t)^i \frac{(k-1)!}{2^ii!(k-1-i)!} \widehat{\int_\sone} f\left(A,[A,A]^{i},F^{k-1-i}\right).
\end{align*}
Integrating over $\I$ and recalling the beta function gives
\[
S_f(\sA,\Phi) = \sum_{i=0}^{k-1} \left(-\frac{1}{2}\right)^i\frac{k!(k-1)!}{(k+i)!(k-1-i)!} \widehat{\int_\sone} f\left(A,[A,A]^{i},F^{k-1-i}\right),
\]
which has the desired coefficients.
It remains only to use \eqref{eqn:localconn} so that one may substitute $[\sA,\sA] +2[\sA,\Phi] \wedge d\theta$ for $[A,A]$.
Substituting in this expression yields
\begin{multline*}
\sum_{i=0}^{k-1}c_i \widehat{\int_\sone}  f\left(A,[A,A]^{i},F^{k-1-i}\right) = \int_\sone \bigg[\sum_{i=0}^{k-1}c_i f\big( \Phi,[\sA,\sA]^i, \sF^{k-1-i} \big)\\
+ \sum_{i=1}^{k-1} \Big[ 2 i c_i f\big(\sA,[\sA,\Phi],[\sA,\sA]^{i-1},\sF^{k-1-i} \big) \\
- 2(k+i)c_i f\big( \sA,\nabla\Phi,[\sA,\sA]^{i-1},\sF^{k-1-i} \big) \Big]\bigg],
\end{multline*}
giving the result.
\end{proof}

As with \eqref{eqn:stringpotpullback}, if $\psi \colon \sP \to \sQ$ is a morphism of loop group bundles one has
\[
\psi^\ast S_f(\sA,\Phi) = S_f(\psi^\ast\sA,\psi^\ast\Phi)
\]
for any $f \in I^k(\g)$ so that the total string potential forms are natural.

As some first examples of applications of the total string potentials:
\vspace{11pt}
\begin{example}
\label{example:liftingbundlegerbe}
Murray and Stevenson \cite[Theorem 5.1]{MS} obtained a differential form representative for Killingback's string class (Example \ref{example:kstring}) using the lifting bundle gerbe of the loop group bundle $\sQ \to M$.
They first gave a \emph{curving} for this bundle gerbe
\[
B = \frac{i}{2\pi} \int_\sone \Big( \tfrac{1}{2}\lan \sA,\d\sA\ran - \lan\sF,\Phi\ran \Big)
\]
which is a $2$-form on the total space of the $LG$-bundle $\pi\colon \sQ \to M$ that depends on the connection $\sA$ and Higgs field $\Phi$.
It turns out that there is a closed $3$-form $H$ on $M$ such that $dB  = \pi^\ast H$ and the cohomology class of $\tfrac{1}{2\pi i}H$ is the string class.

Using \eqref{eqn:totalantiex} on $\sA$ and $\Phi$ with $f(\cdot,\cdot) := - \frac{1}{8\pi^2} \lan\cdot,\cdot\ran$ gives
\begin{align*}
S_f(\sA,\Phi) &= -\frac{1}{8\pi^2} \int_\sone \Big(  \lan \Phi,\sF\ran -\tfrac{1}{6} \left(  \lan \Phi,[\sA,\sA] \ran + 2 \lan \sA,[\sA,\Phi] \ran\right) + \lan \sA,\nabla\Phi \ran \Big)  \\
&= -\frac{1}{4\pi^2} \int_\sone  \Big( \lan \Phi,\sF\ran -\tfrac{1}{2} \lan \sA,\d\sA \ran\Big)   - d \int_\sone f(\sA,\Phi) 
\end{align*}
so that
\[
S_f(\sA,\Phi) = \frac{1}{2\pi i} B+\mbox{exact}.
\]
Thus, the total string potentials recover the curving of the lifting bundle gerbe of $\sQ \to M$.
\end{example}
\vspace{11pt}
\begin{example}
\label{example:cohgenerators}
Consider the total string potential $S_f(\sA_\infty,\Phi_\infty)$ associated to the standard connection and Higgs field on the path fibration $PG \to G$.
Denote by $\imath \colon \Omega G \to PG$ the fibre inclusion (over, say, the identity of $G$) so that $\imath^\ast\sA_\infty =\Theta$, the Maurer-Cartan form on $\Omega G$, and $\Psi(\gamma) := (\imath^\ast\Phi_\infty)(\gamma) = \gamma^{-1}\d\gamma$ for $\gamma\in\Omega G$.
Since the forms $\nabla\Phi_\infty$ and $\sF_\infty$ are horizontal, by \eqref{eqn:totalantiex}, the $\ad$-invariance of $f$ and the fact that $[\Theta,[\Theta,\Theta]] = 0$ give
\begin{multline*}
\imath^\ast S_f(\sA_\infty,\Phi_\infty) =  \int_\sone \Big[ c_{k-1} f\left( (\Psi,[\Theta,\Theta]^{k-1} \right)\\
%%%
+2(k-1)  c_{k-1} f\left(\Theta,[\Theta,\Psi],[\Theta,\Theta]^{k-2} \right)\Big]\\
%%%
= \left(-\frac{1}{2}\right)^{k-1}\frac{k!(k-1)!}{(2k-2)!} \int_\sone f\left(\Psi ,[\Theta,\Theta]^{k-1} \right).
%%%
\end{multline*}
Notice that this form is closed by Proposition \ref{prop:totalstringclosed}, since $\imath^\ast\pi^\ast s_f(\sA_\infty,\Phi_\infty) = 0$, so it defines a class in $H^\bullet(\Omega G;\RR)$.
This construction coincides with Borel's transgression map \cite[p.~40]{B2} for the universal bundle $PG \to G$.

In the case that $G$ is simply-connected, $H^\bullet(\Omega G;\RR)$ is a polynomial algebra on the even-dimensional classes defined given by pulling back generators of the cohomology of $G$ to $\Omega G\x\sone$ by the evaluation map $\ev \colon (p,\theta) \mapsto p(\theta)$ and integrating over the fibre.
This process yields precisely the cohomology classes of the forms above \cite[Appendix 4.11]{PS}, so that the total string forms of the path fibration give generators for the cohomology of $\Omega G$.
\end{example}

%%%%%%%%%%%%%%%%%%%%%%%%%%%%%%%%%%%%%%%%%%%%%%%%%%%%%%%%%%%%%%%%%%%%%%%%%%%%%%%%%%%%
%%%%%%%%%%%%%%%%%%%%%%%%%%%%%%%%%%%%%%%%%%%%%%%%%%%%%%%%%%%%%%%%%%%%%%%%%%%%%%%%%%%%
%%%%%%%%%%%%%%%%%%%%%%%%%%%%%%%%%%%%%%%%%%%%%%%%%%%%%%%%%%%%%%%%%%%%%%%%%%%%%%%%%%%%

\section{String potentials and secondary characteristic classes}
\label{S:diffcoh}
The total string potential forms introduced above have a remarkable degree of similarity to Chern-Simons forms: total string potentials are to string forms as the Chern-Simons forms are to Chern-Weil forms.

In \cite{ChS1,CS} the Chern-Simons forms were used to construct \emph{secondary} or \emph{differential} characteristic classes for $G$-bundles.
There is some disagreement as to the precise meaning of the phrase `secondary characteristic class' in the literature: in this thesis it is taken to mean characteristic classes valued in ordinary differential cohomology that are natural with respect to connection-preserving bundle maps\footnote{in the case where $G$ is a loop group, secondary characteristic classes are required to be natural with respect to connection- and Higgs field-preserving bundle maps.}.

The similarity between the total string potential forms and the Chern-Simons forms leads one naturally to question to what extent the string potentials may be used to construct secondary characteristic classes for loop group bundles.
This section constructs such classes in a restricted setting, indicating that a more general construction should be possible.

For this discussion, it is necessary to briefly review the theory of ordinary differential cohomology.
Ordinary differential cohomology is a multiplicative differential extension of ordinary (smooth singular) cohomology with integer coefficients (see Appendix \ref{app:diff} for background on differential extensions).
The notion of differential cohomology first appeared in the guise of differential characters in the work of Cheeger and Simons \cite{ChS1}; there are now many different models for ordinary differential cohomology, see for example \cite{Bry,HLZ}.

The sense in which it is meant that a construction is a \emph{model} for ordinary differential cohomology is made clear through the following notion of a character functor, introduced in \cite{SSax}.
\vspace{11pt}
\begin{definition}
\label{defn:charfun}
A \emph{character functor} is a functor $\check F$ from the category $\Man$ of smooth manifolds (with corners) to the category $\gagrp$ of graded abelian groups that is equipped with a set of natural transformations $\{i_1,i_2,\delta_1,\delta_2\}$ such that the diagram
\[
\xy
(-60,40)*+{0}="1";
(60,40)*+{0}="2";
(-30,20)*+{H^{k-1}\RR/\ZZ}="4";
(30,20)*+{H^k\ZZ}="5";
(-60,0)*+{H^{k-1}\RR}="7";
(0,0)*+{\check F^k}="8";
(60,0)*+{H^k\RR}="9";
(-30,-20)*+{\Omega^{k-1} / \Omega^{k-1}_\ZZ}="11";
(30,-20)*+{\Omega^k_\ZZ}="12";
(-60,-40)*+{0}="14";
(60,-40)*+{0}="15";
{\ar^{} "1";"4"};;
{\ar^{} "14";"11"};
{\ar^{\beta} "7";"11"};
{\ar^{\alpha} "7";"4"};
{\ar^{-B} "4";"5"};
{\ar^{} "5";"2"};
{\ar^{i_1} "4";"8"};
{\ar^{i_2} "11";"8"};
{\ar^{d} "11";"12"};
{\ar^{\imath} "5";"9"};
{\ar^{} "12";"15"};
{\ar^{\delta_2} "8";"5"};
{\ar^{\delta_1} "8";"12"};
{\ar^{\deR} "12";"9"};
\endxy
\]
commutes and has exact diagonal sequences.
Here, $H^k R$ denotes degree $k$ cohomology with coefficients in the ring $R$, $\Omega^k_\ZZ$ are the degree $k$ differential forms with integral periods,
\[
H^{k-1}\RR \xrightarrow{\;\;\alpha\;\;} H^{k-1}\RR/\ZZ \xrightarrow{\;\;B\;\;} H^k\ZZ \xrightarrow{\;\;\imath\;\;} H^k\RR
\]
is the Bockstein exact sequence arising from the exact sequence $\ZZ\to\RR\to\RR/\ZZ$ of coefficients and
\[
H^{k-1}\RR \xrightarrow{\;\;\beta\;\;} \Omega^{k-1}/\Omega^{k-1}_\ZZ \xrightarrow{\;\;d\;\;} \Omega^k_\ZZ \xrightarrow{\;\deR\;} H^k\RR
\]
is the exact sequence arising from the de Rham theorem.
This diagram is called the \emph{character diagram} (of the character functor $\check F$).
\end{definition}

Cheeger-Simons differential characters define such a character functor (for a detailed treatment see \cite[Chapter 2]{Green} or \cite{ChS1}).
All of the models for ordinary differential cohomology are related by the following key result
\vspace{11pt}
\begin{theorem}[\cite{SSax}]
\label{theorem:chardiagunique}
Any two character functors $\{\check F,i_1,i_2,\delta_1,\delta_2\}$ and $\{\check F',i'_1,i'_2,\delta'_1,\delta'_2\}$ are equivalent via a unique natural transformation $\check F \to \check F'$ that commutes with the identity on all other parts of the character diagram.
\end{theorem}
\vspace{11pt}
\begin{remark}
This result may be viewed as a version of the Bunke-Schick uniqueness result  (Theorem \ref{theorem:bunkeschick}) for differential extensions of integral cohomology.
\end{remark}

The following discussion centres around the original differential characters of Cheeger and Simons.
Fixing a smooth manifold $M$, define the reduction map
\[
\rho \colon C^k(M;\RR) \lo C^k(M;\RR/\ZZ)
\]
on smooth real-valued singular cochains by setting
\[
\rho(c)(\sigma) := c(\sigma) \mod \ZZ
\]
for any smooth singular $k$-chain $\sigma \in C_k(M;\RR)$, writing $\widetilde c = \rho(c)$ for short.
Using the integration pairing, a $k$-form $\omega \in \Omega^k(M)$ can be viewed as a real-valued $k$-cochain so that one may take the $\mathrm{mod}\;\ZZ$ reduction of $\omega$.
\vspace{11pt}
\begin{definition}[\cite{ChS1}]
\label{defn:differentialchar}
The set of degree $k$ \emph{differential characters} of $M$ is
\[
\check{H}^k(M) = \{ \chi \in \Hom(Z_{k-1}(M), \RR/\ZZ) \mid \chi \circ \partial = \widetilde{F_{\chi}}\; \mathrm{for\; some}\; F_{\chi} \in \Omega^k(M) \}
\]
and $\check{H}^0(M):=\ZZ$.
Here, $Z_{k-1}(M)$ is the group of integral smooth singular $(k-1)$-cycles on $M$.
\end{definition}

Notice that addition on $\Hom(Z_{k-1}, \RR/\ZZ)$ induces a group operation on $\check{H}^k(M)$ and that the degree convention used here is not the original one\footnote{compare Definition \ref{defn:differentialchar} with \cite[Section 1]{ChS1}.}.
Using a chain homotopy between the wedge product of forms and the cup product of cochains, Cheeger and Simons defined a graded ring structure on $\check H^\bullet(M) := \bigoplus_{k} \check H^k(M)$ \cite[Section 1]{ChS1}, however this structure is not required in this thesis.

An important property of differential characters is
\vspace{11pt}
\begin{theorem}[\cite{ChS1}]
\label{theorem:cheegersimons1}
There are natural exact sequences
\[
0 \lo H^{k-1}(M; \RR/\ZZ) \lo \check{H}^k(M) \xrightarrow{\;\; F \;\;} \Omega_{\ZZ}^k(M) \lo 0 
\]
\[
0\lo \Omega^{k-1}(M)/\Omega^{k-1}_{\ZZ}(M) \lo \check{H}^k(M) \xrightarrow{\;\; c \;\;} H^k(M) \lo 0
\]
\[
0\lo H^{k-1}(M;\RR) / H^{k-1}(M) \lo \check{H}^k(M) \xrightarrow{\;\; F\x c \;\;} R^k(M) \lo 0
\]
where
\[
R^k(M) := \{(c,\omega) \in H^k(M)\x\Omega_\ZZ^k(M)  \mid \imath(c) = [\omega] \in H^k(M;\RR)\}.
\]
\end{theorem}

The maps $F$ and $c$ are the \emph{curvature} and \emph{underlying class} morphisms respectively using the terminology of differential extensions.

As a first step in understanding the relationship between the total string potentials and secondary characteristic classes, one now constructs degree $1$ secondary characteristic classes for $\Omega U(n)$-bundles.
The reason for considering this very limited setting is that one may easily construct the desired classes directly from the string potentials, which in this case are functions that depend only on the Higgs field.
Consider the invariant polynomial
\[
\tr \colon \xi \longmapsto \frac{1}{2\pi i}\tr(\xi)
\]
for $\xi \in \mathfrak{u}(n)$.
The total string form of the standard connection and Higgs field on the path fibration $PU(n) \to U(n)$ corresponding to this polynomial is the function
\[
S_{\tr}:= S_{\tr}(\sA_\infty,\Phi_\infty) \colon p \longmapsto \frac{1}{2\pi i}\int_{\sone} \tr(p^{-1}\d p),
\]
where $p \in PU(n)$, so that $p^{-1}\d p \in L\mathfrak{u}(n)$.
In the case that this function is evaluated on a closed path in $PU(n)$ one obtains a `winding number'
\vspace{11pt}
\begin{lemma}[\cite{TWZ}]
\label{lemma:winding}
For any $\gamma \in \Omega U(n)$
\[
\frac{1}{2\pi i} \int_\sone \tr(\gamma^{-1} \d\gamma) \in \ZZ.
\]
\end{lemma}
\begin{proof}
The proof is essentially the same as that of \cite[Lemma 3.3]{TWZ}.
Recall that there is an isomorphism
\[
SU(n) \rtimes U(1) \lo U(n)
\]
given by $(\kappa, g) \mapsto \kappa g$.
Taking $\gamma \in \Omega U(n)$, write $(\kappa(\theta), z(\theta)) \mapsto \gamma(\theta)$ under this isomorphism, so that
\[
\gamma^{-1} \d\gamma = (\kappa z)^{-1} \d (\kappa z) = z^{-1} \kappa^{-1} (\d\kappa) z + z^{-1} \d z.
\]
Applying $\tr$, since $\kappa^{-1}\d\kappa\in \mathfrak{su}(n)$, one obtains
\[
\tr(\gamma^{-1}\d\gamma) = \tr(z^{-1}\d z) = n\cdot z^{-1}\d z.
\]
It therefore suffices to show that
\[
\frac{1}{2\pi i} \int_{\sone} z^{-1}\d z = \frac{1}{2\pi i} \int_0^1 z(t)^{-1} z'(t)\,dt \in \ZZ
\]
for any $z \in \Omega U(1)$.
Notice that the function
\[
f(s) := z(s)^{-1} e^{\int_0^s z(t)^{-1} z'(t)\,dt}
\]
satisfies $f(0) = 1$ and for all $s \in \I$
\[
f'(s) = - z(s)^{-2} z'(s) e^{\int_0^s z(t)^{-1} z'(t)\,dt} + z(s)^{-2} z'(s) e^{\int_0^s z(t)^{-1} z'(t)\,dt}  = 0
\]
since $z' z^{-1} + z (z^{-1})' =0$.
Thus $f$ is constant on $\I$, in particular, $f(1) = 1$ so that 
\[
e^{\int_0^1 z(t)^{-1} z'(t)\,dt} =1 \Longrightarrow \int_0^1 z(t)^{-1} z'(t)\,dt \in 2\pi i \ZZ
\]
as $z(1) = 1$.
\end{proof}

The result of Lemma \ref{lemma:winding} allows one to understand the behaviour of $S_{\tr}$ under the action of $\Omega U(n)$ on $PU(n)$.
Namely, take any $p \in PU(n)$ and $\gamma \in \Omega U(n)$, then
\begin{multline*}
S_{\tr}(p\gamma) = \frac{1}{2\pi i}\int_{\sone} \tr\big((p\gamma)^{-1}\d (p\gamma)\big) =  \frac{1}{2\pi i}\int_{\sone} \tr(p^{-1}\d p) + \frac{1}{2\pi i}\int_{\sone} \tr(\gamma^{-1}\d \gamma) \\
= S_{\tr}(p) + \frac{1}{2\pi i}\int_{\sone} \tr(\gamma^{-1}\d \gamma).
\end{multline*}
In particular, modulo $\ZZ$,
\[
\widetilde{S_{\tr}}(p\gamma) = \widetilde{S_{\tr}}(p)
\]
for all $p \in PU(n)$ and $\gamma \in \Omega U(n)$.
This shows that $\widetilde{S_{\tr}}$ descends to an $\RR/\ZZ$-valued function, or $\RR/\ZZ$-valued $0$-cochain, on $U(n)$.
By Proposition \ref{prop:totalstringclosed} $d S_{\tr} = \ev_{2\pi}^\ast s_{\tr}(\sA_\infty,\Phi_\infty)$ and so, writing $\check S := \widetilde{S_{\tr}}$, one has
\[
\check S \circ \d = \widetilde{s_{\tr}(\sA_\infty,\Phi_\infty)}.
\]
Since every smooth singular $0$-chain is a cycle, $\check S$ is a differential character of degree $1$ on $U(n)$.
The differential character $\check S$ now gives rise to a secondary characteristic class for $\Omega U(n)$-bundles in the following manner.
For any $\Omega U(n)$-bundle $\sQ \to M$ with connection $\sA$ and Higgs field $\Phi$, recall the $\Omega U(n)$-equivariant Higgs field holonomy $\hol_\Phi \colon \sQ \to PU(n)$, which preserves the Higgs field but not necessarily the connection.
Since the total string potential $S_{\tr}(\sA,\Phi) = \tfrac{1}{2\pi i}\int_\sone \tr(\Phi)$ is independent of $\sA$, one has
\[
S_{\tr}(\sA,\Phi) = \hol_\Phi^\ast S_{\tr}
\]
on the nose.
The above argument shows that the mod $\ZZ$ reduction $\widetilde{S_{\tr}(\sA,\Phi)}$ of $S_{\tr}(\sA,\Phi)$ descends to a differential character $\check S(\sQ,\sA,\Phi)$ on $M$.
Moreover, if $f \colon \sP \to \sQ$ is an $\Omega U(n)$-bundle map that preserves the connections and Higgs fields and covers a map $\widetilde f \colon N\to M$, it is clear from the construction that
\[
\widetilde f^\ast \check S(\sQ,\sA,\Phi) = \check S(\sP,f^\ast \sA,f^\ast\Phi),
\]
so that the assignment $(\sQ,\sA,\Phi) \mapsto \check S(\sQ,\sA,\Phi)$ defines a secondary characteristic class.
\vspace{11pt}
\begin{remark}
Observe that a key step in the construction of the secondary characteristic classes $\check S(\sQ,\sA,\Phi)$ above was pulling back by the Higgs field holonomy map, which is not necessarily connection-preserving.
Irrespective of this, however, the resulting differential characters are natural with respect to connection- and Higgs field-preserving maps so that they do indeed define secondary characteristic classes.
\end{remark}

\newpage
\mbox{}
\chapter{$\Omega$ vector bundles\label{ch:four}}

The focus of this chapter is a certain family of Fr\'{e}chet vector bundles, which in this thesis are called \emph{$\Omega$ vector bundles}.
In essence, $\Omega$ vector bundles are Fr\'{e}chet vector bundles with the additional property that the fibres form a smoothly-varying family of finite-rank modules over $L\CC$, the ring of smooth loops in $\CC$ with pointwise addition and multiplication.

The $\Omega$ vector bundles are interesting objects because, as demonstrated below, they are the targets on the infinite-dimensional side of a caloron correspondence for vector bundles that is naturally related to the based caloron correspondence of Chapter \ref{ch:two} via the frame bundle and associated vector bundle functors.
As is the case for principal bundles, the caloron correspondence for vector bundles may be enhanced to incorporate connective data, specificially \emph{module connections} and \emph{vector bundle Higgs fields} (Definitions \ref{defn:vectorhf} and \ref{defn:moduleconnection}).
These data are geometric objects on the total space of an $\Omega$ vector bundle that respect the $L\CC$-module structure on the fibres.

As is the case with finite-rank vector bundles, the collection of $\Omega$ vector bundles over a fixed base manifold $M$ is a monoidal category under the Whitney sum operation $\oplus$.
Passing to isomorphism classes gives an abelian semi-group and, in the case that $M$ is compact, the vector bundle caloron correspondence is used to show that the Grothendieck group completion of this semi-group gives a new model---phrased entirely in terms of smooth bundles over $M$---for the odd $K$-theory of $M$.

As the name suggests, the notion of $\Omega$ vector bundles is used extensively in the construction of the $\Omega$ model of odd differential $K$-theory in Chapter \ref{ch:five}.
As such, this chapter is intended to provide an in-depth treatment of these objects.

%%%%%%%%%%%%%%%%%%%%%%%%%%%%%%%%%%%%%%%%%%%%%%%%%%%%%%%%%%%%%%%%%%%%%%%%%%%%%%%%%%%%%%%%%
%%%%%%%%%%%%%%%%%%%%%%%%%%%%%%%%%%%%%%%%%%%%%%%%%%%%%%%%%%%%%%%%%%%%%%%%%%%%%%%%%%%%%%%%%
%%%%%%%%%%%%%%%%%%%%%%%%%%%%%%%%%%%%%%%%%%%%%%%%%%%%%%%%%%%%%%%%%%%%%%%%%%%%%%%%%%%%%%%%%

\section{Preliminaries}
\label{S:vector}
Before presenting the definition of $\Omega$ vector bundles, one first fixes some terminology and notation.
In the following, take $V  = \CC^n$ for some $n$ and write $G = GL(V)$ for its general linear group.
The action of $G$ on $V$ is the standard action of $GL(V)$ on $V$.
Unless stated otherwise, all vector bundles (whether Fr\'{e}chet or finite rank) and all maps between them are taken to be smooth.
When a base manifold $X$ is given and $W$ is a specified vector space, $\underline{W}$ is understood to mean the trivial vector bundle $X \x W \to X$.

The action of a loop $\gamma$ in $LG$ or $\Omega G$ on a loop $v \in LV$ is always taken to be the loop
\begin{equation}
\label{eqn:looprep}
\gamma(v) \colon \theta \longmapsto \gamma(\theta)(v(\theta))
\end{equation}
in $V$, where the action on the right hand side is the standard action of $G$ on $V$.
This is the \emph{loop} or \emph{pointwise representation} of $LG$ (or $\Omega G$) on $LV$.
The group of all $\CC$-linear endomorphisms of $LV$ is denoted $GL(LV)$ and it is clear that $\Omega G < LG < GL(LV)$.

The reader is assumed to be familiar with the notion of the \emph{frame bundle} $\cF(E) \to M$ of a vector bundle $E\to M$ as well as the related concept of the \emph{associated vector bundle} of a $G$-bundle, recalling that these assignments are functorial.
In particular, if $E \to M$ is a complex vector bundle of rank $n$, a \emph{frame} at $x \in M$ is a linear isomorphism $\CC^n \to E_x$.
In the case that $\sE \to M$ is a Fr\'{e}chet vector bundle with typical fibre $L\CC^n$ a \emph{frame} at $x\in M$ is a linear isomorphism $L\CC^n \to \sE_x$ of $\CC$-vector spaces.
A classical reference on frame bundles and associated vector bundles is \cite{KN1}.

Having established these fundamentals,
\vspace{11pt}
\begin{definition}
\label{defn:frechetvector}
An \emph{$\Omega$ vector bundle over $M$} is a Fr\'{e}chet vector bundle $\sE\to M$ with typical fibre $LV$ and structure group $\Omega G$\footnote{that is, there is a preferred family of local trivialisations of $\sE$ for which the transition functions are valued in $\Omega G$, which is taken as a subgroup of $GL(LV)$ via the loop representation.}.
Notice that this implies that the frame bundle $\cF(\sE)\to M$ of $\sE$ has a chosen reduction of the structure group from $GL(LV)$ to $\Omega G$.
In this context, unless explicitly stated otherwise $\cF(\sE)$ shall be taken to mean this reduction rather than the whole $GL(LV)$-bundle.
\end{definition}
\vspace{11pt}
\begin{remark}
In order to easily distinguish $\Omega$ vector bundles from their finite-dimensional counterparts, similarly to Chapters \ref{ch:two} and \ref{ch:three} the total space of an $\Omega$ vector bundle is usually denoted by a sans-serif character, for example $\sE, \sF$ or $\sH$ as opposed to $E, F$ or $H$.
\end{remark}

It may seem bizarre that in this definition the typical fibre $LV$ is the space of \emph{free} loops in $V$ whereas the structure group $\Omega G$ consists of \emph{based} loops.
The reason for this is explained in Remark \ref{remark:evaluation} below.

A direct consequence of Definition \ref{defn:frechetvector} is that the fibres of any $\Omega$ vector bundle $\sE \to M$ naturally have the structure of $L\CC$-modules.
To see this fix $x \in M$ and choose any local trivialisation
\[
\psi_\alpha \colon \sE|_{U_\alpha} \lo  U_\alpha\x L\CC^n
\] 
of $\sE$ over some open set $U_\alpha \subset M$ containing $x$.
Write $\cC^\infty(U_\alpha,L\CC)$ for the ring of smooth maps $U_\alpha \to L\CC$ with addition and multiplication defined pointwise and notice that the set of sections of the trivial bundle $U_\alpha \x L\CC^n \to U_\alpha$ is naturally a free, finitely generated $\cC^\infty(U_\alpha,L\CC)$-module of rank $n$.
Since $\psi_\alpha$ is $\CC$-linear, it induces the structure of a $\cC^\infty(U_\alpha,L\CC)$-module on $\Gamma(U_\alpha,\sE)$.
Explicitly,
\[
f\cdot s \colon x\longmapsto \psi_\alpha^{-1}\big(f(x)\cdot \psi_\alpha(s(x))\big)
\]
defines the action of $f \in \cC^\infty(U_\alpha,L\CC)$ on $s \in \Gamma( U_\alpha , \sE)$.
It is straightforward to see that this makes $\Gamma(U_\alpha,\sE)$ into a $\cC^\infty(U_\alpha,L\CC)$-module and, moreover, that $\psi_\alpha$ is an isomorphism of $\cC^\infty(U_\alpha,L\CC)$-modules.
In particular, since the isomorphism $\psi_\alpha$ acts fibrewise, the fibre of $\sE$ over $x$ becomes an $L\CC$-module of rank $n$.

If $\psi_\beta \colon \sE|_{U_\beta} \to U_\beta\x L\CC^n$ is another local trivialisation of $\sE$ with $x\in U_\beta$, then on the intersection $U_{\alpha\beta} := U_\alpha\cap U_\beta$ the transition functions $\tau_{\alpha\beta}$ of $\sE$ are valued in $\Omega G$ acting via the loop representation.
Thus the map
\begin{align*}
\psi_\beta\circ \psi_\alpha^{-1} \colon U_{\alpha\beta}\x L\CC^n &\lo U_{\alpha\beta}\x L\CC^n
\\
(x,v) &\longmapsto \left(x,\tau_{\alpha\beta}(x)(v)\right)
\end{align*}
is an isomorphism of $\cC^\infty(U_{\alpha\beta},L\CC)$-modules.
Denote by $\Gamma(U_{\alpha\beta},\sE)^{(\alpha)}$ the set of sections $\Gamma(U_{\alpha\beta},\sE)$ with $\cC^\infty(U_{\alpha\beta},L\CC)$-module structure induced by the restriction of $\psi_\alpha$ to $U_{\alpha\beta}$ as above, with $\Gamma(U_{\alpha\beta},\sE)^{(\beta)}$ defined similarly.
Then there is the commuting diagram
\[
\xy
(0,25)*+{\Gamma(U_{\alpha\beta}\x L\CC^n)}="1";
(50,25)*+{\Gamma(U_{\alpha\beta}\x L\CC^n)}="2";
(0,0)*+{\Gamma(U_{\alpha\beta},\sE)^{(\alpha)}}="3";
(50,0)*+{\Gamma(U_{\alpha\beta},\sE)^{(\beta)}}="4";
{\ar^{\psi_\beta\circ\psi_\alpha^{-1}} "1";"2"};
{\ar^{\psi_\alpha} "3";"1"};
{\ar^{\psi_\beta} "4";"2"};
{\ar^{\id} "3";"4"};
\endxy
\]
of $\cC^\infty(U_{\alpha\beta},L\CC)$-module isomorphisms that restrict to isomorphisms of $L\CC$-modules on each fibre.
In particular the module structure on $\Gamma(U_{\alpha\beta},\sE)$ and, hence, the module structure on $\sE_x$, is independent of the chosen trivialisation.

Denote by $L\mathcal{O}_M$ the sheaf of rings
\[
U \longmapsto \cC^\infty(U,L\CC)
\]
on $M$, where the ring structure on $\cC^\infty(U,L\CC)$ given pointwise by the ring structure of $L\CC$.
Denote by $\Gamma_\sE$ the sheaf of sections
\[
U \longmapsto \Gamma(U,\sE)
\]
on $M$.
A straightforward extension of the above argument implies that $\Gamma_\sE$ is a locally free sheaf of modules of rank $n$ over the ringed space $(M,L\mathcal{O}_M)$.
\vspace{11pt}
\begin{definition}
The \emph{rank} of an $\Omega$ vector bundle $\sE \to M$ is the rank of its sheaf of sections $\Gamma_\sE$ as a sheaf of modules over the ringed space $(M,L\mathcal{O}_M)$.
\end{definition}

\vspace{11pt}
\begin{remark}
\label{remark:framesmodule}
Given any frame $p \in \cF(\sE)_x$, there is some $\gamma \in \Omega G$ such that $p = \psi_\alpha^{-1} \circ \gamma$, where $\psi_\alpha$ is the local trivialisation of $\sE$ as above, restricted to the fibre $\sE_x$.
Since $\gamma$ and $\psi_\alpha^{-1}$ are both isomorphisms of $L\CC$-modules, so is $p$.
In general, any $L\CC$-module isomorphism $L\CC^n \to \sE_x$ is of the form $\psi_\alpha^{-1}\circ \gamma$ for some $\gamma \in LG$.
\end{remark}
\vspace{11pt}
\begin{remark}
\label{remark:evaluation}
Take the $\Omega$ vector bundle $\sE \to M$ with typical fibre $LV$ and structure group $\Omega G$.
Another consequence of Definition \ref{defn:frechetvector} is that there is an \emph{evaluation map} on $\sE$.
Namely, for each $\theta \in \sone$ there is a vector bundle $E_\theta \to M$ with $\rank E_\theta  = \dim V$ and a map $ \ev_\theta \colon \sE \to E_\theta$ of vector bundles covering the identity on $M$.

The bundle $E_\theta$ is constructed by first choosing an open covering $\{U_\alpha\}_{\alpha\in I}$ of $M$ over which there are local trivialisations
\[
\psi_\alpha \colon \sE|_{U_\alpha} \simto U_\alpha \x LV
\]
of $\sE$.
Then over $U_\alpha$, by evaluating loops in $V$ at $\theta$ one obtains
\[
\sE|_{U_\alpha} \xrightarrow{\;\;\psi_\alpha\;\;} U_\alpha \x LV \xrightarrow{\;\;\ev_\theta\;\;} U_\alpha \x V
\]
which, by evaluating the $\Omega G$-valued transition functions of $\sE$ at $\theta$ gives $G$-valued functions over double intersections of the $U_\alpha$.
The clutching construction gives a vector bundle $E_\theta \to M$ with typical fibre $V$ and structure group $G$.
The vector bundle $E_\theta$ is independent of the choices of $U_\alpha$ and $\psi_\alpha$: different choices amount simply to different local trivialisations of $E_\theta$.
Moreover, the construction determines the map
\[
\ev_\theta \colon \sE \lo E_\theta
\]
of vector bundles, which agrees with the regular evaluation of loops at $\theta$ in the above trivialisations.
Notice that
\begin{itemize}
\item
$E_0$ is the trivial bundle since the transition functions of $\sE$ are valued in based loops in $G$, in fact $E_0$ has a distinguished global trivialisation; and
\item
there is a well-defined notion of an element $v \in \sE_x$ being zero at $\theta \in \sone$; namely if $\ev_\theta(v)$ agrees with the zero section in the fibre of $E_\theta$ over $x$.
Therefore one may consistently talk about $v \in \sE_x$ as being either zero or non-zero at $\theta$.
\end{itemize}
It shall be seen as a result of the vector bundle caloron correspondence that there is a vector bundle over $M\x\sone$ whose restriction over $M\x\{\theta\}$ is naturally isomorphic to $E_\theta$: this is the caloron transform of $\sE$.
This is the reason for specifying that the typical fibre is the vector space of free loops in $V$ rather than based loops, since in the latter case one would have $E_0 \cong M\x\{0\}$, which does not have the same rank as $E_\theta$ for any $\theta\in\sone\setminus\{0\}$.
The first observation above, which is a direct consequence of the fact that the structure group is $\Omega G$ rather than $LG$, implies that there is a distinguished choice of \emph{framing} (Definition \ref{defn:vecframe}) of the caloron transform over $M_0$.
\end{remark}

Some basic examples of $\Omega$ vector bundles:
\vspace{11pt}
\begin{example}[the trivial $\Omega$ vector bundle]
The \emph{trivial $\Omega$ vector bundle of type $V$} over $M$ is $\underline{LV} := M\x LV \to M$.
For any $x \in M$ the fibre $\cF(\underline{LV})_x$ consists of precisely those $L\CC$-module isomorphisms $\psi\colon LV \to LV$ such that $\ev_0 \circ \,\psi \circ \imath \colon V \to V$ is the identity, with $\imath \colon V \hookrightarrow LV$ the inclusion of constant loops.
\end{example}
\vspace{11pt}
\begin{example}[the universal $\Omega$ vector bundle]
\label{example:universalomega}
The \emph{universal $\Omega$ vector bundle of type $V$} is the associated vector bundle\footnote{recall that the associated vector bundle to, say, the $\Omega GL_n(\CC)$-bundle $\cF(\sE)$ is $\cF(\sE)\x_{\Omega GL_n(\CC)} L\CC^n$, where the quotient is taken with respect to the $\Omega GL_n(\CC)$-action $(p,v)\gamma:= (p\gamma,\gamma^{-1}(v))$.} to the path fibration $PG\to G$, denoted
\[
\sE(V):= PG \x_{\Omega G}LV \lo G.
\]
It is clear that $\sE(V)$ is universal for $\Omega$ vector bundles modelled over $LV$ with structure group $\Omega G$, since $PG$ is universal for $\Omega G$-bundles.
The universal property of $\sE(V)$ can also be seen as a consequence of the existence of vector bundle Higgs fields (Definition \ref{defn:vectorhf} and Lemma \ref{lemma:higgshol}).
\end{example}

There is a pullback operation on $\Omega$ vector bundles given in exactly the same way as for ordinary vector bundles.
Namely, if $\psi \colon N \to M$ is a smooth map of manifolds and $\sE \to M$ is an $\Omega$ vector bundle, then the regular pullback $\psi^\ast \sE \to N$ (as Fr\'{e}chet vector bundles) is an $\Omega $ vector bundle since the transition functions of $\psi^\ast \sE$ are obtained by pulling back the transition functions of $\sE$ by $\psi$.
\vspace{11pt}
\begin{definition}
A \emph{morphism} from the $\Omega$ vector bundle $\pi_\sE \colon \sE \to M$ to the $\Omega$ vector bundle $\pi_\sF \colon \sF \to N$ (with $n = \rank \sE$ and $m = \rank \sF$) is a pair of smooth maps $f \colon \sE \to \sF$, $\widetilde f\colon M\to N$ such that
\begin{itemize}
\item
$\pi_\sF \circ f = \widetilde f \circ \pi_\sE$; and

\item
for every $x \in M$ the restriction $f \colon\sE_x \to \sF_{\widetilde f(x)}$ is a homomorphism of $L\CC$-modules with the additional property that for any choice of frames $p \colon L\CC^n \to \sE_x$ and $q \colon L\CC^m \to \sF_{\widetilde f(x)}$ the composition $q^{-1} \circ f\circ p$ is given by the pointwise action of some $\gamma\in \Omega \Mat_{(m,n)}(\CC)$, where $\Mat_{(m,n)}(\CC)$ is the space of $m\x n$ complex-valued matrices with its usual action on $\CC^n$.
\end{itemize}
In the sequel, a morphism shall usually be written as a map $f \colon \sE \to \sF$ of total spaces that covers the map $\widetilde f \colon M \to N$ of base spaces.

An \emph{isomorphism} of $\Omega$ vector bundles (based over $M$) is a morphism $f \colon \sE \to \sF$ covering the identity with the additional property that $f$ is a diffeomorphism and $f^{-1} \colon \sF \to \sE$ is also a morphism of $\Omega$ vector bundles covering the identity.
\end{definition}
\vspace{11pt}
\begin{remark}
Any isomorphism $f \colon \sE \to \sF$ of $\Omega$ vector bundles uniquely determines an isomorphism $\cF(f) \colon \cF(\sE) \to \cF(\sF)$ of frame bundles, given simply by postcomposing frames by $f$.

Conversely, given an isomorphism $g \colon \cF(\sE) \to \cF(\sF)$ of the frame bundles then passing to the associated vector bundles there is an isomorphism
\[
\cF(\sE)\x_{\Omega GL_n(\CC)}L\CC^n \lo \cF(\sF) \x_{\Omega GL_n(\CC)}L\CC^n
\]
of $\Omega$ vector bundles given by
\[
[p,v] \longmapsto [g(p),v].
\]
This determines an isomorphism $\sE \to \sF$ of $\Omega$ vector bundles, since there are natural isomorphisms $\cF(\sE)\x_{\Omega GL_n(\CC)}L\CC^n \to \sE$ and $\cF(\sF)\x_{\Omega GL_n(\CC)}L\CC^n \to \sF$ of $\Omega$ vector bundles, i.e.
\[
\cF(\sE)\x_{\Omega GL_n(\CC)}L\CC^n \ni [p,v] \longmapsto p(v) \in \sE.
\]
Thus an isomorphism of $\Omega$ vector bundles may be defined equivalently as a map $f \colon \sE \to \sF$ of $\Omega$ vector bundles that induces an isomorphism of frame bundles.
\end{remark}

The primary reason for dealing with $\Omega$ vector bundles, at least initially, is the vector bundle caloron correspondence.
As is the case for principal bundles, the vector bundle caloron correspondence is best phrased using the terminology of category theory.
\vspace{11pt}
\begin{definition}
For a fixed manifold $M$, write $\ovect(M)$ for the groupoid which has as objects the $\Omega$ vector bundles $\sE \to M$.
Morphisms are given by isomorphisms of $\Omega$ vector bundles covering the identity on $M$.

Write $\ovect_V(M)$ for the subgroupoid of $\ovect(M)$ consisting of those $\Omega$ vector bundles with typical fibre $LV$, with its appropriate morphisms.
\end{definition}

Observe that the assignments $M\mapsto \ovect(M)$ and $M\mapsto \ovect_V(M)$ are functorial, acting on morphisms by pullback.

%%%%%%%%%%%%%%%%%%%%%%%%%%%%%%%%%%%%%%%%%%%%%%%%%%%%%%%%%%%%%%%%%%%%%%%%%%%%%%%%%%%%%%%%%
%%%%%%%%%%%%%%%%%%%%%%%%%%%%%%%%%%%%%%%%%%%%%%%%%%%%%%%%%%%%%%%%%%%%%%%%%%%%%%%%%%%%%%%%%
%%%%%%%%%%%%%%%%%%%%%%%%%%%%%%%%%%%%%%%%%%%%%%%%%%%%%%%%%%%%%%%%%%%%%%%%%%%%%%%%%%%%%%%%%

\subsection{The caloron correspondence for vector bundles}
The caloron correspondence for vector bundles is constructed by using the based caloron transform of Chapter \ref{ch:two} in tandem with the frame bundle and associated vector bundle functors.
\vspace{11pt}
\begin{definition}
\label{defn:vecframe}
Let $E \to X$ be a complex vector bundle and let $X_0 \subset X$ be a smooth submanifold.
Then $E$ is \emph{framed} over $X_0$ if there is a distinguished trivialisation of $E$ over $X_0$.
This distinguished trivialisation is equivalent to a section $s_0 \in \Gamma(X_0,\cF(E))$ of the frame bundle, called a \emph{framing} of $E$ over $X_0$.
Thus the frame bundle of a framed vector bundle is framed (in the sense of Definition \ref{defn:framedbundle}).
\end{definition}
\vspace{11pt}
\begin{remark}
Unless stated otherwise, when a framed vector bundle $E \to X$ is said to have framing $s_0$ over $X_0$, the symbol $s_0$ refers to the section of $\cF(E)$ over $X_0$ determined by the distinguished trivialisation.
\end{remark}
\vspace{11pt}
\begin{definition}
For a fixed manifold $M$, let $\frvect(M)$ be the groupoid of (finite rank) complex vector bundles $E \to M\x\sone$ that are framed over $M_0:= M\x\{0\}$.
Morphisms in $\frvect(M)$ are vector bundle isomorphisms covering the identity that preserve the framing over $M_0$.
\end{definition}

Recall the caloron transform functors $\cC$ and $\cC^{-1}$ from Section \ref{S:based}.
The \emph{caloron transform} for vector bundles is the functor
\[
\cV\colon \ovect(M) \lo \frvect(M)
\]
whose action on objects is given by sending $\sE$ to
\[
\cV(\sE) := \cC(\cF(\sE)) \x_{G} V  \lo M\x\sone,
\]
where $\sE$ has typical fibre $LV$ and structure group $\Omega G$.
The vector bundle $\cV(\sE)$ is equipped with the distinguished trivialisation given by
\[
[s_0(x,0)g, v] \longmapsto (x,0, g(v)) \in M_0 \x V,
\]
where $s_0 \in \Gamma(M_0,\cC(\cF(\sE)))$ is the framing arising from the caloron correspondence for principal bundles.
The action of $\cV$ on a morphism $f \colon \sE \to \sF$ in $\ovect(M)$ is given by
\[
\cV(f) \colon [q,v] \longmapsto \left[\left(\cC\circ\cF(f)\right)(q),v\right],
\]
which preserves the framing and covers the identity on $M\x\sone$, so that $\cV(f)\colon\cV(\sE)\to\cV(\sF)$ is indeed a morphism in $\frvect(M)$.

The \emph{inverse caloron transform} for vector bundles is the functor
\[
\cV^{-1} \colon \frvect(M) \lo \ovect(M)
\]
that sends the object $E \in \frvect(M)$ to the associated vector bundle
\[
\cV^{-1}(E) := \cC^{-1}(\cF(E)) \x_{\Omega G} LV \lo M,
\]
where $E$ has typical fibre $V$ and structure group $G$.
As usual, the action of $\gamma \in \Omega G$ on $v \in LV$ is given by the loop representation \eqref{eqn:looprep}.
The action of $\cV^{-1}$ on the morphism $f\colon E\to F$ of $\frvect(M)$ is given by
\[
\cV^{-1}(f) \colon [q,v] \longmapsto \left[\left(\cC^{-1}\!\circ\cF(f)\right)\!(q),v\right]
\]
which is a morphism in $\ovect(M)$.
\vspace{11pt}
\begin{theorem}
\label{theorem:vector}
The caloron correspondence for vector bundles
\[
\cV\colon \ovect(M) \lo \frvect(M)
\;\;\mbox{and}\;\;
\cV^{-1} \colon \frvect(M) \lo \ovect(M)
\]
is an equivalence of categories.
\end{theorem}
\begin{proof}
It suffices to show that for each $\sE \in \ovect(M)$ and for each $E\in \frvect(M)$ there are natural isomorphisms $\alpha_\sE \colon \cV^{-1}\circ\cV(\sE) \to \sE$ and $\beta_E \colon \cV\circ\cV^{-1}(E) \to E$.

Recall that for any vector bundle $F$ (modelled over $V$ with structure group $G$) there are natural isomorphisms
\[
\eta_{\cF(F)} \colon \cF(\cF(F)\x_G V) \lo \cF(F)
\;\;
\mbox{ and }
\;\;
\eta_F \colon \cF(F)\x_G V \lo F.
\]
In the first instance, $\eta_{\cF(F)}$ is defined as the inverse of the map that sends the frame $p \colon v\mapsto p(v) \in F_x$ to the frame $v \mapsto [p,v]$ and, in the second instance, $\eta_F$ is given by sending $[p,v] \mapsto p(v)$.
Moreover, in the case that $F$ is an $\Omega$ vector bundle, $\eta_F$ is readily seen to be an isomorphism of $\Omega$ vector bundles.
It follows from this and the definitions of  $\cV$ and $\cV^{-1}$ that there are natural isomorphisms $\mu\colon \cF \circ \cV \to \cC\circ \cF$ and $\nu \colon \cF \circ \cV^{-1} \to \cC^{-1}\circ \cF$.

The natural isomorphism $\alpha_\sE$ is given by composing
\[
\cV^{-1}\circ\cV(\sE) = \cC^{-1}(\cF(E))\x_{\Omega G} LV \xrightarrow{\;\;[\cC^{-1}(\mu),\id]\;\;} \cC^{-1}(\cC( \cF(\sE))) \x_{\Omega G} LV
\]
with
\[
\cC^{-1}(\cC( \cF(\sE))) \x_{\Omega G} LV \xrightarrow{\;\;[\alpha_{\cF(\sE)},\id]\;\;} \cF(\sE) \x_{\Omega G} LV
\]
and
\[
\cF(\sE) \x_{\Omega G} LV \xrightarrow{\;\;\eta_\sE\;\;} \sE ,
\]
where $\alpha_{\cF(\sE)}\colon \cC^{-1} \circ \cC (\cF(\sE))\to \cF(\sE)$ is the natural isomorphism of Theorem \ref{theorem:equiv}.
Note that each of the above maps is a natural isomorphism so that $\alpha_\sE$ is also.

The construction of $\beta_E$ is entirely analogous, noting in this case that the framing is preserved.
\end{proof}

From this proof it is also seen that
\vspace{11pt}
\begin{corollary}
\label{cor:frame}
The diagrams of functors
\[
\xy
(-38,0)*+{\small
\xy
(0,25)*+{\ovect_V(M)}="1";
(50,25)*+{\Bun_{\Omega G}(M)}="2";
(0,0)*+{\frvect_V(M)}="3";
(50,0)*+{\frBun_G(M)}="4";
{\ar^{\cF} "1";"2"};
{\ar^{\cC} "2";"4"};
{\ar^{\cV} "1";"3"};
{\ar^{\cF} "3";"4"};
\endxy};
(0,0)*+{\mbox{and}};
(42,0)*+{\small\xy
(0,25)*+{\frvect_V(M)}="1";
(50,25)*+{\frBun_{G}(M)}="2";
(0,0)*+{\ovect_V(M)}="3";
(50,0)*+{\Bun_{\Omega G}(M)}="4";
{\ar^{\cF} "1";"2"};
{\ar^{\cC^{-1}} "2";"4"};
{\ar^{\cV^{-1}} "1";"3"};
{\ar^{\cF} "3";"4"};
\endxy};
\endxy
\]
both commute up to natural isomorphism.
\end{corollary}
\begin{proof}
The natural isomorphisms are $\mu$ for the diagram on the left and $\nu$ for the diagram on the right.
\end{proof}

The functors $\cV$ and $\cV^{-1}$ may be described heuristically in more concrete terms as follows.
By Theorem \ref{theorem:equiv}, up to isomorphism the frame bundle $\cF(\sE)$ of $\sE \in \ovect_V(M)$ may be regarded as loops in some $P \in \frBun_G(M)$.
That is
\[
\cF(\sE)_x \cong  \Gamma_0(\{x\}\x\sone,P)
\]
for any $x\in M$, where the subscript $0$ denotes that the sections agree with the given section of $P$ over $M\x\{0\}$.
The evaluation map $\ev_\theta$ of Remark \ref{remark:evaluation} then acts as
\[
\ev_\theta \colon [p,v] \longmapsto [p(\theta),v(\theta)]
\]
where the image lies in the fibre of the associated vector bundle $P \x_G V$ over $(x,\theta)$.
Thus (up to isomorphism) $\cV(\sE)\to M\x\sone$ is the vector bundle whose restriction to $M\x\{\theta\}$ is $\ev_\theta(\sE) \to M$.

There is also a more direct description of $\cV^{-1}$ that is useful in the sequel.
Recalling the construction of the inverse caloron transform $\cC^{-1}$, in particular the map $\eta$ of \eqref{eqn:etamap}, notice that
\[
\cC^{-1} \circ \cF(E) = \eta^\ast \left( \Omega_{s_0(M_0)}\cF(E)\right)
\]
so that a point $q$ in the fibre of $\cC^{-1} \circ \cF(E)$ over $x \in M$ is precisely a smoothly-varying assignment of a $\CC$-linear isomorphism $q_\theta \colon V \to E_{(x,\theta)}$ for every $\theta \in\sone$.
Consider the $\Omega$ vector bundle $\eta^\ast LE \to M$, which has frame bundle $\eta^\ast( \Omega_{s_0(M_0)}\cF(E)) \to M$.
There is a natural isomorphism
\begin{equation}
\label{eqn:altvcal}
\lambda_E \colon \cV^{-1}(E) \simto \eta^\ast LE
\end{equation}
given by sending $[q,v] \in \cV^{-1}(E)$ to the loop
\[
q(v) \colon \theta \longmapsto q(\theta)(v(\theta)) \in E_{(x,\theta)}.
\]
In particular, $\lambda_E$ is a fibrewise $L\CC$-module isomorphism that allows one to interpret $\cV^{-1}$ as a sort of looping operation on framed vector bundles.

%%%%%%%%%%%%%%%%%%%%%%%%%%%%%%%%%%%%%%%%%%%%%%%%%%%%%%%%%%%%%%%%%%%%%%%%%%%%%%%%%%%%%%%%%
%%%%%%%%%%%%%%%%%%%%%%%%%%%%%%%%%%%%%%%%%%%%%%%%%%%%%%%%%%%%%%%%%%%%%%%%%%%%%%%%%%%%%%%%%
%%%%%%%%%%%%%%%%%%%%%%%%%%%%%%%%%%%%%%%%%%%%%%%%%%%%%%%%%%%%%%%%%%%%%%%%%%%%%%%%%%%%%%%%%

\subsection{Operations on $\ovect$}
\label{SS:vector:monoidal}
This section describes some adaptations of familiar operations on vector bundles, such as the direct sum and tensor product, to the setting of $\Omega$ vector bundles.
These operations on $\ovect(M)$ are first constructed using the caloron correspondence of Theorem \ref{theorem:vector}---in order to easily obtain information such as coherence diagrams and associators---then they are described in concrete terms.

In order to apply the caloron correspondence, however, one must understand how framings on framed vector bundles behave with respect to direct sums and tensor products.
Take two vector bundles $E, F \to M$ with respective ranks $n$ and $m$ that are framed over $M_0$ by the sections $s_0\in \Gamma(M_0,\cF(E))$ and $r_0 \in \Gamma(M_0,\cF(F))$ respectively.
Observe that any local section $s \in\Gamma(U,\cF(E))$ corresponds to a local trivialisation $s^{-1} \colon E|_U \to U\x \CC^n$, where $s^{-1}(e) := (x, s(x)^{-1}(e))$ for $e\in E_x$\label{page:sectiontriv}.
Thus the \emph{direct sum} of $s_0$ and $r_0$ is the framing on $E\oplus F$ corresponding to the trivialisation
\[
(v,w) \longmapsto \left(s_0^{-1}(v),r_0^{-1}(w)\right) \in \CC^n \oplus \CC^m.
\]
This framing turns $E\oplus F$ into a framed vector bundle and it is easy to see that this framed direct sum operation $\oplus$ determines a monoidal product on $\frvect(M)$.

In a similar fashion, the standard tensor product operation $\otimes$ on vector bundles gives a monoidal structure on $\frvect(M)$.
The \emph{tensor product} of the framings $s_0$ and $r_0$ is the framing on $E \otimes F$ corresponding to the trivialisation that acts on homogeneous elements by
\[
 v\otimes w \longmapsto t\left(s_0^{-1}(v) \otimes r_0^{-1}(w)\right).
\]
Here $t \colon \CC^{n}\otimes\CC^{m} \to \CC^{n\x m}$ is the isomorphism that acts on products of the standard basis vectors by 
\[
\mathbf{e}_i \otimes \mathbf{e}_j \longmapsto \mathbf{e}_{(i-1)m +j}.
\]
As with the framed direct sum operation, the framed tensor product $\otimes$ defines a monoidal product on $\frvect(X)$.

The monoidal products $\oplus$ and $\otimes$ on $\frvect(M)$ induce monoidal products on $\ovect(M)$ via the caloron correspondence.
For any $\sE, \sF \in \ovect(M)$, the \emph{direct sum} of $\sE$ and $\sF$ is
\[
\sE\oplus \sF := \cV^{-1} (\cV(\sE)\oplus \cV(\sF)).
\]
That the bifunctor $\oplus$ is a monoidal product on $\ovect(M)$ comes from applying Theorem \ref{theorem:vector} to the coherence diagrams of the framed direct sum operation $\oplus$ on $\frvect(M)$.
\vspace{11pt}
\begin{lemma}
\label{lemma:whitney}
For any $\sE,\sF \in \ovect(M)$ and $x\in M$ there is a natural $L\CC$-module isomorphism
\[
(\sE \oplus \sF){}_x \simto \sE_x \oplus_{L\CC} \sF_x
\]
that varies smoothly with $x \in M$.
In particular $\sE \oplus\sF$ is naturally isomorphic to the regular Whitney sum of $\sE$ and $\sF$ as $\Omega$ vector bundles.
\end{lemma}
\begin{proof}
Comparing with the observation following Corollary \ref{cor:frame} and using the natural isomorphisms $\alpha_\sE$ and $\alpha_\sF$ of Theorem \ref{theorem:vector}, there are natural isomorphisms 
\[
\sE_x \simto \Gamma(\{x\}\x \sone,E)\;\;\mbox{ and }\;\;\sF_x \simto \Gamma(\{x\}\x \sone,F)
\]
of $L\CC$-modules that depend smoothly on $x \in M$, where $E= \cV(\sE)$ and $F = \cV(\sF)$ are the caloron transforms.
By the same argument there is also a natural isomorphism
\[
(\sE \oplus \sF){}_x \simto \Gamma(\{x\}\x\sone, E\oplus F)
\]
depending smoothly on $x$.

It remains only to notice that there is a natural $L\CC$-module isomorphism
\[
\Gamma(\{x\}\x\sone, E\oplus F) \simto \Gamma(\{x\}\x \sone,E) \oplus_{L\CC} \Gamma(\{x\}\x \sone,F)
\]
induced by the projection maps.
\end{proof}

Similarly, the \emph{honed tensor product} $\sE \oast \sF$ of $\sE,\sF \in\ovect(M)$ is defined as
\[
\sE \oast \sF := \cV^{-1}(\cV(\sE)\otimes \cV(\sF)).
\]
As with the direct sum operation on $\ovect(M)$, the fact that $\oast$ is a monoidal product comes directly from the caloron correspondence.
Notice that the honed tensor product is not the same operation as the standard fibrewise tensor product $\otimes$ on $\ovect(M)$.
As the name suggests, the honed tensor product is a finer operation: supposing that $\sE,\sF$ are respectively modelled over $LV$ and $LW$, then $\sE \oast \sF$ has typical fibre $L(V\otimes W)$ instead of the larger space $LV\otimes LW$.
\vspace{11pt}
\begin{lemma}
\label{lemma:oast}
For any $\sE,\sF \in \ovect(M)$ and $x\in M$ there is a   natural $L\CC$-module isomorphism
\[
(\sE \oast \sF){}_x \simto \sE_x \otimes_{L\CC} \sF_x
\]
that smoothly varies with $x$.
\end{lemma}
\begin{proof}
As in the proof of Lemma \ref{lemma:whitney}, there are natural $L\CC$-module isomorphisms 
\[
\sE_x \simto \Gamma(\{x\}\x \sone,E)\;\;\mbox{ and }\;\;\sF_x \simto \Gamma(\{x\}\x \sone,F)
\]
that depend smoothly on $x \in M$, where $E= \cV(\sE)$ and $F = \cV(\sF)$. 
Similarly, there is a natural isomorphism
\[
(\sE\oast\sF)_x \simto \Gamma(\{x\}\x \sone,E\otimes F).
\]
There is also a natural isomorphism of $L\CC$-modules
\[
\Gamma(\{x\}\x\sone,E\otimes F) \simto \Gamma(\{x\}\x\sone,E) \otimes_{L\CC} \Gamma(\{x\}\x\sone,F)
\]
from which the result follows.
\end{proof}
\vspace{11pt}
\begin{remark}
Notice that $\frvect(M)$ is a bimonoidal category with respect to $\oplus$ and $\otimes$, i.e. $\otimes$ distributes over $\oplus$ up to natural isomorphism.
Applying the caloron correspondence implies that $\ovect(M)$ is a bimonoidal category with respect to $\oplus$ and $\oast$.
As a result, the assignment
\[
M \longmapsto \ovect(M)
\]
may be viewed as a contravariant functor
\[
\ovect \colon \Man \lo \bimon
\]
from the category of smooth manifolds to the category of bimonoidal categories, where the action of $\ovect$ on morphisms is given by pullback.
\end{remark}

Importantly for the discussion of Hermitian structures on $\Omega$ vector bundles in Section \ref{S:hermitian} there is also a \emph{dual operation} in $\ovect(M)$.
The dual operation sends the $\Omega$ vector bundle $\sE \to M$ with typical fibre $LV$ to an $\Omega$ vector bundle $\sE^\star \to M$ with typical fibre $L(V^\ast)$, where $V^\ast$ is the dual of the finite-dimensional vector space $V$.
The dual $\sE^\star$ is defined using the machinery of associated vector bundles, however it is shown below that $\sE^\star$ is equivalently given by taking the fibrewise dual of $\sE$ with respect to $L\CC$.

To see how this works, take the $\Omega$ vector bundle $\sE\to M$ with typical fibre $LV$ and structure group $\Omega G$.
An element $\gamma$ of $LG$ or $\Omega G$ acts on $\lambda \in L(V^\ast)$ by
\begin{equation}
\label{eqn:duallooprep}
\gamma(\lambda) \colon \theta\longmapsto  \lambda(\theta) \circ \gamma^{-1}(\theta),
\end{equation}
this is the \emph{dual loop} or \emph{dual pointwise representation} (cf. \eqref{eqn:looprep}).
The notation $\overline{L G}, \overline{\Omega G}$ denotes $LG$ or $\Omega G$ respectively acting in the dual loop representation.
The \emph{honed dual} $\sE^\star$ of $\sE$ is defined as
\[
\sE^\star := \cF(\sE) \x_{\overline{\Omega G}} L(V^\ast).
\]
As with the honed tensor product, the honed dual is a finer operation than simply taking the regular dual of $\sE$ as a Fr\'{e}chet vector bundle since it respects the $L\CC$-module structures on the fibres.
\vspace{11pt}
\begin{lemma}
\label{lemma:altdual}
For any $\sE\in \ovect(M)$ and $x\in M$ there is a natural $L\CC$-module  isomorphism
\[
\sE^\star_x \simto \Hom_{L\CC}( \sE_x ,{L\CC}).
\]
\end{lemma}
\begin{proof}
For any $x \in M$, by definition one has $\sE_x^\star := \cF(\sE)_x \x_{\overline{\Omega G}} L(V^\ast)$.
Any element $[f,\lambda]\in \sE_x^\star$ determines an $L\CC$-module homomorphism $\sE_x \to L\CC$ by sending $v \in \sE_x$ to the loop
\[
\theta \longmapsto \lambda(\theta) \left( f^{-1}(v)(\theta)\right)
\]
in $\CC$, noting that this is well-defined.
This gives a homomorphism $\sE_x^\star \to \Hom_{L\CC}(\sE_x,L\CC)$ of $L\CC$-modules.

Suppose that for every $v \in \sE_x$ and $\theta\in\sone$
\[
\lambda(\theta) \left( f^{-1}(v)(\theta)\right) = 0.
\]
As $f$ is an $L\CC$-module isomorphism $LV \to \sE_x$ one must have $\lambda = 0$ so that $[f,\lambda]$ coincides with the zero section over $x$.
Therefore the map $\sE_x^\star\to\Hom_{L\CC}(\sE_x,L\CC)$ is injective.

To see that this map is surjective take any $\sigma  \in \Hom_{L\CC}(\sE_x,L\CC)$, so that for any $f \in \cF(\sE)_x$ the composition $\sigma \circ f \colon LV \to L\CC$ may be viewed as an element of $L(V^\ast)$, since $f$ is an $L\CC$-module isomorphism $LV\to \sE_x$.
Then for any $v \in \sE_x$
\[
\theta \longmapsto \sigma(v)(\theta) = (\sigma \circ f )(\theta) \big( f^{-1}(v)(\theta)\big)
\]
is the image of $[f,\sigma\circ f] \in \sE_x^\star$, which proves surjectivity.
\end{proof}

From this it follows that
\vspace{11pt}
\begin{corollary}
\label{cor:dualpair}
There is a `dual pairing'
\[
\sE \oast \sE^\star \lo \underline{L\CC},
\]
which is a homomorphism of $\Omega$ vector bundles.
\end{corollary}
\begin{proof}
This is an immediate consequence of the canonical pairing
\[
\sE_x \otimes_{L\CC} \Hom_{L\CC}(\sE_x,L\CC) \lo L\CC
\]
and of Lemmas \ref{lemma:oast} and \ref{lemma:altdual}.
\end{proof}

More generally, given two objects $\sE,\sF \in \ovect(M)$ modelled respectively over $LV$ and $LW$ with structure groups $\Omega G$ and $\Omega H$ (acting by the loop respresentation) one may define an $\Omega$ vector bundle $\Hom_{L\CC} (\sE,\sF) \in \ovect(M)$ using the clutching construction as follows.
Suppose that $\{U_\alpha\}$ is an open covering of $M$ over which $\sE$ and $\sF$ are both trivialised with transition functions $\tau_{\alpha\beta}$ and $\upsilon_{\alpha\beta}$ respectively.
Transition functions for $\Hom_{L\CC} (\sE,\sF)$ are given on $U_{\alpha\beta} := U_\alpha \cap U_\beta$ by
\begin{align*}
U_{\alpha\beta} \x \Hom_{L\CC}(LV,LW) &\lo U_{\alpha\beta} \x \Hom_{L\CC}(LV,LW)\\
(x,A) &\longmapsto (x, \tau_{\alpha\beta}\cdot A \cdot \upsilon_{\beta\alpha}).
\end{align*}
In this manner, $\Hom_{L\CC} (\sE,\sF)$ becomes an $\Omega$ vector bundle over $M$ with typical fibre $\Hom_{L\CC}(LV,LW) \cong L\Hom_\CC(V,W)$ and structure group $\Omega (G\x H)\cong \Omega G \x \Omega H$.
The frame bundle of $\Hom_{L\CC} (\sE,\sF)$ is isomorphic to the fibre product $\cF(\sE)\x_M\cF(\sF)$ and, by construction,
\vspace{11pt}
\begin{lemma}
\label{lemma:ovecthom}
For any $\sE,\sF \in \ovect(M)$ and $x\in M$ there is a natural $L\CC$-module isomorphism
\[
\Hom_{L\CC}(\sE,\sF)_x \simto \Hom_{L\CC}(\sE_x,\sF_x)
\]
that smoothly varies with $x$.
\end{lemma}

This gives an alternative way to obtain the honed dual and Corollary \ref{cor:dualpair}.

%%%%%%%%%%%%%%%%%%%%%%%%%%%%%%%%%%%%%%%%%%%%%%%%%%%%%%%%%%%%%%%%%%%%%%%%%%%%%%%%%%%%%%%%%
%%%%%%%%%%%%%%%%%%%%%%%%%%%%%%%%%%%%%%%%%%%%%%%%%%%%%%%%%%%%%%%%%%%%%%%%%%%%%%%%%%%%%%%%%
%%%%%%%%%%%%%%%%%%%%%%%%%%%%%%%%%%%%%%%%%%%%%%%%%%%%%%%%%%%%%%%%%%%%%%%%%%%%%%%%%%%%%%%%%

\subsection{$\Omega$ vector bundles and $L\cC^\infty(M)$-modules}
\label{SS:vector:module}
In this section, the caloron correspondence for vector bundles is used to obtain a convenient \emph{global} algebraic description of $\Omega$ vector bundles.
Specifically, one proves an adaptation of the well-known Serre-Swan Theorem that states that $\Omega$ vector bundles based over compact $M$ are equivalent to finitely-generated projective modules over the ring $L\cC^\infty(M)$ of smooth maps $M \to L\CC$ (with ring operations defined pointwise\footnote{it is perhaps more appropriate to define $L\cC^\infty(M)$ as the ring resulting from applying the looping functor to $\cC^\infty(M;\CC)$; however these definitions are readily seen to be equivalent with the help of the set-theoretic caloron correspondence \eqref{eqn:setcaloron}.}).
Notice that $L\cC^\infty(M)$ is, in fact, a Fr\'{e}chet space: this is one reason for requiring that $M$ be compact.

This algebraic description of $\Omega$ vector bundles comes about as a direct consequence of the fact that there is an isomorphism
\begin{equation}
\label{eqn:loopringiso}
L\cC^\infty(M) \lo \cC^\infty(M\x\sone;\CC)
\end{equation}
of Fr\'{e}chet spaces respecting the ring structure.
This isomorphism is given by sending $\check r \in L\cC^\infty(M)$ to $r \in \cC^\infty(M\x\sone;\CC)$ with
\[
r \colon (x,\theta) \longmapsto \check r(x)(\theta),
\]
i.e.~the set-theoretic caloron correspondence \eqref{eqn:setcaloron}.
\vspace{11pt}
\begin{theorem}
\label{theorem:module1}
For any $\sE \in \ovect(M)$ the set of global sections $\Gamma(\sE)$ of $\sE$ is a finitely-generated projective $L\cC^\infty(M)$-module.
\end{theorem}
\begin{proof}
Write $E:= \cV(\sE) \to M\x\sone$ for the caloron transform of $\sE$, with $\pi$ the projection $E \to M\x\sone$\, and let $\rank \sE = \rank E =m$, say.
Recall that the smooth Grassmannian manifolds
\[
\Gr_m(\CC^k) := \{ W \subset \CC^k \mid W \mbox{ is a subspace and $\dim W = m$}\} 
\]
come equipped with smooth vector bundles $\gamma^m(\CC^k) \to \Gr_m(\CC^k)$, where
\[
\gamma^m(\CC^k) := \{ (v,W) \mid v \in W\mbox{ and } W \in \Gr_m(\CC^k) \}\subset \CC^k \x\Gr_m(\CC^k),
\]
and that the $\Gr_m(\CC^k)$ are $n$-classifying for complex vector bundles of rank $m$, i.e. for any $n$ there is some $k$ depending on $m$ and $n$ such that every complex vector bundle $E \to X$ of rank $m$ with $\dim X \leq n$ is isomorphic to a pullback of $\gamma^m(\CC^k)$ by some continuous map $X \to \Gr_m(\CC^k)$ that is unique up to homotopy (see, for example, \cite[Section 5]{MilnorStasheff}).
Observe that the vector bundle $\gamma^m(\CC^k) \to \Gr_m(\CC^k)$ has a canonical framing over the point $\bp:= \mathrm{span}\{\mathbf{e}_1,\dotsc,\mathbf{e}_m\} \in \Gr_m(\CC^k)$ corresponding to the trivialisation
\[
\left(\sum_{i=1}^mf^i\mathbf{e}_i,\bp\right) \longmapsto \left(\bp,\sum_{i=1}^m f^i\mathbf{e}_i\right) \in \{\bp\} \x\CC^m,
\] where $\mathbf{e}_i$ denotes the $i$-th standard basis vector in $\CC^m$.

The strategy of this proof is to first construct a map of vector bundles $f\colon E \to \gamma^m(\CC^k)$ that covers a smooth map $\widetilde f\colon M\x\sone \to \Gr_m(\CC^k)$ and preserves the framing so that $\widetilde f$ is constant on $M_0 = M\x\{0\}$: this implies a canonical isomorphism $E\cong \widetilde f ^\ast\gamma^m(\CC^k)$ of bundles framed over $M_0$.
To construct such a map $f$, note that $E$ is trivial over $M_0$ and hence trivial over some (open) tubular neighbourhood $T$ of $M_0$ in $M\x\sone$, where one may choose a trivialisation $t \colon E|_T \to T\x \CC^m$ that agrees with the distinguished trivialisation (i.e. framing) of $E$ over $M_0$.

Proceeding as in the proof of \cite[Lemma 5.3]{MilnorStasheff}, let $T,U_1,\dotsc,U_r$ be a covering of $M\x\sone$ by open sets over which $E$ is trivial, noting that this requires the compactness of $M$.
Without loss of generality, one may suppose that $M_0$ is not contained in any of the $U_1,\dotsc, U_r$.
Write $U_0 := T$ and take open sets $V_0,\dotsc, V_r$ covering $M\x\sone$ such that $\overline{V}_i \subset U_i$ for each $i = 0,\dotsc,r$ and $M_0 \subset V_0$.
Similarly, take $W_0,\dotsc, W_r$ covering $M\x\sone$  with $\overline{W}_i \subset V_i$ and $M_0 \subset W_0$.

Now let $\varrho_i \colon M\x\sone \to \RR$ be a smooth bump function taking the value $1$ on $\overline{W}_i$ and the value $0$ outside of $V_i$.
Since $E$ is trivial over each of the $U_i$, there are smooth maps
\[
t_i \colon E|_{U_i} \lo \CC^m
\]
that restrict to linear isomorphisms on each fibre.
Using the above trivialisation $t$ of $E$ over $T$, one may suppose that for every $x \in M$, $t_0(\widehat{\mathbf{e}}_i(x,0)) = \mathbf{e}_i$  is the $i$-th standard basis vector of $\CC^m$, where the $\widehat{\mathbf{e}}_i$ are the sections of $E$ over $M_0$ determined by the framing.
Define the maps $T_i \colon E \to \CC^m$ by
\[
T_i(e) := 
\begin{cases}
0 &\mbox{ if $\pi(e) \notin V_i$}\\
\varrho_i(\pi(e)) t_i(e) &\mbox{ if $\pi(e) \in U_i$}
\end{cases}
\]
noting that the $T_i$ are smooth and linear on each fibre and that for every $x \in M$
\[
T_j (\widehat{\mathbf{e}}_i(x,0)) = 
\begin{cases}
0 &\mbox{ if $j \neq 0$}\\
\mathbf{e}_i &\mbox{ if $j=0$.}
\end{cases}
\]
Now define the smooth map
\[
\widehat f \colon E \lo \underbrace{\CC^m \oplus\dotsb\oplus\CC^m}_{\text{$r+1$ times}} \cong \CC^{(r+1)m}
\]
by $\widehat f(e) := (T_0(e),T_1(e),\dotsc,T_r(e))$, noting that $\widehat f$ is linear and injective on each fibre.
Taking $k = (r+1)m$, the map $f \colon E \to \gamma^m(\CC^k)$ given by
\[
f(e) := \Big(\widehat f(e),\widehat f\,\big(E_{\pi(e)}\big) \Big)
\]
is a smooth vector bundle map.
Additionally, for any $x \in M_0$, by construction one has
\[
f(\widehat{\mathbf{e}}_i(x,0)) = (\mathbf{e}_i ,\bp) \in \gamma^m(\CC^k)_\bp
\]
so that $f$ preserves the framing.
Thus $f$ covers a smooth map $\widetilde f \colon M\x\sone \to \Gr_m(\CC^k)$ that sends $M_0$ to $\bp$ and there is a canonical isomorphism $E\cong \widetilde f^\ast \gamma^m(\CC^k)$ that preserves the framing.

Recall that there is a diffeomorphism $\Gr_m(\CC^k) \to \Gr_{k-m}(\CC^k)$ given by sending the subspace $W\in \Gr_m(\CC^k)$ to its orthogonal complement $W^\perp\in\Gr_{k-m}(\CC^k)$.
Thus $\Gr_m(\CC^k)$ comes equipped with another smooth vector bundle $\gamma^{k-m}(\CC^k) \to \Gr_m(\CC^k)$, where the projection is given by $(v,W^\perp) \mapsto W$.
Observe that $\gamma^{k-m}(\CC^k)\to \Gr_m(\CC^k)$ also has a canonical framing over the point $\bp \in \Gr_m(\CC^k)$ corresponding to the trivialisation
\[
\left(\sum_{i=m+1}^k f^i\mathbf{e}_i,\bp^\perp \right) \longmapsto \left(\bp,\sum_{i=m+1}^k f^i\mathbf{e}_i\right) \in \{\bp\} \x\CC^m,
\]
and moreover that the framed vector bundle $\gamma^m(\CC^k) \oplus \gamma^{k-m}(\CC^k)$ is isomorphic to the trivial bundle of rank $k$ over $\Gr_m(\CC^k)$ with its obvious framing over $\bp$.

Now, the pullback $F:= \widetilde f^\ast \gamma^{k-m}(\CC^k) \to M\x\sone$ has a canonical framing over $M_0$ and by virtue of its construction there is a smooth isomorphism $E \oplus F \to \underline{\CC}^k$ of framed vector bundles.
Write $\sF := \cV^{-1}(F)$, so that taking the caloron transform and applying natural isomorphisms as necessary gives an isomorphism
\[
\sE \oplus \sF \simto \underline{L\CC}^k
\]
of $\Omega$ vector bundles over $M$.
Notice that $\Gamma(\underline{L\CC}^k)$ is a finitely-generated free module of rank $k$ over $L\cC^\infty(M)$, with basis at $x\in M$ given by $\hat{\mathbf{e}}_i(x) \colon \theta \mapsto \mathbf{e}_i$ with $\mathbf{e}_i$ the canonical basis for $\CC^k$.
In particular, $\Gamma(\sE) \oplus \Gamma(\sF)$ and $\Gamma(\underline{L\CC}^k)$ are isomorphic as $L\cC^\infty(M)$-modules (the module structures on $\Gamma(\sE)$ and $\Gamma(\sF)$ having been previously established).
Thus
\[
\Gamma(\sE) \oplus \Gamma(\sF) \cong \Gamma(\underline{L\CC}^k) \cong L\cC^\infty(M)^k,
\]
which shows that $\Gamma(\sE)$ is a finitely-generated projective $L\cC^\infty(M)$-module.
\end{proof}

There are two immediate consequences of this proof:
\vspace{11pt}
\begin{corollary}
\label{cor:oddkinverse}
Every $\Omega$ vector bundle $\sE \to M$, with $M$ compact, has an \emph{inverse}, i.e.~there is some $\Omega$ vector bundle $\sF \to M$ such that $\sE \oplus \sF$ is trivial.
\end{corollary}
\vspace{11pt}
\begin{corollary}
\label{cor:framedclassmap}
If $M$ is compact, then for any framed vector bundle $E\to M\x\sone$ of rank $m$ there is a smooth map $M\x\sone \to \Gr_m(\CC^k)$ for some $k$ that sends $M_0$ to the basepoint of $\Gr_m(\CC^k)$ with the additional property that $E$ and $f^\ast \gamma^m(\CC^k)$ are isomorphic as framed vector bundles over $M\x\sone$.
\end{corollary}

The converse to Theorem \ref{theorem:module1} also holds, giving the $\Omega$ vector bundle version of the Serre-Swan Theorem:
\vspace{11pt}
\begin{theorem}
\label{theorem:module}
An $L\cC^\infty(M)$-module $P$ is isomorphic to the set of sections $\Gamma(\sE)$ of some $\Omega$ vector bundle $\sE$ over $M$ if and only if $P$ is finitely-generated and projective.
\end{theorem}
\begin{proof}
One direction is Theorem \ref{theorem:module1}.
For the other direction, one closely follows the proof of \cite[Theorem 11.32]{N}.

Suppose that the $L\cC^\infty(M)$-module $P$ is finitely-generated and projective.
Up to isomorphism, the free $L\cC^\infty(M)$-module on $k$ generators is $\Gamma(\underline{L\CC}^k)$ and hence, for some $k$, there is an isomorphism
\[
\psi\colon P\oplus Q \simto \Gamma(\underline{L\CC}^k)
\]
for some $L\cC^\infty(M)$-module $Q$.
Let $\ev_x$ be the evaluation map at $x \in M$ and define
\[
P_x := \{\ev_x \circ\, \psi (p) \mid p\in P\}
\]
with $Q_x$ defined similarly.
As all of the maps involved are $L\cC^\infty(M)$-homomorphisms, $P_x$ and $Q_x$ are $\CC$-linear vector subspaces of $L\CC^k$.
One proceeds now by showing that the $P_x$ are the fibres of some $\Omega$ vector bundle $\sP$ over $M$.

First take any $v \in L\CC^k$ and choose $s \in \Gamma(\underline{L\CC}^k)$ extending $v$, i.e. $s(x) = v$.
Then $s = \psi (p+q)$ for some unique $p\in P$ and $q\in Q$ and hence $s(x) = v \in P_x + Q_x$, which shows that $P_x + Q_x = L\CC^k$.

Suppose that $v \in P_x \cap Q_x$ so that $ v = \ev_x \circ\, \psi (p) = \ev_x \circ\, \psi (q)$ for some $p \in P$ and $q\in Q$.
Writing $\psi (p)$ and $\psi (q)$ in terms of the standard basis elements $\check{\mathbf{e}}_i = \psi(p_i +q_i)$ of $\Gamma(\underline{L\CC}^k)$, one obtains
\[
\psi (p) - \psi (q) = \sum_{i=1}^k f^i\check{\mathbf{e}}_i 
= \sum_{i=1}^k f^i\psi(p_i) + \sum_{i=1}^k f^i\psi (q_i)
\]
for some $f^1,\dotsc,f^k \in L\cC^\infty(M)$.
The functions $f^i \in L\CC$ must vanish at $x$ and by the direct sum decomposition of $\Gamma(\underline{L\CC}^k)$
\[
\psi(p) = \sum_{i=1}^k f^i\psi (p_i),
\]
so that $v= \ev_x\circ\,\psi (p) = 0$.
This shows that $L\CC^k = P_x \oplus Q_x$.

It remains to show that the vector spaces $P_x$ determine an $\Omega$ vector bundle.
Fix $x \in M$ and $\theta\in\sone$ and set $P_x(\theta) := \ev_\theta(P_x)$, so that $P_x(\theta) \oplus Q_x(\theta) = \CC^k$ as $\CC$-vector spaces.
This gives an isomorphism $P_x(\theta) \cong \CC^{n(x,\theta)}$ for some $n(x,\theta) \leq k$.
Now choose $p_1,\dotsc,p_{n(x,\theta)} \in P$ such that the $\ev_x\circ\,\psi(p_i)$ form a basis for $P_x(\theta)$.
Since the $\psi(p_i)$ are smooth sections of $\underline{L\CC}^k$, $n(x,\theta) = \dim P_x(\theta)$ is an upper semi-continuous function of $x \in M$ and $\theta\in\sone$.
The same holds for $\dim Q_x(\theta)$, so since $\dim P_x(\theta) + \dim Q_x(\theta) = k$ is constant, $\dim P_x(\theta)$ is a locally constant function of $x$ and $\theta$.
In particular, there is some $n\leq k$ such that $P_x \cong L\CC^n$ as $L\CC$-modules for all $x \in M$.

For a fixed $x \in M$ choose $n$ elements $v_1,\dotsc,v_n$ generating $P_x$ as an $L\CC$-module, so that the $\ev_\theta (v_i)$ form a basis for $P_x(\theta)$ for all $\theta\in\sone$ with the additional property that $\ev_0 (v_i) = \mathbf{e}_i$, with $\{\mathbf{e}_i\}_{i=1}^n$ the standard basis of $\CC^n$.
Choosing $p_1,\dotsc,p_n \in P$ such that $v_i = \ev_x\circ\,\psi(p_i)$, by compactness of $S^1$ one may find a neighbourhood $U$ of $x$ in $M$ such that the vectors $v_i(z,\theta) := \ev_\theta\circ \ev_z\circ\,\psi(p_i)$ are linearly independent elements of $P_z(\theta)$ for all $(z,\theta) \in U\x\sone$.
Writing $v_i(z)\colon \theta \mapsto v_i(z,\theta)$, the map
\[
\sum_{i=0}^n f^i v_i \longmapsto \left(z, \sum_{i=0}^n f^i(z) v_i(z) \right)
\]
gives a local basis for the family of $L\CC$-modules $P_x$ that depends smoothly on $x$.
Thus the family $P_x$ determines an open covering of $\{U_\alpha\}_{\alpha\in I}$ of $M$ over which there are trivial $\Omega$ vector bundles.
Moreover, by construction these trivial $\Omega$ vector bundles patch together on double overlaps of the $U_\alpha$ via $L\CC$-module isomorphisms that correspond to the action of elements of $\Omega GL_n(\CC)$.
This determines an $\Omega$ vector bundle $\sP \to M$ of rank $n$ whose fibre over $x \in M$ is $P_x$.

By the construction of $\sP$, the map $\psi \colon P \to \Gamma(\sP)$ is an injective homomorphism of $L\cC^\infty(M)$-modules.
Take $s_\sP \in \Gamma(\sP)$ and identify with the corresponding element $s \in \Gamma(\underline{L\CC}^k)$ so that $s = \psi   (p+ q)$ for some unique $p\in P$ and $q\in Q$.
Applying $\ev_x$ to this equation gives $\ev_x \circ\,\psi (q) = 0$ for all $x \in M$ so that $q=0$ and $s_\sP$ lies in the image of $\psi$, in particular $\psi$ is an isomorphism of $L\cC^\infty(M)$-modules.
\end{proof}

This section is concluded by recording a relationship between the sections of $E \in \frvect(M)$ and its caloron transform $\sE \in\ovect(M)$.
This relationship can be thought of as the vector bundle version of \eqref{eqn:secidentg} and is crucial to the constructions of Section \ref{S:vectorgcal}.
\vspace{11pt}
\begin{proposition}
\label{prop:sections}
For each $E \in \frvect(M)$, there is a natural isomorphism
\[
\mu_E \colon \Gamma(E) \simto \Gamma(\sE)
\]
of $L\cC^\infty(M)$-modules, where $\sE:=\cV^{-1}(E)$. 
\end{proposition}
\begin{proof}
Recalling the map $\eta$ of \eqref{eqn:etamap}, each section $s \in \Gamma(E)$ may be viewed as a section $\check s \in \Gamma(\eta^\ast LE)$ via
\[
\check s (x)\colon \theta \longmapsto s(x,\theta)
\]
and conversely.
Moreover, the assignment $s \mapsto \check s$ is an $L\cC^\infty(M)$-module isomorphism so that composing with the map $\lambda_E^{-1}$ of \eqref{eqn:altvcal} gives an $L\cC^\infty(M)$-module isomorphism.
Denote by $\mu_E$ this composition.
\end{proof}

There is a similar map $\mu_\sE$ corresponding to each $\sE \in \ovect(M)$.
In this case, write $E := \cV(\sE)$ and set $\mu_\sE = \alpha_\sE \circ \mu_E$ where $\alpha_\sE$ is the natural isomorphism of Theorem \ref{theorem:vector}.
Then this is an isomorphism 
\[
\mu_\sE \colon \Gamma(E) \simto \Gamma(\sE)
\]
of $L\cC^\infty(M)$-modules.

One occasionally also requires \emph{local} versions of the maps $\mu_E$ and $\mu_\sE$.
It is readily seen from the argument of Proposition \ref{prop:sections} that for every $U \subset M$ there are isomorphisms
\begin{align*}
\mu_E \colon \Gamma(U\x\sone, E) &\lo \Gamma(U,\cV^{-1}(E))\\
\mu_\sE \colon \Gamma(U\x\sone,\cV(\sE)) &\lo \Gamma(U,\sE)
\end{align*}
of $L\cC^\infty(U)$-modules.
Whether the local or global versions of $\mu_E$ and $\mu_\sE$ are used depends on the context.

%%%%%%%%%%%%%%%%%%%%%%%%%%%%%%%%%%%%%%%%%%%%%%%%%%%%%%%%%%%%%%%%%%%%%%%%%%%%%%%%%%%%%%%%%
%%%%%%%%%%%%%%%%%%%%%%%%%%%%%%%%%%%%%%%%%%%%%%%%%%%%%%%%%%%%%%%%%%%%%%%%%%%%%%%%%%%%%%%%%
%%%%%%%%%%%%%%%%%%%%%%%%%%%%%%%%%%%%%%%%%%%%%%%%%%%%%%%%%%%%%%%%%%%%%%%%%%%%%%%%%%%%%%%%%

\section{Higgs fields and module connections}
\label{S:vectorgcal}
\label{SS:vector:higgs}
As in the case for principal bundles, there is a refinement of Theorem \ref{theorem:vector} incorporating connective data that greatly increases the utility of the vector bundle caloron correspondence.
This \emph{geometric vector bundle caloron correspondence} can be obtained directly by applying the frame bundle functors to $\cC$ and $\cC^{-1}$ (as with the construction of $\cV$ and $\cV^{-1}$); however it is perhaps more illuminating to see how the geometric data behaves directly.

To begin with, one must first define the appropriate connective data.
\vspace{11pt}
\begin{definition}
\label{defn:vectorhf}
Let $\sE \to M$ be an $\Omega$ vector bundle.
A \emph{(vector bundle) Higgs field} on $\sE$ is a smooth vector bundle endomorphism $\phi \colon \sE \to \sE$ covering the identity such that
\[
\phi(f v) = f\phi(v) + \d(f)v
\]
for all $f \in L\CC$ and $v \in \sE$.
The map $\d$ is the derivation
\[
\d\colon f \longmapsto \frac{\d f}{\d\theta}
\]
on $L\CC$ and juxtaposition denotes the product given by the $L\CC$-module structure on the fibres.
Write $\cH_\sE$ for the space of Higgs fields on $\sE$, which is an affine space modelled over the vector space of sections of $\Hom_{L\CC}(\sE,\sE)$.
\end{definition}
\vspace{11pt}
\begin{example}[the trivial Higgs field]
The trivial $\Omega$ vector bundle $\underline{L\CC}^k \to M$ has a canonical Higgs field called the \emph{trivial} Higgs field, namely
\[
\d \colon (x,v) \longmapsto \Big(x,\frac{\d v}{\d\theta}\,\Big).
\]
\end{example}
\vspace{11pt}
\begin{example}
\label{example:stdhigg}
Recall from Example \ref{example:universalomega} the universal $\Omega$ vector bundle $\sE(V) \to G$ of type $V$.
Using Lemma \ref{lemma:higgscorrespond} below applied to $PG$ equipped with the standard Higgs field $\Phi_\infty$ gives a vector bundle Higgs field $\phi_V$ called the \emph{standard Higgs field} for $\sE(V)$.
\end{example}

Similarly to Lemma \ref{lemma:hfexist} on the existence of Higgs fields on principal bundles, one has
\vspace{11pt}
\begin{lemma}
Every $\Omega$ vector bundle $\sE \to M$ has a Higgs field.
\end{lemma}
\begin{proof}
A convex combination of Higgs fields is again a Higgs field.
The result follows from the fact that $\sE$ is locally trivial and since $M$ admits smooth partitions of unity.
\end{proof}

As the name might suggest, there is an important relationship between principal bundle Higgs fields and vector bundle Higgs fields:
\vspace{11pt}
\begin{lemma}
\label{lemma:higgscorrespond}
Higgs fields on the $\Omega$ vector bundle $\sE \to M$ are in bijective correspondence with Higgs fields on the frame bundle $\cF(\sE) \to M$.
\end{lemma}
\begin{proof}
Suppose that $\rank\sE=n$.
Let $\phi$ be a Higgs field on $\sE$ and suppose that over the open set $U_\alpha \subset M$ there is a framing
\[
\alpha \colon U_\alpha \x L\CC^n \lo \sE|_{U_\alpha}
\]
so that $\alpha$ is an isomorphism of $L\CC$-modules on each fibre over $U_\alpha$.
This amounts to having $n$ sections $\check {\mathbf{e}}_i \in \Gamma(U_\alpha, \sE)$ that form a basis for the $L\CC$-module structure of each fibre of $\sE$ over $U_\alpha$.
Fixing $x \in U_\alpha$, take any  $v = f^i\check {\mathbf{e}}_i \in \sE_x$ (with summation over repeated indices understood) so that
\[
\phi(v) = f^i \phi(\check {\mathbf{e}}_i) + \d(f^i) \check {\mathbf{e}}_i.
\]
Write $\phi(\check {\mathbf{e}}_i) = \Phi(\alpha)_i^j \check {\mathbf{e}}_j$, where $\Phi_i^j \in L\CC$ for each $i,j = 1,\dotsc,n$, then
\[
\phi(v) = (\d f^j  + f^i\Phi(\alpha)_i^j) \,\check {\mathbf{e}}_j.
\]
Now, $\g = \Lie(G) = \End(\CC^n)$ so that the data $\left\{\Phi(\alpha)_i^j\right\}_{i,j = 1}^n$ determine a map 
\[
\Phi(\alpha)\colon  \sone \to L\g
\]
by sending $\theta$ to the matrix whose $(i,j)$-th entry is $\Phi(\alpha)_i^j(\theta)$.

If
\[
\beta \colon U_\beta \x L\CC^n  \lo \sE|_{U_\beta}
\]
is another such local framing of $\sE$ then on $U_{\alpha\beta}$ one has $\alpha = \beta\circ \tau_{\alpha\beta}$ for some $\tau_{\alpha\beta}$ valued in $\Omega G$.
As before, write $\check{\mathbf{f}}_i$ for the basis for the $L\CC$-module structure given by the framing $\beta$.
Then
\[
\check{\mathbf{e}}_i = ( \tau_{\alpha\beta} )_i^j \check{\mathbf{f}}_j,
\]
with $( \tau_{\alpha\beta})_i^j (\theta)$ the $(i,j)$-th component of $\tau_{\alpha\beta}(\theta) \in G$.
Choosing any $x \in U_{\alpha\beta}$ and any element $v = f^i\check {\mathbf{e}}_i = f^i( \tau_{\alpha\beta})_i^j \check{\mathbf{f}}_j \in \sE_x$ one obtains
\begin{align*}
\phi(v) &= \left(\d \!\left\{f^i ( \tau_{\alpha\beta} )_i^j \right\}  + f^i ( \tau_{\alpha\beta} )_i^k\Phi(\beta)_k^j\right) \check {\mathbf{f}}_j \\
%%%
&= (\d f^j) \check {\mathbf{e}}_j  + f^i \big(\d  ( \tau_{\alpha\beta} )_i^j \big) \check {\mathbf{f}}_j + f^i ( \tau_{\alpha\beta} )_i^k\Phi(\beta)_k^j \check {\mathbf{f}}_j.
\end{align*}
Comparing this with the above expression yields
\[
\Phi(\alpha) = \tau_{\alpha\beta}^{-1} \d(\tau_{\alpha\beta}) + \ad(\tau_{\alpha\beta}^{-1}) \Phi(\beta).
\]
The data $\{\Phi(\alpha)\}$ define a Higgs field $\Phi$ on $\cF(\sE)$ as follows: given a frame $a \in \cF(\sE)_x$ choose some local framing $\alpha$ of $\sE$ extending $a$.
Then take $\Phi(a)$ to be the restriction of $\Phi(\alpha)$ to the fibre of $\sE$ over $x$, noting that this is independent of the chosen extension of $a$.
Since the action of $\Omega G$ on $\cF(\sE)$ is given by precomposition, the above equation implies that if $b = R_\gamma(a)$ for some  $\gamma \in \Omega G$ then
\[
\Phi(b) = \ad(\gamma^{-1})\Phi(a) + \gamma^{-1}\d\gamma
\]
which is the required twisted equivariance condition.
It remains to notice that the map $\Phi \colon \cF(\sE) \to L\g$ is indeed smooth, since $\phi$ is.
It is evident from this argument that $\phi$ uniquely determines $\Phi$ and conversely so the result follows.
\end{proof}

Many of the operations on $\Omega$ vector bundles may also be adapted to Higgs fields.
The first of these operations is pullback: given an $\Omega$ vector bundle $\sE \to M$ equipped with Higgs field $\phi$ and a smooth map $f \colon N \to M$, one may define the pullback $f^\ast\phi$ as the unique Higgs field on $f^\ast \sE$ that acts on pullback sections $f^\ast s$ by
\[
\left(f^\ast\phi\right)(f^\ast s) := f^\ast\! \left(\phi(s) \right)
\]
for any section $s$ of $\sE$.
It follows from the argument of Lemma \ref{lemma:higgscorrespond} that the Higgs field on $\cF(f^\ast \sE)$ corresponding to $f^\ast \phi$ is the pullback of the Higgs field on $\cF(\sE)$ corresponding to $\phi$.
In the case that $f\colon \sF\to \sE$ is an isomorphism of $\Omega$ vector bundles
\[
 f^\ast \phi := f^{-1} \circ \phi \circ f.
\]
 The direct sum and honed product of $\Omega$ vector bundles also extend to Higgs fields.
If $\sE \to M $ and $\sF\to M$ have Higgs fields $\phi$ and $\psi$ respectively, then via the identification of Lemma \ref{lemma:whitney} the formula
\[
(\phi\oplus\psi)(v,w):= \big(\phi(v),\psi(w)\big)
\]
defines a Higgs field on $\sE\oplus\sF$, called the \emph{direct sum} of $\phi$ and $\psi$.
Similarly, using the identification of Lemma \ref{lemma:oast}
\[
(\phi \oast \psi) (v \otimes w) := \phi(v) \otimes w + v \otimes\psi(w) 
\]
defines the \emph{product} Higgs field on $\sE \oast\sF$, the products $\otimes$ above understood to be over the ring $L\CC$.

As in the principal bundle case, there is a notion of the \emph{holonomy} of a Higgs field $\phi$ on $\sE \to M$.
Writing $\Phi$ for the Higgs field on $\cF(\sE)$ associated to $\phi$ via Lemma \ref{lemma:higgscorrespond}, the holonomy of $\phi$ is simply the homomorphism of $\Omega$ vector bundles
\[
\hol_\phi \colon \cF(\sE)\x_{\Omega G}LV \lo \sE(V)
\]
that sends 
\[
\hol_\phi \colon [p,v] \longmapsto [\hol_\Phi(p),v],
\]
noting that this is well-defined.
As in the case of $\Omega G$-bundles, $\hol_\phi$ descends to a map $\hol_\phi \colon M\to G$ (in fact,  $\hol_\phi$ and $\hol_\Phi$ agree as maps $M\to G$) so that
\vspace{11pt}
\begin{lemma}
\label{lemma:higgshol}
The holonomy map $\hol_\phi$ is a smooth classifying map for $\sE$.
\end{lemma}
\vspace{11pt}
\begin{remark}
Comparing with the discussion preceding Proposition \ref{prop:higgsfieldholonomy}, one can also see that, up to isomorphism, the pullback Higgs field $\hol_\phi^\ast \phi_V$ coincides with the original Higgs field $\phi$ on $\sE$.
\end{remark}

The second geometric datum required for the geometric caloron correspondence for vector bundles is a \emph{module connection}, which is a connection on the $\Omega$ vector bundle $\sE \to M$ that respects the $L\cC^\infty(M)$-module structure.
First notice that for any $\sE \in \ovect(M)$ there is an \emph{absorbing isomorphism}
\[
m_\sE \colon \Omega^\bullet(M;L\CC) \otimes_{L\cC^\infty (M)} \Gamma(\sE) \lo \Omega^\bullet(M) \otimes\Gamma(\sE).
\]
This is simply the natural map
\[
\big( \Omega^\bullet(M) \otimes L\cC^\infty(M) \big) \otimes_{L\cC^\infty(M)} \Gamma(\sE) \simto \Omega^\bullet(M) \otimes \Gamma(\sE)
\]
since by definition $\Omega^\bullet(M;L\CC) = \Omega^\bullet(M) \otimes L\cC^\infty(M)$.
The unadorned tensor products above are understood to be over the ring $\cC^\infty(M;\CC)$ and the map $m_\sE$ may be thought of as absorbing all of the $\sone$ dependence of an $L\CC$-valued form into the $\Gamma(\sE)$ factor.
\vspace{11pt}
\begin{definition}
\label{defn:moduleconnection}
A connection $\Delta \colon \Gamma(\sE) \to \Omega^1(M) \otimes \Gamma(\sE)$ on the $\Omega$ vector bundle $\sE \to M$ is a \emph{module connection} if for all $z \in L\cC^\infty(M)$ and $v \in \Gamma(\sE)$
\[
\Delta(zv) = z\Delta(v) + m_\sE(dz \otimes v)
\]
where $dz \in \Omega^1(M;L\CC)$ is the standard exterior derivative of $z$.
\end{definition}
\vspace{11pt}
\begin{remark}
In this thesis, module connections are commmonly denoted by $\Delta$ (rather than the more standard notation $\nabla$ for a connection) in order to emphasise the fact that they are \emph{module} connections.
This notation also helps one to differentiate between a module connection and its caloron transform.
\end{remark}
\vspace{11pt}
\begin{example}[the trivial module connection]
\label{ex:triviconn}
The trivial $\Omega$ vector bundle $\underline{L\CC}^k \to M$ has a canonical module connection, namely
\[
\delta\colon v \longmapsto m_{\underline{L\CC}^k}(d v)
\]
for $v \in \Gamma(\underline{L\CC}^k)$.
This is the \emph{trivial} module connection.
\end{example}
\vspace{11pt}
\begin{example}
Using Lemma \ref{lemma:moduleconnectioncorrespond} below applied to $PG$ equipped with its standard connection $\sA_\infty$ gives a module connection $\Delta_V$ on $\sE(V) \to G$, called the \emph{standard  module connection}.
\end{example}
The existence of module connections is guaranteed by
\vspace{11pt}
\begin{lemma}
Every $\Omega$ vector bundle $M$ has a module connection.
\end{lemma}
\begin{proof}
A convex combination of module connections is again a module connection.
The result follows from local triviality since $M$ admits smooth partitions of unity.
\end{proof}

Similarly to the relationship between Higgs fields on $\sE$ and its frame bundle $\cF(\sE)$, there is a relationship between module connections on $\sE$ and principal connections on the frame bundle $\cF(\sE)$.
This relationship reflects the fact that the structure group of $\sE$ is $\Omega G$ instead of the larger group $GL(LV)$:
\vspace{11pt}
\begin{lemma}
\label{lemma:moduleconnectioncorrespond}
Fix an $\Omega$ vector bundle $\sE \to M$ with typical fibre $LV$ and structure group $\Omega G$.
Module connections on $\sE$ are in bijective correspondence with $LG$-connections on $\cF(\sE)$.
\end{lemma}
\begin{proof}
Let $\rank\sE=n$.
Let $\Delta$ be a module connection on $\sE$ and suppose that
\[
\alpha \colon U_\alpha \x L\CC^n \lo \sE|_{U_\alpha}
\]
is a framing over the open set $U_\alpha \subset M$, so $\alpha$ is an isomorphism of $L\CC$-modules on each fibre over $U_\alpha$.
As in the proof of Lemma \ref{lemma:higgscorrespond}, let $\check {\mathbf{e}}_i \in \Gamma(U_\alpha,\sE)$ be the sections corresponding to $\alpha$.
Then for any $v \in \Gamma(U_\alpha,\sE)$, observe that $v = f^i\check {\mathbf{e}}_i$ for some $f^1,\dotsc,f^n \in L\CC$ so that
\begin{align*}
\Delta (v) &= f^i \Delta (\check {\mathbf{e}}_i) + m_\sE (d f^i \otimes \check {\mathbf{e}}_i)
\end{align*}
since $\Delta$ is a module connection.
Write $\Delta(\check {\mathbf{e}}_i) = (\sA_\alpha)_i^j \check {\mathbf{e}}_j$, where $(\sA_\alpha)_i^j \in \Omega^1(U_\alpha, L\CC)$ for $i,j = 1,\dotsc,n$.
Thus
\begin{align*}
\Delta (v) &= f^i (\sA_\alpha)_i^j \check {\mathbf{e}}_j + m_\sE (d f^i \otimes \check {\mathbf{e}}_i).
\end{align*}
Recalling that $\g = \Lie(G) =  \End(\CC^n)$, the collection $\left\{(\sA_\alpha)_i^j\right\}_{i,j=1}^n$ determines an $L\g$-valued $1$-form $\sA_\alpha$ on $U_\alpha$.

If
\[
\beta \colon U_\beta \x L\CC^n  \lo \sE|_{U_\beta}
\]
is another such framing then on $U_{\alpha\beta}:=U_\alpha\cap U_\beta$ one has $\alpha = \beta\circ \tau_{\alpha\beta}$ for some $\tau_{\alpha\beta}$ valued in $\Omega G$.
As before, write $\check{\mathbf{f}}_i$ for sections corresponding to the framing $\beta$.
Then
\[
\check{\mathbf{e}}_i = (\tau_{\alpha\beta})_i^j \check{\mathbf{f}}_j
\]
with $( \tau_{\alpha\beta})_i^j \colon \sone \to \CC$ the map given by taking the $(i,j)$-th component of $\tau_{\alpha\beta} \colon \sone \to \g$.
Now take any $v = f^i\check {\mathbf{e}}_i = f^i( \tau_{\alpha\beta})_i^j \check{\mathbf{f}}_j \in \Gamma(U_{\alpha\beta},\sE)$ so that
\begin{align*}
\Delta(v) &= f^i (\tau_{\alpha\beta})_i^j (\sA_\beta)_j^k \check {\mathbf{f}}_k + m_\sE (d (f^i (\tau_{\alpha\beta})_i^j) \otimes \check {\mathbf{f}}_j) \\
%%%
&= f^i (\tau_{\alpha\beta})_i^j (\sA_\beta)_j^k \check {\mathbf{f}}_k + m_\sE (d f^i \otimes \check {\mathbf{e}}_i) + f^i m_\sE (d (\tau_{\alpha\beta})_i^j \otimes \check {\mathbf{f}}_j).
\end{align*}
Comparing this with the above expression yields
\[
\sA_\alpha = \tau_{\alpha\beta}^{-1} d\tau_{\alpha\beta} + \ad(\tau_{\alpha\beta}^{-1}) \sA_\beta,
\]
which is precisely the local transformation law for an $LG$-connection on $\cF(\sE)$.
Thus the data $\{\sA_\alpha \}$ are the local connection $1$-forms of an $L\g$-valued connection $\sA$ on $\cF(\sE)$ (i.e. $\alpha^\ast \sA = \sA_\alpha$, with $\alpha$ viewed as a local section of $\cF(\sE)$).
This argument also shows that module connections on $\sE$ uniquely determine $LG$-connections on $\cF(\sE)$ and conversely.
\end{proof}
\vspace{11pt}
\begin{remark}
\label{remark:vectorprincipalcurv}
The curvature of the module connection $\Delta$ is the $L\g$-valued $2$-form defined by its action on the section $s \in \Gamma(\sE)$ as
\[
\sR(X,Y) s= \Delta_X\Delta_Y s- \Delta_Y\Delta_X s- \Delta_{[X,Y]}s
\]
for tangent vectors $X$ and $Y$ to $M$.
If $\sA$ is the principal connection on the frame bundle corresponding to $\Delta$ and $\sF = d\sA +\tfrac{1}{2}[\sA,\sA]$ is its curvature form, then by a straightforward adaptation of the standard argument (for finite rank vector bundles) one may show that $\sR = \sF$.
\end{remark}
\vspace{11pt}
\begin{remark}
Notice that Lemma \ref{lemma:moduleconnectioncorrespond} is phrased in terms of $LG$-connections on $\cF(\sE)$ rather than $\Omega G$-connections.
This is because, via the clutching construction for example, $\cF(\sE)$ may be taken to be an $LG$-bundle.
\end{remark}
\vspace{11pt}
\begin{definition}
\label{defn:basedmodcon}
A module connection $\Delta$ on $\sE$ is \emph{based} if the corresponding connection form $\sA$ on $\cF(\sE)$ is an $\Omega G$-connection.
The space of all based module connections on the $\Omega$ vector bundle $\sE\to M$ is denoted $\cM_\sE$ and is an affine space modelled over $\Omega^1(M;\Omega\g)$.
\end{definition}
\vspace{11pt}
Given an $\Omega$ vector bundle $\sE \to M$ equipped with (based) module connection $\Delta$ and a smooth map $f \colon N \to M$, one may define the pullback $f^\ast\Delta$ as the unique (based) module connection on $f^\ast \sE$ that acts on pullback sections $f^\ast s$ by
\[
\left(f^\ast\Delta\right)_X(f^\ast s) := f^\ast\! \left(\Delta_{df(X)}s \right)
\]
for any section $s$ of $\sE$ and vector field $X$ on $N$.
Observe that not every section of $f^\ast \sE$ is a pullback; however it is sufficient to define $f^\ast\Delta$ on pullback sections since these determine  local frames for $f^\ast\sE$.
It is straightforward to verify that the connection on $\cF(f^\ast \sE)$ corresponding to $f^\ast \Delta$ via Lemma \ref{lemma:moduleconnectioncorrespond} is given by pulling back the connection on $\cF(\sE)$ corresponding to $\Delta$.
In the case that $f\colon \sF\to \sE$ is an isomorphism of $\Omega$ vector bundles,
\[
f^\ast \Delta := \big(\!\id\otimes f^{-1}\big)\Delta f
\]
omitting composition signs.

If $\sE \to M $ and $\sF\to M$ have module connections $\Delta$ and $\Delta'$ respectively, then via the identification of Lemma \ref{lemma:whitney} the formula
\[
\big( \Delta \oplus \Delta' \big)(v,w):= \big(\Delta v,\Delta'w\big)
\]
defines a module connection on $\sE\oplus\sF$ called the \emph{direct sum} of $\Delta$ and $\Delta'$.
Similarly, using the identification of Lemma \ref{lemma:oast} the formula
\[
\big(\Delta \oast \Delta')_X(v\otimes w) := \Delta_X v \otimes w + v \otimes \Delta'_X w
\]
defines the \emph{product} module connection on $\sE \oast\sF$ where, as was the case for Higgs fields, the products $\otimes$ above understood to be over the ring $L\CC$ and not $\CC$.
If $\Delta$ and $\Delta'$ are based, so too are their direct sum and product.
\vspace{11pt}
\begin{definition}
Let $\ovect^c(M)$ be the groupoid whose objects are triples $(\sE,\Delta,\phi)$ where $\sE \in \ovect(M)$ and $\Delta, \phi$ are respectively a based module connection and Higgs field on $\sE$.
Morphisms of $\ovect^c(M)$ are the morphisms of $\ovect(M)$ respect the additional data.

Write $\ovect_V^c(M)$ for the subgroupoid of $\ovect^c(M)$ consisting of those $\Omega$ vector bundles over $M$ with based module connection and Higgs field that are modelled over $LV$, with its appropriate morphisms.
\end{definition}

As with $\ovect$, one may view the assignment
\[
M \longmapsto \ovect^c(M)
\]
as a contravariant functor
\[
\ovect^c \colon \Man \lo \bimon
\]
acting on morphisms by pullback.

%%%%%%%%%%%%%%%%%%%%%%%%%%%%%%%%%%%%%%%%%%%%%%%%%%%%%%%%%%%%%%%%%%%%%%%%%%%%%%%%%%%%%%%%%
%%%%%%%%%%%%%%%%%%%%%%%%%%%%%%%%%%%%%%%%%%%%%%%%%%%%%%%%%%%%%%%%%%%%%%%%%%%%%%%%%%%%%%%%%
%%%%%%%%%%%%%%%%%%%%%%%%%%%%%%%%%%%%%%%%%%%%%%%%%%%%%%%%%%%%%%%%%%%%%%%%%%%%%%%%%%%%%%%%%

\subsection{The geometric caloron correspondence}
\label{SS:vector:geometric}
One now has almost all of the machinery required for the geometric caloron correspondence for vector bundles.
The only notion not yet defined is a special class of connections that, in a fashion completely analogous to Definition \ref{caloron:framedconnection}, respect the framing data of a framed vector bundle.
\vspace{11pt}
\begin{definition}
Suppose that $E \to X$ is framed over $X_0 \subset X$ with framing $s_0$, say.
A connection $\nabla \colon \Gamma(E) \to\Omega^1(M) \otimes \Gamma(E)$ is \emph{framed (with respect to $s_0$)} if
\[
\big(\!\id\otimes \,s_0^{-1}\big) \imath_{X_0}^\ast \nabla = d \circ s_0^{-1},
\]
with $\imath_{X_0} \colon X_0 \hookrightarrow X$ the inclusion map.
\end{definition}
\vspace{11pt}
\begin{remark}
\label{remark:framed}
Given such a framed bundle $E\to X$, observe that the framing determines trivialising sections $\mathbf{e}_1,\dotsc,\mathbf{e}_n$ of $E$ over $X_0$.
The condition that the connection $\nabla$ on $E$ be framed with respect to $s_0$ is precisely the requirement that the sections $\mathbf{e}_i$ are \emph{parallel} for $\nabla$, i.e.~$\nabla_Y\mathbf{e}_i = 0$ for each $i$ and tangent vector $Y$ to $X_0$.
\end{remark}

To see that the above definition really is the vector bundle version of Definition \ref{caloron:framedconnection}, one observes
\vspace{11pt}
\begin{lemma}
\label{lemma:framedconnections}
A connection $\nabla$ on $E \in \frvect(X)$ is framed if and only if the associated connection $A$ on $\cF(E)$ is framed.
\end{lemma}
\begin{proof}
Following the standard argument (compare with Lemma \ref{lemma:moduleconnectioncorrespond}) gives
\[
\big(\!\id\otimes \,s_0^{-1}\big) \imath_{X_0}^\ast \nabla = d\circ s_0^{-1} +s_0^\ast A,
\]
where $A$ is the connection on $\cF(E)$ corresponding to $\nabla$.
Thus $\nabla$ is framed if and only if $s_0^\ast A$ vanishes.
\end{proof}

This result guarantees the existence of framed connections on framed vector bundles since their frame bundles have framed connections.

It is clear from the definition that pullbacks, direct sums and (tensor) products of framed connections are also framed (with respect to the pullback, direct sum and product framings respectively).
\vspace{11pt}
\begin{definition}
Write $\frvect^c(M)$ for the groupoid whose objects are pairs $(E,\nabla)$ with $E \in \frvect(M)$ and $\nabla$ a framed connection on $E$.
Morphisms of $\frvect^c(M)$ are given by the connection-preserving morphisms of $\frvect(M)$.

Write $\frvect_V^c(X)$ for the subgroupoid consisting of pairs $(E,\nabla)$, where $E \in \frvect_V(X)$, equipped with the appropriate morphisms.
\end{definition}

Notice that $\frvect^c(M)$ is a bimonoidal category with respect to the operations of (framed) direct sum $\oplus$ and (framed) tensor product $\otimes$.

One is now finally in a position to discuss the \emph{geometric} caloron correspondence for vector bundles.
The isomorphisms $\mu_E$ and $\mu_\sE$ of Proposition \ref{prop:sections} play a crucial role in the construction; write $\check s = \mu_E(s)$ so that any section with a check on it is understood to be a section of an $\Omega$ vector bundle and unadorned sections live on finite-rank bundles.
This notation is chosen in order to be consistent with the shorthand used to denote the isomorphism $L\cC^\infty(M) \to \cC^\infty(M\x\sone;\CC)$ of \eqref{eqn:loopringiso}.

To construct the caloron correspondences, one must first know how to construct a framed connection on $E := \cV(\sE)$ from a based module connection and Higgs field on $\sE$.
\vspace{11pt}
\begin{proposition}
\label{prop:construction:finiteconnection}
Given $(\sE,\Delta,\phi) \in \ovect^c(M)$, the formula
\begin{equation}
\label{eqn:construction:finiteconnection}
\nabla := \big( \!\pr_M^\ast \otimes\, \mu_\sE^{-1} \big) \circ \Delta \circ \mu_\sE + d\theta \otimes \big(\mu_\sE^{-1} \circ \phi \circ \mu_\sE \big)
\end{equation}
defines a framed connection on $E = \cV(\sE)$, where $\pr_M \colon M\x\sone \to M$ is the projection map.
\end{proposition}
\begin{proof}
To see that the map $\nabla \colon \Gamma(E) \to \Omega^1(M) \otimes \Gamma(E)$ defined by \eqref{prop:construction:finiteconnection} is a connection, since it is clearly $\CC$-linear one need only verify that it satisfies the Leibniz property.
Take any $s \in \Gamma(E)$ and $g \in \cC^\infty(M\x\sone;\CC)$, then
\begin{align*}
\nabla(gs) &= \big( \!\pr_M^\ast \otimes\, \mu_\sE^{-1} \big) \Delta (\check g\check s) + d\theta \otimes \big(\mu_\sE^{-1} \circ \phi\, (\check g\check s)  \big)\\
%%%
&=\big(\! \pr_M^\ast \otimes \,\mu_\sE^{-1} \big)\big( \check g\Delta(\check s) + m_\sE(d\check g \otimes \check s)\big) + d\theta \otimes \mu_\sE^{-1}\big(\check g \phi\, (\check s) +\d(\check g) \check s  \big)\\
%%%
&= g \big(\! \pr_M^\ast \otimes\, \mu_\sE^{-1} \big) \Delta(\check s) + (d_M g )\otimes s + g\,d\theta\otimes\mu_\sE^{-1}(\phi(\check s)) + (\d g )\,d\theta \otimes s\\
%%%
&=g \nabla(s) + d g \otimes s,
\end{align*}
since $dg = d_M g+ \d g \,d\theta$, where $d_M$ momentarily denotes differentiation in the $M$ direction on $M\x\sone$.

Let $s_0$ be the framing of $E$; it remains only to show that $\nabla$ is framed with respect to $s_0$.
For each $x \in M$ choose an open neighbourhood $U \subset M$ of $x$ such that $E$ is trivial over $U \x\sone$---this is always possible, for example if $U$ is contractible.
Choose a framing $s_U$ of $E$ over $U\x\sone $ such that $s_U$ and $s_0$ agree over $U_0 := U \x\{0\}$.

Taking $n = \rank E= \rank\sE$, the framing $s_U$ determines $n$ sections $\mathbf{e}_1,\dotsc,\mathbf{e}_n$ that form a basis for $\Gamma(U\x\sone,E)$ as a $\cC^\infty(U\x\sone,\CC)$-module.
Let $\check{\mathbf{e}}_i := \mu_\sE(\mathbf{e}_i)$, then the $\check{\mathbf{e}}_i$ determine a basis for $\Gamma(U,\sE)$ as an $L\cC^\infty(U)$-module and therefore determine a section $\check s_U\in\Gamma(U,\cF(\sE))$.
Then
\begin{align*}
\nabla \mathbf{e}_i &= \big( \!\pr_M^\ast \otimes\, \mu_\sE^{-1} \big) \Delta \check{\mathbf{e}}_i + d\theta \otimes \big(\mu_\sE^{-1} \circ \phi (\check{\mathbf{e}}_i) \big)\\
%%%
&=  \big( \!\pr_M^\ast \otimes\, \mu_\sE^{-1} \big) \big(\check s_U^\ast \sA\big)_i^j \check{\mathbf{e}}_j +d\theta \otimes \left(\mu_\sE^{-1} \big(\Phi(\check s_U)_i^j \check{\mathbf{e}}_j\big) \right)
\end{align*}
so that at $(x,\theta) \in U\x\sone$
\[
\nabla \mathbf{e}_i (x,\theta) = \big(\check s_U^\ast \sA_x(\theta)\big)_i^j {\mathbf{e}}_j(x,\theta) +  d\theta \otimes \big(\Phi(\check s_U(x))(\theta)\big)_i^j \mathbf{e}_j (x,\theta),
\]
where $\sA$ and $\Phi$ are the connection and Higgs field corresponding to $\Delta$ and $\phi$ on $\cF(\sE)$.
But
\[
\nabla\mathbf{e}_i = \big(s_U^\ast A\big)_i^j \mathbf{e}_j,
\]
with $A$ the connection on $\cF(E)$ corresponding to $\nabla$.
Pulling back by the canonical inclusion $\imath_{U} \colon U\x\{0\} \hookrightarrow U \x\sone$ gives
\[
\imath_{s_0}^\ast A_x = (\imath_{U}^\ast s_U^\ast A)_{x} = \check s_U^\ast \sA_x(0) = 0
\] 
since $\sA$ is $\Omega\g$-valued.
This shows that $\nabla$ is framed, since $A$ is.
\end{proof}
\vspace{11pt}
\begin{corollary}
\label{cor:geoframe1}
The diagram of functors
\[
\xy
(0,25)*+{\ovect_V^c(M)}="1";
(50,25)*+{\Bun^c_{\Omega G}(M)}="2";
(0,0)*+{\frvect_V^c(M)}="3";
(50,0)*+{\frBun^c_G(M)}="4";
{\ar^{\cF} "1";"2"};
{\ar^{\cC} "2";"4"};
{\ar^{\cV} "1";"3"};
{\ar^{\cF} "3";"4"};
\endxy
\]
commutes up to the natural isomorphism $\mu$ of Corollary \ref{cor:frame}.
\end{corollary}
\begin{proof}
This follows readily from the argument of Proposition \ref{prop:construction:finiteconnection} since the expression
\[
\alpha^\ast A_{(x,\theta)} = \check \alpha^\ast\sA_x(\theta) + \Phi(\check\alpha(x))(\theta) d\theta
\]
is the pullback of \eqref{eqn:con:lgcalfinconn} by the local section $\alpha \in \Gamma(U\x\sone,\cF(E))$ corresponding to the section $\check \alpha \in \Gamma(U,\cF(\sE))$.
\end{proof}

The \emph{geometric caloron transform} for vector bundles is the functor
\[
\cV \colon \ovect^c(M) \lo \frvect^c(M)
\]
that sends the object $(\sE,\Delta,\phi)$ to the object $(E,\nabla)$, with $E = \cV(\sE)$ and $\nabla$ given as in Proposition \ref{prop:construction:finiteconnection}.
The action of $\cV$ on the morphism $f$ of $\ovect^c(M)$ is simply $\cV(f)$ as defined above, since
\vspace{11pt}
\begin{lemma}
\label{lemma:consistency1}
Writing $(E,\nabla) = \cV(\sE,\Delta,\phi)$ and $(F,\nabla') = \cV(\sF,\Delta',\phi')$, if $f \colon \sE \to \sF$ preserves the connective data, i.e. $f^\ast\Delta' = \Delta$ and $f^\ast\phi' = \phi$, then
\[
\cV(f)^\ast \nabla' = \nabla
\]
so that $\cV(f) \colon E \to F$ is also connection-preserving.
\end{lemma}
\begin{proof}
It suffices to observe that $f\circ\mu_\sE = \mu_\sF \circ \cV(f)$ and $f^{-1}\circ\mu_\sF = \mu_\sE \circ \cV(f^{-1})$, so the result follows from \eqref{eqn:construction:finiteconnection}.
\end{proof}

For the other direction of the geometric caloron correspondence for vector bundles, one must construct a based module connection and Higgs field on $\sE:= \cV^{-1}(E)$ from a framed connection on $E \in \frvect(M)$.
\vspace{11pt}
\begin{proposition}
\label{prop:construction:moduleconnection}
Given $(E,\nabla) \in \frvect^c(M)$, the formula
\begin{equation}
\label{eqn:construction:moduleconnection}
\Delta := \big(\imath_{M}^\ast\otimes \mu_E \big)\circ\nabla \circ \mu_E^{-1}
\end{equation}
defines a based module connection on $\sE = \cV^{-1}(E)$, where $\imath_M \colon M\x\{0\} \hookrightarrow M\x\sone$ is the inclusion.
\end{proposition}
\begin{proof}
Since $\Delta$ is clearly a $\CC$-linear map $\Gamma(\sE) \to \Omega^1(M)\otimes\Gamma(\sE)$ it suffices to verify that it satisfies the module connection condition as this implies the Leibniz rule.
Take any $\check g \in L\cC^\infty(M)$ and $\check s \in \Gamma(\sE)$, then
\begin{align*}
\Delta(\check g \check s) &= \big(\imath_{M}^\ast\otimes \mu_E \big)\nabla(gs)\\
&= \big(\imath_{M}^\ast\otimes \mu_E \big)(g \nabla(s) +dg\otimes s)\\
&= \check g \Delta (\check s) + m_\sE(d\check g\otimes \check s)
\end{align*}
since $(\imath_M^\ast\otimes \mu_E)(dg\otimes s) = m_\sE(d\check g\otimes \check s)$ and $(\imath_M^\ast\otimes \mu_E)(g \nabla(s)) =\check g \Delta (\check s) $.
Thus $\Delta$ is indeed a module connection.

To see that $\Delta$ is based, as in Lemma \ref{lemma:moduleconnectioncorrespond} let $\check{\alpha}$ be a framing of $\sE$ over $U$ corresponding to the trivialising sections $\check {\mathbf{e}}_1,\dotsc,\check {\mathbf{e}}_n$.
Then for any $v = f^i \check {\mathbf{e}}_i \in \Gamma(U,\sE)$
\begin{align*}
\Delta (v) &= f^i (\check{\alpha}^\ast\sA)_i^j \check {\mathbf{e}}_j + m_\sE (d f^i \otimes \check {\mathbf{e}}_i)\\
&= f^i (\check{\alpha}^\ast\sA)_i^j \check {\mathbf{e}}_j + \delta \circ \check{\alpha}^{-1} (v)
\end{align*}
with $\delta$ the trivial module connection of Example \ref{ex:triviconn} and $\sA$ the $LG$-connection on $\cF(\sE)$ corresponding to $\Delta$.
To demonstrate that $\Delta$ is based, one wishes to show that $\sA$ is $\Omega \g$-valued.
Similarly to the proof of Proposition \ref{prop:construction:finiteconnection}, take $\mathbf{e}_i := \mu_E^{-1}(\check{\mathbf{e}}_i)$ with $\alpha$ the corresponding section of $\cF(E)$ over $U\x\sone$.
Then
\begin{align*}
\Delta \check{\mathbf{e}}_i = \big(\imath_{M}^\ast\otimes \mu_E \big)\nabla \mathbf{e}_i = \big(\imath_{M}^\ast\otimes \mu_E \big) (\alpha^\ast A)^j_i \mathbf{e}_j,
\end{align*}
with $A$ the principal connection corresponding to $\nabla$, which implies
\[
\check\alpha^\ast\sA_x(\theta) = \alpha^\ast A_{(x,\theta)}
\]
for all $x \in U$ and $\theta\in \sone$.
Since $\alpha$ agrees with the framing $s_0$ of $E$ over their mutual domain and $s_0^\ast A= 0$, setting $\theta = 0$ in this expression gives $\check\alpha^\ast\sA_x(0)=0$ as required.
\end{proof}

To see how to obtain a Higgs field on $\sE$, first write $\d$ for the canonical vector field in the $\sone$ direction
on $M\x\sone$ (cf. p.~\pageref{page:d}).
Then
\vspace{11pt}

\begin{proposition}
\label{prop:construction:higgs}
Given $(E,\nabla)\in\frvect^c(M)$, for any $v \in \sE_x$ choose some section $\check s \in \Gamma(\sE)$ extending $v$.
The formula
\begin{equation}
\label{eqn:construction:higgs}
\phi(v) := \left[\mu_E \circ \nabla_{\d} \circ \mu_E^{-1} (\check s)\right](x)
\end{equation}
determines the action of a Higgs field $\phi$ on the fibre $\sE_x$.
\end{proposition}
\begin{proof}
It must first be seen that $\phi$ is well-defined.
Take any two $\check s, \check r \in \Gamma(\sE)$ extending $v$ so that $s(x,\theta) = r(x,\theta)$ for all $\theta\in\sone$.
Then
\[
\nabla_{\d}(s)(x,\theta) = \nabla_{\d}(r)(x,\theta)
\]
for all $\theta \in \sone$ so applying $\mu_E$ and evaluating at $x$ shows that $\phi$ is well-defined.

To see that $\phi$ is a Higgs field, first notice that it is a smooth vector bundle homomorphism $\sE \to \sE$ covering the identity.
Taking any $v \in \sE_x$ and $f \in L\CC$, choose $\check s \in \Gamma(\sE)$ and $\check g \in L\cC^\infty(M)$ extending $v$ and $f$ respectively.
Then
\begin{align*}
\phi(fv) &:=\mu_E \circ \nabla_{\d}(gs)(x) \\
&= \mu_E\big[ g\nabla_{\d} (s) +\d(g) s\big](x)\\
&= \check g(x) \phi(v) + \d(\check g)(x) \,\check s (x)\\
&= f\phi(v) + \d(f)\, v,
\end{align*}
since the ring isomorphism of \eqref{eqn:loopringiso} commutes with differentiation along the $\sone$ direction.
\end{proof}
\vspace{11pt}
\begin{corollary}
\label{cor:geoframe2}
The diagram of functors
\[
\xy
(0,25)*+{\frvect_V^c(M)}="1";
(50,25)*+{\frBun^c_{G}(M)}="2";
(0,0)*+{\ovect_V^c(M)}="3";
(50,0)*+{\Bun^c_{\Omega G}(M)}="4";
{\ar^{\cF} "1";"2"};
{\ar^{\cC^{-1}} "2";"4"};
{\ar^{\cV^{-1}} "1";"3"};
{\ar^{\cF} "3";"4"};
\endxy
\]
commutes up to the natural isomorphism $\nu$ of Corollary \ref{cor:frame}.
\end{corollary}
\begin{proof}
One must show that $\nu$ preserves the connections and Higgs fields.
For connections, this follows directly from the proof of Proposition \ref{eqn:construction:moduleconnection}, since $\check\alpha^\ast\sA_x(\theta) = \alpha^\ast A_{(x,\theta)}$ is the pullback of \eqref{eqn:con:lgcalfrecconn} by the local section $\alpha \in \Gamma(U\x\sone,\cF(E))$ corresponding to the section $\check \alpha \in \Gamma(U,\cF(\sE))$.

For Higgs fields, taking $\check{\mathbf{e}}_i$ and $\mathbf{e}_i$ as in the proof of Proposition \ref{prop:construction:moduleconnection}, for any $x \in U$
\[
\phi(\check{\mathbf{e}}_i(x)) =  \Phi(\check \alpha(x))^j_i \check{\mathbf{e}}_j (x).
\]
Using \eqref{eqn:construction:higgs}
\[
\phi(\check{\mathbf{e}}_i(x)) = \left[ \mu_E \nabla_\d (\mathbf{e}_i )\right](x) = \left[ \mu_E \big((\alpha^\ast A(\d))_i^j \mathbf{e}_j \big)\right]\!(x)
\]
and hence for all $\theta\in \sone$,  $\Phi(\check \alpha(x))(\theta) = \check \alpha(x)^\ast A_\theta(\d)$, where one uses $\nu$ to view $\check \alpha(x)$ as a section $\check \alpha(x) \in \Gamma(\{x\}\x\sone,\cF(E))$.
This expression coincides with \eqref{eqn:con:lgcalhiggs}, so the result follows.
\end{proof}

The \emph{inverse geometric caloron transform} for vector bundles is the functor
\[
\cV^{-1} \colon \frvect^c(M) \lo \ovect^c(M)
\]
that sends the object $(E,\nabla)$ to $(\sE,\Delta,\phi)$, with $\sE = \cV^{-1}(E)$ (as above) and $\Delta$, $\phi$ given respectively by Propositions \ref{prop:construction:moduleconnection} and \ref{prop:construction:higgs}.
As for $\cV$, the action of $\cV^{-1}$ on the morphism $f$ of $\frvect^c(M)$ is simply $\cV^{-1}(f)$ as before, since
\vspace{11pt}
\begin{lemma}
Writing $(\sE,\Delta,\phi) = \cV^{-1}(E,\nabla)$ and $(\sF,\Delta',\phi') = \cV^{-1}(F,\nabla')$, then if $f \colon E \to F$ is connection-preserving, i.e. $f^\ast \nabla' = \nabla$, then
\[
\Delta = \cV^{-1}(f)^\ast \Delta'
\;\;
\mbox{\;and\;}
\;\;
\phi = \cV^{-1}(f)^\ast \phi'
\]
so that $\cV^{-1}(f) \colon \sE \to \sF$ also preserves the connective data.
\end{lemma}
\begin{proof}
The proof is completely analogous to Lemma \ref{lemma:consistency1}.
\end{proof}

One is at last in a position to prove
\vspace{11pt}
\begin{theorem}
\label{theorem:geometric}
The geometric caloron correspondence for vector bundles
\[
\cV\colon \ovect^c(M) \lo \frvect^c(M)
\;\;\mbox{and}\;\;
\cV^{-1} \colon \frvect^c(M) \lo \ovect^c(M)
\]
is an equivalence of categories.
\end{theorem}
\begin{proof}
It suffices to show that the geometric data behaves well with respect to the natural isomorphisms of Theorem \ref{theorem:vector}, which (recalling the definitions of $\alpha_\sE$ and $\beta_E$) follows from Corollaries \ref{cor:geoframe1} and \ref{cor:geoframe2} and the based geometric caloron correspondence of Theorem \ref{theorem:gequiv}.

Alternatively, by unwinding the definitions one observes that $\beta_E = \mu_E^{-1}\circ \mu_{\cV^{-1}(E)}$ and $\alpha_\sE = \mu_{\sE}\circ \mu_{\cV(\sE)}^{-1} $.
\end{proof}

%%%%%%%%%%%%%%%%%%%%%%%%%%%%%%%%%%%%%%%%%%%%%%%%%%%%%%%%%%%%%%%%%%%%%%%%%%%%%%%%%%%%%%%%%
%%%%%%%%%%%%%%%%%%%%%%%%%%%%%%%%%%%%%%%%%%%%%%%%%%%%%%%%%%%%%%%%%%%%%%%%%%%%%%%%%%%%%%%%%
%%%%%%%%%%%%%%%%%%%%%%%%%%%%%%%%%%%%%%%%%%%%%%%%%%%%%%%%%%%%%%%%%%%%%%%%%%%%%%%%%%%%%%%%%

\section{Hermitian structures}
\label{S:hermitian}
This section introduces a version of Hermitian structures suited to the $\Omega$ vector bundle setting.
Interestingly, there are extensions of Theorems \ref{theorem:vector} and \ref{theorem:geometric} that demonstrate that a Hermitian structure on the $\Omega$ vector bundle $\sE$ arises from a Hermitian structure (in the usual sense) on $\cV(\sE)$ via the caloron correspondence and conversely.
Recall the evaluation map of Remark \ref{remark:evaluation} and for the $\Omega$ vector bundle $\sE \to M$ denote by $\overline{\sE} \to M$ the $\Omega$ vector bundle obtained from $\sE$ by complex conjugation.
\vspace{11pt}
\begin{definition}
A \emph{Hermitian structure} on $\sE \in \ovect(M)$ is a smooth section $\llan\cdot,\!\cdot\rran$ of the $\Omega$ vector bundle $(\sE \oast \overline{\sE})^\star \to M$ such that
\begin{enumerate}
\item
$\llan v, {v} \rran (\theta) > 0$ whenever $\ev_\theta (v) \neq 0$; and

\item
$\llan v,{w} \rran(\theta) = \overline{\llan w,{v} \rran}(\theta)$ for all $v, w \in \sE_x$ and $\theta \in \sone$, where the overline denotes complex conjugation.
\end{enumerate}
A \emph{Hermitian $\Omega$ vector bundle} is a pair $(\sE, \sh)$ with $\sE \in \ovect(M)$ and $\sh = \llan\cdot,\!\cdot\rran$ a Hermitian structure on $\sE$.
\end{definition}
\vspace{11pt}
\begin{remark}
\label{remark:hermreduct}
As with finite-rank vector bundles, equipping $\sE \in \ovect_{\CC^n}(M)$ with a Hermitian structure $\sh = \llan\cdot,\!\cdot\rran$ determines a reduction of the structure group of $\sE$ from $\Omega GL_n(\CC)$ to $\Omega U(n)$.
This fact relies on the existence of a local basis ${\check{\mathbf{e}}}_i$ for the fibres of $\sE$ as $L\CC$-modules that is \emph{orthonomal} in the sense that
\[
\llan {\check{\mathbf{e}}}_i, {\check{\mathbf{e}}}_j \rran(\theta) =
\begin{cases}
1 &\mbox{ if } i =j \\
0 &\mbox{ if } i \neq j
\end{cases}
\]
for all $\theta \in\sone$.
The existence of such a local basis is guaranteed by applying a generalisation of the Gram-Schmidt process to a locally trivialising basis of sections $\check{\mathbf{f}}_i$.
The collection of all frames that are orthonormal for the Hermitian structure $\sh$ then gives a reduction $\cF_0(\sE)$ of the frame bundle $\cF(\sE)$, where the former is an $\Omega U(n)$-bundle and the latter is an $\Omega GL_n(\CC)$-bundle.
\end{remark}
In order to see how Hermitian structures behave under the caloron transform functors, one must first understand how Hermitian structures interact with the framing on the object $E \in \frvect(M)$.
The idea is that a Hermitian structure on $E$ must be compatible with the framing in order to take its caloron transform.
Suppose that $E$ is modelled over $\CC^n$, with framing $s_0$ over $M_0$.
The framing $s_0$ is equivalent to $n$ linearly independent sections
\[
\mathbf{e}_i \in \Gamma(M_0, E),\quad i=1,\dotsc, n
\]
that trivialise $E$ over $M_0$.
In what follows, a Hermitian structure $h$ on $E \in \frvect(M)$ is understood to mean a Hermitian structure on $E$ for which the sections $\mathbf{e}_i$ are orthonormal (strictly speaking, such a Hermitian structure ought to be called `framed').
Note that such objects exist: they may be constructed directly by using the Tietze Extension Theorem and convexity.
\vspace{11pt}
\begin{proposition}
\label{prop:hermcorr1}
Any (framed) Hermitian structure $h = \lan\cdot,\!\cdot\ran$ on $E$ uniquely determines a Hermitian structure $\sh = \llan\cdot,\!\cdot\rran$ on $\sE = \cV^{-1}(E)$ and conversely.
\end{proposition}
\begin{proof}
Recalling the isomorphism $\mu_E$ of Proposition \ref{prop:sections}, take any $\check v,\check w\in \sE_x$ so that, for example, $\check v = \mu_E(v)$.
Notice that $v \in \Gamma(\{x\}\x\sone,E)$ so for any $\theta \in \sone$ one has $v(\theta) \in E_{(x,\theta)}$.
With this in mind, set
\[
\llan \check v,\check w \rran (\theta) := \lan v(\theta) ,w(\theta) \ran.
\]
Using the properties of $\mu_E$ and $h$ it follows that $\sh$ is a Hermitian structure on $\sE$.
Moreover, it is apparent that this construction gives a bijective correspondence between Hermitian structures $h$ and $\sh$.
\end{proof}

The exact same argument, using $\mu_\sE$ in place of $\mu_E$, gives
\vspace{11pt}
\begin{proposition}
\label{prop:hermcorr2}
Any Hermitian structure $\sh = \llan\cdot,\!\cdot\rran$ on $\sE$ uniquely determines a (framed) Hermitian structure $h = \lan\cdot,\!\cdot\ran$ on $\cV(\sE)$ and conversely.
\end{proposition}

In order to write the Hermitian caloron correspondence in category-theoretic terms, one must introduce appropriate notation, to wit
\vspace{11pt}
\begin{definition}
Denote by $\ohect(M)$ the groupoid whose objects are Hermitian $\Omega$ vector bundles over $M$ and whose morphisms $f \colon (\sE,\sh) \to (\sF,\sh')$ are those morphisms $f$ of $\ovect(M)$ preserving the Hermitian structure.

Write $\frhect(M)$ for the groupoid whose objects are pairs $(E,h)$ with $E \in \frvect(M)$ and $h$ a (framed) Hermitian structure on $E$.
Morphisms $f \colon (E,h) \to (F,h')$ are the Hermitian structure preserving morphisms of $\frvect$.
\end{definition}

Propositions \ref{prop:hermcorr1} and \ref{prop:hermcorr2} together imply the existence of functors
\[
\cV \colon \ohect(M) \lo \frhect(M)
\;\;\mbox{ and }\;\;
\cV^{-1} \colon \frhect(M) \lo \ohect(M).
\]
The action of these functors on morphisms is given by the action of the standard vector bundle caloron transform functors, noting that Hermitian structure preserving morphisms are mapped to Hermitian structure preserving morphisms (as can be seen directly from the proofs of Propositions \ref{prop:hermcorr1} and \ref{prop:hermcorr2}).
\vspace{11pt}
\begin{theorem}
\label{theorem:hvector}
The caloron correspondence for Hermitian vector bundles
\[
\cV\colon \ohect(M) \lo \frhect(M)
\;\;\mbox{ and }\;\;
\cV^{-1} \colon \frhect(M) \lo \ohect(M)
\]
is an equivalence of categories.
\end{theorem}
\begin{proof}
It suffices to notice that the natural isomorphisms $\beta_E$ and $\alpha_\sE$ of Theorem \ref{theorem:vector} preserve the Hermitian structures since $\beta_E = \mu_E^{-1}\circ \mu_{\cV^{-1}(E)}$ and $\alpha_\sE = \mu_{\sE}\circ \mu_{\cV(\sE)}^{-1} $ (as in Theorem \ref{theorem:geometric}).
\end{proof}

%%%%%%%%%%%%%%%%%%%%%%%%%%%%%%%%%%%%%%%%%%%%%%%%%%%%%%%%%%%%%%%%%%%%%%%%%%%%%%%%%%%%%%%%%
%%%%%%%%%%%%%%%%%%%%%%%%%%%%%%%%%%%%%%%%%%%%%%%%%%%%%%%%%%%%%%%%%%%%%%%%%%%%%%%%%%%%%%%%%
%%%%%%%%%%%%%%%%%%%%%%%%%%%%%%%%%%%%%%%%%%%%%%%%%%%%%%%%%%%%%%%%%%%%%%%%%%%%%%%%%%%%%%%%%

\subsection{Hermitian Higgs fields and module connections}
A natural question to ask, given the caloron correspondences of Theorems \ref{theorem:geometric} and \ref{theorem:hvector}, is whether there is a version of the caloron correspondence for Hermitian vector bundles with Hermitian connections (and Higgs fields).
The answer is affirmative and the result is established by directly combining the geometric and Hermitian caloron correspondences for vector bundles.

Before this correspondence is established, one first records some results relating connective data on Hermitian $\Omega$ vector bundles to connective data on their (unitary) frame bundles.
In particular, one develops a notion of module connections and Higgs fields being compatible with a given Hermitian structure.

In the following, take $(\sE,\sh) \in \ohect(M)$ where $\rank \sE =n$ so that $\sE$ has structure group $\Omega U(n)$.
\vspace{11pt}
\begin{definition}
A based module connection $\Delta$ on $(\sE,\sh)$ is \emph{compatible with the Hermitian structure} $\mathsf{h} = \llan\cdot,\!\cdot\rran$ if
\begin{equation}
\label{eqn:hermcon}
d\llan v,{w} \rran = \llan \Delta(v),{w} \rran + \llan v,{\Delta(w)}\rran
\end{equation}
as $L\CC$-valued forms on $M$ for all $v,w \in \Gamma(\sE)$.
\end{definition}

The following result is the direct analogue of Lemma \ref{lemma:moduleconnectioncorrespond} to the Hermitian setting:
\vspace{11pt}
\begin{lemma}
\label{lemma:hermconn}
If the based module connection $\Delta$ on $\sE$ is compatible with $\sh = \llan\cdot,\!\cdot\rran$ then it uniquely determines an $\Omega\u(n)$-valued connection $\sA$ on $\cF_0(\sE)$ (as in Remark \ref{remark:hermreduct}) and conversely.
\end{lemma}
\begin{proof}
Let $\alpha$ be a framing of $\sE$ over $U$, so that as usual one obtains the local ($L\CC$-module) basis ${\check{\mathbf{e}}}_i$ for the fibres of $\sE$ over $U$.
Requiring $\alpha$ to be a section of $\cF_0(\sE)$ over $U$ implies that the ${\check{\mathbf{e}}}_i$ are orthonormal for $\sh$ and for any $v = f^i{\check{\mathbf{e}}}_i \in \Gamma(U,\sE)$
\[
\Delta(v) = f^i(\sA_\alpha)^j_i {\check{\mathbf{e}}}_j + m_\sE (df^i \otimes {\check{\mathbf{e}}}_i),
\]
where $\sA_\alpha := \alpha^\ast \sA$.
If $w = g^j{\check{\mathbf{e}}}_j  \in \Gamma(U_\alpha,\sE)$, the compatibility condition \eqref{eqn:hermcon} now reads locally as
\[
d (f^i\overline{g}^i) = \overline{g}^idf^i + f^i d\overline{g}^i +  f^i \overline{g}^j (\sA_\alpha)_i^j + f^i \overline{g}^j (\overline{\sA_\alpha})_j^i
\]
where the overline denotes complex conjugation as usual.
Using the Leibniz rule, this implies
\[
(\sA_\alpha)_i^j + (\overline{\sA_\alpha})_j^i  = 0
\]
and hence $(\sA_\alpha) \in \Omega(U,\Omega\u(n))$ as $\Delta$ is based.
This shows that the connection $\sA$ on $\cF_0(\sE)$ is $\Omega\u(n)$-valued when restricted to orthonormal frames, so gives an $\Omega U(n)$-connection on $\cF_0(\sE)$.
The result follows from Lemma \ref{lemma:moduleconnectioncorrespond}.
\end{proof}

The compatibility condition for Higgs fields is
\vspace{11pt}
\begin{definition}
A Higgs field $\varphi$ on $(\sE,\sh)$ is \emph{compatible with the Hermitian structure} $\mathsf{h} = \llan\cdot,\!\cdot\rran$ if
\begin{equation}
\label{eqn:hermhiggs}
\d \llan v,{w} \rran = \llan \varphi(v),{w} \rran + \llan v,{\varphi(w)}\rran
\end{equation}
for all $v,w \in \Gamma(\sE)$, where $\d$ as usual denotes differentiation in the $\sone$ direction.
\end{definition}

The analogue of Lemma \ref{lemma:higgscorrespond} in the Hermitian setting is
\vspace{11pt}
\begin{lemma}
\label{lemma:hermhiggs}
If the Higgs field $\varphi$ is compatible with $\sh = \llan\cdot,\!\cdot\rran$ then it uniquely determines an $L\u(n)$-valued Higgs field $\Phi$ on $\cF_0(\sE)$ and conversely.
\end{lemma}
\begin{proof}
Let ${\check{\mathbf{e}}}_i$ be trivialising orthonormal sections over $U \subset M$ as in the proof of Lemma \ref{lemma:hermconn}, so that for any $v = f^i{\check{\mathbf{e}}}_i \in \Gamma(U,\sE)$
\[
\varphi(v) = (\d f^i + f^j \Phi(\alpha)_j^i) {\check{\mathbf{e}}}_i.
\]
If $w = g^j{\check{\mathbf{e}}}_j \in \Gamma(U,\sE)$, the compatibility condition \eqref{eqn:hermhiggs} then reads locally as
\[
\d (f^i\overline{g}^i) = \overline{g}^i\d f^i + f^i \d\overline{g}^i +  f^i \overline{g}^j \Phi(\alpha)_i^j + f^i \overline{g}^j \overline{\Phi(\alpha)}_j^i.
\]
Using the Leibniz rule for $\d$ gives
\[
\Phi(\alpha)_i^j + \overline{\Phi(\alpha)}_j^i  = 0
\]
and hence $\Phi(\alpha)$ is $L\u(n)$-valued.
The rest of the argument is essentially that of Lemma \ref{lemma:higgscorrespond}.
\end{proof}

\vspace{11pt}
\begin{definition}
Denote by $\ohect^c(M)$ the groupoid whose objects are quadruples $(\sE,\sh,\Delta,\phi)$ where $(\sE,\sh) \in \ohect(M)$ and $\Delta$ and $\phi$ are respectively a based module connection and Higgs field on $\sE$ that are compatible with $\sh$.
Morphisms are given by connection- and Higgs field-preserving morphisms of $\ohect(M)$.

Similarly, write $\frhect^c(M)$ for the groupoid whose objects are triples $(E,h,\nabla)$ with $(E,h) \in \frhect(M)$ and $\nabla$ a framed connection on $E$ compatible with $h$.
Morphisms are precisely the connection-preserving morphisms of $\frhect(M)$.
\end{definition}

The geometric caloron correspondence for Hermitian vector bundles relies on the following two results
\vspace{11pt}
\begin{proposition}
\label{prop:hermgcal1}
If the framed connection $\nabla$ on $E \in \frvect(M)$ is compatible with the Hermitian structure $h = \lan\cdot,\!\cdot\ran$ then the module connection $\Delta$ and Higgs field $\phi$ on $(\sE,\sh) := \cV^{-1}(E,h)$ given respectively by \eqref{eqn:construction:moduleconnection} and  \eqref{eqn:construction:higgs} are compatible with $\sh = \llan\cdot,\!\cdot\rran$.
\end{proposition}
\begin{proof}
Recall (cf. Theorem \ref{theorem:geometric}) that $\nabla' = (\beta_E)^\ast \nabla$, where
\[
\nabla' = \big( \!\pr_M^\ast \otimes\, \mu_\sE^{-1} \big) \Delta \circ \mu_\sE + d\theta \otimes \big(\mu_\sE^{-1} \circ \phi \circ \mu_\sE \big)
\]
is the connection on $E' = \cV\circ\cV^{-1}(E)$.
Therefore $\nabla'$ is compatible with the induced Hermitian structure $h' = \lan\beta_E(\cdot),\beta_E(\cdot)\ran = \lan\cdot,\!\cdot\ran_{E'}$ on $E'$.

Take any $\check v,\check w \in \Gamma(\sE)$, then
\begin{align*}
d \lan v,w\ran_{E'} &= \lan \nabla'(v), w\ran_{E'} + \lan v,\nabla'(w)\ran_{E'}
\end{align*}
where, for example, $\check v = \mu_{\sE}(v)$.
Now, for all $\theta \in \sone$
\begin{align*}
d \lan v(\theta),w(\theta)\ran_{E'} &= d_M \lan v(\theta),w(\theta)\ran_{E'} + \d \lan v(\theta),w(\theta)\ran_{E'} \,d\theta \\
&= d_M \llan\check  v,\check w\rran(\theta) + \d \llan\check  v,\check w\rran(\theta) \,d\theta 
\end{align*}
and
\begin{align*}
 \lan \nabla'(v)(\theta), w(\theta)\ran_{E'} &= \llan \Delta(\check v),\check w\rran(\theta) + \llan \phi(\check v),\check w\rran(\theta)\,d\theta
\end{align*}
so that
\[
d\llan \check v,\check {w} \rran = \llan \Delta(\check v),{\check w} \rran + \llan \check v,{\Delta(\check w)}\rran\;\;\mbox{ and }\;\;
\d\llan \check v,\check {w} \rran = \llan \phi(\check v),{\check w} \rran + \llan \check v,{\phi(\check w)}\rran
\]
as required.
\end{proof}

Essentially the same argument gives
\vspace{11pt}
\begin{proposition}
\label{prop:hermgcal2}
If the module connection $\Delta$ and Higgs field $\phi$ on $\sE \in \ovect(M)$ are compatible with the Hermitian structure $\sh = \llan\cdot,\!\cdot\rran$ then the framed connection $\nabla$ on $(E,h) := \cV(\sE,\sh)$ given by \eqref{eqn:construction:finiteconnection} is compatible with $h = \lan\cdot,\!\cdot\ran$.
\end{proposition}

Combining all of these results, one arrives at the geometric caloron correspondence for Hermitian vector bundles:
\vspace{11pt}
\begin{theorem}
\label{theorem:hgeometric}
The geometric caloron correspondence for Hermitian vector bundles
\[
\cV\colon \ohect^c(M) \lo \frhect^c(M)
\;\;\mbox{ and }\;\;
\cV^{-1} \colon \frhect^c(M) \lo \ohect^c(M)
\]
is an equivalence of categories.
\end{theorem}

\vspace{11pt}
\begin{remark}
\label{remark:ounreduct}
As was observed in Remark \ref{remark:hermreduct}, a Hermitian structure on the $\Omega$ vector bundle $\sE \to M$ determines a reduction of the structure group from $\Omega GL_n(\CC)$ to $\Omega U(n)$ (where $n = \rank\sE$).

Conversely, if $\sE$ comes equipped with a preferred reduction of the structure group from $\Omega GL_n(\CC)$ to $\Omega U(n)$, then it has a Hermitian structure $\sh = \llan\cdot,\!\cdot\rran$ as follows.
Pick an open cover $\{U_{\alpha}\}_{\alpha\in I}$ of $M$ over which $\sE$ is trivialised (by maps $\psi_\alpha$, say), so that the transition functions $\{\tau_{\alpha\beta}\}$ are $\Omega U(n)$-valued.
Using the trivialisation
\[
\psi_\alpha \colon \sE|_{U_\alpha}\lo U_\alpha \x L\CC^n 
\]
one may define a Hermitian structure on $\sE|_{U_\alpha}$ that acts on elements $v,w$ in the fibre of $\sE$ over $x\in U_\alpha$ by
\[
\llan v,w\rran_\alpha (\theta):= \big\lan \psi_\alpha(v)(\theta),\psi_\alpha(w)(\theta) \big\ran,
\]
where $\lan\cdot,\!\cdot\ran$ is the standard inner product on $\CC^n$.
Over double intersections $U_{\alpha\beta} := U_\alpha\cap U_\beta$, one has
\begin{align*}
\psi_\beta\circ \psi_\alpha^{-1} \colon U_{\alpha\beta}\x L\CC^n &\lo U_{\alpha\beta}\x L\CC^n
\\
(x,v) &\longmapsto \left(x,\tau_{\alpha\beta}(x)(v)\right)
\end{align*}
so that for any $x \in U_{\alpha\beta}$
\begin{multline*}
\llan v,w\rran_\beta (\theta) = \big\lan \psi_\beta(v)(\theta),\psi_\beta(w)(\theta) \big\ran = \big\lan \tau_{\alpha\beta}(x)(\theta)\cdot\psi_\alpha(v)(\theta),\tau_{\alpha\beta}(x)(\theta)\cdot\psi_\alpha(w)(\theta) \big\ran\\
%%%
=\big\lan \psi_\alpha(v)(\theta),\psi_\alpha(w)(\theta)\big\ran = \llan v,w\rran_\alpha (\theta) 
\end{multline*}
since the $\tau_{\alpha\beta}$ are $\Omega U(n)$-valued.
Therefore the $\llan\cdot,\!\cdot\rran_\alpha$ patch together to give a Hermitian structure $\sh = \llan\cdot,\!\cdot\rran$ on $\sE$ that  does not depend on the choice of cover $\{U_\alpha\}$ of $M$.
From here on, any $\Omega$ vector bundle $\sE \to M$ with structure group $\Omega U(n)$ is understood to be equipped with this Hermitian structure.
\end{remark}
\vspace{11pt}
\begin{example}[the canonical $\Omega$ vector bundles]
\label{example:bmsen}
The associated $\Omega$ vector bundle to the path fibration $PU(n) \to U(n)$ is the \emph{canonical Hermitian $\Omega$ vector bundle} of rank $n$, denoted $\sE(n) \to U(n)$.
Applying Lemmas \ref{lemma:hermconn} and \ref{lemma:hermhiggs} to the standard connection $\sA_\infty$ and Higgs field $\Phi_\infty$ on $PU(n)$, one obtains the based module connection $\Delta(n)$ and Higgs field $\phi(n)$ on $\sE(n)$, noting that these are both compatible with the Hermitian structure.
Since $U(n)$ is universal for $\Omega U(n)$-bundles, $\sE(n) \to U(n)$ is the universal Hermitian $\Omega$ vector bundle of rank $n$, however it is not necessarily universal for Hermitian $\Omega$ vector bundles with connection; this is the reason for the disctinction in terminology between \emph{universal} and \emph{canonical}.
\end{example}

%%%%%%%%%%%%%%%%%%%%%%%%%%%%%%%%%%%%%%%%%%%%%%%%%%%%%%%%%%%%%%%%%%%%%%%%%%%%%%%%%%%%%%%%%
%%%%%%%%%%%%%%%%%%%%%%%%%%%%%%%%%%%%%%%%%%%%%%%%%%%%%%%%%%%%%%%%%%%%%%%%%%%%%%%%%%%%%%%%%
%%%%%%%%%%%%%%%%%%%%%%%%%%%%%%%%%%%%%%%%%%%%%%%%%%%%%%%%%%%%%%%%%%%%%%%%%%%%%%%%%%%%%%%%%

\section{Applications to $K$-theory}
\label{S:oddk}
Having given a detailed treatment of $\Omega$ vector bundles in previous sections, one is now in a position to discuss an interesting application of the theory.
It turns out that $\Omega$ vector bundles give convenient objects with which to describe odd topological $K$-theory.
Despite being extravagant in dimensions, working with $\Omega$ vector bundles is advantageous since it allows one to phrase the odd $K$-theory of $M$ entirely in terms of \emph{smooth} bundles that are based over $M$.
This smoothness is critical to the construction of the $\Omega$ model in Chapter \ref{ch:five}.

\subsection{Topological $K$-theory}
\label{S:topk}
Topological $K$-theory was introduced in the late 1950s by Atiyah and Hirzebruch as a generalised cohomology theory that provides a useful tool for studying vector bundles on topological spaces.
Since its first appearance (topological) $K$-theory has become an extremely active area of algebraic topology that has been used in the proofs of some very far-reaching results, notably the Atiyah-Singer Index Theorem.
Some of the most well-known results of algebraic topology, such as the result of Adams giving the maximum number of linearly independent vector fields on the $n$-sphere, have relatively simple proofs using $K$-theory.
$K$-theory (and, in particular, twisted and differential $K$-theory) has also found uses in physics, where it enjoys an important role in the developing understanding of $T$-duality (see \cite[Section 1.1]{BS1} for a discussion).

Topological $K$ theory is related to the notion of stable isomorphism.
Over a fixed base space $X$, complex vector bundles $E,F \to X$ are \emph{stably isomorphic} if
\[
E\oplus \underline{\CC}^n \cong F\oplus\underline{\CC}^n
\]
for some $n\geq 0$.
If $X$ is compact and Hausdorff, one defines the abelian group $K(X)$ of formal differences $E-F$ of vector bundles over $X$ modulo the equivalence relation $E-F\sim E'-F' \Leftrightarrow E\oplus F'$ and $E'\oplus F$ are stably isomorphic\footnote{the compactness condition on $X$ is required to show that this equivalence relation is transitive.}.
The group operation on $K(X)$ is given by $(E-F) + (E'-F') := (E\oplus E') -(F\oplus F')$, the identity element is $E-E$ (for any choice of $E$) and the inverse of $E - F$ is $F-E$. 
Alternatively, the group $K(X)$ may be obtained from the category $\vect(X)$ of (topological) vector bundles over $X$ by applying the Grothendieck group completion device as follows.

Recall (cf. \cite[Section 2.1]{AtiyahK}) that the Grothendieck construction produces from a given abelian semi-group $(A,\oplus)$ a universal abelian group called $K(A)$ and a homomorphism $\alpha \colon A \to K(A)$ of semi-groups.
The group $K(A)$ is \emph{universal} in the sense that given any abelian group $G$ and semi-group homomorphism $\beta \colon A \to G$ there is a unique homomorphism $\gamma \colon K(A) \to G$ such that $\beta = \gamma\circ\alpha$.
Thus, if $K(A)$ exists it is unique up to isomorphism.

To construct $K(A)$, first write $F(A)$ for the free abelian group generated by elements of $A$ and $E(A)$ for the subgroup of $F(A)$ generated by those elements of the form $a + a' - (a\oplus a')$.
The $K$-group of $A$ is defined to be
\[
K(A) := F(A) / E(A)
\]
with $\alpha \colon A \to K(A)$ the obvious map.

There is an alternative construction of $K(A)$ that is often quite useful.
Recall the diagonal map $\Delta \colon A \to A\x A$, which is a semi-group homomorphism, and define $K(A)$ as the quotient semi-group
\[
K(A) :=( A\x A) / \Delta(A).
\]
Since $(b,a)$ is clearly the additive inverse of $(a,b)$, $K(A)$ is in fact a group and one often writes $(a,b)$ as $a-b$ so that elements of $K(A)$ may be viewed as formal differences of elements of $A$.
The homomorphism $\alpha \colon A \to K(A)$ that characterises the universal property in this case is given by sending $a$ to the coset  of $(a\oplus a, a)$, or the coset of $(a,0)$ when $A$ has unit $0$.
It is readily seen from the construction that the assignment $A \mapsto K(A)$ is functorial.

For any compact Hausdorff space $X$, $\vect(X)$\footnote{the category whose objects are vector bundles over $X$ and whose morphisms are vector bundle maps covering the identity.} is a symmetric monoidal category under the direct sum operation $\oplus$ on vector bundles, with identity the rank zero bundle $X\x\{0\} \to X$.
Passing to isomorphism classes\footnote{recall that for a category $\mathrm{C}$, $\pi_0 \mathrm{C}$ denotes the isomorphism classes of $\mathrm{C}$.} gives the abelian semi-group $(\pi_0\!\vect(X),\oplus)$ and one sets
\[
K(X) := K(\pi_0\!\vect(X)).
\]
Since $A \mapsto K(A)$ and $X \mapsto \pi_0\!\vect(X)$ are both functorial, the assignment $X\mapsto K(X)$ determines a contravariant functor from the category of compact Hausdorff topological spaces to the category of abelian groups.
In the case that $X$ is also a smooth manifold, one may restrict to \emph{smooth} vector bundles over $X$; since every topological vector bundle based over a smooth manifold is isomorphic to a smooth vector bundle\footnote{this can be seen, for example, by choosing a continuous classifying map into some Grassmannian and then using \cite[Proposition 17.8]{BT} to find a homotopic smooth map, as in Theorem \ref{theorem:reduction}.}, this gives the same $K(X)$.

As discussed above, elements of $K(X)$ are \emph{virtual vector bundles} over $X$, that is, formal differences $E - F$, where $E$ and $F$ are (isomorphism classes of) vector bundles over $X$.
The \emph{rank} of the virtual vector bundle $E-F$ is simply the difference $\rank E - \rank F$, which is constant on connected components of $X$.
Using the Serre-Swan Theorem (cf. Remark \ref{remark:oddkprop} or \cite[Section 2.1]{AtiyahK})  every element of $K(X)$ may be written as $E - \underline{\CC}^n$ for some $n$, so for any $E - \underline{\CC}^n, F -\underline{\CC}^m \in K(X)$
\[
E - \underline{\CC}^n = F -\underline{\CC}^m \Longleftrightarrow E\oplus \underline{\CC}^p \cong F\oplus\underline{\CC}^q
\]
for some $p$ and $q$.

When $X$ is a pointed space with basepoint $\bp$, define $\widetilde K(X)$ to be the kernel of the map $\imath^\ast \colon K(X) \to K(\{\bp\})$, where $\imath \colon \{\bp\} \hookrightarrow X$ is the basepoint inclusion.
If $\rank E = m$ and $\rank F = n$, the map $\imath^\ast$ sends $E-F$ to $\underline{\CC}^m - \underline{\CC}^n$ so that elements of $\widetilde{K}(X)$ are virtual vector bundles of rank zero on the connected component of $\bp$ in $X$.
Elements of $K(\{\bp\})$ are virtual vector bundles characterised entirely by their rank, so $K(\{\bp\}) \cong \ZZ$ and there is a split  exact sequence
\[
0\lo \widetilde{K}(X) \lo K(X) \lo \ZZ \lo 0.
\]
The group $\widetilde K(X)$ is referred to as the \emph{reduced} $K$ group of $X$.

One may extend $K(X)$ to a cohomology theory, which in the case that $X$ does not come equipped with a basepoint is done in the following fashion.
Write $X^+ := X \sqcup \{\bp\}$, which is now a pointed space with basepoint $\bp$.
Then for $n\in \ZZ$ define (cf. \cite[Definition 2.2.2]{AtiyahK})
\[
K^n (X) := \widetilde{K}\left(\Sigma^{|n|}(X^+)\right)\!,
\]
where $\Sigma^{|n|}$ denotes the $|n|$-th iterated reduced suspension\footnote{recall that the \emph{reduced suspension} of the pointed space $(Y,y_0)$ is the space $\Sigma Y := (\sone\x Y) / (\sone\vee Y)$, where $\sone\vee Y:= \sone\x\{y_0\} \cup \{0\}\x Y$, equipped with the basepoint $[0,y_0]$.}.
Then, for example,
\[
K^0(X) = K(X) \;\;\mbox{ and }\;\; K^{-1}(X) = \widetilde K\left( (X\x\sone) / (X\x\{0\})\right),
\]
where $\Sigma(X^+) \cong (X\x\sone) / (X\x\{0\})$ is equipped with basepoint $(X\x\{0\}) / (X\x\{0\})$.

It is well-known that
\[
K^\bullet(X) := \bigoplus_{n\in \ZZ} K^n(X)
\]
naturally has the structure of a graded ring\footnote{just as the group structure is induced by $\oplus$, the ring structure is defined using the tensor product $\otimes$---see \cite[Section 2.6]{AtiyahK} for example.} and that the assignment $X \mapsto K^\bullet(X)$ determines a generalised (Eilenberg-Steenrod) cohomology theory.
The famous Bott Periodicity Theorem asserts that
\[
K^{n+2}(X) \cong K^n(X)
\]
for any $n$ and hence it suffices to talk only about the \emph{even} and \emph{odd} $K$-theory of $X$, which are given  by $K^0(X)$ and $K^{-1}(X)$ respectively.
As topological $K$-theory is a $2$-periodic generalised cohomology theory, it has a spectrum consisting of two spaces
\[
K^0(X) \cong [X,BU\x\ZZ] \;\;\mbox{ and }\;\; K^{-1}(X) \cong [X,U]
\]
where
\[
U := \dlim U(n)
\]
is the stabilised unitary group (Appendix \ref{app:frechet}) and $BU$ is its classifying space.
Since spectra are determined up to homotopy equivalence, one may alternatively take $BGL \x\ZZ$ and $GL$ respectively as the classifying spaces for even and odd $K$-theory, where
\[
GL := \dlim GL_n(\CC)
\]
is the stabilised general linear group and $BGL$ is its classifying space.

%%%%%%%%%%%%%%%%%%%%%%%%%%%%%%%%%%%%%%%%%%%%%%%%%%%%%%%%%%%%%%%%%%%%%%%%%%%%%%%%%%%%%%%%%
%%%%%%%%%%%%%%%%%%%%%%%%%%%%%%%%%%%%%%%%%%%%%%%%%%%%%%%%%%%%%%%%%%%%%%%%%%%%%%%%%%%%%%%%%
%%%%%%%%%%%%%%%%%%%%%%%%%%%%%%%%%%%%%%%%%%%%%%%%%%%%%%%%%%%%%%%%%%%%%%%%%%%%%%%%%%%%%%%%%

\subsection{The functor $\cK$}
\label{S:ouroddk}
In this section, the theory of $\Omega$ vector bundles is used to give a model for the odd $K$-theory of an \emph{unpointed} compact manifold $M$ that is phrased entirely in terms of smooth objects based over $M$ instead of topological vector bundles over $\Sigma M^+$.
In the same way that even $K$-theory is given in terms of virtual vector bundles, odd $K$-theory is given by virtual $\Omega$ vector bundles of rank zero.
\vspace{11pt}
\begin{definition}
Let $M$ be a compact manifold and recall that $(\ovect(M),\oplus)$ is an abelian semi-group.
Choose a basepoint $\bp \in M$ and define
\[
\cK(M) := \widetilde K(\pi_0\ovect(M)),
\]
the abelian group obtained by using the Grothendieck group completion and then restricting to the kernel of the map $K(\pi_0\ovect(M)) \to K(\pi_0\ovect(\{\bp\}))$.
This definition is independent of the choice of basepoint $\bp$ (even if $M$ is not connected) as the rank of an $\Omega$ vector bundle is constant.
Since the assignment $M \mapsto \ovect(M)$ is functorial, the assignment $M\mapsto \cK(M)$ clearly determines a contravariant functor that acts on smooth maps by pullback.
\end{definition}

Elements of $\cK(M)$ are \emph{virtual $\Omega$ vector bundles} of rank zero, that is formal differences $\sE - \sF$ where $\sE,\sF \to M$ are $\Omega$ vector bundles with $\rank\sE = \rank\sF$.
Strictly speaking, one should write $[\sE] - [\sF]$ instead, with $[\sE]$ the isomorphism class of $\sE$, however this leads to an excess of notation.
\vspace{11pt}
\begin{remark}
\label{remark:oddkprop}
One observes some  immediate consequences of the definition:
\begin{itemize}
\item
Take any $\sE-\sF\in\cK(M)$, then by Corollary \ref{cor:oddkinverse} there is some $\Omega$ vector bundle $\mathsf{G} \to M$ such that $\sF \oplus \mathsf{G} \cong \underline{L\CC}^n$ for some $n$.
Then
\[
\sE - \sF = \sE + \mathsf{G} - \sF - \mathsf{G} = \sE \oplus \mathsf{G} - \underline{L\CC}^n.
\]
Thus every element of $\cK(M)$ may be written in the form $\sE - \underline{L\CC}^n$ where $n = \rank\sE$.

\item
Suppose $\sE - \sF = 0$ in $\cK(M)$.
Then there is some $\mathsf{G}$ such that $\sE \oplus \mathsf{G} \cong \sF \oplus \mathsf{G}$ as $\Omega$ vector bundles over $M$.
Invoking Corollary \ref{cor:oddkinverse} again gives that there is some $\sH$ such that $\mathsf{G} \oplus \sH$ is trivial and hence
\[
\sE \oplus \mathsf{G} \cong \sF \oplus \mathsf{G} \Longrightarrow \sE \oplus \underline{L\CC}^n \cong \sF \oplus \underline{L\CC}^n
\]
so $\sE$ and $\sF$ are stably isomorphic.
Conversely, if $\sE$ and $\sF$ are stably isomorphic then it is clear that $\sE-\sF = 0$ in $\cK(M)$.

\item
Similarly, $\sE - \sF = \sE'-\sF'$ in $\cK(M)$ if and only if $\sE \oplus \sF' \oplus \underline{L\CC}^n \cong \sE'\oplus\sF \oplus \underline{L\CC}^n$ for some $n$.
In particular, $\sE - \underline{L\CC}^n = \sF - \underline{L\CC}^m$ in $\cK(M)$ if and only if $\sE \oplus \underline{L\CC}^p \cong \sF \oplus \underline{L\CC}^q$ for some $p$ and $q$.
\end{itemize}
\end{remark}

The fact that $\cK(M)$ gives the odd $K$-theory of $M$ comes from the characterisation of $K^{-1}(M)$ given above, that is
\[
K^{-1}(M) = \widetilde{K}\big(\Sigma (M^+)\big) = \widetilde K\left( (M\x\sone) / (M\x\{0\})\right)\!.
\]
The underlying idea (glossing over smoothness requirements for now) is that a vector bundle on $(M\x\sone) / (M\x\{0\})$ is equivalent to a vector bundle over $M\x\sone$ with a distinguished framing over $M_0 := M\x\{0\}$ and hence, via the caloron transform, to an $\Omega$ vector bundle over $M$.

Consider the \emph{reduction functor}
\[
\cR \colon \ovect(M) \lo \vect\big((M\x\sone)/\,M_0\big)
\]
defined as follows.
Given an $\Omega$ vector bundle $\sE \to M$, its caloron transform $E:= \cV(\sE)$ is a framed vector bundle over $M\x\sone$, so has a distinguished section $s_0 \in \Gamma(M_0,\cF(E))$.
Define an equivalence relation on $E$ over $M_0$ by setting
\[
e \sim e' \Longleftrightarrow \pr_2 \circ \,s^{-1}_0(e) = \pr_2 \circ \,s^{-1}_0(e')
\]
recalling the notation of p.~\pageref{page:sectiontriv}, with $\pr_2\colon M_0 \x\CC^n \to \CC^n$ the projection, and extend this to all of $E$ by the identity.
Denote by $E/s_0$ the quotient space of $E$ given by this equivalence relation, then
\[
E / s_0\lo (M\x\sone) / \,M_0
\]
is a topological vector bundle (see \cite[Lemma 1.4.7]{AtiyahK}, for example).
The action of the reduction functor on $\sE$ is given by
\[
\cR(\sE) := \cV(\sE)/s_0 \in \vect\big((M\x\sone)/\,M_0\big).
\]
If $f\colon E\to F$ is a morphism in $\frvect(M)$ (with framings $s_E$ and $s_F$ on $E$ and $F$ respectively) it is clear from the construction that the induced map $[f] \colon E/s_E \to F/s_F$ is an isomorphism.
The action of $\cR$ on morphisms is then given by sending $f \colon \sE \to \sF$ to $[\cV(f)] \colon \cR(\sE) \to \cR(\sF)$.

Recall that a functor $F \colon \mathrm{C} \to \mathrm{D}$ is \emph{essentially surjective} if every object of $\mathrm{D}$ is isomorphic to an object in the image of $F$.
\vspace{11pt}
\begin{theorem}
\label{theorem:reduction}
The reduction functor
\[
\cR \colon \ovect(M) \lo \vect\big((M\x\sone)/\,M_0\big)
\]
is essentially surjective, in particular the induced map on isomorphism classes is surjective.
\end{theorem}
\begin{proof}
To see that $\cR$ is essentially surjective, take any vector bundle $E$ over $(M\x\sone)/M_0$ and let $\pi \colon M\x\sone \to (M\x\sone)/M_0$ be the quotient map.
Consider the pullback $E' := \pi^\ast E$, which is a topological vector bundle over $M\x\sone$ with the additional property that for  any two $x,x' \in M$ there is an isomorphism $E'_{(x,0)} \to E'_{(x',0)}$ given by
\[
E'_{(x,0)}\ni(x,0,e) \longmapsto (x',0,e) \in E'_{(x',0)}.
\]
Fix some $\bp \in M$ and let $\rank E = m$, so that picking a frame $p\colon \CC^m \to E'_{(\bp,0)}$ determines a framing $r_0$ of $E'$ over $M_0$, in particular $E'|_{M_0}$ is trivial.
Note that $E' /r_0$ is clearly isomorphic to $E$.

Pick a classifying map $f \colon M\x\sone \to \Gr_m(\CC^k)$ for $E'$ so that $E' \cong f^\ast \gamma^m(\CC^k)$.
As both $M\x\sone$ and $\Gr_m(\CC^k)$ are smooth manifolds, one may find a smooth map $g$ homotopic to $f$ \cite[Proposition 17.8]{BT}.
Since the pullbacks of any bundle by homotopic maps are isomorphic (see, for example, \cite[Lemma 1.4.3]{AtiyahK}) $E'$ is isomorphic to some smooth vector bundle $F := g^\ast\gamma^m(\CC^k) \to M\x\sone$, which is trivial over $M_0$.
Pick an isomorphism $f\colon E' \to F$ and consider the framing $f\circ r_0$ induced on $F$, which is not necessarily smooth.
However, $f \circ r_0$ is homotopic to a smooth framing $s_0$ of $F$, so since the quotient construction depends on the framing only up to homotopy \cite[Lemma 1.4.7]{AtiyahK} there is an isomorphism $F/s_0 \cong E'/r_0 \cong E$. 

Setting $\sF :=\cV^{-1}(F)$, recall that $\cV(\sF)$ and $F$ are isomorphic as framed vector bundles and hence so too are their quotients with respect to their respective framings.
In particular, $\cR(\sF)$ is isomorphic to $F/s_0$ and hence to $E$, so $\cR$ is essentially surjective.
\end{proof}

One wishes to show that $\cR$ induces a bijection on isomorphism classes.
It suffices to verify that whenever $\cR(\sE)$ and $\cR(\sF)$ are isomorphic in $\vect((M\x\sone)/\,M_0)$ then $\sE$ and $\sF$ are isomorphic in $\ovect(M)$.
Suppose that there is an isomorphism $f\colon \cR(\sE) \to \cR(\sF)$, where $\cR(\sE) = E/s_0$ and $\cR(\sF) = F/r_0$, say.
Similarly to the proof of Theorem \ref{theorem:reduction} above, one can see that $f$ lifts to a continuous isomorphism $f'\colon E\to F$ that respects the framings, i.e. $r_0 = f\circ s_0$.
One now requires a smooth isomorphism $E \to F$ that respects the framings, since applying the caloron transform and natural isomorphisms as necessary implies the result.

Let $m = \rank E = \rank F$ and recall from Corollary \ref{cor:framedclassmap} that $E$ and $F$ are smoothly isomorphic as framed vector bundles to some $f^\ast \gamma^m(\CC^k)$ and $g^\ast \gamma^m(\CC^{k'})$ respectively, where the smooth maps $f \colon M\x\sone \to \Gr_m(\CC^k)$ and $g\colon M\x\sone \to \Gr_m(\CC^{k'})$ send $M_0$ onto the respective basepoints.
Denote by $\widetilde f$, $\widetilde g$ the maps on $(M\x\sone)/M_0$ induced by $f$ and $g$.

Recall, since $\Gr_m(\CC^k)$ is a quotient of $GL_k(\CC)$, that there are embeddings
\[
\Gr_m(\CC^m) \hookrightarrow \Gr_m(\CC^{m+1})\hookrightarrow \Gr_m(\CC^{m+2}) \hookrightarrow \dotsb
\]
lifting to smooth vector bundle maps
\[
\gamma^m(\CC^m) \hookrightarrow \gamma^m(\CC^{m+1}) \hookrightarrow \gamma^m(\CC^{m+2}) \hookrightarrow \dotsb
\]
that respect the framings.
If $\vect_m(X)$ denotes the space of isomorphism classes of rank $m$ complex vector bundles over $X$, recall also
\vspace{11pt}
\begin{theorem}
The map
\[
\dlim [X,\Gr_m(\CC^k)] \lo \vect_m(X)
\]
induced by pulling back the bundles $\gamma^m(\CC^k) \to \Gr_m(\CC^k)$ is an isomorphism for all compact Hausdorff spaces $X$.
\end{theorem}
For a proof see \cite[Theorem 1.4.15]{AtiyahK}.
Since $(M\x\sone)/M_0$ is compact and Hausdorff, there is a homotopy $\widetilde h \colon (M\x\sone)/M_0 \x\I \to \Gr_m(\CC^K)$ from $\widetilde f$ to $\widetilde g$ for some $K \geq k,k'$.
This implies that there is a homotopy $h \colon M\x\sone \to \Gr_m(\CC^K)$ from $f$ to $g$ that sends $M_0\x\I$ to the basepoint of $\Gr_m(\CC^K)$.
Since $f$ and $g$ are homotopic relative to $M_0$, they are smoothly homotopic relative to $M_0$ \cite[Theorem 6.29]{Lee}.
\vspace{11pt}
\begin{lemma}
\label{lemma:pointygrassman}
If $f,g \colon M\x\sone \to \Gr_m(\CC^k)$ are smooth maps both sending $M_0$ to the basepoint $\bp$ of $\Gr_m(\CC^k)$ that are smoothly homotopic relative to $M_0$, then there is a smooth isomorphism $f^\ast \gamma^m(\CC^k) \cong g^\ast\gamma^m(\CC^k)$ of framed bundles.
\end{lemma}
\begin{proof}
Take any such homotopy $h \colon M\x\sone\x\I \to \Gr_m(\CC^k)$ from $f$ to $g$, writing $h_t(x,\theta) = h(x,\theta,t)$.
Recall that the  framing of $\gamma^m(\CC^k)$ over $\bp$ is given by taking the first $m$ standard basis vectors $\mathbf{e}_1,\dotsc,\mathbf{e}_m \in \CC^{m+k}$.
By pulling back this framing by $h$ (or, more precisely, by pulling back the corresponding trivialisation) one obtains a framing of $h^\ast\gamma^m(\CC^k)$ over $M_0\x\I$, namely the framing corresponding to the trivialising sections $h^\ast\mathbf{e}_1,\dotsc,h^\ast\mathbf{e}_m \in \Gamma(M_0\x\I,h^\ast\gamma^m(\CC^k))$.

Now, if $\nabla$ is a connection on $\gamma^m(\CC^k)$ then the pullback connection $h^\ast\nabla$ acts on  the $h^\ast\mathbf{e}_i$ by
\[
\big(h^\ast\nabla\big){}_X (h^\ast\mathbf{e}_i) = h^\ast\big(\nabla_{dh(X)} \mathbf{e}_i \big) = 0
\]
for any vector $X$ tangent to $M_0 \x\I$, as $dh(X) = 0$.
Let $\rho_{(x,\theta)} \colon \I \to \Gr_m(\CC^k)$ be the smooth curve $\rho_{(x,\theta)}(t) = h(x,\theta,t)$ and let $\varrho_t \colon h_0^\ast \gamma^m(\CC^k) \to h_t^\ast\gamma^m(\CC^k)$ be the smooth isomorphism determined by parallel transport along the curves $\rho_{(x,\theta)}$.

Since the tangent vector $\dot\rho_{(x,0)}(t)$ to $\rho_{(x,\theta)}$ lies tangent to $M_0 \x\I$ for any $t \in \I$, by the above one has
\[
t\mapsto \big(h^\ast\nabla\big){}_{\dot\rho_{(x,0)}(t)} (h^\ast\mathbf{e}_i) = 0
\]
so that the sections $h^\ast\mathbf{e}_i$ are parallel for the pullback connection.
It follows from the uniqueness of parallel transport that the isomorphisms $\varrho_t$ preserve the framing.
But $h_0 = f$ and $h_1 = g$ so there is a smooth isomorphism $f^\ast\gamma^m(\CC^k) \cong g^\ast\gamma^m(\CC^k)$ preserving the framing.
\end{proof}

Using this result, one at last has that $f^\ast\gamma^m(\CC^K)$ and $g^\ast \gamma^m(\CC^K)$ are smoothly isomorphic as framed vector bundles and, hence, that $E$ and $F$ are smoothly isomorphic as framed vector bundles.
This completes the proof of
\vspace{11pt}
\begin{theorem}
The map on isomorphism classes induced by the reduction functor $\cR$ is injective and hence, by Theorem \ref{theorem:reduction}, is a bijection.
\end{theorem}

Notice also  that if $E, F$ are framed bundles over $M\x\sone$ with framings $s_0, r_0$ respectively then there is an isomorphism
\[
E/s_0 \oplus F/r_0 \cong (E\oplus F)/(s_0\oplus r_0)
\]
and hence $\cR(\sE\oplus\sF) \cong \cR(\sE)\oplus\cR(\sF)$ since $\cV(\sE\oplus\sF) \cong \cV(\sE)\oplus\cV(\sF)$.
\vspace{11pt}
\begin{corollary}
\label{cor:kgroups1}
The map on isomorphism classes
\[
\pi_0\ovect(M) \lo \pi_0\!\vect\big((M\x\sone)/\,M_0\big)
\]
induced by $\cR$ is an isomorphism of semi-groups.
\end{corollary}

Passing to the group completions, observe that the isomorphism
\[
K(\pi_0\ovect(M)) \lo K\left( (M\x\sone) / (M\x\{0\})\right)
\]
given by $\sE - \sF \mapsto \cR(\sE) - \cR(\sF)$ preserves the rank, since $\rank \cR(\sE) = \rank \sE$.
Restricting to the rank zero subgroups gives
\vspace{11pt}
\begin{theorem}
\label{theorem:oddk}
For any compact manifold $M$, the map
\[
\cR \colon \cK(M) \lo K^{-1}(M)
\]
given by
\[
\cR \colon \sE-\sF \longmapsto \cR(\sE) - \cR(\sF)
\]
is an isomorphism.
In particular, $\cK(M)$ is the odd $K$-theory of $M$.
\end{theorem}
\vspace{11pt}
\begin{remark}
\label{remark:homotopyversion}
Another way of obtaining the isomorphism $\cK \to K^{-1}$ of Theorem \ref{theorem:oddk} is to use the homotopy-theoretic model of odd $K$-theory, i.e.~$K^{-1}(M) \cong [M,GL]$.
The group operation on $K^{-1}(M)$ is given by $[g]+ [h]:= [g\oplus h]$, with $g\oplus h$ the pointwise block sum of $g$ and $h$.
At the level of homotopy, the group operations given by block sum and the matrix product are equal so the inverse of $[g] \in K^{-1}(M)$ is the homotopy class of the map $g^{-1}\colon x\mapsto (g(x))^{-1}$ (see \cite{TWZ}, for example).
As $M$ is compact the image of any map $f \colon M \to GL$ is contained in $GL_n(\CC)$ for some sufficiently large $n$ (Proposition \ref{prop:compactimage}) and  one may find a smooth map $g \colon M\to GL_n(\CC)$ homotopic to $f$ \cite[Proposition 17.8]{BT}.
Similarly, if any two smooth maps $f,g \colon M\to GL_n(\CC)$ are continuously homotopic then they are also smoothly homotopic \cite[Theorem 6.29]{Lee}.

One constructs the isomorphism $\cR\colon \cK(M) \to [M,GL]$ by sending $\sE - \underline{L\CC}^n \in \cK(M)$ to the homotopy class $[g]$, where $g$ is any classifying map (equivalently, Higgs field holonomy) for $\sE$.
Observe that $\cR$ is surjective, since any homotopy class $[f] \in [M,GL]$ has a smooth representative $g \colon M\to GL_n(\CC)$, in which case $g^\ast \sE(\CC^n) - \underline{L\CC}^n$ is mapped to $[f]$ by $\cR$\footnote{recalling the universal $\Omega$ vector bundles of Example \ref{example:universalomega}.}.
To see that $\cR$ is injective, suppose that $\cR(\sE-\underline{L\CC}^n) = [\id]$ so that any classifying map for $\sE$ is homotopic to the constant map sending $M$ to the identity in $GL$.
Let $g \colon M \to GL_n(\CC)$ be any smooth classifying map for $\sE$, then there is a smooth homotopy from $g \oplus\id$ to $\id$ inside $GL_{n+k}(\CC)$ for some $k$ sufficiently large.
But $(g\oplus\id)^\ast \sE(\CC^{n+k}) \cong g^\ast\sE(\CC^n) \oplus \underline{L\CC}^k$ so that $g^\ast\sE(\CC^n) \oplus \underline{L\CC}^k \cong \underline{L\CC}^{n+k}$ and hence $\sE -\underline{L\CC}^n = 0$ in $\cK(M)$.
That $\cR$ is a group homomorphism follows from the fact that $g^\ast\sE(\CC^n) \oplus h^\ast \sE(\CC^m) \cong (g\oplus h)^\ast\sE(\CC^{n+m})$.
\end{remark}
\vspace{11pt}
\begin{remark}
\label{remark:unversion}
By considering isomorphism classes of Hermitian structured $\Omega$ vector bundles over $M$ one obtains a Hermitian version of $\cK(M)$, denoted temporarily by $\cKr(M)$.
The argument of Remark \ref{remark:homotopyversion} demonstrates that there is an isomorphism $\cKr(M) \cong [M,U]$, with $U$ the stabilised unitary group.
As $U$ and $GL$ are homotopy equivalent this implies that the map $\cKr(M) \to \cK(M)$ given by discarding the Hermitian structure is an isomorphism and, consequently, that one may view elements $\sE - \sF \in \cK(M)$ as virtual \emph{Hermitian} $\Omega$ vector bundles.
\end{remark}

%%%%%%%%%%%%%%%%%%%%%%%%%%%%%%%%%%%%%%%%%%%%%%%%%%%%%%%%%%%%%%%%%%%%%%%%%%%%%%%%%%%%%%%%%
%%%%%%%%%%%%%%%%%%%%%%%%%%%%%%%%%%%%%%%%%%%%%%%%%%%%%%%%%%%%%%%%%%%%%%%%%%%%%%%%%%%%%%%%%
%%%%%%%%%%%%%%%%%%%%%%%%%%%%%%%%%%%%%%%%%%%%%%%%%%%%%%%%%%%%%%%%%%%%%%%%%%%%%%%%%%%%%%%%%

\subsection{The total string form and the Chern character}
\label{SS:oddk:string}
The primary motivation for using $\Omega$ vector bundles to define odd $K$-theory is that they are smooth objects that may be readily equipped with extra differential form data in the guise of based module connections and Higgs fields.
As demonstrated in Chapter \ref{ch:five}, this allows one to construct a differential refinement of $\cK$ that provides a codification of $\Omega$ vector bundles with connective data.
To provide additional motivation for this discussion it is shown that the odd Chern character may be constructed on $\cK$ using characteristic classes of the underlying $\Omega$ vector bundles.

Let $M$ be a compact manifold.
Recall that the Chern character is a natural map
\[
ch \colon K^\bullet(M) \lo H^\bullet(M;\CC)
\]
that induces an isomorphism
\[
ch \colon K^\bullet(M)\otimes \CC \lo H^\bullet(M;\CC)
\]
of $\ZZ_2$-graded rings (with the $\ZZ_2$-grading on $H^\bullet \CC$ given by taking the even and odd degree parts).
In the case of even $K$-theory, the Chern character may be realised explicitly as follows.
Recall from Section \ref{S:topk} that an element of $K^0(M)$ is a formal difference $E -F$ with $E$ and $F$ (isomorphism classes of) smooth complex vector bundles over $M$.
Choosing a connection $\nabla$ on $E$, write $R$ for the curvature of $\nabla$ and consider the even complex-valued form
\begin{equation}
\label{eqn:totalchern}
\Ch(\nabla) := \tr\left(e^{\tfrac{1}{2\pi i} R} \right) = \sum_{j=0}^\infty \frac{1}{j!} \bigg(\frac{1}{2\pi i} \bigg)^j \tr(\underbrace{R\wedge\dotsb \wedge R}_{\text{$j$ times}})
\end{equation}
on $M$.
If $E$ has typical fibre $\CC^n$, one can check that the degree $2k$ piece of $\Ch(\nabla)$ is exactly
\[
\cw_{\overline\tr_k}(A) = \overline\tr_k (\underbrace{F,\dotsc,F}_{\text{$k$ times}}),
\]
where $A$ is the connection on the frame bundle corresponding to $\nabla$, with curvature form $F = R$, and $\overline{\tr}_k$ is the \emph{(normalised) $k$-th symmetrised trace}, i.e.~the $\CC$-valued invariant polynomial of degree $k$ on $\mathfrak{gl}_n(\CC)$ given by
\begin{equation}
\label{eqn:traceline}
\overline\tr_k (\xi_1,\dotsc,\xi_k) := \frac{1}{(k!)^2(2\pi i)^k} \sum_{\sigma\in S_k} \tr(\xi_{\sigma(1)} \dotsb \xi_{\sigma(k)})
\end{equation}
for $\xi_1,\dotsc,\xi_k \in \mathfrak{gl}_n(\CC)$, with $S_k$ the group of permutations on $\{1,\dotsc,k\}$.
It follows from Theorem \ref{theorem:chernweil} that $\Ch(\nabla)$ is closed and the class of $\Ch(\nabla)$ in de Rham cohomology is independent of $\nabla$.
Write $\Ch(E) \in H^{even}(M;\CC)$ for the image of the de Rham class $[\Ch(\nabla)]$ under the de Rham isomorphism.
It is well-known that
\[
\Ch(\nabla \oplus \nabla') = \Ch(\nabla) + \Ch(\nabla') \;\;\mbox{ and }\;\; \Ch(\nabla\otimes\nabla') = \Ch(\nabla) \wedge \Ch(\nabla'),
\]
so the  even Chern character
\[
ch \colon K^0(M) \lo H^{even}(M;\CC)
\]
given by
\[
ch \colon E-F \longmapsto \Ch(E) - \Ch(F)
\]
is a ring homomorphism.

In the case of odd $K$-theory, it is easier to discuss the Chern character using the homotopy-theoretic model for $K^{-1}$ from the end of Section \ref{S:topk}, i.e.~$K^{-1}(M) \cong [M,GL]$.
As mentioned above, the group operation on $K^{-1}(M)$ is given by block sum and $[g^{-1}]$ is the inverse of $[g]$.
The odd Chern character of $[g] \in K^{-1}(M)$ may be computed directly by taking a smooth representative $g\in[g]$ and setting
\begin{equation}
\label{eqn:oddchern}
ch([g]) =  \sum_{j=0}^\infty \frac{-j!}{(2j+1)!} \bigg(\!\!-\frac{1}{2\pi i} \bigg)^{j+1} \tr(\underbrace{g^{-1}dg\wedge\dotsb\wedge g^{-1}dg}_{\text{$2j+1$ times}} ) \mod \mbox{ exact}
\end{equation}
(see \cite[Proposition 1.2]{G}, for example).
Since $(g\oplus h)^{-1} d(g\oplus h) = (g^{-1} dg)\oplus (h^{-1}dh)$ and $(g^{-1})^{-1} dg^{-1} = -(dg) g^{-1}$ one has
\[
ch([g]+[h]) = ch([g\oplus h])= ch([g]) + ch([h]) \;\;\mbox{ and }\;\; ch(-[g]) = ch([g^{-1}])= -ch([g]),
\]
which explicitly demonstrates  that the odd Chern character $ch \colon K^{-1}(M) \to H^{odd}(M;\CC)$ is a group homomorphism.

The present aim is to define the odd Chern character directly on $\cK(M)$ using the \emph{total string forms} (Definition \ref{defn:stringform}) of the underlying $\Omega$ vector bundles.
Given the $\Omega$ vector bundle $\sE \to M$ equipped with based module connection $\Delta$ and Higgs field $\phi$, write $\sA$ and $\Phi$ for the corresponding connection and Higgs field on the frame bundle $\cF(\sE)$.
Recall from Remark \ref{remark:vectorprincipalcurv} that if $\sR$ is the curvature of $\Delta$ and $\sF$ is the curvature of $\sA$ then $\sR = \sF$, in particular $\sR$ is $\Omega\g$-valued.
Define the Higgs field covariant derivative as
\[
\Delta(\phi) := \nabla\Phi = d\Phi +[\sA,\Phi]-\d\sA,
\]
where the notation is chosen so as to indicate that $\Delta(\phi)$ depends both on $\Delta$ and on $\phi$.
\vspace{11pt}
\begin{definition}
\label{defn:stringform}
The \emph{total string form} of the object $(\sE,\Delta,\phi) \in \ovect^c(M)$ is
\[
s(\Delta,\phi) :=  \sum_{j=1}^\infty \frac{1}{(j-1)!}\bigg(\frac{1}{2\pi i}\bigg)^j \int_{\sone}  \tr\big( \underbrace{\sR\wedge\dotsb\wedge\sR}_{\text{$j-1$ times}} \wedge \,\Delta(\phi)\big),
\]
which is an odd complex-valued form on $M$.
\end{definition}

Comparing with the string forms on $\cF(\sE)$, similarly to $\Ch(\nabla)$ notice that  the degree $2k-1$ piece of $s(\Delta,\phi)$ is
\begin{equation}
\label{eqn:jthstringalt}
s_{\overline\tr_k}(\sA,\Phi) =  k \int_\sone \overline{\tr}_k(\nabla\Phi,\underbrace{\sF,\dotsc,\sF}_{\text{$k-1$ times}}),
\end{equation}
with $\overline{\tr}_k$ the $k$-th symmetrised trace.
In particular, $s(\Delta,\phi)$ is closed and its cohomology class, which is independent of $\Delta$ and $\phi$, is a characteristic class for $\Omega$ vector bundles.
Similarly to \eqref{eqn:totalchern} one also has
\begin{equation}
s(\Delta,\phi) =  \frac{1}{2\pi i} \int_\sone \tr\left( e^{\tfrac{1}{2\pi i} \sR} \wedge \Delta(\phi)\right)\!.
\end{equation}
The following result demonstrates that the total string forms are additive
\vspace{11pt}
\begin{lemma}
\label{lemma:stringprop}
For any $(\sE,\Delta,\phi) , (\sF,\Delta',\phi') \in \ovect^c(M)$
\[
s(\Delta\oplus\Delta',\phi\oplus\phi') = s(\Delta,\phi) + s(\Delta',\phi').
\]
\end{lemma}
\begin{proof}
The curvature of $\Delta \oplus \Delta'$ is the block sum $\sR_\Delta \oplus \sR_{\Delta'}$, with $\sR_\Delta$, $\sR_{\Delta'}$ respectively the curvatures of $\Delta$ and $\Delta'$.
Similarly, the Higgs field covariant derivative of $\phi\oplus\phi'$ is the block sum $\Delta(\phi) \oplus \Delta'(\phi')$.
Since trace is additive, for every $j>0$
\[
\tr\left((\sR_\Delta\oplus\sR_{\Delta'})^{j-1}) \wedge (\Delta(\phi)\oplus\Delta'(\phi')\right) = \tr\left( \sR_\Delta^{j-1} \wedge \Delta(\phi)\right) + \tr\left( \sR_{\Delta'}^{j-1} \wedge \Delta'(\phi')\right)
\]
and the result follows.
\end{proof}

For any $\Omega$ vector bundle $\sE \to M$ write $s(\sE) \in H^{odd}(M;\CC)$ for the image of the cohomology class $[s(\Delta,\phi)]$ under the de Rham isomorphism, with $\Delta$ and $\phi$ any choice of based module connection and Higgs field on $\sE$. 
By Lemma \ref{lemma:stringprop} the map
\begin{equation}
\label{eqn:oddchernchars}
\sE \longmapsto s(\sE)
\end{equation}
is a semi-group homomorphism $\pi_0\ovect(M) \to H^{odd}(M;\CC)$, which extends to a group homomorphism
\[
s \colon \cK(M) \lo H^{odd}(M;\CC).
\]
Explicitly, the action of $s$ on $\sE - \sF \in \cK(M)$ is
\[
s \colon \sE - \sF \longmapsto s(\sE) - s(\sF).
\]
One now shows that $s$ is indeed the odd Chern character on $\cK(M)$, realised in terms of $\Omega$ vector bundles.
\vspace{11pt}
\begin{theorem}
\label{theorem:stringchern}
The diagram of group homomorphisms
\[
\xy
(0,25)*+{\cK(M)}="1";
(50,25)*+{K^{-1}(M)}="2";
(25,0)*+{H^{odd}(M;\CC)}="3";
{\ar^{\cR} "1";"2"};
{\ar_{ch} "2";"3"};
{\ar^{s} "1";"3"};
\endxy
\]
commutes, with $\cR$ the map isomorphism of Theorem \ref{theorem:oddk}.
In particular,
\[
s \colon \cK(M) \lo H^{odd}(M;\CC)
\]
is the odd Chern character.
\end{theorem}
\begin{proof}
In the homotopy-theoretic model for $K^{-1}$, the isomorphism $\cR$ is given by
\[
\sE - \sF \longmapsto [g] \oplus [h]^{-1} = [g\oplus h^{-1}]
\]
where $g, h \colon M\to GL_n(\CC)$ are smooth classifying maps for $\sE$ and $\sF$ respectively for some sufficiently large $n$ (see Remark \ref{remark:homotopyversion}). 

Choose based module connections $\Delta$, $\Delta'$ and Higgs fields $\phi$, $\phi'$ on $\sE$ and $\sF$ respectively.
By Lemma \ref{lemma:higgshol} one may take $g = \hol_\phi$ and $h = \hol_{\phi'}$, in which case by Lemma \ref{lemma:universalstringform} and Theorem \ref{theorem:stringdiag1} one has, for example,
\[
s(\Delta,\phi) = \sum_{j=1}^\infty g^\ast \tau\left (\overline{\tr}_j \right ) = 
\sum_{j=1}^\infty \left(-\frac{1}{2}\right)^{j-1}\frac{j!(j-1)!}{(2j-1)!} \, g^\ast \,\overline{\tr}_j\big(\Theta, [\Theta,\Theta]^{j-1} \big) +\mbox{exact},
\]
with $\Theta$ the Maurer-Cartan form on $GL_n(\CC)$.
Now,
\[
\overline{\tr}_j\big(\Theta, [\Theta,\Theta]^{j-1} \big) = \frac{2^{j-1}}{j!(2\pi i)^j}\tr(\Theta^{2j-1}),
\]
and $g^\ast \Theta = g^{-1} dg$, so that
\[
s(\Delta,\phi) = \sum_{j=1}^\infty \frac{-(j-1)!}{(2j-1)!} \left(-\frac{1}{2\pi i}\right)^{j} \tr\left((g^{-1}dg)^{2j-1}\right)+\mbox{exact}.
\]
Shifting the degree of the sum, by the additivity of trace
\[
s(\sE - \sF) = \sum_{j=0}^\infty \frac{-j!}{(2j+1)!} \left(-\frac{1}{2\pi i}\right)^{j+1} \tr\left(\left[\big(g\oplus h^{-1}\big)^{-1}d\big(g\oplus h^{-1}\big)\right]^{2j+1}\right)+ \mbox{exact}.
\]
Comparing this with the expression \eqref{eqn:oddchern} for the odd Chern character gives
\[
s(\sE - \sF) = ch([g\oplus h^{-1}]) = ch(\cR(\sE-\sF))
\]
as required.
\end{proof}

If $\Theta$ now denotes the Maurer-Cartan form on $GL$ (see Appendix \ref{app:frechet}), set
\begin{equation}
\label{eqn:tauform}
\tau_{GL} :=  \sum_{j=0}^\infty \frac{-j!}{(2j+1)!} \left(-\frac{1}{2\pi i}\right)^{j+1} \tr\left(\Theta^{2j+1}\right)
\end{equation}
and define the space of closed odd degree complex-valued forms
\begin{equation}
\label{den:wedgegl}
\wedge_{GL}(M) := \{ g^\ast \tau_{GL}\mid g \colon M \lo GL \mbox{ is smooth}\}.
\end{equation}
Under the pointwise block sum and matrix inversion operations on the space of smooth maps $M\to GL$ one has
\[
(g \oplus h)^\ast \tau_{GL} = g^\ast\tau_{GL} + h^\ast\tau_{GL}
\;\;\mbox{ and }\;\;
(g^{-1})^\ast\tau_{GL} = -g^\ast\tau_{GL}
\]
so that $\wedge_{GL}(M)$ is a subgroup of $\Omega^{odd}(M;\CC)$.
The proof of Theorem \ref{theorem:stringchern} gives
\vspace{11pt}
\begin{corollary}
\label{cor:imagecherncharacter}
If $s \colon \cK(M) \to H^{odd}_\deR(M;\CC)$ is the odd Chern character on $\cK(M)$ valued in complex de Rham cohomology, then $\im\,s = \wedge_{GL}(M)$ mod exact and so $\wedge_{GL}(M)$ generates the odd complex de Rham cohomology of $M$.
\end{corollary}

In light of Remark \ref{remark:unversion},  there is also a Hermitian version of the odd Chern character that maps $\cK(M)$ into real-valued singular cohomology.
It is given by sending the virtual Hermitian $\Omega$ vector bundle $\sE - \sF$ to the cohomology class of the closed real-valued differential form $s(\Delta,\phi) - s(\Delta',\phi')$, where $\Delta$ and $\Delta'$, $\phi$ and $\phi'$ are any choice of based module connections and Higgs fields on $\sE$ and $\sF$ respectively that are compatible with the Hermitian structure.
The reason that $s(\Delta,\phi)$, for example, is a real-valued form is that
\[
s(\Delta,\phi) = \widehat{\int_\sone} \Ch(\nabla)
\]
where $\nabla$ is the caloron-transformed connection: since $\nabla$ is compatible with the caloron-transformed Hermitian structure, if $R$ denotes its curvature then the form $\tfrac{1}{i} R$ takes values in Hermitian matrices so $\Ch(\nabla)$ is a real-valued form.

One may prove a Hermitian version of Theorem \ref{theorem:stringchern} following the same argument as above, and setting
\begin{equation}
\label{eqn:htauform}
\tau :=  \sum_{j=0}^\infty \frac{-j!}{(2j+1)!} \left(-\frac{1}{2\pi i}\right)^{j+1} \tr\left(\Theta^{2j+1}\right)
\end{equation}
with $\Theta$ the Maurer-Cartan form on $U$, one defines the space of closed odd degree real-valued forms
\begin{equation}
\label{den:wedgeu}
\wedge_{U}(M) := \{ g^\ast \tau\mid g \colon M \lo U \mbox{ is smooth}\}.
\end{equation}
Similarly to $\wedge_{GL}(M)$, $\wedge_U(M)$ has an abelian group structure induced by the block sum operation on smooth maps $M \to U$.
One also has
\vspace{11pt}
\begin{corollary}
\label{cor:himagecherncharacter}
If $s \colon \cK(M) \to H^{odd}_\deR(M)$ is the odd Chern character on $\cK(M)$ valued in de Rham cohomology, then $\im\,s = \wedge_{U}(M) $ mod exact and so $\wedge_U(M)$ generates the odd de Rham cohomology of $M$.
\end{corollary}

%\newpage
%\mbox{}
\chapter{The $\Omega$ model\label{ch:five} for odd differential $K$-theory}
In this chapter, the theory of $\Omega$ vector bundles developed in Chapter \ref{ch:four} is used to construct a differential extension of odd topological $K$-theory---the $\Omega$ model---which is shown to be a model for odd differential $K$-theory.

A differential extension $\check E$ of the generalised (Eilenberg-Steenrod) cohomology theory $E$ is a geometric refinement of $E$ that naturally includes additional differential form data.
When the underlying cohomology theory $E$ is multiplicative, the differential extension $\check E$ is \emph{multiplicative} if it has a graded ring structure that respects the cohomological and differential form data.  
In their seminal paper \cite{HS}, Hopkins and Singer demonstrated that every generalised cohomology theory has a differential extension.
Following this, Bunke and Schick  \cite{BS3,BS2,BS1} provided an axiomatic characterisation of differential extensions together with a result,  under certain mild conditions on the underlying cohomology theory, that guarantees the uniqueness of multiplicative differential extensions up to unique isomorphism (these axioms and uniqueness results are recorded in Appendix \ref{app:diff}).
In point of fact, this unique isomorphism can be constructed without a multiplciative structure provided that the differential extension has a special pushforward map called an \emph{$\sone$-integration}.
In particular, since $K$-theory satisfies the requisite conditions, any two differential extensions of $K$-theory with $\sone$-integration are isomorphic via a unique isomorphism: this is what is meant by differential $K$-theory.

Differential $K$-theory has been studied extensively in recent years.
Consequently, there are a variety of different constructions of differential $K$-theory; for a detailed survey see  \cite{BS1}.
As with ordinary $K$-theory, the additive structure of differential $K$-theory splits into even and odd degree parts resulting in \emph{even} and \emph{odd} differential $K$-theory.
The Bunke-Schick uniqueness results are sufficient to guarantee that any two differential extensions of the even part of $K$-theory are uniquely isomorphic, even without a multiplicative structure or an $\sone$-integration.
Consequently, any differential extension of the even part of $K$-theory defines even differential $K$-theory.
However, this is not true for odd $K$-theory and in fact one can show that there are infinitely many non-isomorphic differential extensions of odd $K$-theory \cite{BS2}.

It is particularly important to have accessible geometric models for differential extensions, not least because this is the most common way that differential extensions appear in physics.
In gauge theory, for example, gauge fields are described by ordinary differential cohomology and it is suggested by Freed \cite{Freed2} that differential $K$-theory (or rather, twisted differential $K$-theory) describes Ramond-Ramond fields in type II string theory.

In \cite{SSvec}, Simons and Sullivan construct a simple geometric model for the even degree part of differential $K$-theory using vector bundles with connection.
The principle behind this construction is that, insofar as even $K$-theory is related to stable isomorphism classes of vector bundles, even differential $K$-theory is related to stable isomorphism classes of vector bundles with connection, where the stability condition is extended to connective data using the Chern-Simons forms.
This model of even differential $K$-theory provides a straightforward codification of vector bundles with connection.

In this chapter, an analogous geometric model for odd differential $K$-theory is constructed by using $\Omega$ vector bundles and their associated geometric data in place of vector bundles: this is the $\Omega$ model for odd differential $K$-theory.
Motivated by the results of Section \ref{S:ouroddk}, where it was shown that virtual $\Omega$ vector bundles of rank zero give odd $K$-theory, one shows that the odd part of differential $K$-theory may be obtained in a similar fashion by using $\Omega$ vector bundles equipped with based module connections and Higgs fields.
The string potentials of Chapter \ref{ch:three} play a central role in the construction of the $\Omega$ model, which gives a straightforward geometric model for odd differential $K$-theory complementing the work of Simons and Sullivan.

To aid the discussion, an explicit isomorphism is constructed between the $\Omega$ model and an elementary differential extension of odd $K$-theory described recently by Tradler, Wilson and Zeinalian \cite{TWZ}.
This isomorphism provides a homotopy-theoretic interpretation of the $\Omega$ model as well as a proof that the TWZ extension does indeed define odd differential $K$-theory, this latter result being speculated but not proved in \cite{TWZ}.

Throughout this chapter, unless stated otherwise $M$ shall always be taken to be a compact finite-dimensional manifold, possibly with corners, and all $\Omega$ vector bundles are taken to be Hermitian.

%%%%%%%%%%%%%%%%%%%%%%%%%%%%%%%%%%%%%%%%%%%%%%%%%%%%%%%%%%%%%%%%%%%%%%%%%%%%%%%%%%%%%%%%%
%%%%%%%%%%%%%%%%%%%%%%%%%%%%%%%%%%%%%%%%%%%%%%%%%%%%%%%%%%%%%%%%%%%%%%%%%%%%%%%%%%%%%%%%%
%%%%%%%%%%%%%%%%%%%%%%%%%%%%%%%%%%%%%%%%%%%%%%%%%%%%%%%%%%%%%%%%%%%%%%%%%%%%%%%%%%%%%%%%%

\section{Structured vector bundles}
The primary ingredient in the construction of the $\Omega$ model is the notion of \emph{structured} $\Omega$ vector bundles.
These are $\Omega$ vector bundles equipped with a certain equivalence class of based module connections and Higgs fields.

%%%%%%%%%%%%%%%%%%%%%%%%%%%%%%%%%%%%%%%%%%%%%%%%%%%%%%%%%%%%%%%%%%%%%%%%%%%%%%%%%%%%%%%%%
%%%%%%%%%%%%%%%%%%%%%%%%%%%%%%%%%%%%%%%%%%%%%%%%%%%%%%%%%%%%%%%%%%%%%%%%%%%%%%%%%%%%%%%%%
%%%%%%%%%%%%%%%%%%%%%%%%%%%%%%%%%%%%%%%%%%%%%%%%%%%%%%%%%%%%%%%%%%%%%%%%%%%%%%%%%%%%%%%%%

\subsection{The Simons-Sullivan model}
\label{S:simonssullivan}

The construction of the $\Omega$ model is motivated by the Simons-Sullivan model\footnote{in this thesis, one considers only the Hermitian version of the Simons-Sullivan construction---there is also a version that does not use Hermitian structures.} for even differential $K$-theory developed in \cite{SSvec}.
Consider a complex, finite-rank Hermitian vector bundle $E\to M$ with compatible connection $\nabla$.
Recall that the Chern character of $E$ may be represented by the real-valued closed form
\[
\Ch(\nabla) =  \sum_{j=0}^\infty \frac{1}{j!} \bigg(\frac{1}{2\pi i} \bigg)^j \tr(\underbrace{R\wedge\dotsb \wedge R}_{\text{$j$ times}}),
\]
with $R$ the curvature of $\nabla$.
If $\gamma\colon t\mapsto \nabla_t$ is a smooth path of connections\footnote{defined similarly to Definition \ref{defn:smoothconn}.} on $E$ with $B_t$ the time derivative of the path $\gamma$, the Chern-Simons form of $\gamma$ is
\[
\CS(\gamma) := \sum_{j=1}^\infty \frac{1}{(j-1)!}\bigg(\frac{1}{2\pi i} \bigg)^j \int_0^1 \tr(B_t\wedge \underbrace{R_t\wedge\dotsb \wedge R_t}_{\text{$j-1$ times}})\,dt,
\]
with $R_t$ the curvature of $\nabla_t$.
It is well-known that $d\CS(\gamma) = \Ch(\nabla_1) - \Ch(\nabla_0)$.

Since spaces of connections are affine, there is always a smooth path between any pair of connections on $E$.
If there is a path $\gamma$ from $\nabla_0$ to $\nabla_1$ such that $\CS(\gamma)$ is exact then it turns out that $\CS(\nu)$ is exact for \emph{any} smooth path $\nu$ from $\nabla_0$ to $\nabla_1$ \cite[Proposition 1.1]{SSvec} and, in particular, $\Ch(\nabla_0) = \Ch(\nabla_1)$.
Setting
\[
\CS(\nabla_0;\nabla_1) := \CS(\gamma) \mod\mbox{exact}
\]
for any smooth path $\gamma$ from $\nabla_0$ to $\nabla_1$, one defines an equivalence relation on the space of connections on $E$ by setting $\nabla_0 \sim \nabla_1\Leftrightarrow\CS(\nabla_0;\nabla_1) = 0 \mod\mbox{exact}$.
\vspace{11pt}
\begin{definition}[\cite{SSvec}]
A \emph{structured vector bundle over $M$} is a pair $\bm{E} = (E,[\nabla])$ with $E \to M$ a Hermitian vector bundle and $\nabla$ an equivalence class of compatible connections on $E$.
\end{definition}
\vspace{11pt}
\begin{example}[The trivial structured bundle of rank $n$]
The trivial structured bundle of rank $n$ over $M$ is $\underline{\bm{n}} = (\underline{\CC}^n,[d])$, with $\underline{\CC}^n = M\x\CC^n \to M$ the trivial bundle of rank $n$ (with its canonical Hermitian structure) and $d$ the trivial connection.
\end{example}

The equivalence relation on connections described above behaves well with respect to pullbacks.
In particular, one may pull back structured vector bundles by smooth maps and there is a natural notion of isomorphism for structured vector bundles.
The direct sum and tensor product operations on vector bundles with connection also give rise to well-defined operations $\oplus$ and $\otimes$ on structured vector bundles; for details the interested reader is referred to \cite{SSvec}.
\vspace{11pt}
\begin{definition}[\cite{SSvec}]
Define $\struuct(M)$ to be the set of all isomorphism classes of structured vector bundles over $M$, which is a commutative semi-ring under the operations $\oplus$ and $\otimes$.
To avoid excessive notation, elements of $\struuct(M)$ are denoted $\bm{E}$ rather than $[\bm{E}]$.
The assignment
\[
M \longmapsto \struuct(M)
\]
determines a contravariant functor
\[
\struuct\colon \Man \lo \rig
\]
from the category of smooth compact manifolds with corners to the category of semi-rings.
\end{definition}
The following special types of structured bundles are important for understanding how structured vector bundles give rise to even differential $K$-theory.
\vspace{11pt}
\begin{definition}[\cite{SSvec}]
\label{defn:ssst}
A structured vector bundle $\bm{E} = (E,[\nabla])$ is \emph{stably trivial} if there is some $k$ such that $\sE \oplus \underline{\CC}^k \cong\underline{\CC}^{n+k}$.
Write
\[
\struuct^T(M) := \left\{\bm{E} \in \struuct(M) \mid \bm{E} \mbox{ is stably trivial} \right\}
\]
for the semi-group of those isomorphism classes in $\struuct(M)$ containing a stably trivial structured vector bundle.

A structured vector bundle $\bm{E}$ is \emph{stably flat} if $\bm{E} \oplus \underline{\bm{k}} = \underline{\bm{n+k}}$ for some $k$.
Write
\[
\struuct^F(M) := \left\{\bm{E} \in \struuct(M) \mid \bm{E} \mbox{ is stably flat} \right\}
\]
for the semi-group of those isomorphism classes in $\struuct(M)$ containing a stably flat structured vector bundle.
Clearly $\struuct^F(M) \subset \struuct^T(M)$ and is a sub-semi-group.
\end{definition}

The Grothendieck group completion applied to $\struuct(M)$ gives the commutative ring
\begin{equation}
\label{eqn:ssktheory}
\edK(M) := K\big(\struuct(M)\big),
\end{equation}
where, as usual, elements of $\edK(M)$ are denoted as formal differences $\bm{E}-\bm{F}$.
Every structured bundle $\bm{E}$ has an \emph{inverse}, i.e. a structured bundle $\bm{F}$ such that $\bm{E}\oplus\bm{F} \cong \underline{\bm{n}}$ for some $n$, so one may write any element of $\edK(M)$ in the form $\bm{E} - \underline{\bm{n}}$ \cite[Theorem 1.8]{SSvec}.
One may also show that $\bm{E} - \underline{\bm{n}} = 0\in\edK(M)$ if and only if $\bm{E}$ is stably flat.

There is a natural transformation of functors
\[
\Ch \colon \struuct(M) \lo \Omega^{even}_{d=0}(M)
\]
given by $\Ch \colon \bm{E} = (E,[\nabla]) \mapsto \Ch(\nabla)$. In fact, $\Ch$ is a semi-ring homomorphism, that is
\[
\Ch(\bm{E} \oplus \bm{F}) = \Ch(\bm{E}) + \Ch(\bm{F})\;\;\mbox{ and }\;\; \Ch(\bm{E} \otimes \bm{F}) = \Ch(\bm{E}) \wedge \Ch(\bm{F})
\]
\cite[(1.11)]{SSvec}.
Passing to $\edK(M)$ gives the ring homomorphism
\[
\check{\Ch} \colon \bm{E} - \bm{F} \longmapsto \Ch(\bm{E}) - \Ch(\bm{F}).
\]
Write $\delta \colon \edK(M) \to K^0(M)$ for the obvious forgetful map that discards the connective data, i.e.~$\bm{E} - \bm{F} \mapsto E- F$.
Recalling that $K^0(M)$ may be defined entirely in terms of smooth vector bundles, it follows that $\delta$ is surjective.
The maps $\delta$ and $\check{\Ch}$ fit into the commuting diagram of ring homomorphisms
\[
\xy
(0,25)*+{\edK(M)}="1";
(50,25)*+{K^0(M)}="2";
(0,0)*+{\Omega^{even}_{d=0}(M)}="3";
(50,0)*+{H^{odd}(M;\RR)}="4";
{\ar^{\delta} "1";"2"};
{\ar^{\deR} "3";"4"};
{\ar^{\check{\Ch}} "1";"3"};
{\ar^{ch} "2";"4"};
\endxy
\]
with $\deR$ the natural map given by the de Rham isomorphism and $ch$ the even Chern character.
This gives part of the data required for a differential extension of the even part of $K$-theory: $\check{\Ch}$ is the \emph{curvature} morphism and $\delta$ is the \emph{underlying class} morphism of $\edK(M)$.

The final datum that makes $\edK(M)$ a differential extension of $K^0(M)$ is the \emph{action of forms} morphism $a \colon \Omega^{odd}(M)/\im\,d \to \edK(M)$, along with the associated exact sequence.
It is useful for the constructions in the sequel to examine how this map works in some detail.

Say that a connection $\nabla$ on $E$ is \emph{Flat} if it has trivial holonomy around every closed curve \cite[Definition 1.7]{SSvec}.
This implies that the curvature of $\nabla$ vanishes and (by finding parallel sections for $\nabla$) that $E$ is isomorphic to the trivial bundle $\underline{\CC}^n$ with the trivial connection.
Recall from \eqref{den:wedgeu} the space of odd forms $\wedge_{U}(M)$, then
\vspace{11pt}
\begin{lemma}[\cite{SSvec}]
\label{lemma:sscs1}
For any pair of Flat connections $\nabla$, $\nabla'$ on $E$ that are compatible with the Hermitian structure
\[
\CS(\nabla;\nabla') \in \wedge_{U}(M)\!\! \mod\mathrm{exact}.
\]
\end{lemma}

Digressing for the moment, one also has
\vspace{11pt}
\begin{lemma}
\label{lemma:flatuns}
Let $d$ be a Flat connection on the Hermitian vector bundle $E \to M$ and $g \colon M \to U(n)$ a smooth map, with $n = \rank E$.
Then there is a Flat connection $\bar d$ on $E$ and a smooth path $\gamma$ of connections compatible with the Hermitian structure from $d$ to $\bar d$ such that
\[
\CS(\gamma) = g^\ast \tau,
\]
with $\tau$ the form of \eqref{eqn:htauform}.
\end{lemma}
\begin{proof}
Compare with the proof of \cite[Lemma 2.1]{SSvec}.
By taking a global framing of $E$ that is parallel for $d$, one may suppose that $E = \underline{\CC}^n$ with its standard Hermitian structure and that $d$ is the trivial product connection.
Then, if $\Theta$ is the Maurer-Cartan form on $U(n)$ define the Flat compatible connection $\bar d := d + g^\ast\Theta = d+ g^{-1}dg$.

Define a smooth path of connections by $\gamma(t) := \nabla_t = d+ tg^\ast \Theta$ noting that the curvature of $\nabla_t$ is given by $td(g^\ast\Theta) + t^2 g^\ast\Theta\wedge g^\ast\Theta = (t^2-t) g^\ast(\Theta\wedge\Theta)$ since $d\Theta = -\Theta\wedge\Theta$.
Then
\begin{multline*}
\CS(\gamma) = \sum_{j=1}^\infty \frac{1}{(j-1)!}\bigg(\frac{1}{2\pi i} \bigg)^j \int_0^1 (t^2-t)^{j-1}\,dt \cdot \tr\big((g^\ast\Theta)^{2j-1}\big)\\
%%%
= \sum_{j=0}^\infty \frac{-j!}{(2j+1)!} \left(-\frac{1}{2\pi i}\right)^{j+1} \tr\big((g^\ast\Theta)^{2j+1}\big) = g^\ast\tau,
\end{multline*}
which is the desired result.
\end{proof}

Given any stably trivial $\bm{E} \in \struuct^T(M)$, take trivial Hermitian vector bundles $F$ and $H$ such that $E \oplus F\cong H$ and set
\[
\widehat{\CS}(\bm{E}) := \CS(\nabla^H;\nabla\oplus\nabla^F) \mod \wedge_{U}(M)
\]
for any choice of Flat connections $\nabla^H$ on $H$ and $\nabla^F$ on $F$ compatible with the Hermitian structure, noting that this is independent of choices by Lemma \ref{lemma:sscs1}.
This defines a semi-group homomorphism \cite[Proposition 2.3]{SSvec}
\[
\widehat{\CS} \colon \struuct^T(M) \lo \Omega^{odd}(M)/ \widehat\wedge_{U}(M),
\] 
where $\widehat\wedge_{U}(M) :=  \wedge_{U}(M) + d\Omega^{even}(M)$.
\vspace{11pt}
\begin{theorem}[\cite{SSvec}]
\label{theorem:cshat}
$\widehat{\CS}$ is surjective with kernel $\struuct^F(M)$.
Therefore
\[
\widehat{\CS} \colon \struuct^T(M)/\struuct^F(M) \lo \Omega^{odd}(M)/ \widehat\wedge_{U}(M)
\]
is a semi-group isomorphism, which implies that $\struuct^T(M)/\struuct^F(M)$ is a group.
\end{theorem}

By postcomposing $\widehat{\CS}^{-1}$ with the isomorphism
\begin{equation}
\label{eqn:ssstsfmap}
\Gamma \colon \struuct^T(M)/\struuct^F(M) \lo \ker\delta
\end{equation}
that sends $\{\bm{E}\} \mapsto \bm{E} - \underline{\bm{n}}$, with $n = \rank E$, one obtains the exact sequence
\[
0 \lo \Omega^{odd}(M)/ \widehat\wedge_{U}(M) \xrightarrow{\;\;\imath\;\;} \edK(M) \xrightarrow{\;\;\delta\;\;} K^0(M) \lo 0
\]
where $\imath = \Gamma\circ\widehat{\CS}^{-1}$.
One defines the action of forms morphism $a$ as the composition
\begin{equation}
\label{eqn:ssactionofforms}
\Omega^{odd}(M) /\im\,d  \lo \Omega^{odd}(M) /\,\widehat\wedge_{U}(M) \xrightarrow{\;\;\imath\;\;} \edK(M),  
\end{equation}
where the first map is the natural projection given by the identification
\[
\Omega^{even}(M) /\,\widehat\wedge_{U}(M) = \big(\Omega^{even}(M)/\im\,d\big) /\!\wedge_{U}(M).
\]
Corollary \ref{cor:himagecherncharacter} demonstrates that $\wedge_{U}(M) \mod \mbox{exact}$ is the image of the odd Chern character, so  the  sequence
\[
K^{-1}(M) \xrightarrow{\;\;ch\;\;} \Omega^{odd}(M) /\im\,d \xrightarrow{\;\;a\;\;} \edK(M) \xrightarrow{\;\;\delta\;\;} K^0(M) \lo 0
\]
is exact.
One may also show that $\check {\Ch} \circ a = d$  \cite[Proposition 3.2]{SSvec} so that $\edK$ is indeed a differential extension of even $K$-theory.
By the Bunke-Schick uniqueness result (Theorem \ref{theorem:bunkeschickk}) it follows that $\edK$ is naturally isomorphic to the even part of any other differential extension of $K$-theory via a unique isomorphism.
The Simons-Sullivan model for even differential $K$-theory has a definite geometric appeal since it provides a clear codification of vector bundles with connection.
Moreover, the construction of this model involves less geometric data than other models; compare with the Freed-Lott model \cite{FL}, which requires an additional differential form.
\vspace{11pt}
\begin{remark}
\label{remark:csactionoforms}
As a concluding remark to round off this section, one makes an observation that is required in the sequel.
Given any $\omega \in \Omega^{odd}(M)$, by definition $a(\{\omega\}) := \bm{E} - \underline{\bm{n}}$ where $\bm{E} = (E,[\nabla]) \in \struuct(M)$ is such that
\begin{itemize}
\item
$\rank E = n$ and $E$ is trivial; and

\item
for any chosen representative $\nabla \in [\nabla]$ there is a smooth path $\gamma'$ of connections from some Flat connection $\bar d$ to $\nabla$ such that the Chern-Simons form
\[
\CS(\gamma') = \omega \mod \widehat\wedge_{U}(M),
\]
i.e.~$\CS(\gamma') = \omega + g^\ast \tau+d\chi$ for some smooth $g \colon M \to U$.
Invoking Lemma \ref{lemma:flatuns}, after possibly stabilising by a trivial bundle with its trivial connection, this implies that there is a  path of connections $\gamma$ from some Flat connection $d$ to $\nabla$ such that
\[
\CS(\gamma) = \omega \mod \mbox{exact}.
\]
\end{itemize}
\end{remark}

%%%%%%%%%%%%%%%%%%%%%%%%%%%%%%%%%%%%%%%%%%%%%%%%%%%%%%%%%%%%%%%%%%%%%%%%%%%%%%%%%%%%%%%%%
%%%%%%%%%%%%%%%%%%%%%%%%%%%%%%%%%%%%%%%%%%%%%%%%%%%%%%%%%%%%%%%%%%%%%%%%%%%%%%%%%%%%%%%%%
%%%%%%%%%%%%%%%%%%%%%%%%%%%%%%%%%%%%%%%%%%%%%%%%%%%%%%%%%%%%%%%%%%%%%%%%%%%%%%%%%%%%%%%%%

\subsection{Structured $\Omega$ vector bundles}
The construction of the $\Omega$ model proceeds by analogy with the Simons-Sullivan model; accordingly the first notion required is that of \emph{structured} $\Omega$ vector bundles.

Take the Hermitian $\Omega$ vector bundle $(\sE,\sh)$ over $M$ (of rank $n$, say) equipped with based module connection $\Delta$ and Higgs field $\phi$ that are compatible with $\sh$.
From this point onward, it shall be standard to omit explicit reference to the Hermitian structure $\sh$ and refer to $\sE$ simply as a Hermitian vector bundle, with a choice of $\sh$ assumed.
Write $\sR$ for the curvature of $\Delta$ and $\Delta(\phi)$ for the Higgs field covariant derivative of $\phi$, recalling the total string form
\[
s(\Delta,\phi) =  \sum_{j=1}^\infty \frac{1}{(j-1)!}\bigg(\frac{1}{2\pi i}\bigg)^j \int_{\sone}  \tr( \underbrace{\sR\wedge\dotsb\wedge\sR}_{\text{$j-1$ times}} \wedge \,\Delta(\phi)),
\]
which is the odd Chern character on $\cK$.
Just as the string forms play the role of the Chern character on $\cK$, the relative string potentials play the part of relative Chern-Simons forms in the $\Omega$ model.

To see how this works, as in Section \ref{S:antistring} define a \emph{smooth $n$-cube} of based module connections on $\sE$ as a based module connection $\hat{\Delta}$ on the $\Omega$ vector bundle $\sE \x \I^n \to M\x\I^n$ such that $\hat \Delta_X s= 0$ for any vector field $X$ on $M \x\I^n$ that is vertical for the projection $M\x\I^n \to M$ and section $s$ pulled back to $M\x\I^n$ from $M$.
A \emph{smooth $n$-cube} of Higgs fields on $\sE$ is a Higgs field $\hat\phi$ on $\sE\x\I^n$ and since $\sE$ is Hermitian, smooth $n$-cubes of connections and Higgs fields are required to be compatible with the Hermitian structure induced on $\sE\x\I^n$.

Recalling the notation of Definitions \ref{defn:vectorhf} and \ref{defn:basedmodcon}, let $\cH_\sE$ and $\cM_\sE$ respectively be the spaces of Higgs fields and based module connections on $\sE$ that are compatible with the Hermitian structure.
A \emph{smooth map}
\[
\gamma \colon \I^n \lo \cM_\sE\x\cH_\sE
\]
is an assignment
\[
\gamma\colon {(t_1,\dotsc,t_n)} \longmapsto \left(\Delta_{(t_1,\dotsc,t_n)},\phi_{(t_1,\dotsc,t_n)}\right)
\]
with
\[
\Delta_{(t_1,\dotsc,t_n)} = \varsigma_{(t_1,\dotsc,t_n)}^\ast \Delta^\gamma\;\;\mbox{ and }\;\;\phi_{(t_1,\dotsc,t_n)} = \varsigma_{(t_1,\dotsc,t_n)}^\ast\phi^\gamma
\] 
for some smooth $n$-cubes of based module connections $\Delta^\gamma$ and Higgs fields $\phi^\gamma$ on $\sE$.
As in the principal bundle case, specifying such a smooth map $\gamma$ is equivalent to specifying smooth $n$-cubes $\Delta^\gamma$ and $\phi^\gamma$.
A straightforward argument using the proofs of Lemmas \ref{lemma:higgscorrespond} and \ref{lemma:moduleconnectioncorrespond} demonstrates that smooth maps
\[
\I^n \lo \cM_\sE\x\cH_\sE \;\;\mbox{ and }\;\; \I^n \lo \cA_{\cF(\sE)}\x\cH_{\cF(\sE)}
\]
are in bijective correspondence, with $\cF(\sE)$ the unitary frame bundle of $\sE$ determined by the Hermitian structure.

Take a smooth path $\gamma \colon \I \to \cM_\sE \x\cH_\sE$, writing $\cF\gamma$ for the corresponding smooth path $\I \to \cA_{\cF(\sE)}\x\cH_{\cF(\sE)}$.
\vspace{11pt}
\begin{definition}
\label{defn:antistringform}
For $j >0$, the \emph{$j$-th string potential} of the path $\gamma$ is the $(2j-2)$-form
\[
S_j(\gamma) := j \int_0^1 \int_\sone \Big( (j-1) \overline{\tr}_j\big(\sB_t , \underbrace{\sR_t,\dotsc,\sR_t}_{j-2\;\text{times}}, \Delta_t(\phi) \big) + \overline{\tr}_j\big(\underbrace{\sR_t,\dotsc,\sR_t}_{j-1\;\text{times}},\varphi_t \big)\Big)\,dt
\]
where $\sR_t$ is the curvature of $\Delta_t$, $\Delta_t(\phi)$ is the Higgs field covariant derivative of $\phi_t$ (with respect to $\Delta_t$) and
\[
\sB_t := \varsigma_t^\ast(\cL_{\d_t} \Delta^\gamma)\;\;\mbox{ and }\;\;\varphi_t := \varsigma_t^\ast(\cL_{\d_t} \phi^\gamma)
\]
are the time derivatives of $\Delta_t$ and $\phi_t$ along the path, which are both $\Omega \mathfrak{u}(n)$-valued forms.
The invariant polynomials $\overline{\tr}_k$ on $\mathfrak{u}(n)$ are the symmetrised traces of \eqref{eqn:traceline}.

The \emph{total string potential} of $\gamma$ is the form
\[
S(\gamma) = \sum_{j=1}^\infty S_j(\gamma),
\]
whose degree $2k-2$ piece is the $k$-th string potential.
The forms $S_j (\gamma)$ and $S(\gamma)$ appear to live on the total space $\sE$ but are in fact basic by Theorem \ref{theorem:stringycs} below.
\end{definition}

If $\cF\gamma(t) := (\sA_t,\Phi_t)$ is the smooth path of connections and Higgs fields on the frame bundle $\cF(\sE)$ corresponding to $\gamma$ then $\sR_t = \sF_t$ and $\Delta_t(\phi) = \nabla\Phi_t$, where $\sF_t$ is the curvature of the connection $\sA_t$.
The arguments of Lemmas \ref{lemma:higgscorrespond} and \ref{lemma:moduleconnectioncorrespond} also show that
\[
\varsigma_t^\ast\left(\cL_{\d_t} \sA^{\cF(\gamma)}\right) = \varsigma_t^\ast(\cL_{\d_t} \Delta^\gamma)\;\;\mbox{ and }\;\; \varsigma_t^\ast\left(\cL_{\d_t} \Phi^{\cF(\gamma)}\right)= \varsigma_t^\ast(\cL_{\d_t} \Phi^\gamma)
\]
and hence the $j$-th string potential satisfies
\begin{equation}
\label{eqn:jthstringpotalt}
S_j(\gamma) = S_{\overline{\tr}_j}(\cF(\gamma)),
\end{equation}
where the right hand side is a relative string potential form as in Chapter \ref{ch:three}.
\vspace{11pt}
\begin{theorem}
\label{theorem:stringycs}
For any smooth path $\gamma \colon \I \to \cM_\sE\x\cH_\sE$, the $j$-th and total string potentials descend to $M$ and the total string potential satisfies
\[
dS(\gamma) = s(\Delta_1,\phi_1) - s(\Delta_0,\phi_0).
\]
\end{theorem}
\begin{proof}
Applying Theorem \ref{theorem:antistring} to the $j$-th string potential shows that $S_j(\gamma)$ is basic and
\[
dS_j(\gamma) = dS_{\overline{\tr}_j}(\cF(\gamma)) = s_{\overline{\tr}_j}(\sA_1,\Phi_1) - s_{\overline{\tr}_j}(\sA_0,\Phi_0),
\]
which is precisely the degree $2j-1$ piece of $s(\Delta_1,\phi_1) - s(\Delta_0,\phi_0)$.
\end{proof}

Similarly to the expression \eqref{eqn:alternateanti} relating the (principal bundle) string potentials and string forms, by \eqref{eqn:jthstringalt} and \eqref{eqn:jthstringpotalt} one has
\begin{equation}
\label{eqn:altstringpot}
S(\gamma) = \int_0^1  \varsigma_t^\ast \imath_{\d_t} s(\Delta^\gamma,\phi^\gamma)\,dt.
\end{equation}
Another result that transfers over readily from the principal bundle case is
\vspace{11pt}
\begin{proposition}
\label{prop:stringpotsameend}
If $\gamma_0,\gamma_1 \colon \I \to \cM_\sE\x\cH_\sE$ are smooth paths with the same endpoints then $S(\gamma_1) - S(\gamma_0)$ is exact.
\end{proposition}

Due to this last result, one may use the string potentials to define an equivalence relation that is analogous to Chern-Simons exactness.
For any two pairs $(\Delta_0,\phi_0)$ and $(\Delta_1,\phi_1)$ of based module connections and Higgs fields on $\sE$ compatible with the given Hermitian structure, there is a smooth path $\gamma$ from $(\Delta_0,\phi_0)$ to $(\Delta_1,\phi_1)$.
Set
\begin{equation}
\label{eqn:stringdata}
\cS(\Delta_0,\phi_0;\Delta_1,\phi_1) := S(\gamma) \mod \mbox{exact}
\end{equation}
for any such path, which is well-defined by Proposition \ref{prop:stringpotsameend}.
Moreover,
\[
\cS (\Delta_0,\phi_0; \Delta_2,\phi_2) = \cS (\Delta_0,\phi_0; \Delta_1,\phi_1) + \cS (\Delta_1,\phi_1; \Delta_2,\phi_2)
\]
and since $S(\gamma) = 0$ for a constant path one has $\cS (\Delta_0,\phi_0; \Delta_1,\phi_1) = -\cS (\Delta_1,\phi_1; \Delta_0,\phi_0)$.

Define pairs  $(\Delta_0,\phi_0)$ and  $(\Delta_1,\phi_1)$ as above to be \emph{equivalent}, denoted $(\Delta_0,\phi_0) \sim (\Delta_1,\phi_1)$ if and only if
\[
\cS(\Delta_0,\phi_0;\Delta_1,\phi_1) = 0\mod\mbox{exact}.
\]
This is an equivalence relation and an equivalence class $[\Delta,\phi]$ is a \emph{string datum} on $\sE$.
By direct analogy with structured vector bundles
\vspace{11pt}
\begin{definition}
\label{defn:struct}
A \emph{structured $\Omega$ vector bundle over $M$} is a pair $\bm{\sE} = (\sE,[\Delta,\phi])$, with $\sE\to M$ a Hermitian $\Omega$ vector bundle and $[\Delta,\phi]$ a string datum on $\sE$.
The \emph{rank} of the structured $\Omega$ vector bundle $\bm{\sE}$ is $\rank\bm{\sE} := \rank\sE$.
\end{definition}
\vspace{11pt}
\begin{example}
The trivial structured $\Omega$ vector bundle over $M$ of rank $k$ is 
\[
\underline{\bm{L\CC}}^k = \big(\underline{L\CC}^k, [\delta,\d]\big)
\]
where $\delta$ is the trivial module connection and $\d$ is the trivial Higgs field on the trivial $\Omega$ vector bundle $\underline{L\CC}^k$ over $M$ equipped with its canonical Hermitian structure.
\end{example}
An important type of structured $\Omega$ vector bundle is
\vspace{11pt}
\begin{example}
\label{example:canonstruct}
The canonical structured $\Omega$ vector bundle of rank $n$ is
\[
\bm{\sE}(n) := \big( \sE(n), [\Delta(n),\phi(n)] \big),
\]
where $\sE(n) \to U(n)$ is the canonical Hermitian $\Omega$ vector bundle of rank $n$ with its standard geometric data $\Delta(n)$, $\phi(n)$---see Example \ref{example:bmsen}.
\end{example}

Given a structured $\Omega$ vector bundle $\bm{\sE} = (\sE,[\Delta,\phi])$ over $M$ and a smooth map $f \colon N \to M$, one may define the pullback $f^\ast \bm{\sE}$ as $(f^\ast \sE,[f^\ast\Delta,f^\ast\phi])$.
This is well-defined since the string forms and string potentials are natural so that
\[
f^\ast \cS(\Delta_0,\phi_0;\Delta_1,\phi_1) = \cS(f^\ast\Delta_0,f^\ast\phi_0;f^\ast\Delta_1,f^\ast\phi_1).
\]
Structured $\Omega$ vector bundles $\bm{\sE} = (\sE,[\Delta,\phi])$ and $\bm{\sF}= (\sF,[\Delta',\phi'])$ are \emph{isomorphic} if there is an isomorphism $f \colon \sF \to \sE$ of $\Omega$ vector bundles such that $[\Delta',\phi'] = f^\ast[\Delta,\phi]:= [f^\ast\Delta,f^\ast\phi]$.

In this thesis it is important to know how structured $\Omega$ vector bundles behave under smooth homotopy.
Suppose that $f_t \colon N \to M$ is a family of smooth maps depending smoothly on the parameter $t \in\I$, equivalently  a smooth map $F\colon N\x\I \to M$ where $F(\cdot,t) = f_t$.
Take a Hermitian $\Omega$ vector bundle $\sE \to M$ with compatible based module connection $\Delta$ and Higgs field $\phi$ and for $t\in\I$ consider the pullbacks $f_t^\ast(\sE)$, which are Hermitian $\Omega$ vector bundles with compatible connections $f_t^\ast\Delta$ and Higgs fields $f_t^\ast\phi$.

Let $E$ be the caloron transform of $\sE$ equipped with the Hermitian caloron-transformed connection $\nabla$.
Define the family of smooth curves $\rho_{x} \colon \I \to M$ by $\rho_{x}(t) := f_t(x)$ for $x\in N$.
Similarly to \cite[p.~584]{SSvec} and Lemma \ref{lemma:pointygrassman}, let  $\varrho_t \colon (f_0\x\id)^\ast E \to  (f_t\x\id)^\ast E$  be the smooth isomorphisms determined by parallel transport along the family of curves $\rho_{x,\theta}(t):=(f_t(x),\theta)$, observing that the $\varrho_t$ cover the identity on $N\x\sone$.

Set $F:=  (f_0\x\id)^\ast E$ and $\nabla_t:=  \varrho_t^\ast (f_0\x\id)^\ast\nabla$, noting that for all $t\in\I$ the connection $\nabla_t$ is compatible with the induced Hermitian structure on $F$.
Write $\dot\rho_{x,\theta}(t)$ for the tangent vector to $\rho_{x,\theta}$ at $t$, then
\[
\CS(\nabla_0,\nabla_1) = \int_0^1 (f_t\x\id)^\ast \imath_{\dot\rho_{x,\theta}(t)} \Ch(\nabla)\,dt \mod\mbox{exact}
\]
(cf. \cite[(1.8)]{SSvec}).
Recall that
\[
s(\Delta,\phi) = \widehat{\int_\sone} \Ch(\nabla),
\]
then by Lemma \ref{lemma:intpullsquare} and the fact that $\dot\rho_{x,\theta}(t)$ has no $\sone$-component,
\[
\widehat{\int_\sone} (f_t\x\id)^\ast = f_t^\ast\widehat{\int_\sone}
\;\;\mbox{ and }\;\;
\widehat{\int_\sone}\imath_{\dot\rho_{x,\theta}(t)} = \imath_{\dot\rho_{x}(t)}\widehat{\int_\sone}
\] 
so after integrating over the fibre one obtains
\begin{equation}
\label{eqn:homotopicstringpot}
\cS(\Delta_0,\phi_0;\Delta_1,\phi_1) = \int_0^1 f_t^\ast \imath_{\dot\rho_{x}(t)} s(\Delta,\phi)\,dt \mod\mbox{exact},
\end{equation}
where $\Delta_t, \phi_t$ correspond to the connection $\nabla_t$.
This argument shows that $f_t^\ast\sE \cong f_0^\ast \sE$ for any $t\in\I$ and also that
\begin{equation}
\label{eqn:homotopicstringpot2}
\cS(f_0^\ast\Delta,f_0^\ast\phi;f_1^\ast\Delta,f_1^\ast\phi) = \int_0^1 f_t^\ast \imath_{\dot\rho_{x}(t)} s(\Delta,\phi)\,dt \mod\mbox{exact},
\end{equation}
abusing notation slightly by omitting the isomorphism $f_1^\ast\sE \cong f_0^\ast \sE$.

The direct sum operation $\oplus$ on $\ovect(M)$ extends to structured $\Omega$ vector bundles
\vspace{11pt}
\begin{lemma}
\label{lemma:antistringprop}
Let $\sE,\sF$ be $\Omega$ vector bundles over $M$ and take any pair of smooth paths $\gamma \colon \I \to \cM_\sE \x\cH_\sE$ and $\sigma \colon \I \to \cM_\sF \x\cH_\sF$.
These together determine a smooth path
\[
\gamma \oplus \sigma \colon \I \to \cM_{\sE \oplus \sF} \x\cH_{\sE \oplus \sF}
\]
via $\gamma\oplus\sigma(t) := (\Delta^\gamma_t \oplus \Delta^\sigma_t, \phi^\gamma_t \oplus \phi^\sigma_t)$.
Then
\[
S(\gamma\oplus \phi) = S(\gamma) + S(\phi)
\]
and by Theorem \ref{theorem:stringycs}
\[
dS(\gamma\oplus \phi) = s(\Delta^\gamma_1 \oplus \Delta^\sigma_1, \phi^\gamma_1 \oplus \phi^\sigma_1) - s(\Delta^\gamma_0 \oplus \Delta^\sigma_0, \phi^\gamma_0 \oplus \phi^\sigma_0).
\]
Moreover, for any choice of pairs $(\Delta,\phi)$, $(\bar\Delta,\bar\phi)$ on $\sE$ and $(\Delta',\phi')$, $(\bar\Delta',\bar\phi')$ on $\sF$
\[
\cS(\Delta\oplus \Delta',\phi\oplus\phi';\bar\Delta\oplus \bar\Delta',\bar\phi\oplus\bar\phi') = 
\cS(\Delta,\phi;\bar\Delta,\bar\phi) + \cS(\Delta',\phi';\bar\Delta',\bar\phi') 
\]
\end{lemma}
\begin{proof}
Compare with \cite[Lemma 1.4]{SSvec}.
The proof of the first two assertions is essentially that of Lemma \ref{lemma:stringprop}, since the curvatures and time derivatives are block sums.

The last assertion follows from
\begin{multline*}
\cS(\Delta\oplus \Delta',\phi\oplus\phi';\bar\Delta\oplus \bar\Delta',\bar\phi\oplus\bar\phi') = 
\cS(\Delta\oplus \Delta',\phi\oplus\phi';\Delta\oplus \bar\Delta',\phi\oplus\bar\phi')\\
%%%
 + \cS(\Delta\oplus \bar\Delta',\phi\oplus\bar\phi';\bar\Delta\oplus \bar\Delta',\bar\phi\oplus\bar\phi')
\end{multline*}
and by using the first assertion, i.e.
\[
\cS(\Delta\oplus \Delta',\phi\oplus\phi';\Delta\oplus \bar\Delta',\phi\oplus\bar\phi') = \cS(\Delta',\phi';\bar\Delta',\bar\phi')
\]
and similarly for the second term.
\end{proof}

It follows immediately that the direct sum operation is well-defined on string data, so if $\bm{\sE} = (\sE,[\Delta,\phi])$ and $\bm{\sF} = (\sF,[\Delta',\phi'])$ are structured $\Omega$ bundles over $M$ then their \emph{direct sum}
\[
\bm{\sE} \oplus \bm{\sF} := \left(\sE\oplus \sF,\left[\Delta\oplus\Delta',\phi\oplus\phi'\right]\right)
\]
is a structured $\Omega$ bundle over $M$.
Note there is an obvious isomorphism $\bm{\sE} \oplus\bm{\sF} \cong \bm{\sF}\oplus\bm{\sE}$.

In order to define the $\Omega$ model, it is necessary to consider a special class of structured Hermitian $\Omega$ vector bundles, namely those structured $\Omega$ vector bundles that are pullbacks of $\bm{\sE}(n)$ for some $n$ (Example \ref{example:canonstruct}).
Observe that if $g \colon M \to U(n)$ and $h\colon M \to U(m)$ are smooth maps then the block sum
\[
g\oplus h \colon x \longmapsto
\begin{pmatrix}
g(x) & 0\\
0 & h(x)
\end{pmatrix}
\]
is a smooth map $g\oplus h \colon M \to U(n+m)$ and there is an isomorphism of structured $\Omega$ vector bundles
\[
g^\ast\bm{\sE}(n) \oplus h^\ast\bm{\sE}(m) \simto (g\oplus h)^\ast\bm{\sE}(n+m).
\]
The trivial structured $\Omega$ vector bundle $\underline{\bm{L\CC}}^n$ coincides with the pullback of $\bm{\sE}(n)$ by a constant map $M\to U(n)$.
\vspace{11pt}

\begin{definition}
Define $\struct(M)$ to be the set of isomorphism classes of structured $\Omega$ vector bundles
\[
\struct(M) := \{ [g^\ast\bm{\sE}(n)] \mid \mbox{$g$ is a smooth map $M \to U(n)$ for some $n$}\}.
\]
Since $h^\ast\bm{\sE}(m) \oplus g^\ast\bm{\sE}(n) \cong g^\ast\bm{\sE}(n) \oplus h^\ast\bm{\sE}(m) \cong (g\oplus h)^\ast \bm{\sE}(n+m)$ for any smooth maps $g \colon M \to U(n)$, $h \colon M \to U(m)$, $\struct(M)$ is an abelian semi-group under the direct sum operation on structured $\Omega$ vector bundles.
A smooth map $f \colon N \to M$ induces the pullback map $f^\ast \colon \struct(M) \to \struct(N)$  respecting the semi-group structure, so that the assignment
\[
M \longmapsto \struct(M)
\] 
determines a contravariant functor from the category of smooth compact manifolds with corners to the category of abelian semi-groups.
\end{definition}
\vspace{11pt}
\begin{remark}
\label{remark:whyrestrict}
$\struct(M)$ is a special set of isomorphism classes of structured $\Omega$ vector bundles.
The primary reason for making this restriction is Theorem \ref{theorem:inversesexist} below,  which guarantees that every element $\bm{\sE}$ of $\struct(M)$ has an `inverse', i.e.~that there is some $\bm{\sF} \in\struct(M)$ such that $\bm{\sE}\oplus\bm{\sF}$ is trivial.
One expects that any arbitrary structured $\Omega$ vector bundle $\bm{\sE}$ over $M$ (not necessarily isomorphic to a pullback of some $\bm{\sE}(n)$) should have an inverse, however a proof of this fact has not yet been found.
\end{remark}

From Theorem \ref{theorem:stringycs} and the definition of string data the map
\begin{equation}
\label{eqn:stringsemihom}
s \colon \struct(M) \lo \Omega^{odd}_{d=0}(M)
\end{equation}
given by 
\[
(\sE,[\Delta,\phi]) \longmapsto s(\Delta,\phi)
\]
is a well-defined semi-group homomorphism.
It is shown below in Section \ref{S:omega}  that after passing to the group completion, this is the curvature morphism of the $\Omega$ model.

The following theorem is crucial to the construction of the $\Omega$ model and the proof is based on related work appearing in \cite{SSvec,TWZ}.
\vspace{11pt}
\begin{theorem}
\label{theorem:inversesexist}
Let $g \colon M \to U(n)$ be any smooth map.
Then there is an isomorphism of structured $\Omega$ vector bundles
\[
g^\ast\bm{\sE}(n) \oplus (g^{-1})^\ast \bm{\sE}(n) \cong \underline{\bm{L\CC}}^{2n}.
\]
\end{theorem}
\begin{proof}
As remarked above, the structured $\Omega$ vector $g^\ast\bm{\sE}(n) \oplus (g^{-1})^\ast \bm{\sE}(n)$ coincides with the pullback of $\bm{\sE}(2n)$ by the block sum map
\[
g\oplus g^{-1} \colon x \longmapsto
\begin{pmatrix}
g(x) & 0\\
0 & g(x)^{-1}
\end{pmatrix}
\in U(2n)
\]
and the structured $\Omega$ vector bundle $\underline{\bm{L\CC}}^{2n}$ is the pullback of $\bm{\sE}(2n)$ by the constant map at the identity.
The method of this proof is to construct a certain smooth homotopy from $g\oplus g^{-1}$ to $\id$ and then use \eqref{eqn:homotopicstringpot2}.

Similarly to \cite[Lemma 3.6]{TWZ}, define the smooth map $R \colon\J \to U(2n)$ by setting
\[
R(t):=
\begin{pmatrix}
\cos t & -\sin t\\
\sin t & \cos t
\end{pmatrix}
\]
which is a block matrix consisting of $n\x n$ blocks.
Writing
\[
A = 
\begin{pmatrix}
g& 0\\
0 & 1
\end{pmatrix}
\;\;\mbox{ and }\;\;
B= 
\begin{pmatrix}
1 & 0\\
0 & g^{-1}
\end{pmatrix},
\]
for $t \in \J$ consider the map $X_t \colon M \to U(2n)$ given by
\[
X_t:= 
A
R(t)
B
R(t)^{-1}.
\]
The assignment $M\x\J \to U(2n)$ sending $(x,t) \mapsto X_t(x)$ is a smooth homotopy from $g\oplus g^{-1}$ to the identity and by direct calculation
\[
X_t^{-1} \d_t X_t = R(t) B^{-1} J B R(t)^{-1} - J
\]
where
\[
J = 
\begin{pmatrix}
0 & -1\\
1 & 0
\end{pmatrix}\!.
\]

Applying the exterior derivative to $X_t$ (considered as a $U(2n)$-valued function on $M$) gives
\begin{multline*}
dX_t = (dA)R(t)BR(t)^{-1} + AR(t)(dB)R(t)^{-1}\\= \begin{pmatrix}
dg& 0\\
0 & 0
\end{pmatrix}
R(t)
B
R(t)^{-1}
+
AR(t)
\begin{pmatrix}
0& 0\\
0 & -g^{-1}(dg) g^{-1}
\end{pmatrix}
R(t)^{-1}
\end{multline*}
since $d(gg^{-1}) = (dg)g^{-1} +g dg^{-1} = 0$.
Consequently, for any $j \geq 0$
\[
\big(X_t^{-1} dX_t \big)^{2j}
=
R(t) B^{-1} \big[R(t)^{-1} A^{-1}(dA) R(t) + (dB)B^{-1} \big]^{2j} B R(t)^{-1}.
\]
Another direct calculation gives
\[
R(t)^{-1} A^{-1}(dA) R(t) + (dB)B^{-1} = 
\begin{pmatrix}
\cos^2 t & -\cos t\,\sin t\\
-\cos t\,\sin t & - \cos^2t
\end{pmatrix}
g^{-1}dg
\]
and so
\[
\big[R(t)^{-1} A^{-1}(dA) R(t) + (dB)B^{-1} \big]^{2j} = 
\begin{pmatrix}
\cos^{2j} t & 0\\
0 & \cos^{2j}t
\end{pmatrix}
\big(g^{-1}dg\big)^{2j}.
\]
By the cyclic property of $\tr$ and using $R^{-1}(t)JR(t) = J$,
\begin{align*}
\tr\Big( X_t^{-1}\d_t X_t \cdot \big(X_t^{-1}dX_t \big)^{2j}  \Big)& = 
%%%
\tr \bigg( J
\begin{pmatrix}
\cos^{2j} t & 0\\
0 & \cos^{2j}t
\end{pmatrix}
\big(g^{-1}dg\big)^{2j}
 \bigg)\\
%%%
&\qquad\quad - \tr\bigg(
BJB^{-1}
\begin{pmatrix}
\cos^{2j} t & 0\\
0 & \cos^{2j}t
\end{pmatrix}
\big(g^{-1}dg\big)^{2j}
 \bigg)\\
%%%
&=
\tr\bigg(
\begin{pmatrix}
0& -\cos^{2j} t\\
 \cos^{2j}t &0
\end{pmatrix}
\big(g^{-1}dg\big)^{2j}
 \bigg)\\
%%%
&\qquad\quad-
\tr\bigg(
\begin{pmatrix}
0& -g\cos^{2j} t \\
 g^{-1}\cos^{2j}t &0
\end{pmatrix}
\big(g^{-1}dg\big)^{2j}
 \bigg)\\
%%%
&=0.
\end{align*}
Changing track for the moment observe that \eqref{eqn:homotopicstringpot2} and Lemma \ref{lemma:universalstringform} give
\begin{multline*}
\cS\big( g^\ast \Delta(n)\oplus (g^{-1})^\ast\Delta(n),g^\ast \phi(n)\oplus (g^{-1})^\ast\phi(n); \delta,\d \big)\\
%%%
=\int_0^{\tfrac{\pi}{2}} X_t^\ast \imath_{\dot\rho_{x}(t)} \bigg(  \sum_{j=0}^\infty \frac{-j!}{(2j+1)!} \left(-\frac{1}{2\pi i}\right)^{j+1} \tr\left(\Theta^{2j+1}\right) \bigg)dt\mod\mbox{exact}
\end{multline*}
where $\dot\rho_x(t) = \d_t X_t(x)$ is tangent to the curve $t\mapsto X_t(x)$.
But by the above
\begin{multline*}
X_t^\ast \imath_{\dot\rho_x(t)} \tr\big(\Theta^{2j+1}\big) = (2j+1) X_t^\ast \tr\big(\imath_{\dot\rho_x(t)}\Theta \cdot \Theta^{2j} \big)\\
%%%
 = (2j+1) \tr\big(X_t^{-1}\d_t X_t\cdot (X_t^{-1}dX_t)^{2j} \big) = 0
\end{multline*}
since $X_t^\ast \imath_{\dot\rho_{x}(t)} \Theta = X_t^{-1}\d_t X_t$ and $X_t^\ast\Theta = X_t^{-1}dX_t$.

This shows that the pairs $(g^\ast \Delta(n)\oplus (g^{-1})^\ast\Delta(n),g^\ast \phi(n)\oplus (g^{-1})^\ast\phi(n))$ and $(\delta,\d)$ determine the same string datum and, hence, that $g^\ast\bm{\sE}(n) \oplus (g^{-1})^\ast \bm{\sE}(n) \cong \underline{\bm{L\CC}}^{2n}$.
\end{proof}

%%%%%%%%%%%%%%%%%%%%%%%%%%%%%%%%%%%%%%%%%%%%%%%%%%%%%%%%%%%%%%%%%%%%%%%%%%%%%%%%%%%%%%%%%
%%%%%%%%%%%%%%%%%%%%%%%%%%%%%%%%%%%%%%%%%%%%%%%%%%%%%%%%%%%%%%%%%%%%%%%%%%%%%%%%%%%%%%%%%
%%%%%%%%%%%%%%%%%%%%%%%%%%%%%%%%%%%%%%%%%%%%%%%%%%%%%%%%%%%%%%%%%%%%%%%%%%%%%%%%%%%%%%%%%

\section{The $\Omega$ model}
\label{S:omega}
Having gathered most of the required results, one is at last in a position to define the $\Omega$ model.
Define $\dK(M)$ to be the rank zero subgroup of $K(\struct(M))$, that is
\begin{equation}
\label{eqn:odddk}
\dK(M) := \big\{\, \bm{\sE} - \bm{\sF} \in  K \left(\struct(M)\right)  \,\big|\, \rank(\bm{\sE} - \bm{\sF} )=\rank\bm{\sE} - \rank\bm{\sF}= 0\,\big\}.
\end{equation}
\begin{remark}
Note that $\dK(M)$ is defined by directly enforcing the zero rank condition, in contrast to the definition of $\cK(M)$ in Section \ref{S:ouroddk}, which used a restriction-to-basepoint map (though of course, one could have required rank zero from the beginning).
This is because a point has non-trivial odd differential $K$-theory  even though the odd $K$-theory of a point is trivial---see Example \ref{example:odddkpoint}.
\end{remark}

One declares structured $\Omega$ vector bundles $\bm{\sE}$ and $\bm{\sF}$ over $M$ to be \emph{stably isomorphic} if there is some $n$ such that $\bm{\sE} \oplus \underline{\bm{L\CC}}^n \cong \bm{\sF} \oplus \underline{\bm{L\CC}}^n$.
By Theorem \ref{theorem:inversesexist} and the definition:
\begin{itemize}
\item
$\bm{\sE} - \bm{\sF} = \bm{\sE}' - \bm{\sF}'$ in $\dK(M)$ if and only if there is some $\bm{\mathsf{G}} \in \struct(M)$ such that $\bm{\sE} \oplus \bm{\sF}' \oplus \bm{\mathsf{G}}\cong  \bm{\sE}' \oplus \bm{\sF}\oplus \bm{\mathsf{G}}$.
By Theorem \ref{theorem:inversesexist} there is some $\bm{\sH} \in \struct(M)$ such that $\bm{\mathsf{G}} \oplus \bm{\sH} = \underline{\bm{L\CC}}^k$ for some $k$, which implies $\bm{\sE} - \bm{\sF} = \bm{\sE}' - \bm{\sF}'$ in $\dK(M)$ if and only if $\bm{\sE}\oplus\bm{\sF}'$ and $\bm{\sE}'\oplus\bm{\sF}$ are stably isomorphic.

\item
Taking any $\bm{\sE} - \bm{\sF} \in \dK(M)$, there is some $\bm{\mathsf{G}} \in \struct(M)$ such that $\bm{\sF} \oplus\bm{\mathsf{G}} = \underline{\bm{L\CC}}^k$ for some $k$ and so
\[
\bm{\sE} - \bm{\sF} = \bm{\sE} + \bm{\mathsf{G}}- \bm{\sF} -\bm{\mathsf{G}} = \bm{\sE}\oplus\bm{\mathsf{G}} - \underline{\bm{L\CC}}^k.
\]
Thus any element of $\dK(M)$ may be written in the form $\bm{\sE} - \underline{\bm{L\CC}}^n$ where $n = \rank\bm{\sE}$.

\item
Combining the previous two observations, $\bm{\sE} - \underline{\bm{L\CC}}^n = 0$ in $\dK(M)$ if and only if there is some $k$ such that $\bm{\sE}\oplus\underline{\bm{L\CC}}^k \cong \underline{\bm{L\CC}}^{n+k}$.
\end{itemize}

As mentioned above, the string form map \eqref{eqn:stringsemihom} extends to a group homomorphism
\begin{equation}
\label{eqn:stringhom}
\check s \colon \dK(M) \lo \Omega^{odd}_{d=0}(M)
\end{equation}
given explicitly by sending
\[
\check s \colon \bm{\sE} - \bm{\sF} \longmapsto s(\Delta,\phi) - s(\Delta',\phi')
\]
where $\bm{\sE} = (\sE,[\Delta,\phi])$ and $\bm{\sF}=(\sF,[\Delta,\phi'])$.
There is also a surjective map
\[
\delta \colon \dK(M) \lo \cK(M)
\]
that discards the connective data, namely
\[
\delta \colon \bm{\sE} - \bm{\sF} \longmapsto \sE - \sF.
\]
In fact, the maps $\check s$ and $\delta$ are natural transformations of functors and, recalling the odd Chern character map $s$ of \eqref{eqn:oddchernchars}, one has the commuting diagram of homomorphisms
\[
\xy
(0,25)*+{\dK(M)}="1";
(50,25)*+{\cK(M)}="2";
(0,0)*+{\Omega^{odd}_{d=0}(M)}="3";
(50,0)*+{H^{odd}(M;\RR)}="4";
{\ar^{\delta} "1";"2"};
{\ar^{\deR} "3";"4"};
{\ar^{\check s} "1";"3"};
{\ar^{s} "2";"4"};
\endxy
\]
with $\deR$ the natural map given by the de Rham isomorphism.
These data form the beginnings of the sought-after differential extension of $\cK(M)$: $\delta$ is the \emph{underlying class} and $\check s$ is the \emph{curvature}.

In order to easily obtain the remaining data required of a differential extension---namely, the \emph{action of forms} together with its associated exact sequence---one relates $\dK$ to a construction due to Tradler, Wilson and Zeinalian \cite{TWZ}.

%%%%%%%%%%%%%%%%%%%%%%%%%%%%%%%%%%%%%%%%%%%%%%%%%%%%%%%%%%%%%%%%%%%%%%%%%%%%%%%%%%%%%%%%%
%%%%%%%%%%%%%%%%%%%%%%%%%%%%%%%%%%%%%%%%%%%%%%%%%%%%%%%%%%%%%%%%%%%%%%%%%%%%%%%%%%%%%%%%%
%%%%%%%%%%%%%%%%%%%%%%%%%%%%%%%%%%%%%%%%%%%%%%%%%%%%%%%%%%%%%%%%%%%%%%%%%%%%%%%%%%%%%%%%%

\subsection{The $\TWZ$ extension and the action of forms on $\dK$}
\label{S:other}
In a recent paper \cite{TWZ} the authors construct a differential extension of odd $K$-theory, referred to in this thesis as the \emph{$\TWZ$ extension} and denoted $\dl$.
The $\TWZ$ extension refines the homotopy-theoretic model of odd topological $K$-theory ($K^{-1}(M) \cong [M,U]$) by using an equivalence relation on the space $\Map(M,U)$ of smooth maps $M \to U$ that is finer than smooth homotopy.
It is not shown in \cite{TWZ} that $\dl$ defines the odd part of differential $K$-theory.

In this section, one briefly reviews the $\TWZ$ extension before constructing an explicit natural isomorphism $\dK\to \dl$ that is used to easily define the action of forms on $\dK$.
Results relating to the $\TWZ$ extension are stated without proof: the interested reader is referred to \cite{TWZ}.
\vspace{11pt}
\begin{definition}[\cite{TWZ}]
Two smooth maps $g_0, g_1 \in \Map(M,U)$ are \emph{$\CS$-equivalent}, written $g_0\sim_\CS g_1$, if there exists a smooth homotopy $G \colon M \x\I \to U$ from $g_0$ to $g_1$ such that
\begin{equation}
\label{eqn:twzcsequiv}
\sum_{j=0}^\infty  \frac{-j!}{(2j)!} \left(-\frac{1}{2\pi i}\right)^{j+1} \int_0^1 \tr\big(g_t^{-1} \d_t g_t \cdot (g_t^{-1}dg_t)^{2j} \big)\,dt = 0 \mod \mbox{exact}, 
\end{equation}
where $g_t := G(\cdot,t) \colon M\to U$.
The set of equivalence classes
\[
\dl(M) := \Map(M,U) /\! \sim_\CS
\]
inherits an abelian group structure from the block sum operation on $\Map(M,U)$ and the assignment $M \mapsto \dl(M)$ defines a contravariant functor on the category of compact manifolds with corners to the category of abelian groups that acts on morphisms by pullback.
\end{definition}
\vspace{11pt}
\begin{remark}
The condition \eqref{eqn:twzcsequiv} is realised below as the homotopy-theoretic analogue of the total string potential condition that defines string data in the $\Omega$ model.
\end{remark}

As $\dl$ is a differential extension of $K^{-1}$, it comes with curvature, underlying class and action of forms maps.
However, it turns out to be more convenient to work with a different model for $\dl$ described in \cite[Appendix A]{TWZ}, which is defined on a larger generating set.
\vspace{11pt}
\begin{definition}[\cite{TWZ}]
Define
\[
S(M) := \{ (g,\chi) \in \Map(M,U) \x \Omega^{even}(M) \},
\]
which is a monoid under the operation $(g_1,\chi_1) + (g_2,\chi_2) := (g_1\oplus g_2,\chi_1+\chi_2)$ with identity $(\id,0)$, where $\id\colon M\to U$ is the constant map sending $M$ to the identity in $U$.
Pairs $(g_0,\chi_0), (g_1,\chi_1) \in S(M)$ are \emph{CS-equivalent}, denoted $(g_0,\chi_0) \sim_\CS (g_1,\chi_1)$, if there exists a smooth homotopy $G \colon M \x\I \to U$ from $g_0$ to $g_1$ such that\footnote{the right hand side of \eqref{eqn:twzcsequiv2} differs from the expression appearing in \cite[p.~27]{TWZ} by a factor of $-1$, as the original expression fails to preserve the curvature map described below.}
\begin{equation}
\label{eqn:twzcsequiv2}
\sum_{j=0}^\infty  \frac{-j!}{(2j)!} \left(-\frac{1}{2\pi i}\right)^{j+1} \int_0^1 \tr\big(g_t^{-1} \d_t g_t \cdot (g_t^{-1}dg_t)^{2j} \big)\,dt = \chi_0 -\chi_1 \mod \mbox{exact},
\end{equation}
with $g_t := G(\cdot,t) \colon M\to U$ as above.
The set of equivalence classes
\[
\dL(M) := S(M) /\!\sim_\CS
\]
forms an abelian group under the $+$ operation on $S(M)$, where the inverse of $[g,\chi]$ is $[g^{-1},-\chi]$ and the identity is $[\id,0]$.
The assignment $M \mapsto \dL(M)$ defines a contravariant functor on the category of compact manifolds with corners to the category of abelian groups that acts on morphisms by pullback.
\end{definition}
\vspace{11pt}
\begin{proposition}[\cite{TWZ}]
\label{prop:tradlermodels}
The natural map $\dl \to \dL$ given by
\[
[g]\longmapsto [g,0]
\]
is an isomorphism respecting the differential extension data so that, in particular, it is an isomorphism of differential extensions.
\end{proposition}

This survey of the $\TWZ$ extension is concluded by writing down the extension data of $\dL$:

{\textbf{Curvature.}}
The curvature map $F \colon \dL(M) \to \Omega^{odd}_{d=0}(M)$ is given by
\[
F \colon [g,\chi] \longmapsto \sum_{j=0}^\infty \frac{-j!}{(2j+1)!} \left(-\frac{1}{2\pi i}\right)^{j+1} \tr\left((g^{-1}dg)^{2j+1}\right) + d\chi.
\]
Notice that if $g_t \colon M\x\I\to U$ is a smooth homotopy from $g$ to $h$ then (using the expression appearing in \cite[p.~8]{TWZ})
\begin{multline*}
d\bigg(\sum_{j=0}^\infty  \frac{-j!}{(2j)!} \left(-\frac{1}{2\pi i}\right)^{j+1} \int_0^1 \tr\big(g_t^{-1} \d_t g_t \cdot (g_t^{-1}dg_t)^{2j} \big)\,dt\bigg)\\
%%%
= \sum_{j=0}^\infty \frac{-j!}{(2j+1)!} \left(-\frac{1}{2\pi i}\right)^{j+1} \tr\left((h^{-1}dh)^{2j+1}\right)\\
%%%
 - \sum_{j=0}^\infty \frac{-j!}{(2j+1)!} \left(-\frac{1}{2\pi i}\right)^{j+1} \tr\left((g^{-1}dg)^{2j+1}\right)\!,
\end{multline*}
so applying $d$ to \eqref{eqn:twzcsequiv2} shows that $F$ is well-defined on $\dL$.

{\textbf{Underlying class.}}
The underlying class map $c \colon \dL(M)  \to K^{-1}(M)$ sends
\[
c \colon [g,\chi] \longmapsto [g],
\]
where on the left hand side the equivalence class is modulo $\CS$-equivalence and on the right hand side the equivalence class is modulo smooth homotopy.

{\textbf{Action of forms.}}
The action of forms map $b \colon \Omega^{even}(M)/\im\,d \to \dL(M)$ is given simply by
\[
b \colon \{\chi\} \longmapsto [\id,\chi]
\]
with $\id \colon M \to U$ the constant map to the identity.

One constructs a natural isomorphism $\dK \to \dL$ by first associating to each  $\bm{\sE}$ in $\struct(M)$ an element of $\dL(M)$ and then passing to the group completion.
Take some $\bm{\sE} = (\sE,[\Delta,\phi]) \in \struct(M)$ and fix a representative $(\Delta,\phi)\in[\Delta,\phi]$.
Recall from Proposition \ref{prop:canform} and Lemma \ref{lemma:higgshol} that $\hol_\phi$ is a smooth classifying map for $\sE$ and that
\[
s(\Delta,\phi)  = \hol_\phi^\ast\tau + d\chi_{\Delta,\phi}
\]
for some even form $\chi_{\Delta,\phi}$ determined entirely by $\Delta$ and $\phi$.
\vspace{11pt}
\begin{proposition}
\label{prop:etastruct}
The map $\eta \colon \struct(M) \to \dL(M)$ that acts by
\[
\eta \colon \bm{\sE} \longmapsto [\hol_\phi,\chi_{\Delta,\phi}].
\]
$\eta$ is a well-defined semi-group homomorphism, viewing $\hol_\phi$ as a map $M \to U$ in the obvious way.
\end{proposition}
\begin{proof}
Fix a Hermitian $\Omega$ vector bundle $\sE\to M$.
Suppose that the pairs $(\Delta,\phi)$ and $(\Delta',\phi')$ on $\sE$ are both compatible with the Hermitian structure and define the same string datum, i.e. for any smooth path $\gamma(t) = (\Delta^\gamma_t,\phi^\gamma_t)$ of based module connections and Higgs fields compatible with the Hermitian structure that joins $(\Delta,\phi)$ to $(\Delta',\phi')$ the total string potential $S(\gamma)$ is exact.

Invoking Proposition \ref{prop:canform}, on $M\x \I$ one has $s(\Delta^\gamma, \phi^\gamma) = \hol_{\phi^\gamma}^\ast \tau + d\chi_{\Delta^\gamma,\phi^\gamma}$ where by Corollary \ref{cor:canpullback} the even-degree form $\chi_{\Delta^\gamma,\phi^\gamma}$ pulls back to $\chi_0:= \chi_{\Delta,\phi}$ and $\chi_1:=\chi_{\Delta',\phi'}$ on $M\x\{0\}$ and $M\x\{1\}$ respectively.
By \eqref{eqn:altstringpot}
\[
S(\gamma) = \int_0^1 \varsigma_t^\ast \imath_{\d_t} s(\Delta^\gamma,\phi^\gamma)\,dt 
= \int_0^1  \varsigma_t^\ast \imath_{\d_t} \big( \hol_{\phi^\gamma}^\ast \tau + d\chi_{\Delta^\gamma,\phi^\gamma}\big)\,dt.
\]
Writing $g_t = \hol_{\phi^\gamma}(\cdot,t)$, similarly to the proof of Theorem \ref{theorem:inversesexist} one has
\[
\int_0^1  \varsigma_t^\ast \imath_{\d_t}(\hol_{\phi^\gamma}^\ast \tau )\,dt = \sum_{j=0}^\infty \frac{-j!}{(2j)!} \left(-\frac{1}{2\pi i}\right)^{j+1} \int_0^1 \tr\big(g_t^{-1} \d_tg_t \cdot (g_t^{-1}dg_t)^{2j} \big)\,dt.
\]
Moreover, using Cartan's Magic Formula and Stokes' Theorem,
\[
\int_0^1 \varsigma_t^\ast \imath_{\d_t} d \chi_{\Delta^\gamma,\phi^\gamma}\,dt= \int_0^1 \varsigma_t^\ast \left( \frac{d}{dt} \chi_{\Delta^\gamma,\phi^\gamma}  - d( \imath_{\d_t}^\ast \chi_{\Delta^\gamma,\phi^\gamma}) \right) dt = \chi_1 - \chi_0 + \mbox{exact}.
\]
$S(\gamma)$ is exact since $(\Delta,\phi)$ and $(\Delta',\phi')$ belong to the same string datum and so
\[
\sum_{j=0}^\infty \frac{-j!}{(2j)!} \left(-\frac{1}{2\pi i}\right)^{j+1}  \int_0^1 \tr\big(g_t^{-1} \d_t g_t \cdot (g_t^{-1}dg_t)^{2j} \big)\,dt =\chi_0 - \chi_1  + \mbox{exact}.
\]
Therefore $[\hol_\phi,\chi_{\Delta,\phi}] = [\hol_{\phi'},\chi_{\Delta',\phi'}]$ so that $\eta$ is well-defined.

Note that $\hol_{\phi \oplus \phi'} = \hol_\phi \oplus \hol_{\phi'}$ and, by comparing the proofs of Lemma \ref{lemma:stringprop} and Proposition \ref{prop:canform}, one may show that $\chi_{\Delta\oplus\Delta',\phi\oplus\phi'} = \chi_{\Delta,\phi} + \chi_{\Delta',\phi'}$ so it follows that $\eta$ is a semi-group homomorphism.
\end{proof}

Passing to the group completion and taking the rank zero subgroup
\vspace{11pt}
\begin{theorem}
The semi-group homomorphism $\eta$ induces a group isomorphism
\[
\check \eta \colon \dK(M) \lo \dL(M)
\]
given explicitly by
\[
\bm{\sE} - \bm{\sF} \longmapsto [\hol_\phi,\chi_{\Delta,\phi}] - [\hol_{\phi'},\chi_{\Delta',\phi'}] = \big[\hol_\phi\oplus\hol_{\phi'}^{-1},\chi_{\Delta,\phi} - \chi_{\Delta',\phi'} \big],
\]
where $\bm{\sE} = (\sE,[\Delta,\phi])$ and $\bm{\sF} = (\sF,[\Delta',\phi'])$.
\end{theorem}
\begin{proof}
It suffices to show that $\check \eta$ is a bijection since Proposition \ref{prop:etastruct} above guarantees that $\check \eta$ is a homomorphism.

To see that $\check\eta$ is injective suppose that $\check \eta(\bm{\sE} - \underline{\bm{L\CC}}^n) = [\id,0]$, where $\bm{\sE} = g^\ast \bm{\sE}(n)$, say.
Since $\eta(\underline{\bm{L\CC}}^n) = [\id,0]$ and $\eta(\bm{\sE}) = [g,0]$, this implies that $(g, 0) \sim_\CS (\id,0)$. 
Consequently there is a smooth homotopy $G\colon M\x\I \to U$ from $\id $ to $g$ such that
\[
\sum_{j=0}^\infty \frac{-j!}{(2j)!} \left(-\frac{1}{2\pi i}\right)^{j+1} \int_0^1 \tr\big(g_t^{-1} \d_t g_t \cdot (g_t^{-1}dg_t)^{2j} \big)\,dt  = \int_0^1\varsigma_t^\ast\imath_{\d_t} G^\ast \tau\,dt = 0\mod \mbox{exact},
\] 
with $g_t := G(\cdot,t)$ as usual.
Since $M\x\I$ is compact the image of $G$ is contained in $U(n+k)$ for some $k$ (Proposition \ref{prop:compactimage}) and so $G$ may be viewed as a smooth homotopy from $g \oplus \id\colon M \to U(n+k)$ to $\id\colon M\to U(n+k)$.
By \eqref{eqn:homotopicstringpot2}, following a similar argument to the proof of Theorem \ref{theorem:inversesexist}, one has
\begin{multline*}
\cS\big(g^\ast\Delta(n)\oplus\delta,g^\ast\phi (n)\oplus\d;\delta,\d\big)\vphantom{\sum_j^1} \\
%%%
= \sum_{j=0}^\infty \frac{-j!}{(2j)!} \left(-\frac{1}{2\pi i}\right)^{j+1} \int_0^1 \tr\big(g_t^{-1} \d_t g_t \cdot (g_t^{-1}dg_t)^{2j} \big)\,dt\mod \mbox{exact}\\ 
%%%
= \int_0^1\varsigma_t^\ast\imath_{\d_t} G^\ast \tau \,dt= 0\mod \mbox{exact}.
\end{multline*}
This gives that $g^\ast\bm{\sE}(n) \oplus \underline{\bm{L\CC}}^{k} \cong (g\oplus\id)^\ast\bm{\sE}(n+k) \cong\underline{\bm{L\CC}}^{n+k}$ as structured $\Omega$ vector bundles and hence $\bm{\sE} - \underline{\bm{L\CC}}^n = 0$ in $\dK(M)$,  completing the proof of injectivity.

Proposition \ref{prop:tradlermodels} implies that every $(g,\chi) \in S(M)$ is $\CS$-equivalent to $(h,0)$ for some smooth $h\colon M \to U(n)$.
But then $\check\eta(h^\ast\bm{\sE}(n) - \underline{\bm{L\CC}}^n) = [h,0] = [g,\chi]$, giving surjectivity.
\end{proof}

Notice that by virtue of its construction the isomorphism $\check\eta \colon \dK \to \dL$ is a natural transformation of functors and also that $\check\eta$ respects the curvature and underlying class morphisms:

{\textbf{Curvature.}}
Take $\bm{\sE} - \bm{\sF} \in \dK(M)$ as above and observe that
\[
F\circ\check\eta\left( \bm{\sE} - \bm{\sF}\right) = \hol_\phi^\ast\tau + d\chi_{\Delta,\phi}- \hol_{\phi'}^\ast\tau  - d\chi_{\Delta',\phi'}.
\]
But $s(\Delta,\phi) = \hol_\phi^\ast\tau + d\chi_{\Delta,\phi}$ and similarly for $s(\Delta',\phi')$ so that $F\circ \check\eta = \check s$.

{\textbf{Underlying class.}}
For $\bm{\sE} - \bm{\sF} \in \dK(M)$ the composition $c\circ \check\eta\left( \bm{\sE} - \bm{\sF}\right) = [\hol_\phi\oplus\hol_{\phi'}^{-1}]$ is precisely the underlying class map of $\dK(M)$ mapping into the homotopy-theoretic model $K^{-1}(M) = [M,U]$ so that $c\circ\check\eta = \delta$ (cf. Remark \ref{remark:unversion}).

As mentioned above, one of the reasons for relating the $\Omega$ model to the $\TWZ$ extension is to easily obtain an action of forms map for $\dK$, which one now defines as the map
\[
a \colon \Omega^{even}(M) /\im\, d \lo \dK(M)
\]
given by $a(\{\omega\}) := \check\eta^{-1}([\id,\omega]) = \check\eta^{-1} \circ b(\{\omega\})$, with $b$ the action of forms map for $\dL$.
Using the properties of $b$, it follows that the sequence
\begin{equation}
\label{eqn:actionofforms}
K^0(M) \xrightarrow{\;\;ch\;\;} \Omega^{even}(M)/\im\,d \xrightarrow{\;\;a\;\;} \dK(M) \xrightarrow{\;\;\delta\;\;} \cK(M) \lo 0
\end{equation}
is exact.
Since $\check\eta$ respects the curvature maps it also follows that $\check s\circ a = F \circ b = d$ and so
\vspace{11pt}
\begin{theorem}
The functor $M \mapsto \dK(M)$ together with the maps $\check s$, $\delta$ and $a$ define a differential extension of odd $K$-theory.
\end{theorem}
\vspace{11pt}
\begin{proposition}
The natural map $\check \eta \colon \dK \to \dL$ is an isomorphism of differential extensions of odd $K$-theory.
\end{proposition}
 \vspace{11pt}
\begin{example}[$\dK$ on a point]
\label{example:odddkpoint}
It is clear that $\Omega^{even}(\ast) =\Map(\ast,\RR) =   \RR$ and also that $ch(K^0(\ast)) = \ZZ$.
Since $K^{-1}(\ast) = 0$, the exact sequence \eqref{eqn:actionofforms} implies that $\dK(\ast) \cong \RR/\ZZ$.
\end{example}

This section is concluded by a useful interpretation of the action of forms map on $\dK$ that is required in the next section.
\vspace{11pt}
\begin{remark}
\label{remark:oactionoforms}
For any $\{\omega\} \in \Omega^{even}(M)/\im\,d$, observe that by Proposition \ref{prop:tradlermodels} one has $b(\{\omega\}) = [\id,\omega] = [g,0]$ for some smooth $g \colon M\to U$.
Thus there is a smooth homotopy $G \colon M\x\I \to U$ from $g$ to $\id$ such that
\[
\int_0^1\varsigma_t^\ast\imath_{\d_t} G^\ast \tau \,dt= \sum_{j=0}^\infty \frac{-j!}{(2j)!} \left(-\frac{1}{2\pi i}\right)^{j+1} \int_0^1 \tr\big(g_t^{-1} \d_t g_t \cdot (g_t^{-1}dg_t)^{2j} \big)\,dt  = - \omega \mod \mbox{exact},
\]
where $g_t := G(\cdot,t)$.
By Proposition \ref{prop:compactimage}, the image of $G$ is contained inside some $U(n)$ and so by \eqref{eqn:homotopicstringpot2} one has
\[
\cS(\delta,\d;g^\ast\Delta(n)\oplus\delta,g^\ast\phi(n)\oplus\d) = \omega \mod \mbox{exact}
\]
on $(g\oplus\id)^\ast\bm{\sE}(n)$.
Observe that $a(\{\omega\}) = (g\oplus\id)^\ast\bm{\sE}(n) - \underline{\bm{L\CC}}^n$ since $\check\eta$ acting on the right hand side gives $[g,0]$ and $\check\eta$ is an isomorphism.

Therefore the action of forms map $a$ sends the class $\{\omega\}$ to the virtual structured $\Omega$ vector bundle $\bm{\sE} - \underline{\bm{L\CC}}^n$, where $\bm{\sE} = (\sE,[\Delta,\phi])\in \struct(M)$ is such that
\begin{itemize}
\item
$\rank \sE = n$ and $\sE$ is trivial; and

\item
for any chosen representative $(\Delta,\phi) \in [\Delta,\phi]$ of the string datum there exists a smooth path $\gamma$ of connections from  the trivial pair $(\delta,\d)$ (in some trivialisation of $\sE$) to $(\Delta,\phi)$ such that the total string potential
\[
S(\gamma) = \omega \mod\mbox{exact}.
\] 
\end{itemize}
Compare this to Remark \ref{remark:csactionoforms}, which gives a similar statement for the action of forms map in the Simons-Sullivan model.
\end{remark}

%%%%%%%%%%%%%%%%%%%%%%%%%%%%%%%%%%%%%%%%%%%%%%%%%%%%%%%%%%%%%%%%%%%%%%%%%%%%%%%%%%%%%%%%%
%%%%%%%%%%%%%%%%%%%%%%%%%%%%%%%%%%%%%%%%%%%%%%%%%%%%%%%%%%%%%%%%%%%%%%%%%%%%%%%%%%%%%%%%%
%%%%%%%%%%%%%%%%%%%%%%%%%%%%%%%%%%%%%%%%%%%%%%%%%%%%%%%%%%%%%%%%%%%%%%%%%%%%%%%%%%%%%%%%%

\subsection{Odd differential $K$-theory}
\label{S:structcal}
It has been shown so far that $\dK$ is a differential extension of odd $K$-theory.
However, as remarked previously one can show that in the absence of a multplicative structure there is an infinite family of non-isomorphic differential extensions of odd $K$-theory \cite{BS2}.
In this section, the caloron transform is adapted to structured vector bundles in order to prove that $\dK(M)$ does indeed give odd differential $K$-theory, that is $\dK(M)$ is isomorphic to the odd part of any differential extension of $K$-theory with $\sone$-integration.

Fix a model $\check K^\bullet$ of differential $K$-theory, i.e.~a differential extension of $K^\bullet$, with
\begin{itemize}
\item
{curvature} morphism $\check{ch} \colon \check K^\bullet(M) \to \Omega^\bullet_{d=0}(M)$;

\item
{underlying class} morphism $\check{\delta} \colon \check K^\bullet(M) \to K^\bullet(M)$;

\item
{action of forms} morphism $\check{a} \colon \Omega^{\bullet-1}(M)/\im\,d \to \check K^\bullet(M)$; and

\item
{$\sone$-integration} $\widehat{\int_\sone} \colon \check K^{\bullet}(M\x\sone) \to \check K^{\bullet-1}(M)$.
\end{itemize}
According to the uniqueness theorem of Bunke and Schick (Theorem \ref{theorem:bunkeschickk}) there is a unique natural isomorphism $\Phi_0 \colon \edK(M) \to \check K^0(M)$ that respects all of the structure (except, of course, the $\sone$-integration), where $\edK$ is even differential $K$-theory \`{a} la Simons-Sullivan.

Using the isomorphism $\Phi_0$ together with the caloron correspondence, one now defines an isomorphism
\[
\Phi_1 \colon \dK(M) \lo \check K^{-1}(M).
\]
For any $\bm{\sE} - \bm{\sF} \in \dK(M)$, where $\bm{\sE} = (\sE,[\Delta,\phi])$ and $\bm{\sF} = (\sF,[\Delta',\phi'])$ say, set
\begin{equation}
\label{eqn:oddkhom}
\Phi_1 \big(\bm{\sE} - \bm{\sF} \big) := \widehat{\int_\sone} \Phi_0 \big((E,[\nabla]) - (F,[\nabla']) \big).
\end{equation}
In this expression, $(E,\nabla)$ and $(F,\nabla') $ are respectively the caloron transforms of $(\sE,\Delta,\phi)$ and $(\sF,\Delta',\phi')$ for some chosen representatives $(\Delta,\phi)\in[\Delta,\phi]$ and $(\Delta',\phi')\in[\Delta',\phi']$.
The definition of $\Phi_1$ appears to depend explicitly on these choices and in general there is no canonical choice.
Nevertheless
\vspace{11pt}
\begin{proposition}
\label{prop:welldefinedkmap}
The map
\[
\Phi_1 \colon \dK(M) \lo \check K^{-1}(M)
\]
is a well-defined group homomorphism.
\end{proposition}
\begin{proof}
This proof makes use of the homotopy formula (Theorem \ref{theorem:homotopyformula}), which describes how differential extensions behave under smooth homotopies.
Let $\bm{\sE} = (\sE,[\Delta,\phi])$ and $\bm{\sF} = (\sF,[\Delta',\phi'])$ be structured $\Omega$ vector bundles as above and take any two choices of pairs of representatives
\[
(\Delta_0,\phi_0), (\Delta_1,\phi_1) \in [\Delta,\phi] \;\;\mbox{ and }\;\; (\Delta'_0,\phi'_0), (\Delta'_1,\phi'_1) \in [\Delta',\phi'].
\]
Write $\nabla_0, \nabla_1$ and $\nabla'_0,\nabla'_1$ for the corresponding connections on the caloron transforms $E = \cV(\sE)$ and $F = \cV(\sF)$.
Showing $\Phi_1(\bm{\sE} - \bm{\sF})$ is well-defined is thus equivalent to showing
\[
\widehat{\int_\sone} \Phi_0 \big((E,[\nabla_0]) - (F,[\nabla'_0]) \big) = \widehat{\int_\sone} \Phi_0 \big((E,[\nabla_1]) - (F,[\nabla'_1]) \big).
\]
Define smooth paths $\gamma \colon \I \to \cM_\sE \x\cH_\sE$ and $\gamma' \colon \I \to \cM_\sF \x\cH_\sF$ by
\[
\gamma(t) := \big(t\Delta_1+(1-t)\Delta_0,t\phi_1+(1-t)\phi_0\big)
\]
with $\gamma'$ defined similarly.
Then by definition of string data the string potentials $S(\gamma)$ and $S(\gamma')$ are exact.
As in the principal bundle case, the caloron transform gives a smooth path of (framed) connections $\cV\gamma$ from $\nabla_0$ to $\nabla_1$ on $E$  and a  smooth path of (framed) connections $\cV\gamma'$ from $\nabla'_0$ to $\nabla'_1$ on $F$.

Write $\nabla_t$ and $\nabla'_t$ for the connections on $E\x\I$ and $F\x\I$ corresponding to the smooth paths $\cV\gamma$ and $\cV\gamma'$.
By the naturality of $\Phi_0$, the homotopy formula (Theorem \ref{theorem:homotopyformula}) applied to the even differential $K$-class $\widehat x:=\Phi_0( (E\x\I,[\nabla_t]) - (F\x\I,[\nabla'_t]))$ gives
\begin{equation}
\label{eqn:homotopyformula}
\Phi_0 \big((E,[\nabla_1]) - (F,[\nabla'_1]) \big)- \Phi_0\big((E,[\nabla_0]) - (F,[\nabla'_0]) \big)= \check a\bigg(\int_0^1 \varsigma_t^\ast \imath_{\d_t} \check{ch} (\widehat x)\,dt\bigg).
\end{equation}
But since $\Phi_0$ respects the curvature maps, one has
\[
\check{ch}(\widehat x) =  \check{\Ch}\Big((E\x\I,[\nabla_t]) - (F\x\I,[\nabla'_t])\Big)=\Ch(\nabla_t) - \Ch(\nabla'_t)
\]
and hence
\[
\int_0^1 \varsigma_t^\ast \imath_{\d_t} \check{ch} (\widehat x)\,dt = \int_0^1  \varsigma_t^\ast \imath_{\d_t} \big( \Ch(\nabla_t) - \Ch(\nabla'_t)\big)\,dt = \CS(\cV\gamma) - \CS(\cV\gamma').
\]
The second equality uses \cite[(1.5)]{SSvec}, which is the Chern-Simons version of the expression \eqref{eqn:alternateanti}, i.e.
\[
\CS(\cV\gamma) = \int_0^1 \varsigma_t^\ast \imath_{\d_t} \Ch(\nabla_t)\,dt.
\]
In this setting the formula \eqref{eqn:csanti} relating string potentials and Chern-Simons forms gives
\[
S(\gamma) = \widehat{\int_\sone} \CS(\cV\gamma) \;\;\mbox{ and }\;\;
S(\gamma') = \widehat{\int_\sone} \CS(\cV\gamma')
\]
so that
\[
\widehat{\int_\sone}\int_0^1 \varsigma_t^\ast \imath_{\d_t} \check{ch} (\widehat x)\,dt  = S(\gamma) - S(\gamma') 
\]
is exact.
Recall that the $\sone$-integration commutes with $\check a$ and so
\[
\widehat{\int_\sone}  \check a\bigg(\int_0^1 \varsigma_t^\ast \imath_{\d_t} \check{ch} (\widehat x)\,dt\bigg) = \check a \big( S(\gamma) - S(\gamma')  \big) = 0,
\]
since $\check a$ vanishes on exact forms by definition.
Thus applying the $\sone$-integration map to \eqref{eqn:homotopyformula} yields
\[
\widehat{\int_\sone} \Phi_0 \big((E,[\nabla_0]) - (F,[\nabla'_0]) \big) = \widehat{\int_\sone} \Phi_0 \big((E,[\nabla_1]) - (F,[\nabla'_1]) \big),
\]
which shows that $\Phi_1$ is well-defined.
It is clear from the definition and the properties of $\cV$ and $\widehat{\int_\sone}$ that $\Phi_1$ is a homomorphism.
\end{proof}

One proceeds by showing that $\Phi_1$ respects the curvature, underlying class and action of forms maps:

{\textbf{Curvature.}}
To see that $\Phi_1$ respects the curvature maps, i.e. $\check{ch} \circ \Phi_1 = \check s$, notice that for $\bm{\sE} - \bm{\sF} \in \dK(M)$ as above,
\[
\check s \big( \bm{\sE} - \bm{\sF}  \big)= s(\Delta,\phi) - s(\Delta',\phi') = \widehat{\int_\sone} \big( \Ch(\nabla) - \Ch(\nabla') \big)
\]
using \eqref{eqn:stringform}.
However, since the $\sone$-integration commutes with the curvature on $\check K^\bullet$,
\begin{multline*}
\check{ch} \circ \Phi_1 \big( \bm{\sE} - \bm{\sF}  \big) = \widehat{\int_\sone} \check{ch} \left[\Phi_0\big((E,[\nabla])-(F,[\nabla'])\big)\right]\\
%%%
= \widehat{\int_\sone} \check{\Ch} \big((E,[\nabla])-(F,[\nabla'])\big) = \widehat{\int_\sone} \big( \Ch(\nabla) - \Ch(\nabla') \big),
\end{multline*}
so that $\check{ch} \circ \Phi_1 = \check s$ as required.

{\textbf{Underlying class.}}
\label{page:int}
To see that $\Phi_1$ respects the underlying class maps, one must show that $\check\delta \circ \Phi_1 = \delta$.
Now, the map $\dK(M) \to K^0(M\x\sone)$
\[
\bm{\sE} - \bm{\sF} \longmapsto \delta\big((E,[\nabla]) - (F,[\nabla']) \big) = E - F
\]
is well-defined and coincides with the composition $q^\ast \circ \delta$, where $q \colon M\x\sone \to \Sigma(M^+)$ is the projection and $\delta\colon \dK(M)\to K^{-1}(M)$ is as above.
Note also that
\[
\check \delta \circ \Phi_1 \big(\bm{\sE} - \bm{\sF}\big) = \check\delta\widehat{\int_\sone} \Phi_0 \big((E,[\nabla]) - (F,[\nabla']) \big) = \widehat{\int_\sone} \delta \big((E,[\nabla]) - (F,[\nabla']) \big)  =\widehat{\int_\sone} (E-F)
\]
so it suffices to show $\widehat{\int_\sone} \circ \,q^\ast = \id$ acting on the image of $\delta$.
Recall (cf. \cite{BS2}) that if $\pr \colon M\x\sone \to M$ is the projection and $\imath \colon M\to M\x\sone$ is the embedding $m \mapsto (m,0)$ then there is a splitting $K^\bullet(M\x\sone) \cong \im\pr^\ast\oplus \ker \imath^\ast$ since $\pr\circ\, \imath = \id_M$.
Moreover, the quotient map $q\colon M\x\sone \to \Sigma(M^+)$ induces an isomorphism $q^\ast K(\Sigma(M^+))\to \ker\imath^\ast $.
The $\sone$-integration on $K^0$ is defined as the composition of the projection onto $\ker\imath^\ast$ with $(q^\ast)^{-1}$, from which it follows that $\widehat{\int_\sone} \circ q^\ast = \id$ on the image of $\delta$ since $q^\ast(\im\,\delta) \subset \ker\imath^\ast$.
Therefore $\check \delta\circ\Phi_1  =\delta$ as required.

{\textbf{Action of forms.}}
One wishes to show that $\Phi_1 \circ a = \check a$.
To do this, first recall Remarks \ref{remark:csactionoforms} and \ref{remark:oactionoforms} characterising the action of forms maps in the Simons-Sullivan and the $\Omega$ models.
For $\{\omega\} \in \Omega^{even}(M)/\im\,d$ the element $a(\{\omega\}) \in \dK(M)$ is defined as $\bm{\sE} - \underline{\bm{L\CC}}^n$ where $\bm{\sE} = (\sE,[\Delta,\phi])$ is such that $\rank \sE = n$, $\sE$ is trivial and $\cS(\delta,\d;\Delta,\phi) = \omega$ mod exact for any choice $(\Delta,\phi) \in[\Delta,\phi]$.
Fixing such a choice, there is a smooth path $\gamma$ joining $(\delta,\d)$ to $(\Delta,\phi)$ on $\sE$ such that
\[
S(\gamma) = \omega\mod\mbox{exact}.
\]
Writing $(E,\nabla)$ for the caloron transform of $(\sE,\Delta,\phi)$, observe that
\[
\Phi_1 \circ a(\{\omega\}) = \widehat{\int_\sone} \Phi_0 \big((E,[\nabla]) - \underline{\bm{n}}\big).
\]
But, following a similar argument one also has
\[
(E,[\nabla]) - \underline{\bm{n}}= a \big(\big\{\tfrac{1}{2\pi} \omega \wedge d\theta + \alpha\big\} \big)
\]
for some $\alpha$ in the kernel of $\widehat{\int_\sone}$ since by \eqref{eqn:csanti}
\[
\widehat{\int_\sone} \CS(\cV\gamma) = S(\gamma) = \widehat{\int_\sone} \tfrac{1}{2\pi} \omega\wedge d\theta. 
\]
Hence
\[
\widehat{\int_\sone} \Phi_0 \big((E,[\nabla]) - \underline{\bm{n}}\big) = \widehat{\int_\sone} \Phi_0\, a  \big(\big\{\tfrac{1}{2\pi} \omega \wedge d\theta + \alpha\big\} \big) = \check a \widehat{\int_\sone}\big\{\tfrac{1}{2\pi} \omega \wedge d\theta + \alpha\big\} = \check a (\{\omega\}),
\]
giving $\Phi_1 \circ a =\check a$ as required.

It now remains only to show that $\Phi_1$ is an isomorphism:
\vspace{11pt}
\begin{theorem}
\label{theorem:odddifferentialk}
The homomorphism
\[
\Phi_1 \colon \dK(M) \lo K^{-1}(M)
\]
is an isomorphism that respects all the differential extension structure.
\end{theorem}
\begin{proof}
The exact sequences \eqref{eqn:actionofforms} and \eqref{eqn:diffextseq} imply exact sequences
\[
0\lo \Omega^{even}(M)/\im\,ch \xrightarrow{\;\;a\;\;} \dK(M) \xrightarrow{\;\;\delta\;\;} K^{-1} \lo 0
\]
and
\[
0\lo \Omega^{even}(M)/\im\,ch \xrightarrow{\;\;\check a\;\;} \check K^{-1}(M) \xrightarrow{\;\;\check\delta\;\;} K^{-1} \lo 0.
\]
One thus has the commuting diagram
\[
\xy
(0,25)*+{0}="1";
(35,25)*+{\Omega^{even}(M)/\im\,ch}="2";
(70,25)*+{\dK(M)}="3";
(105,25)*+{K^{-1}(M)}="4";
(140,25)*+{0}="5";
(0,0)*+{0}="7";
(35,0)*+{\Omega^{even}(M)/\im\,ch}="8";
(70,0)*+{\check K^{-1}(M)}="9";
(105,0)*+{K^{-1}(M)}="10";
(140,0)*+{0}="11";
{\ar^{} "1";"2"};
{\ar^{a} "2";"3"};
{\ar^{\delta} "3";"4"};
{\ar^{} "4";"5"};
{\ar^{} "1";"7"};
{\ar^{} "7";"8"};
{\ar^{\check a} "8";"9"};
{\ar^{\check \delta} "9";"10"};
{\ar^{} "10";"11"};
{\ar^{} "2";"8"};
{\ar^{\Phi_1} "3";"9"};
{\ar^{} "4";"10"};
{\ar^{} "5";"11"};
\endxy
\]
where the rows are exact and all unlabelled vertical arrows are the identity.
The five-lemma applied to this diagram gives that $\Phi_1$ is an isomorphism.
\end{proof}

At last, one has that the $\Omega$ model $\dK$ is isomorphic to the odd part of any model of differential $K$-theory.
Observe that the isomorphism of Theorem \ref{theorem:odddifferentialk} does not depend on the choice of model for differential $K$-theory used to define it; take any two models $\check K^\bullet$ and $\check L^\bullet$  for differential $K$-theory, then by Theorem \ref{theorem:bunkeschickk} there are unique natural isomorphisms $\Psi \colon \check K^\bullet \to \check L^\bullet$, $\Phi_0 \colon \edK \to \check K^0$ and $\Phi'_0 \colon \edK \to \check L^0$  respecting all of the structure and, by uniqueness, $\Phi'_0 = \Psi\circ\Phi_0$.

Define maps $\Phi_1 \colon \dK \to \check K^{-1}$ and $\Phi'_1 \colon \dK \to \check L^{-1}$ as per \eqref{eqn:oddkhom}, which are isomorphisms respecting all of the structure by Theorem \ref{theorem:odddifferentialk}.
For any $\bm{\sE} - \bm{\sF} \in \dK(M)$ one then has
\begin{multline*}
\Phi'_1 \big(\bm{\sE} - \bm{\sF} \big) := \widehat{\int_\sone} \Phi'_0 \big((E,[\nabla]) - (F,[\nabla']) \big) \\= \widehat{\int_\sone} \Psi\circ \Phi_0 \big((E,[\nabla]) - (F,[\nabla']) \big) = \Psi \circ \Phi_1 \big(\bm{\sE} - \bm{\sF} \big)
\end{multline*}
since $\Psi$ commutes with the $\sone$-integration maps.
Therefore $\Phi_1' = \Psi\circ\Phi_1$, in particular the isomorphism $\Phi_1 \colon \dK \to \check K^{-1}$ defined by \eqref{eqn:oddkhom} does not depend on the choice of model of differential $K$-theory.
Observe also that $\Phi_1$ is a natural isomorphism, which follows from the naturality of $\Phi_0$, the fact that $\widehat{\int_\sone} (f\x\id)^\ast = f^\ast \widehat{\int_\sone}$ and the functoriality of the caloron transform.
Thus
\vspace{11pt}
\begin{theorem}
\label{theorem:wehaveoddk}
$\dK$ is isomorphic to the odd part of any model of differential $K$-theory via a canonical natural isomorphism respecting all of the structure, in particular it defines odd differential $K$-theory.
\end{theorem}
\vspace{11pt}
\begin{remark}
The caloron correspondence plays a crucial role in the proof of this last result; it features explicitly in \eqref{eqn:oddkhom} as a sort of `differential suspension map'.
\end{remark}
\vspace{11pt}
\begin{remark}
In Theorem \ref{theorem:wehaveoddk}, the isomorphism is \emph{canonical} rather than \emph{unique}: in order to establish uniqueness using the results of Bunke and Schick, one requires either a multiplicative structure or an $\sone$-integration.

Define the \emph{structured differential extension} of $K$-theory by setting
\[
\check{\mathcal{K}}^{n}(M) :=
\begin{cases}
\edK(M) &\mbox{ if $n$ is even}\\
\dK(M) &\mbox{ if $n$ is odd},
\end{cases}
\]
noting that the properties of $\edK(M)$ and $\dK(M)$ imply that this is indeed a differential extension of $K^\bullet$.
Fixing a model $\check K^\bullet$ of differential $K$-theory as above, denote by
\[
\Phi \colon \check{\mathcal{K}}^\bullet \lo \check K^\bullet
\]
the natural isomorphism that acts in even degree by $\Phi_0$ and in odd degree by $\Phi_1$, observing that this respects the curvature, underlying class and action of forms maps.
Define
\begin{equation}
\label{eqn:intoperation}
\widehat{\int_\sone} \colon \check{\mathcal{K}}^\bullet(M\x\sone) \lo \check{\mathcal{K}}^{\bullet-1}(M)
\end{equation}
by requiring that the diagram
\[
\xy
(0,25)*+{\check{\mathcal{K}}^\bullet(M\x\sone)}="1";
(50,25)*+{\check{\mathcal{K}}^{\bullet-1}(M)}="2";
(0,0)*+{\check K^\bullet(M\x\sone)}="3";
(50,0)*+{\check K^{\bullet-1}(M)}="4";
{\ar^{\widehat{\int_\sone}} "1";"2"};
{\ar^{\Phi} "2";"4"};
{\ar^{\widehat{\int_\sone}} "3";"4"};
{\ar^{\Phi} "1";"3"};
\endxy
\]
commute, i.e. for $x \in \check{\mathcal{K}}^\bullet(M\x\sone)$
\[
\widehat{\int_\sone} x := \Phi^{-1} \widehat{\int_\sone} \Phi(x).
\]
This is readily seen to define an $\sone$-integration map, so  $\check{\mathcal{K}}^\bullet$ must be isomorphic to any other differential extension of $K$-theory with $\sone$-integration via a unique natural isomorphism preserving all the structure.
Since $\Phi$ is natural and by definition of the $\sone$-integration on $\check{\mathcal{K}}^\bullet$ preserves all of the additional structure it is this unique isomorphism.
Thus the map $\Phi_1$ defined in \eqref{eqn:oddkhom} is the unique isomorphism (in odd degree) that respects this particular $\sone$-integration.

Similarly, in the case where $\check K^\bullet$ is multiplicative one may also use the isomorphism  $\Phi \colon \check{\mathcal{K}}^\bullet \to \check K^\bullet$ to endow $\check{\mathcal{K}}^\bullet$ with a multiplicative structure by setting
\[
x\cdot y:= \Phi^{-1}\big( \Phi(x) \cdot \Phi(y) \big),
\]
so that $\check{\mathcal{K}}^\bullet$ becomes a multiplicative differential extension of $K$-theory.
By the above, neither this $\sone$-integration nor this multiplicative structure depend on the choice of $\check K^\bullet$.
\end{remark}

\chapter*{Conclusions and further work \label{ch:conclusion}}
\addcontentsline{toc}{chapter}{Conclusions and further work }

The focus of this thesis was the construction and applications of certain differential forms associated to loop group bundles.
These \emph{string potential forms} were constructed using the caloron correspondence developed in \cite{MS,MV,V} and play the role of Chern-Simons forms for loop group bundles.
Fixing a loop group bundle $\sQ \to M$, the total string potentials are forms that live on the total space $\sQ$ and depend on a fixed choice of connection and Higgs field for $\sQ$.
The relative string potentials are forms on the base $M$ that depend on a smooth path of connections and Higgs fields for $\sQ$.
In both cases the string potential forms describe some aspect of the dependence of the string forms---the loop group version of Chern-Weil forms---on a particular choice of connection and Higgs field.

In the case that $G$ is a compact, simple, simply-connected Lie group, it was shown that the total string potential forms can be used to recover the curving of the lifting bundle gerbe of the $LG$-bundle $\sQ \to M$ (Example \ref{example:liftingbundlegerbe}).
When $G$ is a simply-connected Lie group, it was shown that the total string potentials of the standard connection and Higgs field for the path fibration $PG \to G$ pull back to give generators for the cohomology of the based loop group $\Omega G$ (Example \ref{example:cohgenerators}).
Moreover, in a restricted setting the total string potential forms were used to define secondary characteristic classes for $\Omega U(n)$-bundles.
Due to their relationship with the Chern-Simons forms and motivated by this toy example, one expects to be able to define such secondary characteristic classes corresponding to each total string potential.
The existence of these classes is an interesting problem that merits further research. 

Motivated by the work of Simons and Sullivan in \cite{SSvec}, who used Chern-Simons forms to define a model $\edK$ for the even part of differential $K$-theory, this thesis applied the relative string potential forms to a similar purpose.
Together with the notion of $\Omega$ vector bundles developed in Chapter \ref{ch:four}, the relative string potential forms were used to construct a differential extension of the odd part of $K$-theory, referred to as the $\Omega$ model and denoted $\dK$.
A caloron correspondence for vector bundles was developed for $\Omega$ vector bundles and played a decisive role in showing that the $\Omega$ model defines odd differential $K$-theory; i.e.~that $\dK$ is isomorphic to the odd part of any model of differential $K$-theory.
An isomorphism of differential extensions was constructed between $\dK$ and an elementary differential extension of odd $K$-theory appearing in \cite{TWZ}, referred to as the $\TWZ$ extension.
The effect of this isomorphism is two-fold as it provides a `homotopy-theoretic' interpretation of the $\Omega$ model via classifying maps and also shows that the $\TWZ$ extension defines odd differential $K$-theory; a result not previously obtained.

There are a few unanswered questions that arise from the treatment of string potential forms and odd differential $K$-theory presented in this thesis.
Firstly, as remarked above, there ought to be a way of using total string potentials to define secondary characteristic classes for loop group bundles.
In the case that such a construction does indeed exist, these classes should be related to the differential Chern character mapping from $\dK$ into  differential cohomology with values in $\QQ$ (cf. \cite[Theorem 6.1]{BS1}).

Secondly, Simons and Sullivan showed that $\edK$ has a multiplicative structure induced by the tensor product operation $\otimes$ on vector bundles with connection.
Since there is a (honed) tensor product $\oast$ for $\Omega$ vector bundles equipped with connective data, it seems reasonable that there should be a way of constructing a multiplicative structure directly on $\check{\mathcal{K}}^\bullet:= \edK(M)\oplus\dK(M)$ using $\otimes$ and $\oast$.
It should also be possible to construct an $\sone$-integration map directly on $\check{\mathcal{K}}^\bullet$ and, additionally, one should be able to define the action of forms map on $\dK$ independently of the work of Tradler, Wilson and Zeinalian.

A final question that arises from the treatment presented in this thesis relates to the construction of the $\Omega$ model, which
uses structured $\Omega$ vector bundles that are classified by the canonical bundles $\bm{\sE}(n)$.
The canonical bundles are classifying for $\Omega$ vector bundles equipped with Higgs fields, though this is not necessarily true when connections are taken into account.
The reason for this restriction, as observed in Remark \ref{remark:whyrestrict}, is that it guarantees the existence of inverses, a fact that is frequently implicitly used throughout Chapter \ref{ch:five}.
It is interesting, therefore, to ask whether such inverses exist for all structured $\Omega$ vector bundles, in which case one should be able to define the $\Omega$ model in terms of all such bundles, rather than just pullbacks of the canonical bundles.

% add as many as you like

% Set up the naming conventions for equations in the appendices
\renewcommand{\theequation}{\Alph{chapter}.\arabic{section}.\arabic{equation}}

\appendix
% Import the appendice files
\chapter{Infinite-dimensional manifolds\label{app:frechet}}

This thesis deals extensively with infinite-dimensional manifolds, in particular with loop groups, loop vector spaces and limits of certain directed systems of manifolds---these are all examples of infinite-dimensional manifolds.
An infinite-dimensional manifold may be thought of as an ordinary smooth manifold that, rather than being modelled over $\RR^n$ or $\CC^n$, is modelled over a generalised sort of Euclidean space.
These spaces belong to a special family of locally convex topological vector space---usually the Fr\'{e}chet spaces, but more generally sequentially complete locally convex topological vector spaces.

This appendix gives a brief treatment of Fr\'{e}chet spaces and manifolds, as well as discussing direct limit manifolds.
The exposition here follows closely that of \cite{Glock,H,V}.

%%%%%%%%%%%%%%%%%%%%%%%%%%%%%%%%%%%%%%%%%%%%%%%%%%%%%%%%%%%
%%%%%%%%%%%%%%%%%%%%%%%%%%%%%%%%%%%%%%%%%%%%%%%%%%%%%%%%%%%
%%%%%%%%%%%%%%%%%%%%%%%%%%%%%%%%%%%%%%%%%%%%%%%%%%%%%%%%%%%

\section{Fr\'{e}chet spaces and Fr\'{e}chet manifolds}
Fr\'{e}chet spaces are special locally convex topological vector spaces that admit a natural notion of differentiability.
Importantly for the purposes of this thesis, this allows one to talk about maps between Fr\'{e}chet spaces being \emph{smooth}.
Recall
\vspace{11pt}
\begin{definition}
A \emph{seminorm} on an $\RR$- or $\CC$-linear vector space $V$ is an $\RR$-valued function $\|\!\cdot\!\| \colon V \to \RR$ such that
\begin{enumerate}
\item
$\| v \| \geq 0$ for all $v \in V$;

\item
$\| v+w \| \leq \| v \|+\| w \|$ for all $v, w \in V$; and

\item
$\| av \| = |a|\cdot \| v \| $ for all $v \in V$ and scalars $a$.
\end{enumerate}
\end{definition}
Notice that the only property distinguishing seminorms from norms is the fact that a seminorm need not be positive-definite; that is for a seminorm $\|\!\cdot\!\|$ on $V$ one may have some $v\neq 0$ in $V$ for which $\|v\| = 0$. 
\vspace{11pt}

\begin{definition}
A \emph{locally convex topological vector space}\footnote{an equivalent definition of local convexity is that the origin has a neighbourhood basis of absolutely convex absorbent sets; that is, convex sets $U$ such that $tx \in U$ for all $x \in U$ and $|t| \leq 1$ (\emph{absolutely convex}) with the additional property that for all $x \in V$ there is some $t$ such that $tx \in U$ (\emph{absorbent}).} is a vector space $V$ with topology defined by a family of seminorms $\{\|\!\cdot\|_i \mid i \in I\}$ for some index set $I$.
Thus a neighbourhood basis of $v \in V$ is given by all sets of the form
\[
U_{\varepsilon}^J(v) := \left\{  w\in V  \mid \|w-v\|_j <\varepsilon \text{ for all $j \in J$}  \right\}
\]
where $J \subset I$ is finite and $\varepsilon >0$.
\end{definition}

A locally convex topological vector space is Hausdorff if and only if $v = 0 \Leftrightarrow \|v\|_i = 0$ for all $i\in I$ and metrisable if and only if the topology may be defined by a countable family of seminorms.
\vspace{11pt}
\begin{definition}
\label{defn:frechetspace}
A \emph{Fr\'{e}chet space} is a sequentially complete Hausdorff metrisable locally convex topological vector space.
\end{definition}

Immediate examples of Fr\'{e}chet spaces are Banach spaces, where the family of seminorms is just the given norm.
Another example of key importance in this thesis is
\vspace{11pt}

\begin{example}
Let $X$ be a compact manifold and $E \to X$ a vector bundle.
The space $\Gamma(E)$ of smooth sections of $E$ then forms a Fr\'{e}chet space.
The seminorms inducing the topology in this case are given by first choosing Riemannian metrics and connections on the vector bundles $TX$ and $E$; for any $v \in \Gamma(E)$, write $D^j v$ for its $j$-th covariant derivative and set
\[
\| v\|_n := \sum_{j=0}^n \,\sup_{x\in X} \,|D^jv(x) |,
\]
where $|\!\cdot\!|$ is the norm induced by the Riemannian metric.
Note that the compactness of $X$ here is crucial.
\end{example}

As mentioned above, there is a sensible notion of maps between Fr\'{e}chet spaces being smooth, given by the \emph{G\^{a}teaux derivative}.
Suppose $V$ and $W$ are Fr\'{e}chet spaces, $U \subseteq V$ is open and $f\colon U \to W$ is a continuous map, then
\vspace{11pt}
\begin{definition}
\label{defn:gatderiv}
The \emph{(G\^{a}teaux) derivative} of $f$ at $v \in U$ in the direction of $h \in V$ is
\[
Df(v)\{h\} := \lim_{t\to 0} \frac{f(v+th) - f(v)}{t}
\]
and $f$ is \emph{differentiable} at $v$ in the direction of $h$ if this limit exists.
One says that $f$ is \emph{continuously differentiable} if the above limit exists for all $(v,h) \in U\x V$ and if 
\[
Df \colon U\x V \lo W
\]
is continuous.

Higher derivatives are defined inductively via
\[
D^{k+1}f(v)\{h_1,\dotsc,h_{k}\} := \lim_{t\to 0} \frac{D^kf(v+th_k)\{h_1,\dotsc,h_{k-1}\} - D^kf(v)\{h_1,\dotsc,h_{k-1}\}}{t}
\]
and one says that $f$ is of class $\cC^k$ if
\[
D^k f \colon U \x V^k \lo W
\]
exists and is continuous.
One says that $f$ is \emph{smooth}, or of class $\cC^\infty$, if $f$ is $\cC^k$ for all $k$.
\end{definition}

The usual definition of a manifold then generalises directly.
\vspace{11pt}
\begin{definition}
\label{defn:frecman}
A \emph{Fr\'{e}chet manifold} is a Hausdorff topological space with an atlas of coordinate charts valued in Fr\'{e}chet spaces such that the coordinate transition maps are all smooth maps .
\end{definition}

There are direct generalisations of many other familiar concepts from differential geometry to the Fr\'{e}chet setting, in particular one may talk about tangent bundles of Fr\'{e}chet manfolds,  Fr\'{e}chet principal bundles and Fr\'{e}chet vector bundles.
For a comprehensive review, see \cite{H}.

%%%%%%%%%%%%%%%%%%%%%%%%%%%%%%%%%%%%%%%%%%%%%%%%%%%%%%%%%%%
%%%%%%%%%%%%%%%%%%%%%%%%%%%%%%%%%%%%%%%%%%%%%%%%%%%%%%%%%%%
%%%%%%%%%%%%%%%%%%%%%%%%%%%%%%%%%%%%%%%%%%%%%%%%%%%%%%%%%%%

\section{Mapping manifolds}
There is one particular type of Fr\'{e}chet manifold that features prominently in this thesis, namely mapping manifolds.
The interested reader is referred to \cite[Section I.4]{H} for a detailed discussion.
\vspace{11pt}

\begin{example}[\cite{H}]
\label{ex:fibreman}
Let $X$ be a compact finite-dimensional manifold and suppose that $\pi \colon B \to X$ is a surjective submersion, with $B$ a finite-dimensional manifold.
Consider the space of sections $\Gamma(B)$ of $B$, that is all the smooth maps $f\colon X\to B$ such that $\pi\circ f  =\id$.
Associated to each such section $f$ there is a vector bundle $V_f B \to X$ called the \emph{vertical tangent bundle} to $B$ at $f$.
The fibre of $V_f B$ at $x \in X$ consists of those tangent vectors to $B$ at $f(x)$ that are annihilated by $d\pi$.

Then, provided it is not empty, the space of sections $\Gamma(B)$ is a Fr\'{e}chet manifold.
Using tubular neighbourhoods one can construct a diffeomorphism between a neighbourhood of the zero section of the vector bundle $V_fB$ and a neighbourhood of the image of $f$ in $B$.
Thus one obtains a bijective correspondence between sections near zero in $\Gamma(V_fB)$ and sections near $f$ in $\Gamma(B)$.
These maps provide coordinate charts for $\Gamma(B)$ since the transition maps are smooth.
Notice that in this example the coordinate charts are not all valued in the same Fr\'{e}chet space.
\end{example}

Taking a smooth manifold $Y$ and setting $B = X\x Y$ in Example \ref{ex:fibreman} gives the Fr\'{e}chet manifold
\[
\Map(X,Y) := \Gamma(X\x Y)
\]
since a section $X\to X\x Y$ is equivalent to a map $X \to Y$.
An important feature of such mapping manifolds is that if $X, Y$ and $Z$ are (finite-dimensional) manifolds with $X$ and $Y$ compact then precomposition by a smooth map $ f \colon X \to Y$ defines a smooth map
\[
f^\ast \colon \Map(Y,Z) \lo \Map(X,Z).
\]
Similarly,  postcomposition by a smooth map $g\colon Y \to Z$  defines a smooth map
\[
g_\ast \colon \Map(X,Y) \lo \Map(X,Z),
\]
where here $Y$ need not be compact.

Another important feature of mapping manifolds is that there is an easy description of their tangent bundles using kinematic tangent vectors.
A smooth path $f \colon (a,b) \to \Map(X,Y)$ is equivalent to a smooth map $f \colon (a,b) \x X \to Y$ and so for each $x \in X$ one obtains a path $t \mapsto f(t)(x)$ in $Y$.
The tangent vector $f'(t)(x)$ in the $t$ direction is then a tangent vector to $Y$ at $f(t)(x)$, so that $f'(t)$ is a section of the pullback $f^\ast TY$.
Thus, one has the canonical identification
\[
T_f \Map(X,Y) = \Gamma(f^\ast TY).
\]
This identification is especially useful, for instance when one wishes to discuss differential forms on $\Map(X,Y)$.

%%%%%%%%%%%%%%%%%%%%%%%%%%%%%%%%%%%%%%%%%%%%%%%%%%%%%%%%%%%
%%%%%%%%%%%%%%%%%%%%%%%%%%%%%%%%%%%%%%%%%%%%%%%%%%%%%%%%%%%
%%%%%%%%%%%%%%%%%%%%%%%%%%%%%%%%%%%%%%%%%%%%%%%%%%%%%%%%%%%

\section{The path fibration}
Following \cite{V}, recall that for a connected Lie group the path fibration has total space
\[
PG := \{ p\colon \RR \to G \mid \text{$p$ is smooth, $p(0) = 1$ and $p^{-1}\d p$ is periodic with period $2\pi$} \}
\]
and the projection
\[
\ev_{2\pi} \colon PG\lo G
\]
is given by evaluation at $2\pi$.
Notice that the total space $PG$ is isomorphic to the space of connections on $\sone\x G \to \sone$, since a connection $A$ on this trivial bundle uniquely determines a periodic path by solving $A = p^{-1}\d p$ subject to the initial condition $p(0)=1$ and, conversely, $p \in PG$ determines a connection via $p^{-1}\d p$.
Thus $PG$ is smoothly contractible as spaces of connections are affine.

To see that $PG\to G$ is in fact an $\Omega G$-bundle, first take any two $p, q\in \ev_{2\pi}^{-1}(\{g\})$ for some fixed $g \in G$.
Let $f(t) := (p^{-1}q)(t+2\pi)$, noticing that $f(0) = 1$ and also that $f$ satisfies the same differential equation as $p^{-1}q$.
By the Picard-Lindel\"{o}f Theorem, one must have $f = p^{-1}q$ so that $p^{-1}q$ is in fact $2\pi$-periodic.
Therefore any two $p, q$ in the same fibre of $\ev_{2\pi}$ are related by $p = q\gamma$ for some $\gamma \in \Omega G$.
From this it follows readily that the right action of $\Omega G$ on the fibres is free and transitive.

To establish the local triviality of $PG$, consider a normal neighbourhood $U \subset G$ of the identity (so that $\exp$ is a diffeomorphism $V \to U$ for some subset $V \subset \g$).
As in \cite{V}, one may define a map
\[
U\x\Omega G \lo \ev_{2\pi}^{-1}(U)
\]
by sending
\[
(g,\gamma) \longmapsto \exp(t\xi)\gamma(t),
\]
with $\xi$ satisfying $\exp(2\pi \xi) = g$.
The inverse of this map is given explicitly by
\[
p \longmapsto \left(p(2\pi), \exp(tp(2\pi)^{-1}) p \right).
\]
This gives a trivialisation over the neighbourhood $U$ of the identity.
To obtain a trivialisation over any $h \in G$, one considers the open sets $\{Uh\}_{h\in G}$.
As $G$ is connected, for any $h \in G$ one may choose a periodic path $\widetilde{h} \in PG$ with $\widetilde{h}(2\pi) = h$, then the map
\[
Uh\x\Omega G \lo \ev_{2\pi}^{-1}(Uh)
\]
given by sending
\[
(gh,\gamma) \longmapsto \widetilde{h}(t) \exp(t\xi)\gamma(t),
\]
with $\xi$ as above, is a local trivialisation.

This shows that $PG \to G$ is indeed an $\Omega G$-bundle and, since $PG$ is contractible, it is a model for the universal bundle.

%%%%%%%%%%%%%%%%%%%%%%%%%%%%%%%%%%%%%%%%%%%%%%%%%%%%%%%%%%%
%%%%%%%%%%%%%%%%%%%%%%%%%%%%%%%%%%%%%%%%%%%%%%%%%%%%%%%%%%%
%%%%%%%%%%%%%%%%%%%%%%%%%%%%%%%%%%%%%%%%%%%%%%%%%%%%%%%%%%%

\section{Direct limit manifolds}
This appendix is concluded by a brief discussion of direct limit manifolds and Lie groups following \cite{Glock}.
Recall that a \emph{directed system} in a category $\mathrm{C}$ is given by a directed set $(I,\leq)$ and a pair $S:= \left(\{X_i\}_{i\in I}, \{\phi_{ji}\}_{i\leq j} \right)$ where each $X_i$ is an object of $\mathrm{C}$ and each $\phi_{ji}$ is a morphism $X_i \to X_j$ in $\mathrm{C}$ such that
\begin{enumerate}
\item
$\phi_{ii} = \id_{X_i}$ for each $i\in I$; and

\item
$\phi_{kj} \circ \phi_{ji} = \phi_{ki}$ for all $i\leq j\leq k$.
\end{enumerate}
A \emph{cone over $S$} is a  pair $(X,\{\phi_i\}_{i\in I})$ where $X\in \mathrm{C}$ and the $\phi_i \colon X_i \to X$ are such that the diagram
\[
\xy
(0,25)*+{X_i}="1";
(50,25)*+{X_j}="2";
(25,0)*+{X}="3";
{\ar^{\phi_{ji}} "1";"2"};
{\ar^{\phi_i} "1";"3"};
{\ar_{\phi_j} "2";"3"};
\endxy
\]
commutes whenever $i\leq j$.

A cone $(X,\{\phi_i\}_{i\in I})$ over $S$ is a \emph{direct limit} of $S$ if it has the following universal property: for each cone $(Y,\{\psi_i\}_{i\in I})$ over $S$ there exists a unique morphism $u \colon X\to Y$ such that the diagram
\[
\xy
(0,25)*+{X_i}="1";
(50,25)*+{X_j}="2";
(25,0)*+{X}="3";
(25,-25)*+{Y}="4";
{\ar^{\phi_{ji}} "1";"2"};
{\ar^{\phi_i} "1";"3"};
{\ar_{\phi_j} "2";"3"};
{\ar_{\psi_i} "1";"4"};
{\ar^{\psi_j} "2";"4"};
{\ar_{u} "3";"4"};
\endxy
\]
commutes for all $i\leq j$.
It is important to remark that direct limits do not always exist in a given category $\mathrm{C}$, but when they do exist they are unique up to unique isomorphism.
The direct limit is written
\[
X:= \dlim X_i
\]
when the directed system $S$ is understood.

Suppose $T= \left(\{Y_i\}_{i\in I}, \{\psi_{ji}\}_{i\leq j} \right)$ is another directed system over the same index set $I$, with morphisms $f_i \colon X_i \to Y_i$ such that $f_j \circ \phi_{ji} = \psi_{ji} \circ f_i$ for all $i\leq j$.
The family of morphisms $\{f_i\}_{i\in I}$ is said to be \emph{compatible} with the directed systems $S$ and $T$. If $(Y,\{\psi_i\}_{i\in I})$ is a cone over $T$ it follows that $(Y,\{\psi_i\circ f_i\}_{i\in I})$ is a cone over $S$ and one writes $\dlim f_i$ for the unique morphism $X= \dlim X_i \to Y$ determined by the above universal property.
The directed systems $T$ and $S$ are \emph{equivalent} if the $f_i$  are isomorphisms for each $i$.
\vspace{11pt}
\begin{example}[Sets]
If $S:= \left(\{X_i\}_{i\in I}, \{\phi_{ji}\}_{i\leq j} \right)$ is a directed system of sets, then one takes
\[
\dlim X_i := \bigsqcup_{i\in I} X_i/\!\sim
\]
where if $x_i \in X_i$ and $x_j \in X_j$ then $x_i \sim x_j\Leftrightarrow \exists k \in I$ such that $\phi_{kj}(x_j) = \phi_{ki}(x_i)$.
The maps $\phi_i \colon X_i \to \dlim X_i$ are the obvious induced maps.
\end{example}
\vspace{11pt}
\begin{example}[Topological spaces]
If $S:= \left(\{X_i\}_{i\in I}, \{\phi_{ji}\}_{i\leq j} \right)$ is a directed system of topological spaces, then the direct limit of this system in the category of topological spaces is given by equipping the direct limit $\dlim X_i$ of underlying sets with the final topology with respect to the maps $\phi_i$.
Therefore a subset $U \subset \dlim X_i$ is open (resp.~closed) if and only if $\phi_i^{-1}(U) \subset X_i$ is open (resp.~closed) for each $i$.
\end{example}
\vspace{11pt}
\begin{example}[Topological groups]
If $S:= \left(\{G_i\}_{i\in I}, \{\phi_{ji}\}_{i\leq j} \right)$ is a directed system of topological groups, one may realise the set-theoretic direct limit $G:= \dlim G_i$ (with the corresponding maps $\phi_i \colon G_i \to G$) as a topological group also.

Equip $G$ with the final topology as above, writing $\mu_i \colon G_i \x G_i \to G_i$ for the multiplication maps and $\iota_i \colon G_i\to G_i$ for the inversion maps.
The family of maps $\{\mu_i\}_{i\in I}$ is compatible with the directed systems $T:=  \left(\{G_i \x G_i \}_{i\in I}, \{\phi_{ji} \x \phi_{ji}\}_{i\leq j} \right)$ and $S$.
As $G\x G$ is the direct limit of $T$ in the category of topological spaces, the multiplication on $G$ is given by the limit map $\mu:= \dlim \mu_i$.
The inversion operation on $G$ is given by $\iota:= \dlim \iota_i$.
Note that $G$ is not necessarily Hausdorff; some authors require this condition in the definition of a topological group.
\end{example}

The following result on direct limits of topological spaces is used extensively in this thesis%
\vspace{11pt}
\begin{proposition}[\cite{Glock}]
\label{prop:compactimage}
Suppose $X$ is the direct limit of the sequence $X_1 \subseteq X_2\subseteq \dotsb$ of topological spaces.
Then
\begin{itemize}
\item
if each $X_n$ is locally compact then $X$ is Hausdorff

\item
if each $X_n$ is $T_1$ then every compact subset of $X$ is contained in some $X_n$.
\end{itemize}
\end{proposition}

One is now in a position to discuss direct limits of directed systems of smooth manifolds.
Direct limits of manifolds are more subtle than, say, directs limits of sets or of topological spaces since it is not immediately clear how to obtain a smooth manifold structure on the na\"{i}ve set-theoretic direct limit.

For the following discussion, it is necessary to slightly  broaden the notion of infinite-dimensional manifolds as discussed thus far in this thesis.
If one relaxes the metrisability condition of Definition \ref{defn:frechetspace} one obtains a (Hausdorff) \emph{sequentially complete locally convex (s.c.l.c.)  topological vector space}.
One may then talk about derivatives in Hausdorff s.c.l.c. vector spaces as per Definition \ref{defn:gatderiv}, in particular there is a notion of \emph{smooth manifolds} modelled over Hausdorff s.c.l.c. vector spaces, as in Definition \ref{defn:frecman}.
\vspace{11pt}
\begin{example}
Denote by $\RR^\infty$ the real vector space of $\RR$-valued sequences with finite support.
The set $I$ of finite-dimensional subsets of $\RR^\infty$ forms a directed set under inclusion and so one obtains the directed system $(\{i\}_{i\in I}, \{\phi_{ji}\}_{i\leq j})$ of topological vector spaces, with $\phi_{ji} \colon i \to j$ the inclusion.
This directed system has direct limit $\RR^\infty$ with the \emph{finite topology}, which is the final topology with respect to the inclusion maps $i \hookrightarrow \RR^\infty$.

The finite topology on $\RR^\infty$ coincides with the finest vector space and finest locally convex vector space topologies.
Thus, since $\RR^\infty$ is sequentially complete, it is a Hausdorff s.c.l.c. topological vector space.
Notice that $\RR^\infty$ is not metrisable so it is not a Fr\'{e}chet space (see \cite{Glock,Hansen} for details).
\end{example}

\vspace{11pt}
\begin{theorem}[\cite{Glock}]
If $M_1 \subseteq M_2 \subseteq M_3 \subseteq \dotsb$ is a sequence of finite-dimensional paracompact smooth manifolds such that $M_i$ is a smooth closed submanifold of $M_{i+1}$ then there exists a unique smooth manifold structure on the direct limit $M := \dlim M_i$ that makes $M$ the direct limit of its submanifolds $M_n$ in the category of smooth manifolds.
If $M_i$ is modelled over $\RR^{n_i}$ then $M$ is modelled over $\dlim \RR^{n_i}$.
\end{theorem}

The proof of this fact presented in \cite{Glock} uses tubular neighbourhoods to extend charts on $M_i$ to charts on $M_{i+1}$.
The result of this procedure is a sequence $U_1 \subseteq U_2 \subseteq U_3\subseteq \dotsb$ of coordinate charts (with $U_i \subseteq M_i$) and a compatible family of maps $U_i \to \RR^{n_i}$.
Passing to the direct limit gives a coordinate chart for $M$ and the collection of all charts obtained in this fashion gives a smooth atlas for $M$ over $\dlim \RR^{n_i}$.

In the case where each $M_i = G_i$ is a Lie group and the submanifold embeddings $G_i \hookrightarrow G_{i+1}$ are Lie group homomorphisms one obtains
\vspace{11pt}
\begin{theorem}[\cite{Glock}]
If $G_1 \subseteq G_2 \subseteq G_3 \subseteq \dotsb$ is a sequence of finite-dimensional paracompact smooth Lie groups such that $G_i$ is a closed Lie subgroup of $G_{i+1}$ then there exists a unique smooth Lie group structure on the direct limit $G := \dlim G_i$ that makes $G$ the direct limit of its Lie subgroups $G_n$ in the category of smooth Lie groups.

Moreover, if $G_i$ has Lie algebra $\g_i$ then the Lie algebra of $G$ is $\g := \dlim \g_i$.
\end{theorem}

As a concluding remark, one notices that if $\Theta$ is the Maurer-Cartan form on $G$---defined by its action on $\chi \in T_g G$ as $\Theta_g(\chi) := dL_{g^{-1}} (\chi)$---and $\Theta_i$ is the Maurer-Cartan form on $G_i$ then $\Theta$ pulls back to $\Theta_i$ under the submanifold inclusion $G_i \hookrightarrow G$.

%%%%%%%%%%%%%%%%%%%%%%%%%%%%%%%%%%%%%%%%%%%%%%%%%%%%%%%%%%%
%%%%%%%%%%%%%%%%%%%%%%%%%%%%%%%%%%%%%%%%%%%%%%%%%%%%%%%%%%%
%%%%%%%%%%%%%%%%%%%%%%%%%%%%%%%%%%%%%%%%%%%%%%%%%%%%%%%%%%%

\newpage
\mbox{}
\chapter{Integration over the fibre\label{app:int}}

Integration over the fibre is an important tool that is used throughout this thesis.
Fibre integration may be thought of as a special type of pushforward map that arises naturally in the context of differential forms---in which case one is literally integrating a differential form over the fibres of some fibre bundle---and in singular and de Rham cohomology.

In this appendix, one considers integration over the fibre on trivial bundles of the form $M\x N\to M$, where $M$ and $N$ are oriented manifolds of dimensions $m$ and $n$ respectively that admit smooth partitions of unity.

%%%%%%%%%%%%%%%%%%%%%%%%%%%%%%%%%%%%%%%%%%%%%%%%%%%%%%%%%%%
%%%%%%%%%%%%%%%%%%%%%%%%%%%%%%%%%%%%%%%%%%%%%%%%%%%%%%%%%%%
%%%%%%%%%%%%%%%%%%%%%%%%%%%%%%%%%%%%%%%%%%%%%%%%%%%%%%%%%%%

\section{For differential forms}
The treatment presented here largely follows that of \cite{BT,Green}.
One begins by noticing that every $\omega \in \Omega^k(M\x N)$ is a sum of \emph{decomposable} forms, that is
\vspace{11pt}
\begin{lemma}
\label{lemma:decompose}
Any $\omega \in \Omega^k(M\x N)$ may be written as a sum of forms of the form
\[
 f \cdot \pr_M^\ast \omega_M \wedge \pr_N^\ast \omega_N,
\]
where $f \in \cC^\infty(M\x N)$ and $\omega_M$, $\omega_N$ are forms on $M$ and $N$ respectively  whose degrees add to $k$.
The maps $\pr_M$, $\pr_N$ are the projections of  $M\x N$ onto $M$ and $N$ respectively.
\end{lemma}
\begin{proof}
Compare with \cite[Lemma 2.4.1]{Green}.
 This is straightforward in the case when $M = \RR^m$ and $N = \RR^n$ since there are global coordinate charts.

For general $M$ and $N$, let $\{(U_\alpha,\phi_\alpha)\}$ and $\{(V_\beta,\psi_\beta)\}$ be atlases for $M$ and $N$ respectively.
Choose smooth partitions of unity $\{\rho_\alpha\}$ and $\{\sigma_\beta\}$ subordinate to $\{U_\alpha\}$ and $\{V_\beta\}$ respectively and choose a family of smooth bump functions $\{b_{\alpha\beta}\}$ such that $b_{\alpha\beta}$ is supported in $U_\alpha \x V_\beta$ and identically $1$ on the support of $\rho_\alpha \sigma_\beta$.

If $\phi_\alpha(x) = (x^1(x),x^2(x),\dotsc, x^m(x))$ and  $\psi_\beta(y) = (y^1(y),y^2(y),\dotsc, y^n(y))$, the restriction $\omega_{\alpha\beta}$ of $\omega$ to $U_\alpha\x V_\beta$ may be written as
\[
\omega_{\alpha \beta} =  \sum_{|I|+|J|=k} \omega_{\alpha\beta}^{IJ} =  \sum_{|I|+|J|=k} (\phi_\alpha\x\psi_\beta)^\ast \left( f^{IJ}_{\alpha\beta} \,dx^I \wedge dy^J\right)
\]
for some smooth functions $f^{IJ}_{\alpha\beta}$, summing over all multi-indices $I$ and $J$ of total length $k$.
Set
\[
f_{\vp}^{IJ}  := \sum_{\alpha,\beta} b_{\alpha\beta} (\phi_\alpha\x\psi_\beta)^\ast f^{IJ}_{\alpha\beta},\;\; \omega^I_M := \sum_\alpha \rho_\alpha \,d\phi_\alpha^I\;\;\mbox{ and }\;\; \omega^J_N := \sum_\beta \sigma_\beta \,d\psi_\beta^J.
\]
Then, for fixed multi-indices $I$ and $J$,
\begin{align*}
f_{\vp}^{IJ} \pr_M^\ast \omega_M^I  \wedge \pr_N^\ast\omega_N^J &=
%%%
\sum_{\alpha,\beta} b_{\alpha\beta}\left(  (\phi_\alpha\x\psi_\beta)^\ast f^{IJ}_{\alpha\beta} \right)\left(\pr_M^\ast \rho_\alpha \,d\phi_\alpha^I\right) \wedge \left( \pr_N^\ast \sigma_\beta \,d\psi_\beta^J\right)\\
%%%
&= \sum_{\alpha,\beta} \rho_\alpha\sigma_\beta  b_{\alpha\beta}\left(  (\phi_\alpha\x\psi_\beta)^\ast f^{IJ}_{\alpha\beta} \right)\left(\pr_M^\ast  \,d\phi_\alpha^I\right) \wedge \left( \pr_N^\ast \,d\psi_\beta^J\right)\\
%%%
&= \sum_{\alpha,\beta} (\rho_\alpha\sigma_\beta  b_{\alpha\beta}) \omega_{\alpha\beta}^{IJ}\\
%%%
&= \sum_{\alpha,\beta} \rho_\alpha\sigma_\beta  \omega_{\alpha\beta}^{IJ}.
\end{align*}
Summing over all multi-indices $I$ and $J$ with $|I| + |J| = k$ gives $\omega$ and, hence, the result.
\end{proof}

Assuming $N$ to be compact, given a form $\omega_N \in \Omega^k(N)$ and a smooth function $f$ on $M\x N$ write $\int_N f\omega_N$ for the function on $M$ that sends
\[
x \longmapsto \int_{N} f_x \,\omega_N
\]
where $f_x(y) := f(x,y)$, noticing that this is zero if $k \neq n$.
\vspace{11pt}
\begin{definition}
\label{defn:integrationoverfibre}
The \emph{integration over the fibre map} on differential forms is the map $\widehat{\int_N} \colon \Omega^k(M\x N) \to \Omega^{k-n}(M)$ that sends the decomposable form $ f \cdot \pr_M^\ast \omega_M \wedge \pr_N^\ast \omega_N$ to
\[
\widehat{\int_N}  f \cdot \pr_M^\ast \omega_M \wedge \pr_N^\ast \omega_N := \int_N f\omega_N\cdot\omega_M.
\]
Integration over the fibre is defined for general forms by invoking Lemma \ref{lemma:decompose} and extending by linearity, noticing that the result is independent of the chosen decomposition.
\end{definition}

As an immediate consequence of the definition
\vspace{11pt}
\begin{lemma}
\label{lemma:intpullsquare}
If $f \colon M \to M'$ is a smooth map, then the diagram
\[
\xy
(0,25)*+{\Omega^\bullet(M'\x N)}="1";
(50,25)*+{\Omega^{\bullet-n}(M')}="2";
(0,0)*+{\Omega^\bullet(M\x N)}="3";
(50,0)*+{\Omega^{\bullet-n}(M)}="4";
{\ar^{\widehat{\int_N}} "1";"2"};
{\ar^{f^\ast} "2";"4"};
{\ar^{(f\x\id)^\ast} "1";"3"};
{\ar^{\widehat{\int_N}} "3";"4"};
\endxy
\]
commutes.
\end{lemma}

When $M$ is also compact one also has the following
\vspace{11pt}
\begin{lemma}
\label{lemma:totalint}
For any $\omega\in\Omega^{m+n}(M\x N)$
\[
\int_{M\x N}\omega = \int_M \widehat{\int_N}\omega.
\]
\end{lemma}
This is proved using the local decompositions of Lemma \ref{lemma:decompose} and Fubini's Theorem on $\RR^m\x \RR^n$.

Importantly for constructing integration over the fibre maps in de Rham cohomology, in the case where $N$ is without boundary one also has
\vspace{11pt}
\begin{lemma}
\label{lemma:intwithd}
The integration over the fibre map commutes with the exterior derivative, that is
\[
\widehat{\int_N} d = d \,\widehat{\int_N}.
\]
\end{lemma}
\begin{proof}
As both $\widehat{\int_N}$ and $d$ are linear, it suffices to prove the result for decomposable forms.
Taking coordinate charts and partitions of unity exactly as in Lemma \ref{lemma:decompose}, let $\omega$ be a decomposable form on $M\x N$.
Then, for some fixed multi-indices $I$ and $J$,
\[
\omega = \sum_{\alpha,\beta} \rho_\alpha \sigma_\beta\, (\phi_\alpha\x\psi_\beta)^\ast \left( f_{\alpha\beta} \,dx^I \wedge dy^J\right).
\]
Writing $g_{\alpha\beta} = (\phi_\alpha\x\psi_\beta)^\ast f_{\alpha\beta}$ and taking the exterior derivative, one obtains
\begin{align*}
d \omega &= \sum_{\alpha,\beta}\left[ g_{\alpha\beta} \sigma_\beta d\rho_\alpha + g_{\alpha\beta} \rho_\alpha d \sigma_\beta + \rho_\alpha \sigma_\beta d g_{\alpha\beta}  \right] \wedge d\phi_\alpha^I \wedge d\psi_\beta^J\\
%%%
&= \sum_{\alpha,\beta}\bigg[ g_{\alpha\beta} \sigma_\beta d\rho_\alpha + g_{\alpha\beta} \rho_\alpha d \sigma_\beta + \rho_\alpha \sigma_\beta \sum_{i=1}^m  \frac{\d g_{\alpha\beta}}{\d \phi_\alpha^i} \, d\phi_\alpha^i + \rho_\alpha \sigma_\beta \sum_{j=1}^n  \frac{\d g_{\alpha\beta}}{\d \psi_\beta^j} \, d\psi_\beta^j  \bigg]\! \wedge d\phi_\alpha^I \wedge d\psi_\beta^J.
\end{align*}
If $|J|=n$ integrating over the fibre gives
\begin{align*}
\widehat{\int_N} d \omega &= \sum_{\alpha,\beta} \bigg[\int_N\sigma_\beta  g_{\alpha\beta}\, d\psi_\beta^J \cdot d\rho_\alpha   +  \rho_\alpha \sum_{i=1}^m \int_N \sigma_\beta \, \frac{\d g_{\alpha\beta}}{\d \phi_\alpha^i} \,d \psi_\beta^J\cdot  d \phi_\alpha^i \bigg]\! \wedge d\phi_\alpha^I \\
%%%
&= \sum_{\alpha,\beta} \left[\int_N\sigma_\beta  g_{\alpha\beta}\, d\psi_\beta^J \cdot d\rho_\alpha  + \rho_\alpha\cdot d \bigg(\! \int_N \sigma_\beta  g_{\alpha\beta} \,d \psi_\beta^J\bigg) \!\right]\!\wedge d\phi_\alpha^I \\
%%%
&= \sum_{\alpha,\beta} d \!\left[\rho_\alpha  \int_N\sigma_\beta  g_{\alpha\beta}\, d\psi_\beta^J \right]\!\wedge d\phi_\alpha^I \\
%%%
&= d \,\widehat{\int_N} \omega
\end{align*}
as required.
If $|J| < n$ then
\begin{align*}
\widehat{\int_N} d \omega &= \sum_{\alpha,\beta} (-1)^{|I|} \rho_\alpha \int_N \bigg[ g_{\alpha\beta}\,d \sigma_\beta \wedge d\psi_\beta^J + \sigma_\beta \sum_{j=1}^n \frac{\d g_{\alpha\beta}}{\d \psi_\beta^j} \,d \psi_\beta^j \wedge  d \psi_\beta^J\bigg]\! \cdot d\phi_\alpha^I \\
%%%
 &= \sum_{\alpha,\beta} (-1)^{|I|} \rho_\alpha \int_N d( \sigma_\beta g_{\alpha\beta} \wedge d\psi_\beta^J) \cdot d\phi_\alpha^I\\
%%%
&=0
\end{align*}
using Stokes' Theorem, since $N$ is without boundary.
On the other hand,
\[
\widehat{\int_N} \omega = 0\Longrightarrow  d \,\widehat{\int_N} \omega = 0.
\]
Since the result is clearly true for $|J| > n$ (as this implies $\omega = 0$), the result is proved.
\end{proof}

This implies that there is a well-defined integration over the fibre map in de Rham cohomology, namely
\[
\widehat{\int_N} \colon H^k_\deR (M\x N) \lo H^{k-n}_\deR (M)
\]
given by
\[
\widehat{\int_N}  \colon [\omega] \longmapsto \bigg[ \widehat{\int_N} \omega \bigg].
\]
By Lemma \ref{lemma:intpullsquare}, the diagram
\[
\xy
(0,25)*+{H_\deR^\bullet(M'\x N)}="1";
(50,25)*+{H_\deR^{\bullet-n}(M')}="2";
(0,0)*+{H_\deR^\bullet(M\x N)}="3";
(50,0)*+{H_\deR^{\bullet-n}(M)}="4";
{\ar^{\widehat{\int_N}} "1";"2"};
{\ar^{f^\ast} "2";"4"};
{\ar^{(f\x\id)^\ast} "1";"3"};
{\ar^{\widehat{\int_N}} "3";"4"};
\endxy
\]
commutes for any smooth map $f \colon M \to M'$.

%%%%%%%%%%%%%%%%%%%%%%%%%%%%%%%%%%%%%%%%%%%%%%%%%%%%%%%%%%%
%%%%%%%%%%%%%%%%%%%%%%%%%%%%%%%%%%%%%%%%%%%%%%%%%%%%%%%%%%%
%%%%%%%%%%%%%%%%%%%%%%%%%%%%%%%%%%%%%%%%%%%%%%%%%%%%%%%%%%%

\section{For singular cohomology}
One can also define integration over the fibre for singular cohomology with coefficients in any ring $R$.
The construction, which is necessarily algebraic, relies on a choice of cross product for singular chains.
The underlying principle is similar to that of the de Rham isomorphism theorem, in which integration of differential forms is viewed as an algebraic operation.

Let $\triangle^n$ denote the standard $n$-simplex in $\RR^{n+1}$.
Following \cite[pp.~277--278]{Hatcher}, one now outlines how to define a product on singular chains
\[
\x \colon C_m(M;R) \otimes C_n(N;R) \to C_{m+n}(M\x N;R).
\]
The crucial ingredient is that there is a natural decomposition of the product $\triangle^m \x \triangle^n$ into $(m+n)$-simplices given as follows.
Label the vertices of $\triangle^m$ by $v_0,\dotsc, v_m$ and the vertices of $\triangle^n$ by $w_0,\dotsc,w_n$.
Thinking of pairs $(i,j)$ with $0\leq i\leq m$ and $0\leq j\leq n$ as points on an $m \x n$ grid in $\RR^2$, let $\gamma$ be any path from $(0,0)$ to $(m,n)$ formed by a sequence of upward or rightward edges in this grid.
To any such path $\gamma$ one associates the linear map
\[
\ell_\gamma \colon \triangle^{m+n} \lo \triangle^m \x \triangle^n
\]
given by sending the $k$-th vertex of $\triangle^{m+n}$ to $(v_{i_k},w_{j_k})$, where $(i_k,j_k)$ is the $k$-th vertex of the path $\gamma$.
Write $|\gamma|$ for the number of squares of the grid lying below the path $\gamma$.

If $a \colon \triangle^m \to M$ and $b \colon \triangle^n \to N$ are singular simplices in $M$ and $N$ respectively, define
\begin{equation}
\label{eqn:intinsingcoh}
a\x b := \sum_{\gamma} (-1)^{|\gamma|}(a\x b) \circ \ell_\gamma,
\end{equation}
which is a singular $(m+n)$-chain in $M\x N$ and the sum ranges over all `edgepaths' $\gamma$.
This defines the simplicial product on homogeneous elements of $C_m(M;R)\otimes C_n(N;R)$, which is extended to all elements by linearity.
Notice that the product $\x$ means different things on both sides of the equation \eqref{eqn:intinsingcoh}: on the left hand side it denotes the simplicial product and on the right hand side it denotes the product map $a \x b \colon \triangle^m\x\triangle^n \to M\x N$.

From the definition of the singular chain product, it is easy to verify that if $f \colon M \to M'$ and $g \colon N \to N'$ are continuous maps and $a \in C_m(M)$ and $b\in C_n(N)$ then
\[
(f\x g)_\ast(a\x b) = f_\ast a \x g_\ast b.
\]
Moreover one can show
\[
\d(a\x b) = \d a \x b +(-1)^m a\x \d b
\]
so that the product of two cycles is a cycle and the product of a cycle with a boundary is a boundary.
Thus the simplicial product induces a product on singular homology.
\vspace{11pt}
\begin{definition}
The \emph{slant product} of $u \in C^{m+n}(M\x N;R)$ and $z \in C_n(N)$ is the singular cochain $u/z \in C^m(M;R)$ given by
\[
u/z\,(c) := u(c \x z)
\]
for $c \in C_m(M)$.
\end{definition}

The interaction of the slant product with the coboundary operator is given by
\vspace{11pt}
\begin{lemma}
\label{lemma:slantbound}
For any $u \in C^{m+n}(M\x N;R)$ and $z \in C_n(N)$
\[
\delta(h/z) = \delta h / z +(-1)^{m} h / \d z.
\]
\end{lemma}
\begin{proof}
For any $c \in C_{m+1}(M)$
\begin{align*}
\delta(h/z) (c) :\!\!&=h(\d c\x z)\\
%%%
&= h(\d c\x z) + ((-1)^m + (-1)^{m+1}) h(c \x \d z)\\
%%%
&= h\left(\d (c\x z)\right)  + (-1)^{m}h(c \x \d z)\\
%%%
&= \delta h / z + (-1)^m h/ \d z
\end{align*}
as required.
\end{proof}

Thus, the slant product on (co)chains induces the \emph{slant product in (co)homology}
\[
/ \colon H^{m+n}(M\x N;R) \otimes H_n(N) \lo H^m(M;R). 
\]
One is now able to define integration over the fibre for singular cohomology.
Take $N$ to be without boundary, with $z \in H_n(N)$ the fundamental class of $N$.
\vspace{11pt}
\begin{definition}
The \emph{integration over the fibre map} for singular cohomology is the map $\widehat{\int_N} \colon H^k(M\x N;R) \to H^{k-n}(M;R)$ given by
\[
\widehat{\int_N} u := u / z.
\]
\end{definition}
Following the above discussion, one notices that for any continuous map $f \colon M\to M'$ the integration over the fibre map for singular cohomology fits into the commuting diagram
\[
\xy
(0,25)*+{H^\bullet(M'\x N;R)}="1";
(50,25)*+{H^{\bullet-n}(M';R)}="2";
(0,0)*+{H^\bullet(M\x N;R)}="3";
(50,0)*+{H^{\bullet-n}(M;R)}="4";
{\ar^{\widehat{\int_N}} "1";"2"};
{\ar^{f^\ast} "2";"4"};
{\ar^{(f\x\id)^\ast} "1";"3"};
{\ar^{\widehat{\int_N}} "3";"4"};
\endxy
\]

It is interesting to examine how differential forms behave under the slant product when viewed as singular cochains via the integration pairing:
\vspace{11pt}
\begin{proposition}
\label{prop:slantform}
If the cycle $z \in Z_n(N)$ represents the fundamental class of $N$ then for any $\omega \in \Omega^{m+n}(M\x N)$ one has
\[
\omega / z = \widehat{\int_N}\omega.
\]
\end{proposition}
\begin{proof}
Compare with \cite[Lemma 2.4.9]{Green}.
From the outset, one may suppose that $\omega$ is a decomposable form, so that $\omega = f \cdot \omega_M \wedge \omega_N$ (suppressing the pullback maps).
Write $z = \sum_i z_i \sigma_i$ for some singular $n$-simplices $\sigma_i \colon \triangle^n \to N$.
For any singular $m$-simplex $c \colon \triangle^m \to M$
\begin{align*}
\omega / z\,(c) :\!\!&= \sum_i z_i\, \omega(c\x \sigma_i)\\
%%%
&= \sum_i z_i\, (c\x \sigma_i)^\ast \omega(\id_m \x \id_n) \\
%%%
&= \sum_i z_i\, \int_{\id_m\x \id_n} (c\x \sigma_i)^\ast \omega 
\end{align*}
where $\id_m$, $\id_n$ denote the identity maps on $\triangle^m$, $\triangle^n$ respectively so that $\id_m \x \id_n$ is a singular $(m+n)$-chain in $\triangle^m \x \triangle^n$.
Even though $\id_m \x \id_n$ is not the identity on $\triangle^m \x \triangle^n$, by construction of the simplicial product it is a decomposition of $\triangle^m \x \triangle^n$ into $(m+n)$-simplices whose orientations cancel on internal faces.
Thus
\begin{align*}
\sum_i z_i\, \int_{\id_m\x \id_n} (c\x \sigma_i)^\ast \omega &= \sum_i z_i\, \int_{\triangle^m \x \triangle^n} (c\x \sigma_i)^\ast \omega\\
%%%
&= \sum_i z_i\, \int_{\triangle^m \x \triangle^n} (c\x \id_n)^\ast (\id_m\x \sigma_i)^\ast\omega \\
%%%
&= \sum_i z_i\, \int_{\triangle^m} \widehat{\int_{\triangle^n}} (c\x \id_n)^\ast (\id_m\x \sigma_i)^\ast\omega\\
%%%
&= \sum_i z_i\, \int_{\triangle^m}  c^\ast \widehat{\int_{\triangle^n}} (\id_m\x \sigma_i)^\ast\omega,
\end{align*}
where the third line invokes Lemma \ref{lemma:totalint} and the fourth line invokes Lemma \ref{lemma:intpullsquare}.
Notice that for any $x \in \triangle^m$
\[
(\id_m\x \sigma_i)^\ast f_x = \sigma_i^\ast f_x 
\]
so that, recalling the definition of the integration over the fibre map for differential forms, one has
\begin{align*}
\sum_i z_i\, \int_{\triangle^m}  c^\ast \widehat{\int_{\triangle^n}} (\id_m\x \sigma_i)^\ast\omega &= \sum_i z_i\, \int_{c} \left( \int_{\triangle^n} \sigma_i^\ast(f \omega_N) \cdot \omega_M\right)\\
%%%
&=\int_{c} \left( \sum_i z_i \int_{\sigma_i}f \omega_N  \right) \omega_M\\
%%%
&=\int_{c} \left( \int_{z}f \omega_N  \right) \omega_M\\
%%%
&=\int_{c} \left( \int_{N}f \omega_N  \right) \omega_M\\
%%%
&=\widehat{\int_N}\omega\, (c),
\end{align*}
where the fourth line follows since $z$ represents the fundamental class of $N$.
Since everything in sight is linear, one has
\[
\omega / z\,(c) = \widehat{\int_N}\omega\,(c)
\]
for any $c\in C_m(M)$ and, hence, the result.
\end{proof}

\newpage
\mbox{}
\chapter{Differential extensions\label{app:diff}}

This appendix provides the background material on differential extensions required for the discussion of differential $K$-theory in Chapter \ref{ch:five}.
The basic idea is that a differential extension of a cohomology theory combines cohomological information with differential form data in a non-trivial fashion.

Differential extensions of generalised cohomology theories are described by a relatively small set of axioms, originally proposed in \cite{BS3} and outlined below.
Following the axioms, some useful results are recorded.
For a comprehensive review of differential extensions, particularly differential $K$-theory, see \cite{BS1}.

%%%%%%%%%%%%%%%%%%%%%%%%%%%%%%%%%%%%%%%%%%%%%%%%%%%%%%%%%%%
%%%%%%%%%%%%%%%%%%%%%%%%%%%%%%%%%%%%%%%%%%%%%%%%%%%%%%%%%%%
%%%%%%%%%%%%%%%%%%%%%%%%%%%%%%%%%%%%%%%%%%%%%%%%%%%%%%%%%%%

\section{Axioms for differential extensions}
When dealing with differential extensions, the starting point is the generalised (Eilenberg-Steenrod) cohomology theory that one wishes to extend.
Suppose that $E$ is a generalised cohomology theory equipped with a natural transformation
\[
ch \colon E(X)\lo H(X; V)
\]
called the \emph{Chern character} of $E$, where $V$ is some graded $\RR$-vector space and $H(X;V)$ is singular cohomology with coefficients in $V$.
\vspace{11pt}

\begin{definition}[\cite{BS1}]
\label{defn:diffext}
A \emph{differential extension} of the pair $(E,ch)$ is a functor
\[
M \longmapsto \check{E}(M)
\]
from the category of smooth compact manifolds (with corners) to the category of $\ZZ$-graded groups together with the natural transformations
\begin{enumerate}
\item $R \colon \check E(M) \lo \Omega_{d=0}(M;V) := \Omega_{d=0}(M) \otimes_{\RR} V $  (the \emph{curvature});
\item $cl \colon \check E(M) \lo E(M)$ (the \emph{underlying class}); and
\item $a \colon \Omega(M;V) / \im \,d \lo \check E(M)$ (the \emph{action of forms}),
\end{enumerate}
with $d$ the usual exterior derivative.
The natural transformations $R, cl$ and $a$ are required to satisfy
\begin{enumerate}
\item
the diagram
\[
\xy
(0,25)*+{\check E(M)}="1";
(50,25)*+{E(M)}="2";
(0,0)*+{\Omega_{d=0}(M;V)}="3";
(50,0)*+{H(M;V)}="4";
{\ar^{cl} "1";"2"};
{\ar^{ch} "2";"4"};
{\ar^{R} "1";"3"};
{\ar^{\deR} "3";"4"};
\endxy
\]
commutes, with $\deR$ the map induced by the de Rham isomorphism;
\item
$R\circ a = d$;
\item
$a$ is of degree $1$; and
\item
the sequence
\begin{equation}
\label{eqn:diffextseq}
E^{\bullet-1}(M) \xrightarrow{\;\;ch\;\;} \Omega^{\bullet-1}(M;V) /\im\,d  \xrightarrow{\;\;a\;\;} \check E^\bullet(M)  \xrightarrow{\;\;cl\;\;} E^\bullet(M) \lo 0
\end{equation}
is exact.
\end{enumerate}
\end{definition}
\vspace{11pt}
\begin{example}
As a first example of a differential extension, observe that any character functor (Definition \ref{defn:charfun}) defines a differential extension of integral singular cohomology.
The Chern character in this case is the map $H\ZZ \to H\RR$ induced by the coefficient inclusion $\ZZ \hookrightarrow \RR$.
\end{example}

In the case that the underlying cohomology theory $E$ is multiplicative (i.e.~is valued in $\ZZ$-graded rings) and the Chern character is a homomorphism of graded rings, one can also require that a differential extension also respects the ring structure.
\vspace{11pt}
\begin{definition}[\cite{BS1}]
A differential extension $\check E$ of a multiplicative cohomology theory $(E,ch)$ is \emph{multiplicative} if
\begin{enumerate}
\item
$\check E$ is a functor to $\ZZ$-graded rings;
\item
$R$ and $cl$ are multiplicative; and
\item
$a(\omega) \cdot x = a(\omega \wedge R(x))$ for all $x \in \check E(M)$ and $\omega \in \Omega(M;V)/\im\,d$, with $\cdot$ the product on $\check E(M)$.
\end{enumerate}
\end{definition}
\vspace{11pt}
\begin{example}
The differential characters of Cheeger-Simons (Definition \ref{defn:differentialchar}) give a multiplicative differential extension of ordinary integer-valued singular cohomology.
\end{example}

There is a special pushforward map for differential extensions that plays a central role in obtaining the Bunke-Schick uniqueness results.
Observe that any generalised cohomology theory $E$ has an \emph{$\sone$-integration map}
\[
\widehat{\int_\sone} \colon E^\bullet (M\x \sone) \lo E^{\bullet-1}(M),
\]
that plays the role of the suspension map (see \cite[Section 1]{BS2} or p.~\pageref{page:int}).
Passing to differential extensions, the corresponding extra structure is
\vspace{11pt}
\begin{definition}
\label{defn:intextension}
A differential extension $\check E$ of $(E,ch)$ has \emph{$\sone$-integration} if there is a natural transformation
\[
\widehat{\int_\sone} \colon \check E^\bullet (M\x \sone) \lo \check E^{\bullet-1}(M)
\]
compatible with the natural transformations $R$ and $cl$ and the $\sone$-integration maps on differential forms and on $E$.
One also requires
\begin{enumerate}
\item
$\widehat{\int_\sone} \pr_M^\ast x = 0$ for each $x \in \check E^\bullet(M)$; and

\item
$\widehat{\int_\sone}(\id_M \x t)^\ast x = - \widehat{\int_\sone} x$ for all $x \in \check E^\bullet(M\x\sone)$, with $t \colon \sone\to\sone$ the (orientation-reversing) map given by complex conjugation. 
\end{enumerate}
\end{definition}

It turns out that under certain mild conditions on the underlying cohomology theory $E$, any multiplicative differential extension of $E$ has a canonical choice of $\sone$-integration \cite{BS2,BS1}.

This section is concluded by recording an important consequence of the axioms.
Since differential extensions contain differential form data they are not homotopy invariant in general; the homotopy formula expresses this failure to be homotopy invariant in terms of differential form data.
\vspace{11pt}
\begin{theorem}[Homotopy Formula]
\label{theorem:homotopyformula}
Let $\check E$ be a differential extension of $(E,ch)$.
If $x \in \check E (M\x\I)$ and $\varsigma_t \colon M \to M\x\I$ is the slice map $\varsigma_t (x) := (x,t)$ then
\[
\varsigma_1^\ast x- \varsigma_0^\ast x = a \bigg(\widehat{\int_\I} R(x) \bigg)
\]
where $\widehat{\int_\I} $ denotes integration over the fibre with respect to the canonical orientation on $\I$.
\end{theorem}
For a proof, see  \cite[Theorem 2.6]{BS1}.
\vspace{11pt}
\begin{remark}
If one suppresses the relevant pullback maps, observe that any form $\omega$ on $M\x\I$ is of the form $\omega = f \cdot \alpha\wedge dt + g \cdot \beta$ for some forms $\alpha,\beta$ on $M$ and smooth functions $f, g$ on $M\x\I$.
By Definition \ref{defn:integrationoverfibre},
\[
\widehat{\int_\I} \omega : = \int_\I f\,dt\cdot \alpha
\]
whereas if $\d_t$ is the canonical vector field on $\I$ and $\varsigma_t \colon m\mapsto (m,t)$ is the slice map
\[
\int_0^1 \varsigma_t^\ast \imath_{\d_t} \omega\,dt= \int_0^1 \varsigma_t^\ast (f\alpha)\,dt = \int_0^1 f\,dt \cdot \alpha = \widehat{\int_\I} \omega.
\]
The Homotopy Formula may then be restated as
\[
\varsigma_1^\ast x- \varsigma_0^\ast x = a \bigg(\int_0^1 \varsigma_t^\ast \imath_{\d_t} R(x) \,dt\bigg),
\]
which is the form used in Section \ref{S:structcal}.
\end{remark}

%%%%%%%%%%%%%%%%%%%%%%%%%%%%%%%%%%%%%%%%%%%%%%%%%%%%%%%%%%%
%%%%%%%%%%%%%%%%%%%%%%%%%%%%%%%%%%%%%%%%%%%%%%%%%%%%%%%%%%%
%%%%%%%%%%%%%%%%%%%%%%%%%%%%%%%%%%%%%%%%%%%%%%%%%%%%%%%%%%%

\section{Uniqueness of differential extensions}
The following result on uniqueness of differential extensions is due to Bunke and Schick.
It was proved originally in \cite{BS2} and appears also in \cite{BS1}.
\vspace{11pt}
\begin{theorem}[\cite{BS2}]
\label{theorem:bunkeschick}
Suppose $E$ is a multiplicative generalised cohomology theory that is rationally even, i.e. $E^{2k-1}(pt) \otimes \QQ = 0$ for all $k$, and that $E^k(pt)$ is a finitely-generated abelian group for all $k \in \ZZ$. 
Then if $\check E'$ and $\check E$ are any two differential extensions of $E$ with $\sone$-integration, there is a unique natural isomorphism $\check E' \to \check E$ compatible with all the structure (including the $\sone$-integration).

If $\check E'$ and $\check E$ do not have $\sone$-integration, there is still a natural isomorphism of the even-degree parts.

If $\check E'$ and $\check E$ are both multiplicative, the natural isomorphism is also multiplicative.
\end{theorem}

It is important to notice that the $\sone$-integration plays a crucial role in this result.
In \cite{BS2}, it is shown that without such a map there is an infinite family of differential extensions of the odd-degree part of $K$-theory that have non-isomorphic group structures.

Observe also that the hypotheses of Theorem \ref{theorem:bunkeschick} are satisfied by topological $K$-theory on compact manifolds (with corners), so as a special case one obtains
\vspace{11pt}
\begin{theorem}[\cite{BS2}]
\label{theorem:bunkeschickk}
Suppose $\check K'$ and $\check K$ are differential extensions of complex $K$-theory.
Then there is a unique natural isomorphism $\check K'^{even}\to \check K^{even}$
compatible with all the structure.

If both extensions are multiplicative this isomorphism extends uniquely to a multiplicative isomorphism $\check K' \to \check K$ compatible with all the structure. 

If both extensions are equipped with $\sone$-integration this isomorphism extends uniquely to an isomorphism $\check K' \to \check K$ compatible with all the structure, including the $\sone$-integration.
\end{theorem}

Since any two differential extensions of $K$-theory with $\sone$-integration are uniquely isomorphic, any such differential extension is called (a model for) \emph{differential $K$-theory}.
Examples include the Freed-Lott model \cite{FL} that uses complex Hermitian vector bundles equipped with extra form data, the cocycle model of Hopkins and Singer \cite{HS}, and a model due to Bunke and Schick that uses geometric families of Dirac operators \cite{BS3}.
The Simons-Sullivan model reviewed at the beginning of Chapter \ref{ch:five} (see also \cite{SSvec}) gives a differential extension of even $K$-theory and hence a model for even differential $K$-theory.
For a detailed survey of existing models of differential $K$-theory see \cite{BS1}.

% Add the bibliography to the table of contents
\addcontentsline{toc}{chapter}{Bibliography}

% What style to use and what file to use for the bibliography
%\bibliographystyle{plain}
%\bibliography
%DELETEHERE

\end{document}